\documentclass[10pt,a4paper]{article}
\usepackage[left=2.54cm,right=2.54cm,top=2.54cm,
bottom=2.54cm]{geometry}
\usepackage{mathrsfs}
\usepackage[english]{babel}
\usepackage{graphicx}
\usepackage{epstopdf}
\usepackage{epsfig,psfrag}
\usepackage{amsmath, amsfonts, amsthm,amssymb}
\usepackage{calligra}
\DeclareMathAlphabet{\mathcalligra}{T1}{calligra}{m}{n}
\usepackage{hyperref}
\usepackage{aurical}
\usepackage{mathrsfs,mathabx}
\usepackage[T1]{fontenc}
\usepackage{yfonts}
\usepackage{color}
\usepackage{pifont}
\newtheorem{lemma}{Lemma}[section]
\newtheorem{theorem}[lemma]{Theorem}
\newtheorem{proposition}[lemma]{Proposition}
\newtheorem{definition}[lemma]{Definition}
\newtheorem{corollary}[lemma]{Corollary}
\newtheorem{remark}[lemma]{Remark}
\newtheorem{example}[lemma]{Example}
\numberwithin{equation}{section}

\newcommand{\g}{\gamma}

\newcommand{\new}{{\small{\mathtt{ New}}}}
\newcommand{\yy}{{\mathcal Y}}%{{ \mbox{\Fontlukas Y}}}
\newcommand{\s}{{\sigma}}
\newcommand{\id}{{\rm Id}}
\newcommand{\jj}{{\vec \jmath}}
\newcommand{\ii}{{\vec \imath}}
\newcommand{\li}{{\mathfrak h}}
\newcommand{\ad}{{\rm ad}}
\newcommand{\ccC}{{\mathscr{C}}}

\newcommand{\wt}[1]{{\widetilde{#1}}}

\newcommand{\wtOmega}{{\widetilde{\Omega}}}

\newcommand{\wtcD}{{\widetilde{\mathcal D}}}

\newcommand{\wtcH}{{\widetilde{\mathcal H}}}

\newcommand{\wtcP}{{\widetilde{\mathcal P}}}

\newcommand{\wtcM}{\widetilde{\mathcal M}}

\newcommand{\whP}{{\widehat P}}

\newcommand{\whcD}{{\widehat{\mathcal D}}}
\newcommand{\whcH}{{\widehat{\mathcal H}}}

\newcommand{\whcT}{{\widehat{\mathcal T}}}

\newcommand{\whcZ}{{\widehat{\mathcal Z}}}

\newcommand{\C}{{\mathbb C}}

\newcommand{\N}{{\mathbb N}}
\newcommand{\Q}{{\mathbb Q}}
\newcommand{\R}{{\mathbb R}}

\newcommand{\T}{{\mathbb T}}
\newcommand{\Z}{{\mathbb Z}}

\newcommand{\cA}{{\mathcal A}}

\newcommand{\cC}{{\mathcal C}}
\newcommand{\cD}{{\mathcal D}}

\newcommand{\cF}{{\mathcal F}}

\newcommand{\cH}{{\mathcal H}}
\newcommand{\cI}{{\mathcal I}}
\newcommand{\cJ}{{\mathcal J}}
\newcommand{\cK}{{\mathcal K}}
\newcommand{\cL}{{\mathcal L}}
\newcommand{\cM}{{\mathcal M}}
\newcommand{\cN}{{\mathcal N}}
\newcommand{\cO}{{\mathcal O}}
\newcommand{\cP}{{\mathcal P}}
\newcommand{\cQ}{{\mathcal Q}}
\newcommand{\cR}{{\mathcal R}}
\newcommand{\cS}{{\mathcal S}}
\newcommand{\cT}{{\mathcal T}}
\newcommand{\cU}{{\mathcal U}}
\newcommand{\cV}{{\mathcal V}}

\newcommand{\cZ}{{\mathcal Z}}

\newcommand{\fA}{{\mathfrak{A}}}

\newcommand{\fh}{{\mathfrak{h}}}

\newcommand{\fR}{{\text{\gothfamily R}}}

\newcommand{\td}{{\mathtt{d}}}

\newcommand{\ti}{{\mathtt{i}}}

\newcommand{\tk}{{\mathtt{d}}}

\newcommand{\tm}{{\mathtt{m}}}

\newcommand{\tC}{{\mathtt{C}}}

\newcommand{\tF}{{\mathtt{F}}}
\newcommand{\tG}{{\mathtt{G}}}

\newcommand{\tK}{{\mathtt{K}}}

\newcommand{\tL}{{\mathtt{L}}}
\newcommand{\tM}{{\mathtt{M}}}
\newcommand{\tN}{{\mathtt{N}}}

\newcommand{\Tan}{{\cS_0}}

\newcommand{\ba}{{\bf a}}

\newcommand{\be}{{\bf e}}
\newcommand{\bi}{{\bf i}}
\newcommand{\bj}{{\bf j}}

\newcommand{\bk}{{\bf k}}
\newcommand{\bm}{{\bf m}}
\newcommand{\bn}{{\bf n}}

\newcommand{\bz}{{\bf z}}
\newcommand{\bN}{{\bf N}}
\newcommand{\bM}{{\bf M}}
\newcommand{\bs}{{\boldsymbol\sigma}}

\newcommand{\oo}{{\omega}}
\newcommand{\al}{{\alpha}}
\newcommand{\bt}{{\beta}}

\newcommand{\derx}{\partial_x}

\newcommand{\norm}[1]{\| #1 \|}
\newcommand{\abs}[1]{\left| #1 \right|}
\newcommand{\0}{{(0)}}
\newcommand{\2}{{(2)}}
\newcommand{\1}{{(1)}}

\newcommand{\la}{\left\langle}
\newcommand{\ra}{\right\rangle}

\newcommand{\im}{{\rm i}}

\newcommand{\re}[1] {\mbox{\Fontlukas T}\!_{#1}}  % {{\mathfrak T}\!_{#1}}
\newcommand{\und}[1]{\underline{#1}}
\newcommand{\di}{{\rm d}}
\newcommand{\e}{{\varepsilon}}
\newcommand{\hh}{ h }

\newcommand{\pix}{{\pi_x}}
\newcommand{\piy}{{\pi_y}}

\newcommand{\uno}{{\mathbb I}}

\newcommand{\meas}{{\rm meas}}
\newcommand{\mix}{{\rm mix}}
\newcommand{\hor}{{\rm hor}}
\newcommand{\out}{{\rm out}}
\newcommand{\diag}{{\rm diag}}
\renewcommand{\line}{{\rm line}}
\newcommand{\fin}{{\rm fin}}

\newcommand{\bd}{{\vec{\bf 2}}}

\newcommand{\kam}{{\eta}}
\newcommand{\bnorm}[1]{{|\mkern-6mu |\mkern-6mu | \,  #1 \,  |\mkern-6mu |\mkern-6mu |}  }

\newcommand{\sQ}{{\mathscr{Q}}}
\newcommand{\sH}{{\mathscr{H}}}

\newcommand{\sS}{{\mathscr{S}}}

\newcommand{\tT}{{\mathtt{T}}}
\newcommand{\fig}{{\rm fg}}

\title{Long time stability of small finite gap solutions \\
of the cubic Nonlinear Schr\"odinger equation on $\T^2$}

\author{
A. Maspero
\footnote{
 International School for Advanced Studies (SISSA), Via Bonomea 265, 34136, Trieste, Italy
  \newline
 \textit{Email: } \texttt{alberto.maspero@sissa.it}} 
, 
M. Procesi
\footnote{
Dipartimento di Matematica e Fisica
Universit\`a degli Studi Roma Tre, 
L.go S. Leonardo Murialdo 1, 00146 Roma, Italy 
 \newline
 \textit{Email: } \texttt{procesi@mat.uniroma3.it}} 
  }

\begin{document}
\maketitle
\begin{abstract}
In this paper we study long time stability of a class of nontrivial,  quasi-periodic solutions depending on one spacial variable of the  cubic defocusing  non-linear  Schr\"odinger equation  on the two dimensional torus.
{We prove that  these quasi-periodic solutions are orbitally stable  for finite but  long times, provided that  their  Fourier support and their frequency vector satisfy some complicated but explicit condition, which we show holds true for most solutions.\\}
 The proof is based on a normal form result.  More precisely we expand the  Hamiltonian in a neighborhood of a quasi-periodic solution, we reduce its quadratic part to diagonal  constant coefficients through a KAM scheme, and finally  we remove its cubic terms with a step of nonlinear Birkhoff normal form. The main difficulty is to impose second and third order Melnikov conditions; this is done by combining the techniques of reduction in order  of pseudo-differential operators with the algebraic analysis of resonant quadratic Hamiltonians.
\end{abstract}
\tableofcontents

\section{Introduction  and main result}
\subsection{The stability  problem}
In this paper we study long time stability of a class of nontrivial,  quasi-periodic solutions depending on one spacial variable of the  cubic defocusing  non-linear  Schr\"odinger equation (NLS) on the two dimensional torus $\T^2$:
\begin{equation}\label{NLS}
\im \partial_t v = -\Delta v + |v|^2 v \ ,  \qquad (x,y)\in\T^2\ .
\end{equation}
In particular we prove that there exists a family of such solutions  which are orbitally stable  in {$H^p(\T^2)$} for finite but  long times, {for any $p >1$}.
This means, essentially, that if we take an initial datum close enough in  $H^p(\T^2)$  to one of these solutions, then we stay in a neighborhood  of the {\em orbit} of the solution for long time, { in the  $H^p(\T^2)$ topology}.\\
Solutions of \eqref{NLS} depending on only one-variable are completely characterized; indeed the
restriction of \eqref{NLS} to 
the subspace  of functions depending only on one-variable, say $x$,  is the 1-dimensional defocusing NLS (dNLS)
\begin{equation}\label{dNLS}
\im \partial_t q = -\partial_{xx} q + |q|^2 q \qquad\qquad x\in\T\ ,
\end{equation}
 which is a well known integrable system \cite{zakharov71, zakharov74}. 
The dynamics of dNLS is  completely understood: the phase space is foliated in invariant tori of finite and infinite dimensions and  the dynamics on each torus is either quasiperiodic (namely it is  a combination of periodic motions with a finite number of different frequencies) or almost periodic (a superimposition of periodic motions with infinitely many different  frequencies).\\ 
 Actually even more is true: Gr\'ebert and Kappeler \cite{grebert_kappeler} showed  that there exists
a globally defined map $\Phi: L^2(\T) \to \ell^2 \times \ell^2, $ $ q \mapsto (z_m, \bar z_m)_{m \in \Z}$, the {\em Birkhoff map}, 
which introduces  {\em Birkhoff coordinates}, namely complex conjugates  canonical coordinates $(z_m, \bar z_m)_{m \in \Z}$,  with the property that the dNLS Hamiltonian, once expressed in such coordinates,  is a real analytic  function depending only on  the actions   $I_m := |z_m|^2$.  
As a consequence, in  Birkhoff coordinates the flow \eqref{dNLS} is conjugated to an infinite chain of nonlinearly coupled oscillators:
\begin{equation}
  \label{dnls.bc}
 \im  \dot z_m = \alpha^{\rm dnls}_m(I) z_m\qquad  \forall m \in \Z \ , 
  \end{equation}  
  where the $\alpha^{\rm dnls}_m(I)$ are frequencies depending only on the actions $(I_m)_{m\in \Z}$. \\
Then, having fixed   an arbitrary number $\tk \in \N$,  an ordered set $\cS_0 := (\tm_1, \ldots, \tm_\tk) \subset \Z$ of modes and a vector $I_\tm = (I_{\tm_i})_{1 \leq i \leq \tk} \subseteq \R^\tk_{>0}$,  the set
\begin{equation}
\label{tori}
\tT^\tk \equiv 
\tT^\tk(\cS_0, I_\tm) := \left\lbrace
(z_m)_{m \in \Z} \colon \quad  |z_{\tm_i}|^2 = I_{\tm_i} \, \mbox{ for } 1 \leq i \leq \tk  \ ,  \quad  
z_m = 0  \, \mbox{ if } m \notin \cS_0 
\right\rbrace
\end{equation}
is an invariant torus of \eqref{dNLS} of dimension $\tk$
which is supported on the set $\cS_0$.
We say that $q(t)$ is a {\em finite gap solution} of \eqref{dNLS} if it is supported (in Birkhoff coordinates) on a finite dimensional torus, i.e.  $\forall t$  $\, \Phi(q(t)) \subseteq \tT^\tk(\cS_0, I_\tm)$, for some  set $\cS_0$ of cardinality $\tk < \infty$ and vector $I_\tm$.
Any finite gap solution is quasiperiodic in time, $q(t) \equiv q(\omega t)$, with a frequency vector $\omega \in \R^\tk$ which  is specified by the values of the actions $I_\tm$ and thus it is associated to the torus itself. 
 In particular, if $\omega$ is nonresonant, than the orbit of the finite gap solution is dense on the torus.
% \red{Note that, given $\cS_0:=(\tm_1,\dots,\tm_\tk)$, the {\em small} $\tk$-gap solutions are {\em essentially supported} on $\cS_0$ at all times. } \\

Clearly  any finite gap solution of \eqref{dNLS} is also a solution of  \eqref{NLS}.
 Of course if we take an initial datum of \eqref{NLS} that is close to the initial datum of a finite gap solution but {\em not $y$-independent}   we expect the  dynamics to be very complicated, but to stay close  to the one of the integrable subspace $\tT^\tk$ for long times.\\
{The purpose of this work is to prove, roughly speaking, that ``many'' tori $\tT^\tk(\cS_0, I_\tm)$ are orbitally stable for long times:  if one chooses an initial datum which is $\delta$-close to one of these tori, then  the solution
 of \eqref{NLS} stays close to the torus for times of order $\delta^2$.}
 In order to make this statement precise we need to introduce some notations.
 For any $p > 1$ we denote $H^p(\T^n)$ the usual Sobolev space of functions whose Fourier coefficients have  finite norm
 $$
 \norm{v}_{H^p(\T^n)} := \left( \sum_{\jj \in \Z^n} \la \jj \ra^{2p} \, |v_\jj|^2 \right)^{1/2} \ ,
 $$
 where $\la \jj \ra :=\sqrt{1 + |\jj|^2}$.  From now on we shall systematically identify $H^p(\T)$ with the closed subspace of $H^p(\T^2)$ of functions depending only on the $x$ variable. Consequently $\Phi^{-1}(\tT^\tk)$ is a closed torus of $H^p(\T)\subset H^p(\T^2)$.
 
 We define now the notion of closeness to a torus $\tT^\tk$ which we will use in the following:
 
\begin{definition}
We say that  $v\in H^p(\T^2)$ is $\delta$-close to the torus $\tT^\tk=\tT^\tk(\cS_0, I_\tm)$ if
\begin{equation}
\label{condom}
{\rm dist}{(v , \Phi^{-1}(\tT^\tk))}_{H^p(\T^2)} \le  \delta \quad \mbox{and}\quad  ||z_m|^2 -I_m|\le \delta^2 \quad \forall m\in \cS_0  \ , 
\end{equation}
where
$$
z  \equiv (z_m, \bar z_m) _{m \in \Z}:= \Phi(v(x,0))\ . $$
\end{definition}
{Namely we ask that the distance between  $v$ and the preimage $\Phi^{-1}(\tT^\tk)$ is of order $\delta$, and furthermore that the Birkhoff actions  of $v(\cdot, 0)$ with indexes in $\cS_0$ are $\delta^2$-close to the corresponding actions $I_\tm$ of the torus.}\\ 
{With this definition of $\delta$-closeness, we can now define $\delta$-orbital stability:
\begin{definition}
\label{def:os}
A torus $\tT^\tk(\cS_0, I_\tm)$ is said to be {\em $\delta$-orbitally stable } for $|t| \leq T$ if there exists a constant $\mathtt K$ (independent of $\delta$) s.t.
for any initial datum $v_0 \in H^p(\T^2)$ which is $\delta$-close to $ \tT^\tk(\cS_0, I_\tm)$, then the solution $v(t,x)$ of \eqref{NLS} stays $\mathtt K \delta$- close to $\tT^\tk(\cS_0, I_\tm)$  for any $|t| \leq T$. 
\end{definition}
We state now our main theorem. 
Up to our knowledge, it is the first result of stability for quasi-periodic solutions in higher dimensional setting.
 \begin{theorem}
	\label{thm:main}
Fix $p>1$. 	For any {\em generic}  choice  of support sites $\cS_0$ there exist $\e_\star,  T_\star, \mathtt k_\star >0$ and for any $0<\e<\e_\star$, there exists a  positive measure Cantor-like set $\cI\in (0,\e)^\tk$ such that the following holds true. For any $I_\tm\in \cI$, any  $\delta< \mathtt k_\star \sqrt \e$, the torus $ \tT^\tk(\cS_0, I_\tm)$     is $\delta$-orbitally stable for $|t| \leq T_\star/\delta^2$. 
{Moreover the constant $\tK$ of Definition \ref{def:os} does not depend on $\e$.}
\end{theorem}
}
In order to make the theorem precise, we need to specify what we mean by {\em generic set $\cS_0$}.
The point is that we cannot prove Theorem \ref{thm:main}  for any arbitrary torus, but we need to impose some restrictions both on the  {\em  Birkhoff support} $\cS_0$ of the torus, and on the values of the {\em Birkhoff actions} $I_\tm$.
 Concerning the Birkhoff support  $\cS_0$, we need to select it in a complicated but explicit way,   see  Definition \ref{defset}--\ref{defar}.  
These sets are generic in the following sense (see Lemma \ref{rem:inf.s0}): given any $R\gg 1$ if we choose the elements of $\cS_0$  randomly  in  $[-R,R]$, then the probability of choosing a ``good'' set (i.e. one for which we can prove our Theorem) goes to one as $R$ goes to infinity.  %which we specify in Lemma \ref{rem:inf.s0}.
  It is possible that the conditions we give on the support can be weakened, but this produces serious technical difficulties. In any case we suspect that some selection of the support of the torus is necessary for our stability result.\\
%   In any case, we prove that there are infinitely many choices of "good" sites, see Remark \ref{rem:inf.s0}. 
%   The precise conditions that we need to impose  are stated below,    see  Definition \ref{defset}--\ref{defar}.\\
 Once we choose the support, we need also to select the Birkhoff actions $I_\tm$ of the torus.
  This selection is inevitable, since we can produce a positive measure set of  actions $I_\tm$ such that
   the torus  $\tT^\tk(\cS_0, I_\tm)$ is  linearly hyperbolic, hence {\em linearly unstable} in $H^p(\T^2)$ (see Remark  \ref{rem:unstable}). Clearly we want to rule out such a  behavior.
   Thus we also impose  conditions   on  the  actions which are  quite explicit and give rise to a Cantor-like set of positive measure.

  Theorem \ref{thm:main} is essentially a result on the $H^p$-{\em orbital stability} of many small finite gap solutions.
   Indeed, consider a finite gap solution $q(t)$ with $\Phi(q(t))\in \tT^\tk(\cS_0,I_\tm)$,  and $\cS_0,I_\tm$ satisfying the conditions of the theorem.
The fact that the actions belong to $(0, \e)^\tk$ implies that any finite  gap $q$ supported on $\tT^\tk(\cS_0, I_\tm)$ is small in size, and more precisely one has the bound $\|q\|_{L^2(\T)}\le \tk \sqrt{\e}$ (see formula \eqref{mass.bc}).
  Then Theorem \ref{thm:main} says essentially that for any initial datum $v_0 \in H^p(\T^2)$ sufficiently close to $q(0)$, the solution $v(t)$ stays close to the torus supporting the  orbit of $q(t)$ for long times 
  (remark that  we cannot hope to control directly the quantity $\norm{v(t) - q( t)}_{H^p(\T^2)}$ for long times).

Our work should be compared  with the results of
Faou, Gauckler,  Lubich \cite{faou13} and Procesi,  Procesi \cite{procesi15}, which both deal with stability issues of nontrivial solutions of  \eqref{NLS}. 
  In \cite{faou13}, the authors consider  only $1$-gap solutions, but are able to prove orbital stability for  times of order  $\delta^{-N}$ for arbitrary  $N >0$. 
  On the other hand, in  \cite{procesi15} the authors prove
only linear stability, namely that the torus is stable for times of order $\delta^{-1}$,  but for a much larger class of tori which depend on both variables $x$ and $y$.   Note however that the class of solutions considered in \cite{procesi15} does not contain $\tk$-gap solutions for $\tk >1$. 
For a more detailed comparison, see Remark \ref{rem:faou} and Remark \ref{rem:procesi} below.
  \\

The proof of Theorem \ref{thm:main} is based on a Birkhoff normal form result.
To develop such a normal form, first we introduce {\em adapted coordinates}, which are  canonical coordinates 
$(\yy, \theta, a, \bar a)$ which describe a neighborhood of the invariant torus $\tT^\tk(\cS_0, I_\tm)$ and such that  $\yy = 0$, $a =  \bar a =0$
is the torus itself.
 In the classical Hamiltonian  language, the  $\yy$ are the variables tangential to the  torus, while the $a, \bar a$ are the normal ones.
In such variables the NLS Hamiltonian takes  the form
\begin{equation}
\label{ham.intro}
 \omega \cdot \yy  + \sum_{\jj} |\jj|^2 |a_\jj|^2 +  \cH^{(\geq 0)}(\yy, \theta, a, \bar a) 
\end{equation}
where $\omega$ are the frequencies associated to the torus, 
 and $ \cH^{(\geq 0)}$ is a perturbation  term  of size $\e$ (recall that ${\e}$ is the size of the actions) and with  a zero of order 2 in the variables $(\yy, a, \bar a)$. 
Then the stability of the torus $\tT^\tk$ is equivalent to 
the stability of the zero solution  for the Hamiltonian equations of \eqref{ham.intro}.\\
The classical methods used in this problem  consist in developing a normal form theory in which one reduces to diagonal, 
$\theta$-independent form the terms of  degree two in $(a, \bar a)$, and removes iteratively the nonresonant terms of higher and higher degree in $(\yy, a, \bar a)$. 
Here we do this up to  order three; namely after 
reducing to constant coefficients the quadratic part of the Hamiltonian, we eliminate also  the cubic terms, and we are left with an Hamiltonian (see \eqref{ham.bnf.fin}) in which the coupling term has a vector field whose Sobolev norm is   $\sim \delta^{3}$; this leads to  the time of stability of Theorem \ref{thm:main}.
\\

Before closing this introduction, we want to put our work in some historical context. It is well known  since the work of Bourgain \cite{Bourgain1993} that \eqref{NLS}  is globally in time well posed in the Sobolev space $ H^p(\T^d)$ for any $p >0$ and $d= 1, 2$. 
Since then, the problem of describing qualitatively the  dynamics of NLS solutions has been widely studied, and the results obtained so far can be essentially divided into three groups: 
in the first group of papers, the authors study stability of periodic or quasiperiodic solutions of NLS on $\T$ or $\T^2$ for long but finite time \cite{zhidkov1, bambusi99, bourgain00, faou13, gallay07a, gallay07b, gallay15, wilson15}.
In the second group, the authors study instability phenomena, showing that (on a possibly very large time scale) the solution goes arbitrary distant from the initial datum \cite{Iteam10, hani14, guardia14, haus15, guardia15, guardia16}.
Finally,  the third group of papers concerns with the existence of quasi-periodic solutions of \eqref{NLS} for any time \cite{eliasson10, geng, wang, procesi13, procesi15, procesi16}. 

We start to describe the results in the first group. 
In dimension $d=1$, it is known since the work of Zhidkov \cite[Sect. 3.3]{zhidkov1} 
that plane waves solution of \eqref{dNLS} (namely solutions of the form $v(t,x) = A e^{\im (m x -\omega t) }$ with $\omega = m^2 + A^2$) are  $H^1$ orbitally stable for any time.
Remark that plane waves solutions are exactly   $1$-gap solutions  of \eqref{dNLS} (see e.g. \cite[Lemma 3.5]{grebert02}), thus their orbit is an invariant torus of dimension $\tk=1$. 
Orbital stability of  $\tk$-finite gap solutions with $\tk >1$  was studied by Bambusi \cite{bambusi99}, who  proved   $H^1$-orbital stability for exponentially long times, namely times  of order $\sim e^{T _0 \delta^{-1/\tk}}$. Furthermore the author does not impose any condition on the choice of the Birkhoff support; this is due to the fact that  resonances in dimension 1 have a much simpler structure.  \\
Finally Bourgain \cite{bourgain00} proved $H^s$ orbital stability of small finite-gap solutions (with $s$ sufficiently large) for times of order   $\delta^{-N}$, with an arbitrary  $N$.\\
In dimension $d=2$, much less is known.  Faou, Gauckler and Lubich \cite{faou13}  proved that plane waves solutions (i.e. $1$-gap solutions)  of \eqref{NLS} are $H^s$ orbitally stable for $s$ large enough and for times  of order $\delta^{-N}$, for arbitrary $N$.
Note that \cite{faou13} consider  {\em large } $1$-gap solutions and control the distance from the torus for times much longer than ours. 
The reason is that for  such solutions the authors know the exact formula for the tangential and normal frequencies  and therefore they can impose Melnikov conditions at any order and perform  arbitrary many  steps of normal form (see Remark \ref{rem:faou} for more details).

We now describe the results of the second group, concerning  instability phenomena for \eqref{NLS}. While the integrability of \eqref{dNLS} prevents any phenomena of growth of Sobolev norms in dimension $1$, the situation in dimension $d>1$ is much more complicated. Bourgain \cite{bourgain96} gives upper bounds on the solutions of \eqref{NLS} of the form
\begin{equation}
\label{upp.b}
\norm{v(t)}_{H^s(\T^2)} \leq t^{2 (s-1)+} \norm{v(0)}_{H^s(\T^2)} \ , \qquad t >0 \ ,
\end{equation}
for $s$ sufficiently large (see also  \cite{staffilani97, sohinger11,  colliander_oh, visciglia} for more recent results on upper bounds of the form  \eqref{upp.b} 
for nonlinear Schr\"odinger equations and \cite{bourgain99,delort10,maro, BGMRb, Mon17} for linear time-dependent Schr\"odinger equations).
Estimate \ref{upp.b} leaves open the question whether it is possible to construct a solution whose Sobolev norm is  unbounded.
 In the last few years many efforts have been done in this direction. 
Colliander et al. \cite{Iteam10} proved, for any $\epsilon, K >0$ and $s >1$,
 the existence of a initial datum $u(0) \in H^s(\T^2)$ and a time $T >0$ s.t. $\norm{u_0}_{H^s(\T^2)} <\epsilon$ 
and $\norm{u(T)}_{H^s(\T^2)} \geq K$. 
This result was later refined by Guardia and Kaloshin \cite{guardia15}, Carles and Faou \cite{carles12} and extended to different Schr\"odinger equations   \cite{haus15, guardia14, guardia16}. 
All these results deal with the instability of the zero solution.
Hani  \cite{hani14} studied  instability of  1-gap solutions, and    proved that for any $0 < s< 1$, $\delta \ll 1, K \gg 1$, there exists a solution of \eqref{NLS} 
with $\norm{v(0) - Ae^{\im m x}}_{H^s(\T^2)} \leq \delta$ and $\norm{v(T)}_{H^s(\T^2)} \geq K$ for some later $T >0$. 
Note that the instability happens in a topology different from the one of \cite{faou13}.\\
The important question of the existence of solutions of \eqref{NLS} with unbounded Sobolev norms, i.e. s.t. 
$$
\overline{\lim_{t \to \infty}} \norm{v(t)}_{H^s(\T^2)} = +\infty
$$
is still an open problem, and it has been solved only in the  product space $\R \times \T^2$ \cite{hani15}.

Finally, we describe the results of the third group, concerning existence of quasi-periodic solution for NLS. While in dimension $d=1$ the existence of KAM quasi-periodic solution is nowadays well understood in several perturbed dNLS  (we mention just the latest contributions \cite{FP15,berti16}; see reference therein), in dimension $d>1$ the situation is much more complicated and it has been addressed only recently
by Eliasson and Kuksin  \cite{eliasson10}, Geng,   Xu and You \cite{geng}, Wang \cite{wang} 
and Procesi with collaborators \cite{procesi13, procesi15, procesi16}. 
In particular in these last papers the authors proved that, 
{\em for most choices} of tangential sites, there exist families of small quasi-periodic solutions  of \eqref{NLS}
(depending on both $x$ and $y$) supported essentially on the tangential sites. 
Such families of solutions give rise to both linearly stable and unstable KAM tori, but the question of nonlinear stability (or instability!) of such tori  is still open (see Remark \ref{rem:procesi} for more comments).

\subsection{Scheme of the Proof}
It is worth to add some words about our strategy. The proof consists in three main steps.
\\
The first step is to introduce  {\em canonical} adapted coordinates in a neighborhood of a finite dimensional invariant torus. 
\\
The second step consists in a  {\em reducibility} argument.  
We consider the  part of  $\cH^{(\geq 0)}$ which is quadratic in $(a, \bar a)$,  call it $\cH^{(0)}$,  and we reduce it to a  diagonal form with constant coefficients. 
\\
The final step consists in applying one step of non linear Birkhoff Normal Form.
\smallskip

\noindent {\it Step one: adapted coordinates.} The question of introducing canonical coordinates in a neighborhood of a solution has been addressed  in several papers, see e.g. 
\cite{kuksin, berti.bolle, BM16}.
Here however the difficulty is that we need to keep track of how such change of variables affects the constants of motion (namely the mass and momentum). 
Here we take advantage of the fact that mass and momentum are left unchanged by the Birkhoff map of dNLS,
thus first we introduce Birkhoff coordinates on the invariant subspace of $y$-independent functions,   and then pass to action-angle coordinates the sites  where the finite gap is supported.
Finally we extend this transformation as the identity to all the phase space.
Actually for later development, it is necessary to know that the Birkhoff coordinates are majorant analytic; this was proved  in \cite{masp_vey}.
In such coordinates the Hamiltonian of NLS has the form \eqref{ham.intro}.

\smallskip

\noindent {\it Step two: reducibility.} This is the most technical part of our argument.  
It is known that reducibility requires   (i) to impose the so called {\em second order Melnikov} conditions
% (which are  non resonance conditions between $\omega$ and the frequencies of the normal modes) 
and (ii) to perform a convergent KAM scheme.
Both these procedures are  particularly delicate in higher spacial dimensions, and  they have  been achieved  for Schr\"odinger equations only in some special cases \cite{EK, procesi13, CHP, procesi16,  GP, BGMR}.
\\
Imposing second order Melnikov conditions means 
requiring lower bounds for expression such  as
\begin{equation}
\label{old.m}
\abs{ \omega \cdot \ell + |\jj_1|^2 \pm |\jj_2|^2 } \,.
\end{equation}
Here there are two problems. First,   since $\omega$ is an integer valued vector up to corrections of order $\e$,
the quantity  \eqref{old.m}  can be of order $\e$ and hence comparable with the size of the perturbation.  Since these expressions appear as {\em denominators} in any diagonalization scheme our problem is {\em not} perturbative. 
Moreover, since we are working in dimension higher than one, it is well known that in order to perform a reduction one needs to know the asymptotics of the eigenvalues associated to the Hamiltonian vector field of $\sum |\jj|^2 |a_\jj|^2 +\cH^\0$.

In order to solve these problems we combine the techniques of reduction in order  of pseudo-differential operators with the algebraic analysis of resonant quadratic Hamiltonians.
More precisely, we construct a change of variables which is {\em not} close to identity and which conjugates
\eqref{ham.intro} to the form 
\begin{equation}
\label{ham.intro2}
\omega \cdot \yy + \sum_{\jj} \wt\Omega_\jj |a_\jj|^2 + \wt \cH^{( 0)}( \theta, a, \bar a) + h.o.t.
\end{equation}
Here $\wt\Omega_\jj-|\jj|^2$ is of size $\e$, while $\wt \cH^{( 0)}( \theta, a, \bar a)$ is a  perturbation of size $\e^2$ and  {\em smoothing} in the following sense: the linear operator associated to its  Hamiltonian vector field is the sum of a $2$-smoothing term plus a term  which gains $2$-derivatives only  in the $x$ direction and is  independent of $y$. This information allows us to extract the asymptotics of the frequencies required for the reduction, just as the T\"oplitz-Lipschtiz matrices in \cite{eliasson10}.
 % In section \ref{fun} we describe this class of Hamiltonians and study their properties.

In order to pass from \eqref{ham.intro} to \eqref{ham.intro2} we first exploit the fact that the Hamiltonian vector field of \eqref{ham.intro} is of the form
\[
2|q(\omega t, x)|^2 \, a(x,y) + q^2(\omega t, x) \,  \bar a(x,y)
\]
up to a finite rank remainder. 
 Due to this special form we can extend to the two dimensional setting the techniques of reduction in orders developed in the one dimensional setting \cite{PT01, IPT05, BBM14}. \\
 After this procedure  the perturbation is still of {\em size} $\e$, but it is now smoothing.
At this point we apply the strategy of \cite{procesi12}, namely we explicitly construct a not close to identity change of variables which diagonalizes the resonant non-perturbative terms of the Hamiltonian (these changes of variables are similar to those used in the problems of almost reducibility, see  \cite{Eli01}). 
This construction requires that we select the sites $\cS_0$. 
 
After this procedure, the Hamiltonian \eqref{ham.intro} is transformed into the Hamiltonian \eqref{ham.intro2},
and the relevant quantity to bound from below is now
\begin{equation}
\label{new.m}
\abs{ \omega \cdot \ell + \wt\Omega_{\jj_1} \pm 
\wt\Omega_{\jj_2} } \, , 
\end{equation}
where the normal frequencies have been corrected by some algebraic functions of the actions of size $\e$. 
Here the main difficulty is to prove that the terms of order $\e$ in such expressions are not identically zero; to show this we use the algebraic techniques of irreducible polynomials,  following  \cite{procesiN}; 
{here it is fundamental to exploit the fact that the mass and momentum are preserved}.\\
At this point we perform a KAM  reducibility scheme which puts the quadratic part of   \eqref{ham.intro2} to constant coefficients, conjugating it  to the Hamiltonian
\begin{equation}
\label{ham.intro3}
\omega\cdot\yy + \sum_{\jj} \Omega_{\jj} |a_\jj|^2 + \cH^{(1)}(\theta, a, \bar a) + h.o.t.
\end{equation}
where $\cH^{(1)}$ is cubic in $(a, \bar a)$.

\smallskip 
\noindent{\em Step three: Birkhoff normal form.}
Finally we perform one step of nonlinear Birkhoff normal form to remove $\cH^{(1)}$ from \eqref{ham.intro3}. 
This is nowadays quite standard (see e.g. \cite{bambusi.grebert}), provided that (i) one is able to   impose {\em third order} Melnikov conditions, i.e.  to give lower bounds to expressions of the form 
\begin{equation}
\label{new.m1}
\abs{ \omega \cdot \ell + \Omega_{\jj_1} \pm 
\Omega_{\jj_2} \pm
\Omega_{\jj_3}   } \, ,
\end{equation}
and (ii) one proves that the  Hamiltonian \eqref{ham.intro3} is majorant analytic.\\
Concerning the estimate of \eqref{new.m1},  we again use  algebraic techniques to prove that the  terms of order $\e$ in such expressions are not identically zero, and the asymptotics of the eigenvalues in order to get quantitative bounds.
Concerning the majorant analyticity of the Hamiltonian, we  prove that each change of variables performed so far preserves this property; it is here that one needs the majorant analyticity of the Birkhoff map.\\
In conclusion we  construct a nonlinear change of variables which conjugates \eqref{ham.intro3} to 
\begin{equation}
\label{ham.intro4}
\omega\cdot\yy + \sum_{\jj} \Omega_{\jj} |a_\jj|^2 + \cR^{(\geq 2)}(\yy, \theta, a, \bar a) \ , 
\end{equation}
where $\cR^{(\geq 2)}$ contains  only terms which are at least  of order four in $(a, \bar a)$, or linear in $\yy$ and quadratic in $(a, \bar a)$, or quadratic in $\yy$. As a consequence, its Hamiltonian vector field is of size $\sim \delta^3$, which implies the stability of zero for times of order $\delta^{-2}$, thus proving our main theorem.
\vspace{1em}
\\
\noindent
{\em A final comment:}
It is a natural question whether it is possible to perform more  steps of nonlinear Birkhoff normal form, removing from \eqref{ham.intro4} monomials of higher and higher order, and obtaining a longer time of stability. 
Performing these steps  requires to be able to impose $n^{th}$ order Melnikov conditions of the form 
\begin{equation}
\label{new.mN}
\abs{ \omega \cdot \ell +
\Omega_{\jj_1} 
 \pm  \ldots 
 \pm
\Omega_{\jj_n}   } \, .
\end{equation}
As before, one should verify that these expressions  do not vanish identically except in  resonant cases; 
this is what we are not able to prove so far.
Indeed,  the only information that we have is on the corrections of order $\e$ to  $ \Omega_{\jj} $,  and one can produce examples where  some linear combinations of them  vanish identically already at order 4.

\vspace{2em}
\noindent
{\em Structure of the paper:}   Section 2 contains some preparation in order to state precisely our result on normal form; furthermore we precise the notion of genericity of $\cS_0$.
In Section 3 we define the class of smoothing Hamiltonians and study their properties.
In Section 4 we construct the adapted coordinates $(\yy, \theta, a, \bar a)$ and show that in these coordinates the Hamiltonian has the form \eqref{ham.intro}.
In Section 5 we begin the step of reducibility, and construct the change of coordinates not close to the identity which conjugates \eqref{ham.intro} to \eqref{ham.intro2}.
In Section 6 we prove that the terms of order $\e$ in expressions of the form \eqref{new.m} and \eqref{new.m1} are not identically zero.  Quantitative  lower bounds for these expressions are proved in Appendix \ref{app:mes.m}.
In Section 7 we conclude the reducibility step by performing a KAM scheme, conjugating \eqref{ham.intro2} to \eqref{ham.intro3}.
In Section 8 we perform the step of nonlinear Birkhoff normal form, conjugating \eqref{ham.intro3} to \eqref{ham.intro4}.
Finally in Section 9 we study the dynamics of \eqref{ham.intro4} and prove the stability of zero for long times.

\vspace{1em}
\noindent{\bf Acknowledgements.}
During the preparation of this paper, we benefited of discussions with many people; in particular we wish to thank Dario Bambusi, Benoit Gr\'ebert, Marcel Guardia,  Zaher Hani, Emanuele Haus, Riccardo Montalto, Stefano Pasquali  and Claudio Procesi   for many discussions and comments.

We were supported by  ERC grant  "HamPDEs",  ANR-15-CE40-0001-02 ''BEKAM`` of the Agence Nationale de
la Recherche and Prin   "Variational methods, with applications to
problems in mathematical physics and geometry".

\section{A Birkhoff normal form result}
In this section we state  our results on the Birkhoff normal form (see Theorem \ref{main2} and Theorem \ref{main3}) and describe the genericity conditions that we need to impose on $\cS_0$.
In order to do this, we need some preparation.

\paragraph{Constants of motion.} NLS on $\T^2$ has three constants of motion that we will constantly use. They are the  Hamiltonian 
\begin{equation}
\label{ham0}
H_{{\rm NLS}}(v) :=  \int_{\T^2} \abs{\nabla v(x,y)}^2 \di x \, \di y + \frac{1}{2}\int_{\T^2}\abs{v(x,y)}^4 \, \di x \, \di y \ , 
\end{equation}
 the mass 
 \begin{equation}
\label{mass0} 
 M(v):= \int_{\T^2} \abs{v(x,y)}^2  \, \di x \, \di y \ ,
 \end{equation}
  and the (vector valued) momentum 
\begin{equation}
\label{momentum0}
P(v):= {\im}\int_{\T^2} \bar v(x,y) \cdot \nabla v(x,y)  \, \di x \, \di y \ .
\end{equation} 

\paragraph{Mass shift.} Since the mass is a constant of motion, we make a trivial phase shift and consider the equivalent Hamiltonian 
\begin{equation}
\label{parto}
H(u) :=  \int_{\T^2} \abs{\nabla u(x,y)}^2 \, \di x \, \di y + \frac{1}{2}\int_{\T^2}\abs{u(x,y)}^4 \, \di x \, \di y -  M(u)^2
\end{equation}
corresponding to the Hamilton equations
\begin{equation}\label{piri0}
\im \partial_t u = -\Delta u + |u|^2 u -2 M(u) u \ , \qquad (x,y)\in\T^2\ .
\end{equation}
Clearly the solutions of \eqref{piri0} differ from the solutions of \eqref{NLS}  only by a phase shift\footnote{In order to show the equivalence we consider any solution $u(x,t)$ of \eqref{piri0} and consider the invertible map 
$$
u\mapsto v= u\; e^{-2 \im M(u) t } \quad \mbox{with inverse}\quad  v\mapsto u= v\; e^{2 \im M(v) t } 
$$
then 
$$
\im v_t = \im u_t e^{-2 \im M(u) t } + 2 M(u) u e^{-2 \im M(u) t } = (-\Delta u + |u|^2 u -2 M(u) u)e^{-2 \im M(u) t } + 2 M(u) u e^{-2 \im M(u) t } = -\Delta v +|v|^2 v.
$$}.
If we pass  $u$ to the Fourier coefficients 
$$
u(x,y,t):= \sum_{\jj=(m,n)\in \Z^2} u_{\jj}(t) \, e^{\im(m x + n y)} 
$$
 the Hamiltonian \eqref{parto} takes the form
\begin{equation}
\label{Ha0}
H(u) = \sum_{\jj=(m,n)\in \Z^2}|\jj|^2 |u_{\jj}|^2 -\frac12 \sum_{\jj}|u_{\jj}|^4 + \frac12\sum_{\jj_i\in \Z^2 \atop \jj_1-\jj_2+\jj_3-\jj_4=0}^{\star}u_{\jj_1}\bar u_{\jj_2}u_{\jj_3}\bar u_{\jj_4}
\end{equation}
where the $\sum^\star$ means the sum over the quadruples $\jj_i$ such that $\{\jj_1,\jj_3\}\neq \{\jj_2,\jj_4\}$, and it is a consequence of having removed the mass in \eqref{parto}.\\
The Hamiltonian \eqref{Ha0} is  analytic 
 in the  phase space 
$$
h^p(\Z^2):= \{u=(u_{\jj})_{\jj\in \Z^2}:\quad \sum_{\jj \in \Z^2}|u_\jj|^{2}\la \jj\ra^{2p}:=|u|_{h^p}^2<\infty\} \ .
$$
% or
%$$
%h^p(\Z^2) \cap h^1(\Z^2) \cap \ell^1(\Z^2)   \ , \qquad \mbox{when }  0 \leq p \leq 1 \ . 
%$$
%The reason to consider different phase spaces according to the value of $p$ is  that we want to have a phase space whose norm is an algebra w.r.t. the convolution of sequences.
which we endow with the  standard symplectic form 
\begin{equation}
\label{sym0}
\im \sum_{\jj \in \Z^2} du_\jj\wedge d\bar u_{\jj} \ . 
\end{equation}

\paragraph{Finite gap solutions of NLS.} As we already pointed out in the introduction, the subspace of functions depending only on the variable $x$ is an invariant subspace. In Fourier coordinates, such subspace is identified  by having Fourier coefficients $u_{(m,n)} = 0$ if $n \neq 0$. 
%$$
%\cU_x^p := \big\{ (u_\jj)_{\jj \in \Z^2} \in h^p(\Z^2) \ \ : \ \ u_{(m,n)}= 0 \mbox{ if } n \neq 0 \ , \ \ \ 
% u_{(m,0)}= q_m  \big\} \equiv h^p(\Z) \ .
% $$
and it  is  clear by the structure of \eqref{Ha0} that such a subspace   is invariant for the dynamics.  
Actually, the  Hamiltonian \eqref{Ha0} restricted on such a subspace is nothing else that the  Hamiltonian 
of the 1-dimensional  defocusing NLS (dNLS) with the mass shift, namely
\begin{equation}
\label{parto.dnls}
H_{\rm dm}(q):= H_{\rm dnls}(q) - M(q)^2 =  \int_{\T} \abs{\nabla q(x)}^2 \, \di x  
+
 \frac{1}{2}\int_{\T}\abs{q(x)}^4 \, \di x 
  -  M(q)^2 \ ,
\end{equation}
 whose equations of motion are
\begin{equation}
\label{dnls.massless}
\im \partial_t q = - \partial_{xx} q + |q|^2 q - 2M(q) q \ , \qquad x \in \T \ .
\end{equation}
Since dNLS is integrable and the mass $M$ is an integral of motion for dNLS,  the Birkhoff map  conjugates \eqref{dnls.massless} to a system of equations of the form
\eqref{dnls.bc},  where the $(\alpha^{\rm dnls}_m(I))_{m \in \Z}$'s are  replaced by new frequencies $(\alpha_m(I))_{m \in \Z}$, see 
subsec. \ref{sub:bm} for more details and references. \\
A consequence of this fact is that any  torus   $\tT^\tk(\cS_0, I_\tm)$ of the form \eqref{tori} is also an invariant torus  for the massless dNLS 
equation \eqref{dnls.massless}, and the dynamics which is induced on it is quasi-periodic with   frequency vector 
$\alpha_\tm(I_\tm) \equiv (\alpha_{\tm_1}(I_\tm), \ldots, \alpha_{\tm_\tk}(I_\tm) )$.
%
% Since the   action-to-frequency  map $I \mapsto \alpha(I)$ is generically a diffeomorphism, we can select the frequencies by modulating the actions. 
Since the map $I_\tm \mapsto \alpha_\tm(I_\tm)$ is highly nonlinear, for technical reasons it is  more convenient to parametrize the vector $\omega:= \alpha_\tm(I_\tm)$  in a linear way; therefore  
 we define
 \begin{equation}
 \label{omega0}
\omega^{(0)}:=(\tm_1^2,\dots,\tm_\tk^2)
 \end{equation}
  and  for $\e$ sufficiently small and  $\lambda \in  [\frac12,1]^\tk$, we   set 
%  \red{non mi sento sicura: ma $\cO_0$ ci serve per l'invertibilta' della mappa azione/frequenza oppure quella e' gia' garantita e ci serve poi per l'ellitticita'? Cosi' sta un po appeso nel nulla} 
 \begin{equation}
 \label{omega.lambda}
 \alpha_\tm(I_\tm)\equiv  \omega^{(0)}- \e \lambda \ =:  \omega(\lambda) \ .   
 \end{equation}
As discusses in subsec. \ref{sub:bm},  the action-to-frequency map  can be inverted, obtaining a map $\lambda \to (I_{\tm_i}(\lambda,\e))_{1 \leq i \leq \tk}$ s.t.
 \begin{equation}\label{boia}
 I_{\tm_i}(\lambda,\e)= \e \lambda_i +  O(\e^2) \ , \qquad 1 \leq i \leq \tk \ .
 \end{equation} 
In the following we will use $\omega \equiv \omega(\lambda)$ as the vector of frequencies of the finite gap solution. Such vector is chosen to be non-resonant so that  the orbit of the finite gap solution is dense on $\tT^\tk(\cS_0, I_\tm)$.

\paragraph{Adapted coordinates  around a finite gap solution.} 
We now introduce local coordinates in $h^p(\Z^2)$ adapted to  the finite dimensional tori \eqref{tori}. 
Remark that, while in Birkhoff coordinates such tori have a very simple form, in the original coordinates the representation is not so trivial, and they have the form 
 \begin{equation}
 \label{finite.gap.formula}
 q(\lambda;\theta,x)= \sqrt{\e}\left(\sum_{i=1}^\tk \sqrt{\lambda_i}e^{\im \theta_i +\im \tm_i x} +\e p(\lambda,\e ;\theta,x)\right)
 \end{equation}
 for some real analytic function $p(\lambda,\e ;\theta,x)$.

To introduce coordinates around such tori we first apply the Birkhoff map $\Phi$ on $(u_{(m,0)})_{m\in \Z}$ leaving the remaining $(u_\jj)_{\jj \in \Z^2 \setminus \Z}$ unchanged, 
and we set  $ (z_m)_{m \in \Z} :=\Phi((u_{(m,0)})_{m \in \Z})$. 
Next   we  pass the  $z_m$ with $m \in \cS_0$ to  symplectic action-angle variables defined close to the torus, setting 
\begin{equation}
\label{a.a.z}
z_{\tm_i}=\sqrt{ I_{\tm_i}(\lambda)+\yy_i }  \ e^{\im \theta_i} \ , \; 1 \leq i \leq \tk \ , \qquad z_{m}= a_{(m,0)} \,,\; m\in \Z\setminus \cS_0 \ , \qquad u_{(m,n)} = a_{(m,n)} \,,\; n \neq 0 .
\end{equation}
In such a way we introduce coordinates $\yy = (\yy_1, \ldots, \yy_\tk) \in \R^\tk$, $\theta = (\theta_1 , \ldots, \theta_\tk) \in \T^\tk$, $a = (a_\jj)_{\jj \in \Z^2 \setminus \cS_0}$
such that  $\yy=0$, $a=0$ describe 
 {torus}
$\tT^\tk \big(\cS_0,  I_\tm(\lambda)\big)$.
Clearly here $\Z^2 \setminus \cS_0  \equiv \Z^2 \setminus (\cS_0 \times \{0\})$. We will use systematically such a  notation. \\
We denote by  $\Lambda: (\yy,\theta,a)\mapsto (u_{\jj})_{\jj\in \Z^2}$ the map 
\begin{equation}
\label{mappa}
u_{(m,n)} = a_{(m,n)} \,,\quad n\neq 0 \,,\qquad (u_{(m,0)})_{m\in \Z}= \Phi^{-1}( (\sqrt{ I_{\tm_i}(\lambda)+\yy_i }  \ e^{\im \theta_i})_{i=1,\ldots,\tk}, (a_{(m,0)})_{m\in \Z\setminus\cS_0} ).
\end{equation}
Since  the Birkhoff map $\Phi$ is symplectic, the symplectic form \eqref{sym0} in the variables $(\yy, \theta, a)$ is given by  
\begin{equation}
\label{yythetaa}
\sum_{i=1}^\tk d\yy_i\wedge d\theta_i +\im  \, \sum_{\jj \in \Z^2 \setminus \cS_0 } d a_\jj \wedge d \bar a_\jj \ . 
\end{equation}
Next we describe the phase space and its  topology.
We fix once and for all a real  $p> 1 $ and define the  phase space  $ \C^\tk \times \T^\tk \times\li^p $ where  
$$
\li^p\equiv \li^p(\cS_0):=
\{\ba=((a_\jj,\bar a_\jj))_{\jj\in \Z^2\setminus \cS_0} \in h^p(\Z^2\setminus \cS_0) \,,\qquad  \norm{a} \equiv |a|_{\hh^p}<\infty\} \ . 
$$  
Finally we define  the complex domain 
$$D(s,r):=\T^\tk_s\times D(r)$$
 where
\begin{align*}
 D(r) & :=\left\lbrace \yy\in \C^\tk:\; 
|\yy|_1:=\sum_{i=1}^\tk|\yy_i| < r^2 \,,
\quad
\ba\in \li^p : \;
 \|\ba\| < r \right\rbrace  \,, \\
  \T^\tk_s & :=\left\lbrace\theta\in \C^\tk: {\rm Re }(\theta)\in \T^\tk \,,\quad |{\rm Im }(\theta)| < s \right\rbrace . 
  \end{align*}
 We will show in Sec. \ref{preparazione} that for $r<\sqrt\e$ the map  \eqref{a.a.z} is well defined and analytic from $D(s,r)$ to a neighborhood of the torus $\tT^\tk(\cS_0, I_\tm(\lambda))$.
 %  \begin{remark}
% % The norm $\norm{\cdot}$ is constructed in such a way that 
% We remark that the space $h^p \cap \ell^1 \cap h^1$ is an algebra:
%  $$
%\|a \star b\|\le \tC \|a\|\|b\|
%\,,\quad
%|a \star b|_{\hh_1}\le \tC (|a|_{\hh_1}\|b\|+ \|a\||b|_{\hh_1})\,.
%$$
%See Lemma \ref{alg1}.
%  \end{remark}

\paragraph{Norm on vector fields.} We describe now how to measure the norm of  Hamiltonian vector fields. 
{ Recall that a Hamiltonian is a \em{real valued} function $\C^\tk\times \T^\tk_s\times \li^p\to \R$}. Given a Hamiltonian function $F(\yy,\theta,\ba)$ we associate to it its Hamiltonian vector field
\begin{equation}
\label{def:XF} 
X_F:=( \ \partial_\theta F, \ -\partial_\yy F, \  -\im \partial_{\bar a} F,  \ \im \partial_{a}F ). 
\end{equation}
More in general consider vector fields which are  functions from $$\C^\tk\times \T^\tk_s\times \li^p \to  \C^{\tk}\times \C^\tk\times \li^p \,:\;(\yy,\theta,\ba)\to (X^{(\yy)},X^{(\theta)},X^{(a)},X^{(\bar a )}) 
$$
 which are analytic in 
% $\theta$ and majorant analytic in $\yy,\ba$ in 
 $D(s,r)$.
On the vector field we use as norm
\begin{equation}
\label{listen}
\bnorm{ X}_r:=|X^{(\theta)}|_{\infty}+\frac{|X^{(\yy)}|_1}{r^2}+ \frac{\|X^{(a)}\|}{r}+ \frac{\|X^{(\bar a)}\|}{r} \ .
\end{equation}
Next we introduce the notion of {\em majorant analytic} Hamiltonians and vector fields. 
We write a real  valued Hamiltonian $h$ in  Taylor-Fourier series which  is well defined  and pointwise absolutely convergent:
\begin{equation}
\label{h.funct}
h(\yy,\theta,\ba)= \sum_{\al ,\bt\in\N^{\Z^2\setminus\cS_0},\ell\in \Z^\tk, l\in \N^\tk } h_{\al,\bt,l,\ell} \,  e^{\im \ell\cdot \theta} \, \yy^l \,  a^\alpha \, \bar a^\beta \,,\quad h_{\al,\bt,l,\ell}= \bar h_{\bt,\al,l,-\ell}\,.
\end{equation}
Correspondingly we expand vector fields in Taylor Fourier series (again  well defined  and pointwise absolutely convergent):
$$
X^{(w)}(\yy,\theta,\ba)= \sum_{\al ,\bt\in\N^{\Z^2\setminus\cS_0},\ell\in \Z^\tk, l\in \N^\tk } X_{\al,\bt,l,\ell}^{(w)} \, e^{\im \ell\cdot \theta} \, \yy^l \,  a^\alpha \, \bar a^\beta\,,
$$
where $w$ denotes the components $\theta_i, \yy_i$ or $a_\jj,\bar a_\jj $.
To a vector field we associate its {\em majorant }
\begin{equation}\label{seriamente}
\und{X}_s^{(w)}[\yy,\ba]:= \sum_{\ell\in \Z^\tk,l\in \N^\tk,\al ,\bt\in\N^{\Z^2 \setminus \cS_0} } |X^{(w)}_{\al,\bt, l, \ell}| \, e^{s\, |\ell|} \, \yy^l \,  a^\alpha \, \bar a^\beta  \ .
\end{equation}
Then we have the following
\begin{definition}
A vector field $X: D(s,r) \to  \C^{\tk}\times \C^\tk\times \li^p$ will be said to be {\em majorant analytic} in $ D(s,r)$  if $\und{X}_s$ defines an analytic vector field  $D(r)\to \C^\tk\times\C^\tk\times \li^p$.
\end{definition}
   Since Hamiltonian functions are defined modulo constants, 
we give the following definition of {\em majorant analytic} Hamiltonian and its   norm:
\begin{definition}
A real valued Hamiltonian $H$ will be said to be {\em majorant analytic} in $D(s,r)$ if its Hamiltonian vector field $X_H$ is majorant analytic in $D(s,r)$. We define its {\em norm} by
\begin{equation}\label{musically}
|H|_{s,r}:=\sup_{(\yy,\ba)\in D(r)} \bnorm{ \und{(X_H)}_s(\yy, \ba)}_{r} \ . 
\end{equation}
\end{definition}

Note that the norm $| \cdot |_{s,r}$ controls the norm of the vector field $X_H$ defined in \eqref{def:XF} in the domain $\T^\tk_{s'}\times D(r')$ for all  $s'< s$, $r'<r$. 

By convention we define the {\em scaling degree} of a monomial $ e^{\im \ell\cdot \theta} \, \yy^l \,  a^\alpha \, \bar a^\beta$ as
\begin{equation}\label{scadeg}
\deg(l, \alpha, \beta) := 2| l | + |\alpha| + |\beta| -2\,,
\end{equation}
and define the {\em projection} on the homogeneous components of {\em scaling degree} $d$ as
$$
H^{(d)} := \sum_{\ell\in \Z^\tk, l \in \N^\tk , \al ,\bt\in\N^{\Z^2\setminus\cS_0}  \atop 
	2| l |+|\al|+|\bt|=d+2}  \, H_{\al,\bt, l, \ell} \, e^{\im \ell\cdot \theta} \yy^l a^\alpha \, \bar a^\beta \ , 
$$
similarly for $H^{(\leq d)}$ and $H^{(\geq d)}$. 
\paragraph{Lipschitz families of  Hamiltonians.} In the following we will consider Hamiltonians  $h(\lambda; \yy, \theta, \ba) \equiv h(\lambda)$ depending on an external parameter $\lambda \in \cO$, {where $\cO$ is some compact set}. Thus we define the  {\em weighted Lipschitz} norm:
\begin{equation}
\label{lip.norm0}
|h|_{s,r}^\cO:= \sup_{\lambda \in \cO}|h(\lambda, \cdot)|_{s,r}+  \sup_{\lambda_1\neq \lambda_2\in \cO} \frac{|h(\lambda_1, \cdot)-h(\lambda_2, \cdot) |_{s,r}}{|\lambda_1-\lambda_2|} \ . 
\end{equation}
It will be  convenient to define the Lipschitz norm also for maps $f:\cO\to E$ with values in a Banach  space $E$, whose norm we denote simply by $\vert \cdot \vert_E$. We pose
\begin{equation}
\label{lip.norm}
\abs{f}^\cO_E:= \sup_{\lambda \in \cO} \abs{f(\lambda)}_{E} +  \sup_{\lambda_1 \neq \lambda_2 \in \cO} \frac{|f(\lambda_1)- f(\lambda_2) |_{E}}{|\lambda_1 - \lambda_2|} \ . 
\end{equation}

%{\color{blue} Ma se non mettiamo le proprieta' che ce lo mettiamo a fare lo scaling degree? Non ci stava qui un teorema sulla norma? }
\paragraph{The Hamiltonian and the constants of motion in the adapted coordinates.} 
When expressed in the coordinates $(\yy, \theta, \ba)$ defined in \eqref{a.a.z}, the Hamiltonian $H$ of \eqref{Ha0} takes the form 
\begin{equation}
\label{ham.1step}
\cH(\yy, \theta, {\bf a}) :=
\omega \cdot \yy 
+ \sum_{\jj=(m,n)\in\Z^2\setminus \cS_0} |\jj|^2 |a_{\jj}|^2 + \cH^{(\geq 0) }(\yy, \theta, {\bf a})
\end{equation}
where $\cH^{(\geq 0) }$ has scaling degree greater than or equal to zero and $\abs{ \cH^{(\geq 0) }}_{s_0,r_0}\leq C \e $ for some $s_0, r_0 >0$, see Section \ref{preparazione} for the details.
In these coordinates the mass $M$ and the momentum $P$ become
\begin{equation}\label{costanti}
	\cM(\yy, \theta, {\bf a})=   \sum_{i=1}^\tk \yy_i +\sum_{\jj\in \Z^2\setminus \cS_0}|a_{\jj}|^2\,,\quad
	\cP(\yy, \theta, {\bf a})=   \begin{bmatrix} \cP_x \\ \cP_y\end{bmatrix}=\sum_i \begin{bmatrix} \mathtt m_i \\ 0 \end{bmatrix} \yy_i+ \sum_{\jj\in \Z^2\setminus \cS_0} \jj \,|a_{\jj}|^2  \, .
	\end{equation}

\paragraph{Genericity condition.}
As we mentioned in the introduction,  we
need to impose some restrictions on the (Birkhoff) support $\cS_0$ of the finite gap solutions. Indeed, we ask $\cS_0$ to fulfill some arithmetic conditions which we now describe.

\begin{definition}
\label{defset} 
Order $\cS_0$ so that $\tm_1<\tm_2<\dots<\tm_\tk$. 
 For every $n \in \Z$ let
$$
\cS_{0n} := \{ (\tm_i, n) :  \ \ 1 \leq i \leq \tk  \} \ .
$$ 
For any  $1 \leq i<j \leq \tk$ let 
\begin{equation}
\label{def:lambda}
\ccC_{i,j}^\pm := \{(m,n)\in \Z^2\,:\; (m-\tm_i)(m-\tm_j)+ n^2 =0\,, \quad \pm n> 0\} \ . 
\end{equation}
Finally denote 
$$
\sS := \bigcup_{n \in \Z\setminus \{0\}} \cS_{0n} \ , \qquad \ccC := \bigcup_{i<j}\ccC_{i,j}\,,\quad  \ccC_{i,j}:=\ccC^+_{i,j}\cup \ccC^-_{i,j} \ . 
$$
\end{definition}
\begin{definition}[Arithmetic genericity]
\label{defar}
We  say that $\cS_0$ is {\em generic} if 
\begin{equation}
\label{gen.cond}
\sS \cap \ccC = \emptyset \ , \qquad \ccC_{i,j} \cap \ccC_{i',j'} = \emptyset \ , \quad \forall \{i, j\} \neq \{i', j'\}.
\end{equation}
Given $\tL\in \N$, we say that $\cS_0$ is $\tL$-generic if it is generic and  moreover
\begin{equation}\label{pop}
\sum \ell_i\tm_i\neq 0 \qquad \forall  0<|\ell|\le \tL.
\end{equation}
\end{definition}
The following lemma explains in which sense the ``good'' sets are generic:
\begin{lemma}
\label{rem:inf.s0}
Fix any $\tL\in \N$. There are infinitely many choices of $\tL$-generic sets. More precisely  for $R\gg 1$ let $B_R$ be the set of all ordered sets $\cS=(\tm_1,\dots,\tm_\tk)$ such that $\max(|\tm_i|)\le R$. Then denoting by $G_R$ the $\tL$-{\em generic} sets in $B_R$ (i.e. those which satisfy Definition \ref{defar},) we have
$$
\lim_{R\to \infty} \frac{|G_R|}{|B_R|} =1.
$$

\end{lemma}
The proof of the lemma is postponed in Appendix \ref{generic.infinity}.

%
%\begin{remark}
%{Fix an arbitrary $\ell \in \Z^\tk$, then the equation  $\sum_{i} \ell_i x_i = 0$ defines a hyperplane in $\C^\tk$. Then condition \eqref{pop} means that the integer points $\tm_1, \ldots, \tm_\tk$ must be chosen outside the union of a finite number of  hyperplanes. This is clearly a generic condition.}
%\end{remark}

Let us   motivate the  genericity conditions.
One of the main problems in developing  perturbation theory for NLS on $\T^2$ is the presence of rectangle-resonances, namely quadruple of integers $\jj_1, \jj_2, \jj_3, \jj_4 \in \Z^2$ such that 
\begin{equation}
\label{rec-res}
\jj_1 - \jj_2 + \jj_3 + \jj_4 = 0 \ , \qquad 
|\jj_1|^2 - |\jj_2|^2 + |\jj_3|^2 - |\jj_4|^2 = 0 \ .
\end{equation}
It is easy to check that $\jj_1, \jj_2, \jj_3, \jj_4$ fulfill \eqref{rec-res} if and only if they form the vertex of a rectangle in $\Z^2$, hence the name rectangle-resonance.
In principle we would like to avoid all the resonances of the form \eqref{rec-res} when two $\jj_i$'s are chosen in $\cS_0$ and two are chosen outside $\cS_0$. One realizes immediately that this is not possible: indeed any two $\jj_i$'s chosen in $\cS_{0n}$, $n \neq 0$, will form a {\em horizontal rectangle} with two sites in $\cS_0$. Similarly,  any two $\jj_i$'s chosen in $\ccC_{ij}$, $1 \leq i < j \leq \tk$, will form a {\em rotated rectangle} with two sites in $\cS_0$ (from the very definition of $\cS_{0n}$ and $\ccC$).\\
Then the genericity condition states that there are no intersection at integer points outside $\cS_0$  between an  horizontal rectangle and a rotated rectangle or between two rotated rectangles.

\begin{remark}
{If $\tk>2$, then there are always rotated rectangles. The reason is the following: if $\tm_i - \tm_k$ is an even number, then the point $(\frac{\tm_i + \tm_k}{2}, \frac{\tm_k - \tm_i}{2})$ has integer coordinates and it forms a right angle with $(\tm_i, 0)$, $(\tm_k,0)$, i.e.  it belongs to $\ccC_{ik}$. 
Clearly if the cardinality of $\cS_0$ is at least 3, there are at least two points whose distance is an even number.}
\end{remark}

\begin{figure}[ht]\centering
\vskip-10pt\begin{minipage}[t]{5cm}
{\centering
{\psfrag{m}{$ \tm_2$}
\psfrag{n}{$\tm_1$}
\psfrag{h}{$\tm_3$}
\psfrag{k}{$\tm_4$}
\includegraphics[width=12cm]{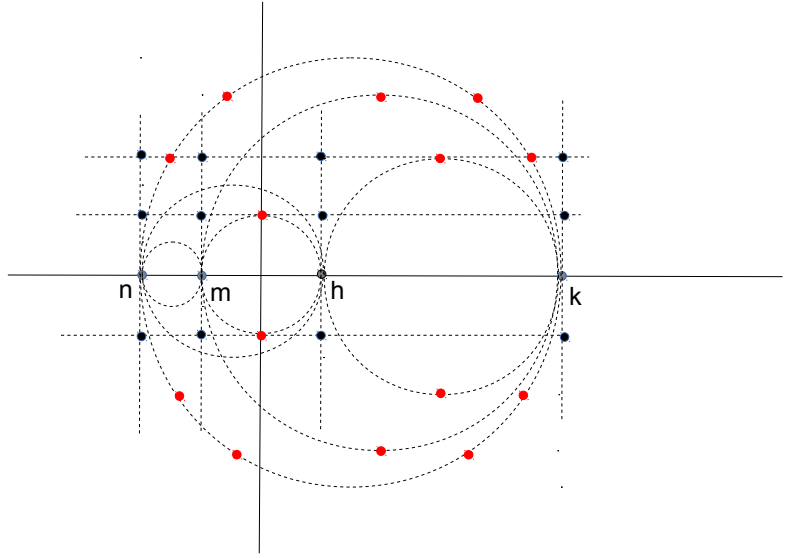}
}}
\end{minipage}
\caption{The red dots on the circles are the $\ccC_{i,j}$  the black dots are $\cS_{0,n}$, the arithmetic genericity condition requires that all the  intersection points of two dotted curves are non-integer. }
\end{figure}

\paragraph{Admissible monomials.}
We  set
	\begin{equation}
		\label{mp.4}
		\begin{aligned}
		&\wtcM:=  \sum_i \yy_i + \sum_{(m , n) \in \Z^2\setminus( \sS \cup \ccC\cup \Tan)  }|a_j|^2 \\
		&\wtcP_x:= \sum_i  \tm_i  \yy_i + \sum_{(m,n) \in \Z^2 \setminus (\sS \cup \Tan\cup \ccC)}\!\!\!\!m \, |a_{(m,n)}|^2 +\!\!
		\sum_{i <k \atop (m,n) \in \ccC^+_{i,k}}\!\!(m - \tm_i)  \left(\ |a_{(m,n)}|^2 - |a_{(\tm_i + \tm_j - m,-n)}|^2 \right)\\
		&\wtcP_y:= \sum_{(m,n) \in \Z^2} n |a_{(m,n)}|^2
		\end{aligned}
		\end{equation} 
		and  give the following 
		\begin{definition}[admissible]
		\label{rem:adm3}
		Given $\bj = (\jj_1, \ldots, \jj_b) \in (\Z^2\setminus\cS_0)^b$, $\ell \in \Z^\tk$ and ${\boldsymbol\sigma}= (\sigma_1, \ldots, \sigma_b) \in \{-1,0, 1\}^b$, we say that $(\bj , \ell, \bs)$
		is {\em admissible}, and denote $(\bj , \ell, \bs) \in {\fA}_b$,  if the monomial $\mathfrak m= e^{\im \theta \cdot \ell} \, a_{\jj_1}^{\sigma_1} \, \ldots a_{\jj_b}^{\sigma_b}$ Poisson commutes with $\wtcM,\wtcP_x, \wtcP_y$.  Here we use the convention that $a^+_\jj=a_\jj, a^-_\jj=\bar a_\jj, a^0_\jj=1$.
	\end{definition}
	\begin{definition}[action preserving]\label{reso}
	Given $\bj = (\jj_1, \ldots, \jj_b) \in (\Z^2\setminus\cS_0)^b$, $\ell \in \Z^\tk$ and $\bs= (\sigma_1, \ldots, \sigma_b) \in \{-1,0, 1\}^b$, we say that $(\bj , \ell, \bs)$ is {\em action preserving}  and denote  $(\bj , \ell, \bs)\in \fR_b$, if  $\ell=0$ and the monomial $\mathfrak m= a_{\jj_1}^{\sigma_1} \, \ldots a_{\jj_b}^{\sigma_b}$ Poisson commutes with  the actions $|a_\jj|^2$ for all $\jj\in \Z^2\setminus\cS_0$.
	\end{definition}

		\paragraph{Birkhoff normal form around a finite gap solution.}
		We are finally ready to state our result on Birkhoff normal form:
\begin{theorem}
	\label{main2}
	There exists $\tL>0$ (depending only on $\td$), such that
	for a $\tL$-generic choice of the set $\Tan$  (in the sense of Definition \ref{defar}) the following holds true.
	Given arbitrary $s_0 >0$, 	there exists $\e_\star >0$ and for any  $0<\e\le \e_\star$,  $r_0\ll \sqrt\e$, 
 there exist a  compact {\em domain} $\cO_1 \subseteq \cO_0 \subseteq [1/2,1]^\tk$,  Lipschitz functions $(\Omega_\jj)_{\jj\in \Z^2\setminus \cS_0}$ defined on $\cO_1$ and  real numbers $\gamma, \tau >0$ s.t. the set 
	\begin{equation}
	\label{3.mc}
	\cC:=\left\{\lambda\in \cO_1:\;	\abs{\omega \cdot \ell + \s_1\Omega_{\jj_1}(\lambda, \e)+ \s_2 \Omega_{\jj_2}(\lambda, \e)+\s_3 \Omega_{\jj_3}(\lambda, \e) } \geq \e \frac{\gamma}{\la \ell \ra^\tau} \ ,\;\forall (\bj , \ell, \bs)\in \fA_3 \setminus \fR_3\right\}
		\end{equation}
has positive measure. 
For each $\lambda \in \cC$ there exists a symplectic change of variables $\cT$,  majorant analytic together with its inverse,  s.t. $\cT$ and $\cT^{-1}$ map $D(s/32, \varrho_0 r) \to D(s,r)$ for all $r \leq r_0$, $s_0/64 \leq s \leq s_0$ (here $\varrho_0>0$ is a parameter depending on $s_0, \max(|\tm_k|^p)$ ). Moreover 
\begin{equation}
\label{ham.bnf.fin}
(\cH\circ \cT)(\yy, \theta, a, \bar a) = \omega \cdot \yy + \sum_{\jj \in \Z^2 \setminus \cS_0} \Omega_\jj\, |a_\jj|^2 + \cR^{(\geq 2)}(\yy, \theta, a, \bar a)
\end{equation}
where $\cR^{(\geq 2)}$ contains just monomials of scaling degree at least 2 and fulfills 
\begin{equation}
\label{R.est0}
\abs{\cR^{(\geq 2)}}_{s/32,\varrho_0 r} \leq C_\cR \,  r^2  \  
\end{equation}
for some positive constant $C_\cR$ independent of $\e$.
\\
The mass $\cM$ and the momentum $\cP$ (defined in \eqref{costanti}) fulfill
$$
\cM\circ \cT= \wtcM\,,\quad \cP\circ \cT= \wtcP \ , 
$$
where $\wtcM$ and $\wtcP$ are defined in \eqref{mp.4}.
\end{theorem}
We are able to describe quite precisely the asymptotics of the frequencies $\Omega_\jj$ of Theorem \ref{main2}. As happens often in higher dimensional settings, such asymptotics depend on the fact that $\jj \in \Z^2 \setminus (\sS \cup \ccC \cup \cS_0)$ or to the sets $\sS$, $\ccC$.
\begin{theorem}
	\label{main3}
Under the same assumptions as Theorem \ref{main2}, for any $ \e\le \e_\star,\lambda\in \cO_1$ the frequencies  $\Omega_\jj \equiv \Omega_\jj(\lambda,\e)$, $\jj = (m,n)\in \Z^2\setminus \cS_0$, have the following asymptotics:
\begin{align}
\label{as.omega0}
&\Omega_\jj(\lambda, \e) =  m^2  + \frac{\varpi_{m}(\lambda, \e)}{\la m \ra} \ , \qquad n  = 0  \ , \\
\label{as.omega}
&\Omega_\jj(\lambda, \e) =  \wtOmega_\jj(\lambda, \e) + \frac{\Theta_{m}(\lambda, \e)}{\la m \ra^2} + \frac{\Theta_{m,n}(\lambda, \e)}{\la m \ra^2 + \la n \ra^2} \ , \qquad n \neq 0 
\end{align}
	where 
	\begin{equation}
	\label{def:omtilde}
	\wtOmega_\jj (\lambda, \e) := 
	\begin{cases}
	m^2 + n^2   , & \jj=(m,n) \not\in \sS \cup \ccC  \ , \\
	\e \mu_i(\lambda) + n^2   \ , & \jj=(\tm_i, n)\ ,  \\
	m^2 + n^2 - \tm_i^2 + \e  \mu_{i,k}^+(\lambda)  \ , & \jj= (m,n) \in \ccC_{i,k} \ , n >0   \\
	 m^2 + n^2 - \tm_k^2 - \e \mu_{i,k}^-(\lambda)   \ , & \jj= (m,n) \in \ccC_{i,k} \ , n <0 
	\end{cases}
	\end{equation}
	Here the $(\mu_i(\lambda))_{1 \leq i \leq \tk }$ are the roots of the polynomial 
	\begin{equation}
	\label{P.poly0bis}
	P(t,\lambda):= \prod_{i=1}^\tk (t + \lambda_i) - 2 \sum_{i=1}^\tk \lambda_i \, \prod_{k \neq i} (t + \lambda_k) \ , 
	\end{equation}
	while for any $1 \leq i<k \leq \tk$ fixed, the  $\mu^+_{i,k}(\lambda), \mu^-_{i,k}(\lambda)$ are the roots of the polynomial
	\begin{equation}
	\label{char02bis}
	Q(t, \lambda_i, \lambda_k) := t^2 - (\lambda_i - \lambda_k) t + 3 \lambda_i \lambda_k \ .
	\end{equation}
Finally $\mu_i$, $\mu_{i,k}^\pm$, $(\varpi_m(\lambda, \e))_{m \in \Z} $,  $(\Theta_m(\lambda, \e))_{m \in \Z}$ and $ (\Theta_{m,n}(\lambda, \e))_{(m,n) \in \Z^2\setminus \Tan}$ fulfill 
\begin{equation}
\label{theta.est}
\begin{aligned} 
& \sum_{1 \leq i \leq \tk} | \mu_i(\cdot) |^{\cO_1} +  \sum_{1 \leq i < k \leq \tk \atop \pm} |  \mu_{ik}^\pm(\cdot) |^{\cO_1} \\
& \qquad
+ \sup_{\e \leq \e_\star  } \frac{1}{\e^2}\Big( 
{\sup_{m \in \Z} |\varpi_m(\cdot, \e)|^{\cO_1}+ }
 \sup_{m \in \Z} |\Theta_m(\cdot, \e)|^{\cO_1} + \sup_{(m,n) \in \Z^2\setminus \Tan} |\Theta_{m,n}(\cdot, \e)|^{\cO_1} \Big) \leq  \tM_0 
 \end{aligned}
\end{equation}
for some positive constant $\tM_0$. 
\end{theorem}

We conclude this section with a list of remarks:
\begin{remark}
\label{rem:mu}
The $(\mu_i(\lambda))_{1 \leq i \leq \tk}$ and the $(\mu_{i,k}^\pm(\lambda))_{1 \leq i<k \leq \tk}$  are distinct  nonzero  algebraic functions which are homogeneous of degree 1.
They depend only on the number $\tk$ of tangential sites, hence  {\em not on the single sites $(\tm_i)_{1 \leq i \leq \td}$}. 
\end{remark}

\begin{remark}
\label{rem:asym}
The asymptotic expansion of the normal frequencies \eqref{as.omega} does not contain any constant term. The reason is that we canceled such term when we subtracted the quantity $M(u)^2$ from the Hamiltonian, see \eqref{parto}. 
Of course if we had not  removed $M(u)^2$, we would have had a constant correction to the frequencies, equal to $\norm{q(\omega t, \cdot)}_{L^2}$.
 Since $q(\omega t, x)$ is a solution of \eqref{NLS}, it enjoys mass conservation, thus $\norm{q(\omega t, \cdot)}_{L^2} = \norm{q(0, \cdot)}_{L^2}$ is independent of time.
\end{remark}

\begin{remark}
\label{rem:faou}
In \cite{faou13} the authors prove that plane waves solutions of \eqref{NLS} are  $H^p$-orbital stable for   times of order $\delta^{-N}$, with $N$ arbitrary chosen, provided $p$ is large enough. 
The reason is the following: plane waves solution have a very specific expression, given by   $A e^{\im (mx - \omega t)}$.
This simple formula  allows to compute exactly the expressions for the  tangential frequency $\omega$ and the normal frequencies $\Omega_\jj$ as a function of $A$, indeed one has  $$
\omega= m^2 + A^2 \ , \qquad \Omega_\jj = \sqrt{|\jj|^4 + 2 |\jj|^2 A^2}  \ . 
$$
Such formulas are much more explicit than the mere asymptotic expansions that we find in \eqref{as.omega}, and allow the authors of \cite{faou13} to impose Melnikov conditions of the form 
$$
\abs{\s_1\Omega_{\jj_1}+\cdots +  \s_n \Omega_{\jj_n}} \geq  \frac{\gamma}{\mu_3(\jj_1, \ldots, \jj_n)^{\alpha}} 
$$
at any order $n$ (here $\mu_3(\jj_1, \ldots, \jj_n)$ is the third largest integer among $|\jj_1|, \ldots, |\jj_n|$). As a consequence they can perform  an arbitrary number of steps of Birkhoff normal form.\\
On the contrary we have a weaker control on the asymptotics of the frequencies, since we know their exact expression only at order $\e$ and not at higher orders in $\e$; this allows us to impose Melnikov conditions up to order three, see \eqref{3.mc}.
\end{remark}

\begin{remark}
\label{rem:procesi}
In \cite{procesi15,procesi16} the authors show that for a generic choice of tangential sites  $\cS_0' \subseteq \Z^2$, it is possible to construct families of 
small quasi-periodic solutions of \eqref{NLS} which are (essentially) supported on $\cS_0'$. Such solutions give rise to finite dimensional KAM tori which depend on both variables $x$ and $y$ and which  are linearly orbitally stable (namely stable for times  $|t| \sim \delta^{-1}$). 
While we borrow ideas and techniques from such papers, the construction that we perform requires extra-care. Indeed the condition of genericity for $\cS_0'$ used in \cite{procesi15,procesi16} does not allow to choose $\cS_0' \subset \Z \times \{ 0\}$, as we do here. 
Furthermore in \cite{procesi15,procesi16} the asymptotics of the normal frequencies $\Omega_\jj$ are less explicit than formula \eqref{as.omega}, therefore it is not clear if it is possible, in the setup of \cite{procesi15,procesi16},  to impose the third order Melnikov conditions \eqref{3.mc} and  obtain nonlinear stability of the KAM tori. 
\end{remark}

\section{Functional setting}
\label{fun}

In this section we introduce the technical apparatus that we will use in the following. \\
Given $\al,\bt\in \N^{\Z^2\setminus \cS_0}$,   to the monomial $e^{\im \ell\cdot \theta}\yy^l a^\al \bar a^\bt$ we associate various numbers.
First we denote by 
\begin{equation}
\label{def.eta}
\eta(\alpha, \beta) := \sum_{\jj \in \Z^2\setminus \cS_0} (\alpha_\jj - \beta_\jj) \ , \quad \qquad \eta(\ell):= \sum_{i=1}^\tk \ell_i \ .
\end{equation} 
The second quantity that we associate to $e^{\im \ell\cdot \theta}\yy^la^\al \bar a^\bt$ are  the momentum  $\pi({\al,\bt})=(\pix,\piy)$ and $\pi(\ell)$ defined by 
\begin{equation}
\label{def.pi}
\pi({\al,\bt})=  \begin{bmatrix} \pi_x(\alpha, \beta) \\ \pi_y(\alpha, \beta) \end{bmatrix} = \sum_{\jj=(m,n)\in \Z^2\setminus \cS_0} 
\begin{bmatrix}  m \\ n \end{bmatrix} (\al_\jj-\bt_\jj) \ , \quad \qquad 
\pi(\ell)= \sum_{i=1}^\tk \tm_i \ell_i  \ . 
\end{equation}
Given a monomial $e^{\im \theta\cdot \ell}\yy^la^\al \bar a^\bt$ we have the commutation rules
\begin{align*}
%\{\cN,e^{\im \theta\cdot \ell}\yy^la^\al \bar a^\bt\}&=\im (\Omega^{(0)}\cdot (\al-\bt)+\omega\cdot \ell )e^{\im \theta\cdot \ell}\yy^la^\al \bar a^\bt\,,\quad
&\{\cM,e^{\im \theta\cdot \ell}\yy^la^\al \bar a^\bt\}=\im (\eta(\alpha, \beta)+\eta(\ell) )e^{\im \theta\cdot \ell}\yy^la^\al \bar a^\bt\\
\{\cP_x,e^{\im \theta\cdot \ell}\yy^la^\al \bar a^\bt\}&=\im (\pi_x(\alpha, \beta)+\pi(\ell) )e^{\im \theta\cdot \ell}\yy^la^\al \bar a^\bt \,,\quad
\{\cP_y,e^{\im \theta\cdot \ell}\yy^la^\al \bar a^\bt\}=\im \, \pi_y(\alpha, \beta)\,  e^{\im \theta\cdot \ell}\yy^la^\al \bar a^\bt
\end{align*} 
\begin{remark}
	\label{leggi_sel}
	A  function  $\cF$ commutes with the mass $\cM$ and the  momentum $\cP$  defined in \eqref{costanti}  if and only if the following {\em selection rules} on its coefficients hold:
	\begin{align*}
	&\{ \cF, \cM\} = 0 \ \ \  \Leftrightarrow   \ \ \  \cF_{\alpha, \beta, \ell} \, (\eta(\alpha, \beta) + \eta(\ell)) = 0 \\
	& \{ \cF, \cP_x\} = 0  \ \ \ \Leftrightarrow  \ \ \    \cF_{\alpha, \beta, \ell} \, (\pi_x(\alpha, \beta) + \pi(\ell)) = 0 \,,\quad  \{ \cF, \cP_y\} = 0  \ \ \ \Leftrightarrow  \ \ \    \cF_{\alpha, \beta, \ell} \, (\pi_y(\alpha, \beta)) = 0 
	\end{align*}
	where $\eta(\alpha, \beta), \eta(\ell)$ are defined in \eqref{def.eta} and $\pi(\alpha, \beta), \pi(\ell)$ are defined in \eqref{def.pi}.
\end{remark}

{From now on we shall always assume that our  Hamiltonians commute with $\cM$, $\cP$, so the selection rules of Remark \ref{leggi_sel} hold.}

%We start by defining  ultraviolet cut offs
%$$
%\Pi_{\leq N} h:= \sum_{\ell\in \Z^\tk,\al ,\bt\in\N^{\Z^2}  \atop 
%\la\ell\ra+\la\pi({\al,\bt})\ra \leq \frac{\sqrt{N}}{4}} h_{\al,\bt,\ell} e^{\im \ell\cdot \theta} \mathfrak m_{\al,\bt}
%$$
%and set $\Pi_{>N}= I-\Pi_{\leq N}$ (note the ugly scaling...maybe we should call this $\Pi_{\leq \sqrt N/4}$).
\begin{definition}
We will denote 
 by $\cA_{s,r}$  the {subspace of real  valued functions in the }  closure of the  monomials ($e^{\im \ell \cdot \theta} \yy^l a^\al \bar a^\beta$) which Poisson commute with $\cM$ and $\cP$,  with respect to the  norm $| \cdot |_{s,r}$ (defined in \eqref{musically}). Such Hamiltonians will be  called {\em regular}.
  Given a compact set $\cO\subset \R^\tk$,  we denote by $\cA^\cO_{s,r}$ the Banach space of Lipschitz maps $\cO\to \cA_{s,r}$ with finite  norm $|\cdot|_{s,r}^\cO$ (defined in \eqref{lip.norm0}).
\end{definition}

%\begin{remark}
%Note that by construction the Birkhoff coordinate $z$ has scaling degree one, $\yy$ has scaling degree two and $\theta$  zero.  The scaling degree of  a monomial $e^{\im \ell\cdot \theta} \yy^l\mathfrak m_{\al,\bt}$ is then $	2|l|+|\al|+|\bt|-2$.
%\end{remark}
In the next proposition we describe some basic properties of the norm $|\cdot|_{s,r}$  which will be used repeatedly in the  paper. 

\begin{proposition}
\label{prop:mer}
 For every $ s,r>0$ the following holds true:
\begin{itemize}
		\item[(i)] {\em Degree decomposition:} 
	given a  Hamiltonian $h\in \cA_{s,R}^\cO$ which is homogeneous of degree $d$,  then  $h\in \cA_{s,r}^\cO$ for all $r>0$ and  one has 
	\begin{equation}\label{scala0}
	|h|_{s,r}^\cO \leq \left(\frac{r}{R}\right)^d |h|_{s,R}^\cO \ . 
	\end{equation}
	The same estimate  holds also if $d$ is the minimal degree and $r \leq R$.
	\item[(ii)] {\em Changes of variables}: Let $h,\, f \in \cA_{s,r}^\cO$. For any $0<s' <s$ and  $0<r' < r$, let $\delta := \min \left( 1-\frac{r'}{r}, s-s'\right)$. If $\delta^{-1} |f|_{s,r}^\cO<\varrho$ sufficiently small then the  Hamiltonian vector field $X_f$ defines a close to identity canonical change of variables
	$\cT_f$ such that 
	$$
	 |h\circ\cT_f|_{s', r'}^\cO \leq (1+C\varrho)|h|_{s,r}^\cO \ , \qquad \forall 0<s' <s \ , \ \ 0<r' < r \ . 
	$$
	\item[(iii)]{\em Remainder estimates}: Let  $f,g \in \cA_{s,r}^\cO$  of minimal scaling  degree respectively $\td_f$ and   $\td_g$ and define the function 
	\begin{equation}\label{taylor}
	\re{\ti}(f; h)=\sum_{l=\ti}^\infty  \frac{(-\ad f)^l}{l!} h\,,\quad \ad(f)h:= \{h,f\} \ . 
	\end{equation} 
Then $\re{\ti}(f; g)$ is of minimal scaling degree $\td_f \ti   +\td_g$ 
	and we have the bound 
	$$
	\abs{\re{\ti}(f; h)}_{s',r'}^\cO \leq C(s) \delta^{-\ti} \left(|f|_{s,r}^\cO\right)^\ti \,  |g|_{s,r}^\cO \ ,  \qquad \forall 0<s' <s \ , \ \ 0<r' < r  \ .
	$$
	Note that the same holds if we substitute in \eqref{taylor}  the sequence $\{\frac{1}{l!}\}$  with any sequence $\{b_l\}$ such that   $|b_l|\leq \frac{1}{l!}$ $\ \forall l$ .
\end{itemize}
\end{proposition}
The proof of the proposition, being quite technical, is postponed in Appendix \ref{app:fun}.

\subsection{Quadratic  Hamiltonians}
\label{ham.pseud}
Inside $\cA_{s,r}$ are the polynomial subspaces, i.e. spaces of fixed scaling degree which we shall denote by $\cA^{(d)}_s$.  By \eqref{scala0}, the polynomials in such spaces  are analytic for all $r>0$. 

Here we want to study a special subspace of $\cA_s^{(0)}$ defined as follows:
\begin{definition}
We denote by $\cQ_s^\cO$ the subspace of $\cA_s^{(0)}$ which contains  Hamiltonians  quadratic in $\ba$.
\end{definition}
In particular $\cQ_s^\cO$ is the subspace of real  valued quadratic functions in the closure of the monomials of the form $e^{\im \ell \cdot \theta} a^\al \bar a^\bt$ with $|\al|+|\bt|=2$.
Now we study some properties  of $\cQ_s^\cO$. 
  For such  Hamiltonians, by \eqref{scala0}, the norm $|\cdot|_{s,r}$  is independent from the domain $D(r)$, hence  we always suppose that $r=1$ and we shall drop it:
$$|h|^\cO_{s,r}=|h|^\cO_{s,1}\equiv |h|^\cO_s \ . $$
 We decompose any quadratic  Hamiltonian $H \in \cQ_s^\cO$, $s >0$,  in three components:
\begin{equation}
H(\lambda;\, \theta, a, \bar a) = H^{\rm line} (\lambda;\,\theta, a, \bar a) + H^{\rm diag}(\lambda;\,\theta, a, \bar a)  + H^{\rm out}(\lambda;\,\theta, a, \bar a)  \ .
\end{equation}
Each component contains  monomials supported in different regions of $\Z^2$, which now we describe  in more details.
\begin{itemize}
\item[(i)] $H^{\rm line}(\lambda;\,\theta, a, \bar a) $ contains  only monomials with $|\alpha| + |\beta| =2$ and  ${\rm supp }\, \al \bigcup {\rm supp }\, \bt \subseteq \Z \times \{0\}$:
$$
H^{\rm line}(\lambda;\,\theta, a, \bar a) = \sum_{m_1, m_2 \in \Z} H^{-}_{m_1, m_2}(\lambda;\,\theta) \ a_{(m_1,0)} \bar a_{(m_2,0)} + 2{\rm Re}( H^{+}_{m_1, m_2}(\lambda;\,\theta) \, a_{(m_1,0)}  a_{(m_2,0)}) \ , 
$$
$$
\mbox{with}\quad H^{-}_{m_1, m_2}= \overline{H^{-}_{m_2, m_1}}\,,\quad  H^{+}_{m_1, m_2}=  H^{+}_{m_2, m_1}.
$$

\item[(ii)] 
$H^{\rm diag}(\lambda;\,\theta, a,\bar a) $ contains monomials with $|\alpha| = |\beta| =1$ and ${\rm supp }\, \al \bigcup {\rm supp}\, \bt \subseteq \Z^2 \setminus \Z $ (namely is the diagonal part of the representative matrix):
$$
H^{\rm diag}(\lambda;\,\theta, a, \bar a) = \sum_{m_1, m_2 \in \Z \atop n \in \Z \setminus \{0\}} H^{-}_{m_1, m_2, n}(\lambda;\,\theta) \ a_{(m_1,n)} \bar a_{(m_2,n)} \,,\qquad  H^{-}_{m_1, m_2, n}= \overline{H^{-}_{m_2, m_1, n}}
$$
\item[(iii)] 
 $H^{\rm out}(\lambda;\,\theta, a, \bar a) $ contains monomials with $|\alpha| = 2$ or $ |\beta| =2$ and  ${\rm supp }\, \al \bigcup {\rm supp }\, \bt \subseteq \Z^2 \setminus \Z $ (namely is the out-diagonal part of the representative matrix):
$$
H^{\rm out}(\lambda;\,\theta, a, \bar a) = 2{\rm Re}(\sum_{m_1, m_2 \in \Z \atop n >0} H^{+}_{m_1, m_2,n}(\lambda;\,\theta) \, a_{(m_1,n)} a_{(m_2,-n)})
$$
\end{itemize}
Note that the sum is over only one index $n$, this is due to the conservation of $\cP_y$. Furthermore  the condition $n>0$ is used in order to avoid repetitions.
It is easy  to see that  $H^{\rm line}$, $H^{\rm diag}$, $H^{\rm out}$ are  in $\cQ_s^\cO$ (see Proposition \ref{riassunto}$(i)$).
We will denote $\cQ^{\cO, \rm diag}_{s}$ (respectively $\cQ^{\cO, \rm out}_{s}$, $\cQ^{\cO, \rm line}_{s}$) the relative subspaces.
\begin{remark}The following commutation rules hold:
\begin{itemize}
\item[(i)] $\{ H^{\rm diag}, K^{\rm diag} \} \in \cQ^{\cO,\rm diag}_s \ $, $\{ H^{\rm diag}, K^{\rm out} \} \in \cQ^{\cO,\rm out}_s \ $, 
$\{ H^{\rm out}, K^{\rm out} \} \in \cQ^{\cO,\rm diag}_s$
\item[(ii)] $\{ H^{\rm diag}, K^{\rm line} \} =  \{ H^{\rm out}, K^{\rm line} \}  = 0$.
\end{itemize}
\end{remark}
 
 %Such concept is the natural extension to matrixes of the concept of pseudodifferential operators. 
Consider a quadratic  Hamiltonian $H(\theta, a, \bar a)$ as above and expand its coefficients in Fourier series
\begin{equation}
H_{m_1, m_2, n}^{\pm}(\lambda;\theta) = \sum_{\ell \in \Z^\tk } H_{m_1, m_2, n}^{\pm, \ell}(\lambda) \ e^{\im \theta \cdot \ell} \ . 
\end{equation}
In such a way, to any  Hamiltonian $H \in \cQ_s$ we can associate a sequence $\{H^{\s,\ell}_{m_1,m_2,n}\}$ as above.
On the contrary,  a sequence  $\{H^{\s,\ell}_{m_1,m_2,n}\}$ corresponds to a Hamiltonian only if it fulfills the constraints which arise from the fact that the corresponding Hamiltonian should be real. \\
Given a sequence $\{H^{\s,\ell}_{m_1,m_2,n}\}$, we define the norm
\begin{equation}
\label{formula}
 |\{H^{\s,\ell}_{m_1,m_2,n}\}|_{s}:= \sup_{\|\ba\|\le 1} \|\{\sum_{m_1,\ell} e^{|\ell|s }|H^{-,\ell }_{m_1,m_2,n}|\, a_{(m_1,n)} +  e^{|\ell|s }| {\overline{ H}}^{+,\ell}_{m_2,m_1,n}| \, \bar a_{(m_1,-n)} \}  \| \ . 
\end{equation}
  Then  $|H|_s  \equiv |\{H^{\s,\ell}_{m_1,m_2,n}\}|_{s}$.
The same for the corresponding Lipschitz norm.

\begin{remark}
\label{sel.rule.q}
By mass and momentum conservation 
\begin{equation}
\label{coef.0} 
H_{m_1, m_2, n}^{\sigma, \ell}(\lambda)\neq 0 \quad  \mbox{ only if } \quad \pi(\ell) + m_1+\sigma m_2 = 0 \ , \  \eta(\ell) + 1+ \sigma = 0 \ ,
\end{equation}
where  $\sigma$ is either  $+1$ or $-1$. 
\end{remark}
\begin{remark}
\label{rem:mm.use}
In particular, for a monomial  $e^{\im  \theta \cdot \ell} \, a_{(m_1, n_1)} \bar a_{(m_2, n_2)}$ mass and momentum are
$$
\eta(\ell) = 0 \ , \ \ \pi(\ell) + m_1 - m_2  = 0 \ , \ \  n_1 - n_2 = 0  \ , 
$$
while for 
$e^{\im  \theta \cdot \ell} \, \bar a_{(m_1, n_1)} \bar a_{(m_2, n_2)}$ they are
$$
\eta(\ell) - 2 = 0 \ , \ \ \pi(\ell) - m_1 - m_2  = 0 \ , \ \  n_1 + n_2 = 0  \ .
$$
\end{remark}
\begin{remark}
\label{rem.m}
A quadratic   Hamiltonian  in $\cA_{s,r}^\cO$ can be canonically identified with a majorant bounded matrix $M(\theta)$ by setting $H= \frac{1}{2}(\ba , J^{-1} M(\theta) \ba )$, where $J^{-1}= \begin{pmatrix}
	0 & -1 \\
	1 & 0
	\end{pmatrix} $. It is easily seen that $| H|_{s}$ is equivalent to the operator norm on $\cL(\li^p,\li^p)$ of the matrix $\sum_{\ell}\und M(\ell) e^{|\ell|s}$, hence we  have the algebra property
\begin{equation}\label{alge}
| \{H,K\}|_s \le C_0   |H|_s |K|_s.
\end{equation}
\end{remark}

 Next we  introduce the concept of {\em order} of a quadratic  Hamiltonian.
\begin{definition}[Order of a Hamiltonian]
	\label{def:pseudo}
Given ${\bf N}=(N_1, N_2) \in \Z^2$, we say that a quadratic  Hamiltonian with coefficients $\{H_{m_1, m_2, n}^{\pm, \ell}\}$ has {\em order} ${\bf N}$ iff there exists $s >0$ s.t. the sequence 
$\{ (\la m_1\ra^{-N_1} + \la n \ra^{-N_2}) H_{m_1, m_2, n}^{\pm, \ell}\}$ has a finite $|\cdot|_{s}^\cO$ norm. We will denote
$$
\lceil H \rfloor^\cO_{s,{\bf N}} :=  | \{(\la m_1\ra^{-N_1}+ \la n \ra^{-N_2}) H_{m_1, m_2, n}^{\pm, \ell} \}|^\cO_s %\equiv \sup_{ |a|_1 \leq 1} \left( \sum_{m_2, n} |a_{m_2,n}| \sum_{m_1, \ell,\pm} \la m_1\ra^{-N_1}\la n \ra^{-N_2} |H_{m_1, m_2, n}^{\pm, \ell}| e^{s  |\ell|} \right) \ . 
$$
We denote by $\cQ^\cO_{s, \bN}$ the subset of  Hamiltonians in $\cQ^\cO_s$ of order $\bN$. In the same way given a   linear map 
$$ \ba \to L \ba \,,\quad (L \ba)_{(m,n)} =\sum_{m_1}\Big(L^{-}_{m_1,m,n}(\lambda;\theta) a_{({m_1,n})}+ L^{+}_{m_1,m,n}(\lambda;\theta) \bar a_{({m_1,-n})}\Big)  $$
we set 
$$
Y:=\sum_{m_1,m_2,n}(\la m_1\ra^{-N_1}+ \la n \ra^{-N_2})\Big(  L^{-}_{m_1,m_2,n}(\lambda;\theta) a_{({m_1,n})}+ L^{+}_{m_1,m_2,n}(\lambda;\theta) \bar a_{({m_1,-n})}\Big) \partial_{a_{(m_2,n)}} + \mbox{c.c.} 
$$
and define
\begin{equation}\label{eye of the tiger}
\lceil L\rfloor_{s,{\bf N}} :=  \sup_{(\yy,\ba)\in D(r)}\bnorm {\underline{Y}_s}_r \ .
\end{equation}
Same for the Lipschitz norm.
\end{definition}
\begin{remark}
\label{rem:mappare}
Note that in the previous definition we could have used the index $m_2$ instead of $m_1$ or we could have used a symmetric weight  (say  $\la m_1\ra +\la m_2\ra$ ) in the definition of order. In fact such definitions are equivalent due to momentum conservation. Our definition is constructed so that, if ${\bf N}=(N,N)$, then a linear operator  $L$ with $\lceil L\rfloor_{s,{\bf N}} <\infty$ maps $\li^{p+N}\to \li^p$.
\end{remark}
\begin{remark}
Let ${\bf N} \in \N^2 $, then clearly
$$
\lceil H \rfloor_{s,{\bf N}}^\cO \leq | H|_s^\cO \leq \lceil H \rfloor_{s,-{\bf N}}^\cO \ . 
$$
\end{remark}
\begin{remark}
For any $\bN \in \N^2$ with non negative components we have
	\begin{equation}
\label{bagnetto}
\lceil A \rfloor_{s,- (N_1,0)}^\cO, \ \lceil A \rfloor_{s,- (0,N_2)}^\cO \le	\lceil A \rfloor_{s,- \bN}^\cO\le  	\lceil A \rfloor_{s,- (N_1,0)}^\cO + \lceil A \rfloor_{s,- (0,N_2)}^\cO
	\end{equation}
\end{remark}
\begin{lemma}[Bony decomposition]
\label{boni}
Set  $c^{-1}:= 2 \max_{1 \leq i \leq \tk}(|\tm_i|)$ and  write 
\begin{align*}
H  &= H^B + H^R  \ , \\
& H^B:= \sum_{{\s \in \pm, \ell \in \Z^\tk, m_1,m_2, n : \atop
		m_1+\s m_2+\pi(\ell)=0}\atop |\ell|\le  c \la m_1 \ra } H_{m_1, m_2, n}^{\s, \ell}(\lambda) \ e^{\im \theta \cdot \ell}a_{(m_1, n)} a^\s_{(m_2, -\s n)} \ , 
	\\
	&
		H^R:=  \sum_{{\s \in \pm, \ell \in \Z^\tk, m_1,m_2, n : \atop
		m_1+\s m_2+\pi(\ell)=0}\atop |\ell|>  c \la m_1 \ra } H_{m_1, m_2, n}^{\s, \ell}(\lambda) \ e^{\im \theta \cdot \ell}a_{(m_1, n)} a^\s_{(m_2, -\s n)} \ . 
\end{align*}
%H=\sum_{{\ell \in \Z^\tk, m_1,m_2: \atop
%		m_1+\s m_2+\pi(\ell)=0}\atop |\ell|\le  c|m_1| } H_{m_1, m_2, n}^{\s, \ell}(\lambda) \ e^{\im \theta \cdot \ell}a_{m_1} a^\s_{m_2} + \sum_{{\ell \in \Z^\tk, m_1,m_2: \atop
%		m_1+\s m_2+\pi(\ell)=0}\atop |\ell|>  c|m_1| } H_{m_1, m_2, n}^{\s, \ell}(\lambda) \ e^{\im \theta \cdot \ell}a_{m_1} a^\s_{m_2}:= H^B+ H^R
%$$
For any $0 < s' <s$ set  
\begin{equation}
\label{delta}
\delta:= s-s' \ .
\end{equation}
We have that for  all $k>0$ 
\begin{equation}
\label{HResto}
\lceil H^R \rfloor_{s',{-(N_1+ k,N_2)}}^\cO \le  k! \, \delta^{-k} (1+ c^{-k})  \lceil H \rfloor_{s,{-(N_1,N_2)}}^\cO \ ,
\end{equation}
namely the part $H^R$ is infinitely smoothing in the $x$ direction.
\end{lemma} 
The proof of the Lemma is postponed in Appendix \ref{app:fun}.
We  now discuss some algebra properties of the norm $\lceil \cdot \rfloor_{s,{\bf N}}^\cO$.
\begin{lemma}
	\label{lem:pseudo00}
	Fix $\bN,\bM \in \N^2$. Let $F\in\cQ^\cO_{s, -\bN}$, $G \in \cQ^\cO_{s,-\bM}$. The following holds true:
	\begin{itemize}
	\item[(i)] For any $0 < s' <s$ one has  $\{ F, G \} \in \cQ^\cO_{s', -\bN-\bM}$  with the quantitative estimate 
	\begin{equation}
	\label{prop.smooth.bra}
	\lceil \{ F, G \} \rfloor_{s', -\bN-\bM}^\cO \leq C_{\bN,\bM} (\,  \, \lceil F \rfloor_{s',-\bN}^\cO \,  \lceil G \rfloor_{s',-\bM}^\cO  + \delta^{-N_1-M_1}| F |_{s}^\cO \,  | G |_{s}^\cO) \ .
	\end{equation}
Here $C=C_{\bN,\bM}>0$ does not depend on $s',s$ and  $\delta$ is defined as in \eqref{delta}. 

\item[(ii)] Let $G \in \cQ_s^\cO$ and  let $F$ fulfill 
\begin{equation}
\label{F.ass}
\lceil F \rfloor_{s, -\bN}^\cO < \eta := \frac{1}{C_{0} } \ ,
\end{equation}
where here (and everywhere below) $C_0$ is the algebra constant of formula \eqref{alge}. 
 Then the Hamiltonian flow $\cT_F$  is a smoothing perturbation of the identity and $\forall \ 0 <s' <s$
\begin{equation}
\label{flow.smooth.est}
\lceil G\circ \cT_F - G \rfloor_{s',-\bN}^\cO \leq C_{\bN}\frac{|G|_{s}^\cO\left(\lceil F \rfloor_{s,-\bN}^\cO+\delta^{-N_1} | F |_{s}^\cO\right)}{1 - \eta^{-1} | F |_{s}^\cO} \ .
\end{equation}

\item[(iii)] Let $G \in \cQ_s^\cO$. Then $\forall \ti \in \N$
the operator $\re{\rm i}(F; G)$ is of order $-\ti \bN$ and $\forall 0 < s' <s$  one has 
$$
\lceil \re{\,\rm i}(F; G) \rfloor_{s',-\ti \bN}^\cO \leq C_{\ti,\bN}\frac{|G|_{s}^\cO\left(\delta^{-N_1} \lceil F \rfloor_{s,-\bN}^\cO\right)^\ti}{1 - \eta^{-1} | F |_{s}^\cO} \ .
$$
Note that the same holds if we substitute in \eqref{taylor}  the sequence $\{\frac{1}{l!}\}$  with any sequence $\{b_l\}$ such that  $\forall l$ one has  $|b_l|\leq \frac{1}{l!}$.

	\end{itemize}
\end{lemma}
The proof of the lemma is postponed in Appendix \ref{app:fun}.

\begin{remark}
	\label{paralin}
	Consider the Poisson bracket $\{F, G\}$ of two quadratic Hamiltonians $F \in \cQ_{s, -\bN}$, $G \in \cQ_{s, -\bM}$. Then
	$$
	\{F,G\}= \{F^B,G^B\}+\{F,G^R\}+\{F^R,G^B\}  
	$$ 
	and the last two terms $\{F,G^R\}$, $\{F^R,G^B\} $ are infinitely smoothing.  
\end{remark}

	\subsection{Quasi-T\"oplitz structure}
Consider the term $\cH^{(\geq 0)}$   in  \eqref{ham.1step}
and denote by $\cH^{(0)}$ its component of scaling degree $0$, which is given explicitly in \eqref{def:H0}. 
It turns out that $\cH^\0$ is quadratic in $\ba$ and its  coefficients  are independent of $n$, for $n\neq 0$. 
We would like to preserve such property, but during the process of normal form that we perform, we cannot avoid generating coefficients which are depending on $n$. 
In turns out that at each step of the normal form procedure (and later on during the KAM iteration), we will have  Hamiltonians whose coefficients have this  specific form:
	$$
	H_{m_1, m_2, n}= H^\hor_{m_1, m_2} + H^\mix_{m_1, m_2,n}
	$$ 
	where $\{H^\hor_{m_1, m_2} \}$ is independent of $n$, while $\{ H^\mix_{m_1, m_2,n}\}$ depends on $n$ but  has order $(-2,-2)$.
	In the rest of the section we will show that the transformations which we will perform preserve such structure. To do so, we need some more notations:
\begin{definition}	
	 We say that a quadratic Hamiltonian $H \in \cQ_s^\cO$ is {\em horizontal} if its coefficients  $\{ H^{\pm, \ell}_{m_1, m_2,n} \}_{ n\neq 0}$ are independent of $n$, i.e. $H^{\pm, \ell}_{m_1, m_2,n} \equiv H^{\pm, \ell}_{m_1, m_2}$  for all $n\neq 0$. We will denote by $\cQ_{s, M}^{\rm hor}  $ the subspace of $\cQ_s^\cO$ of horizontal  Hamiltonians of order $ (M,0)$. 
	 We will denote the corresponding norm as $\lceil \cdot \rfloor_{s, M}$. 
	If $M=0$   we will write just $\cQ_{s}^{\rm hor}  $.
	\end{definition}
	Note that $\cQ_{s, M}^{\rm hor} $ depends on $\cO$. If we need to evidence such dependence we will write $\cQ_{s, M}^{\cO, \hor}$. 
	However, in most of our algorithms $\cO$ is fixed, hence we avoid to write it explicitly. {Note that the condition $H\in  \cQ_{s}^{\cO, \hor}$ does not impose any restriction on the component $H^{\rm line}$.}
%	\begin{remark}\label{pirlino} Note that $\cH^\0\in \cQ_s^\hor$.
%		\end{remark}
%	{\color{red} al momento non e' ancora definita} 
\begin{remark}
\label{rem:hor}
$\cQ_{s, M}^{\rm hor}  $ is closed under  Poisson bracket and composition with flows: if $F, G \in \cQ_s^{\rm hor} $, then for any $0 < s' < s$, 
$$\{ F, G \} \in \cQ_{s'}^{\rm hor} \ , \quad \re{\ti}(F; G) \in \cQ_{s'}^{\rm hor} \ , \quad  G \circ \cT_F \in \cQ_{s'}^{\rm hor} \ .$$
\end{remark}
In the following we want to study the flow generated by an  Hamiltonian of the form $F = F^{\rm hor} + F^\mix $, where $F^\hor \in  \cQ^\hor_{s,-1}$ or $F^\hor \in  \cQ^\hor_{s,-2}$  while $F^\mix \in \cQ_{s, -{\bd}}$ is of order $-\bd:=(-2,-2)$.
\begin{lemma}
\label{lem:smo}
For $a=1,2$, consider $F = F^\hor + F^\mix$ with $F^\hor \in \cQ^\hor_{s,-a}$, $F^\mix \in \cQ_{s,-\bd}^\cO$, and  $G = G^\hor + G^\mix$, $G^\hor \in \cQ^\hor_{s}$, $G^\mix \in \cQ_{s,-\bd}^\cO$.  Set\footnote{here $C_0$ is the algebra constant in \eqref{alge}} $\eta= \min(C_0^{-1},C_{\bd}^{-1})$,  and assume that $\eta^{-1} \abs{F}_{s}^\cO<\frac12 $. The following holds:
\begin{itemize}
	\item[(i)] 
For every $\ti \in \N$ the  Hamiltonian $\re\ti(F; G) $ equals $\re{\ti}(F; G)^\hor + \re{\ti}(F; G)^\mix$ with $\re{\ti}(F; G)^\hor= \re{\ti}(F^\hor; G^\hor) \in \cQ^\hor_{s',-\ti a}$ and $\re{\ti}(F; G)^\mix \in  \cQ_{s',-\bd}^\cO$ for all $0<s'< s$. Furthermore one has the quantitative estimate
\begin{align*}
\lceil \re{\ti}(F; G)^\hor \rfloor_{s', -\ti a }^\cO &\leq C_{a,\ti}\frac{|G^\hor|_{s}^\cO\left(\delta^{-a}\lceil F^\hor \rfloor_{s,-a}^\cO\right)^\ti}{1 - \eta^{-1} \abs{F}_{s}^\cO}\ ,
\end{align*}
for $\ti\ge 0$, here $\delta$ is defined in \eqref{delta}. Similarly for $\ti\ge 1$  we have
\begin{align*}
\lceil \re{\ti}(F; G)^\mix \rfloor_{s',-\bd}^\cO 
&\leq 
 \left(
 \lceil G^\mix\rfloor_{s, -\bd}^\cO \, |F|_s^\cO 
 +  |G|_{s}^\cO \lceil F^\mix \rfloor_{s,-\bd}^\cO
+ |G|_{s}^\cO \, |F|_s^\cO 
  \right)
  \frac{ 2 \delta^{-2} C_{\bd}\left(\eta^{-1} \abs{F }_{s}^\cO\right)^{\ti-1}}{1 - \eta^{-1} \abs{F}_{s}^\cO} 
\end{align*}
 Note that the same holds if we substitute in \eqref{taylor}  the sequence $\{\frac{1}{l!}\}$  with any sequence $\{b_l\}$ such that  $\forall l$ one has  $|b_l|\leq \frac{1}{l!}$.
\item[(ii)]  The Hamiltonian vector field $X_F$ generates a well defined flow $\Phi^\tau_F$ for $\tau\le 1$. Moreover
$$
%|(\Phi^\tau_F -\id) |_{\rm Op}:=  
\sup_{(\yy, \theta, \ba)\in D(s,r)} \bnorm{(\Phi^\tau_F -\id)(\ba,\yy,\theta) }_{r}\le  \tau \eta^{-1} \abs{F}_{s}^\cO \ .
$$
\end{itemize}
\end{lemma}
The lemma is proved in Appendix \ref{app:fun}.
%\end{document}
\begin{lemma}
	\label{lem:smo2}
	Consider $F = F^\hor + F^\mix$and  $G = G^\hor + G^\mix$ such that $F^\hor,G^\hor \in \cQ^\hor_{s,-2}$, $F^\mix,G^\mix \in \cQ_{s,-\bd}^\cO$.  Set\footnote{here $C_0$ is the algebra constant in \eqref{alge}} $\eta= \min(C_0^{-1},C_{\bd}^{-1})$,  and assume that $\eta^{-1} \abs{F}_{s}^\cO<\frac12 $. The following holds:
	\begin{itemize}
		\item[(i)] 
		For every $\ti \in \N$ the  Hamiltonian $\re\ti(F; G) $ equals $\re{\ti}(F; G)^\hor + \re{\ti}(F; G)^\mix$ with $\re{\ti}(F; G)^\hor= \re{\ti}(F^\hor; G^\hor) \in \cQ^\hor_{s,-2}$ and $\re{\ti}(F; G)^\mix \in  \cQ_{s,-\bd}^\cO$. 
		Furthermore one has the quantitative estimates
		\begin{align*}
		\lceil \re{\ti}(F; G)^\hor \rfloor_{s, -2 }^\cO &\leq \frac{|G^\hor|_{s,-2}^\cO\left(\eta^{-1} \lceil F^\hor \rfloor_{s,-2}^\cO\right)^\ti}{1 - \eta^{-1} \abs{F^\hor}_{s,-2}^\cO}\ ,
		\end{align*}
		for $\ti\ge 0$. Similarly for $\ti\ge 1$  we have
			\begin{align*}
		\lceil \re{\ti}(F; G)^\mix \rfloor_{s, -\bd }^\cO &\leq \frac{\lceil G^\mix\rfloor_{s,-\bd}^\cO\left(\eta^{-1} \lceil F \rfloor_{s,-2}^\cO\right)^\ti}{1 - \eta^{-1} (\abs{F}_{s,-2}^\cO)}+
	\frac{\lceil G^\hor\rfloor_{s,-2}^\cO \lceil F^\mix \rfloor_{s,-\bd}\left(\eta^{-1} \lceil F \rfloor_{s,-2}^\cO\right)^{\ti-1}}{1 - \eta^{-1} (\abs{F}_{s,-2}^\cO)}
		\end{align*}
		Note that the same holds if we substitute in \eqref{taylor}  the sequence $\{\frac{1}{l!}\}$  with any sequence $\{b_l\}$ such that  $\forall l$ one has  $|b_l|\leq \frac{1}{l!}$.
	\end{itemize}
\end{lemma}
The lemma is proved in Appendix \ref{app:fun}.\\
 Note that the difference between Lemma \ref{lem:smo} and Lemma \ref{lem:smo2} is that in this last one the quantitative estimates {\em do not} lose regularity in $s$ and do not gain in order.

\section{Preparation of the Hamiltonian}
\label{preparazione}
The aim of this section is to write the   Hamiltonian \eqref{parto}, the mass $M$  \eqref{mass0} and the momentum $P$ \eqref{momentum0} in the adapted variables \eqref{a.a.z} around the tori  \eqref{tori}.

\subsection{The Birkhoff map for the one dimensional cubic NLS}
\label{sub:bm}
First we gather some properties of the  Birkhoff map for the dNLS \eqref{dNLS}. The main references
for this subsection are the book \cite{grebert_kappeler} and the paper \cite{masp_vey}.\\
We introduce the Fourier-Lebesgue spaces for $p, s \geq 0$:
\begin{equation}
\label{ell.s.a}
h^{p,s} \equiv h^{p,s}(\Z) := \left\lbrace
(q_m, \bar q_m)_{m \in \Z} \in \ell^2(\Z)\times\ell^2(\Z) \colon \ \norm{q}_{p,s}^2 := \sum_{m \in \Z} e^{2 s |m|} \langle m \rangle^{2p} \, |q_m|^2 <\infty \right\rbrace
\end{equation}
We shall denote by $B^{p, s}(\rho)$ the ball of radius $\rho$ and center $0$ in the topology of $h^{p,s}$.
\begin{theorem}
\label{thm:dnls}
There exists $\rho_* >0$  and a real analytic,  symplectic,  majorant analytic map $\Phi: B^{0,0}(\rho_*) \to \ell^2\times \ell^2$  with  $\di\Phi(0,0) = \uno $ s.t. $\forall p, s \geq 0$ the following is true.
\begin{itemize}
\item[(i)] The restriction of  $\und{\Phi}$ to $B^{p,s}(\rho_*)$ gives rise to a real analytic map $\und{\Phi} :  B^{p,s}(\rho_*) \to h^{p,s}$. 
Moreover there exists  $C>0$ s.t. for all $0 < \rho <\rho_*$
\begin{equation}
\label{bm.est}
\sup_{\norm{q}_{p,s}  \leq \rho} \norm{(\und{\Phi }- \uno)(q,\bar q)}_{p,s} \leq  C \,  \rho^3  \ .
\end{equation}
The same estimate holds for $\und{\Phi^{-1}}-\uno$, with a different constant.
\item[(ii)]   For $p \geq 1$,   $\Phi$ 
introduces local Birkhoff coordinates for dNLS. 
More precisely the  integrals of motion of dNLS are real analytic functions of the actions
$I_j = |z_j|^2$. In particular, the Hamiltonian $H_{{\rm dnls}}$, 
the mass $M(q):= \int_\T \abs{q(x)}^2 dx$ and the momentum $P(q):= \int_\T \bar q(x) \im \derx q(x) dx$ have the form
\begin{align}
&\left(H_{{\rm dnls}} \circ \Phi^{-1}\right)(z, \bar z) \equiv h_{\rm dnls}\left( (|z_m|^2 )_{m \in \Z}\right) \\
%= \sum_m m^2 |z_m|^2  + \left(\sum_m |z_m|^2 \right)^2 - \frac{1}{2} \sum_{m} |z_m|^4  + O(|z|^6) \ , \\
\label{mass.bc}
&\left(M\circ \Phi^{-1}\right)(z, \bar z) = \sum_{m \in \Z} |z_m|^2 \ , \\
&\left(P\circ \Phi^{-1}\right)(z, \bar z) = \sum_{m \in \Z}  m |z_m|^2 \ . 
\end{align}
\item[(iii)] For $p \geq 2$, define the dNLS  action-to-frequency map  $I \mapsto \alpha^{{\rm dnls}}(I)$, where
$
\alpha^{{\rm dnls}}_m(I) := \frac{\partial h_{\rm dnls}}{\partial I_m}$, $\forall m \in \Z.$
Then, in a neighborhood of $I=0$,  one has the asymptotic expansion 
\begin{equation}
\label{a.fr}
\alpha^{{\rm dnls}}_m(I)  =  m^2 + 2\sum_i I_i - I_m + {\frac{\varpi_m(I)}{\la m\ra}}  \ , \qquad m \in \Z
\end{equation}
where $\varpi_m(I)$ is at least quadratic in $I$ and $\sup_m |\varpi_m(I)| < \infty$.
%The  renormalized frequencies 
%$\alpha^*_m(I) := \alpha^{{\rm dnls}}_m(I)- m^2 -2\sum_i I_i  $
%extend to a real analytic map $I: \ell^2 \to \ell^2$ which is a local diffeomorphism around $I=0$.
\end{itemize}
%Finally item $(i)$ holds also replacing the space $h^{p,s}$ with the space $\ell^1(\Z)$.
\end{theorem}
\begin{proof}
Item $(i)$ is the main content of \cite{masp_vey}, where it is proved that the Birkhoff map is majorant analytic between some Fourier Lebesgue-spaces. Item $(ii)$ is proved in \cite{grebert_kappeler, kuksin_poschel}. Item $(iii)$ in \cite{kapp_liang, kapp_scatt}.
\end{proof}

Consider now the mass shift Hamiltonian $H_{\rm dm}$ defined in \eqref{parto.dnls}.
By Theorem \ref{thm:dnls}, $H_{\rm dm}$ is integrable and one has
$$
\left(H_{\rm dm}\circ \Phi^{-1}\right)(z, \bar z) =  h_{\rm dnls}( (|z_m|^2 )_{m \in \Z}) -  \left(\sum_m |z_m|^2 \right)^2 =:h_{\rm dm}\big( (|z_m|^2 )_{m \in \Z}\big)  \ . 
$$
We define the frequencies $\alpha_m(I) := \frac{\partial h_{\rm dm}}{\partial I_m}$ which by \eqref{mass.bc} 
are given by 
\begin{equation}
\label{freq1:dnls0}
\alpha_m(I) := \alpha^{\rm dnls}_m(I) - 2 \sum_i I_i \ ,
\end{equation}
and by   \eqref{a.fr} one has  the following  expansion in a neighborhood of $I=0$:
\begin{equation}
\label{freq:1dnls}
\alpha_m(I)=    m^2 - I_m + { \frac{ \varpi_m(I)}{\la m\ra}}  \ , \qquad m \in \Z  \ . 
\end{equation}
%\begin{remark}
%\label{freq:2dnls}
%Fix $(z, \bar z)\in B^{p,s}(\rho)$, $\rho \ll 1$. Then  
%$q^{\rm dnls}(t, \cdot) := 
%\Phi^{-1}\left( \left( (z_m \, e^{-\im \alpha_m^{\rm dnls}(I) t} \right)_{m \in \Z} \right)$
%is a solution of  \eqref{dNLS}, while
%$q(t, \cdot) := \Phi^{-1}\left( \left( z_m \, e^{-\im \alpha_m(I) t} \right)_{m \in \Z} \right)$ 
%is a solution of  \eqref{dnls.massless}.
%\end{remark}
Now with  $\cS_0 = (\tm_1, \ldots, \tm_\tk)$ we consider the map 
\begin{equation}
\label{atof.map2}
I_\tm \equiv (I_{\tm_1},\dots,I_{\tm_\tk}) \to (\alpha_{\tm_1}(I_\tm),\dots,\alpha_{\tm_\tk}(I_\tm)) \equiv \alpha_{\tm}(I_\tm) \ ;
\end{equation}
by Theorem \ref{thm:dnls} $(iii)$ this map  is generically a diffeomorphism\footnote{generically in the sense outside a set of measure 0}. 
As we have already  mentioned, we prefer to parametrize the frequencies with a vector of parameters $\lambda = (\lambda_1, \ldots, \lambda_\tk) \in \cO_0 \equiv [1/2, 1]^\tk$, 
 using the fact that, by a standard application of the implicit function theorem, for $\epsilon$ sufficiently small there exists
  an analytic map $ \cO_0 \ni \lambda \to I_\tm(\lambda, \epsilon) \in \R^\tk_{>0}$ of the form  \eqref{boia}
such that 
$$
\alpha_{\tm}( I_\tm(\lambda, \epsilon) ) = \omega^\0 - \epsilon \lambda  =: \omega(\lambda)\ ,
$$
where $\omega^\0$ is defined in \eqref{omega0}. 
 We will call  $\omega(\lambda) \equiv \alpha_\tm(I_\tm(\lambda, \e))$  the {\em tangential frequencies}, while 
 $(\Omega_m(\lambda))_{m \in \Z \setminus \cS_0}$, where 
 \begin{equation}
 \label{norm.freq}
 \Omega_m(\lambda) := \alpha_m(I_\tm(\lambda, \e)) \ , \qquad m \notin \cS_0 
 \end{equation}
  the {\em normal frequencies};  by \eqref{freq:1dnls} the normal frequencies have the following asymptotic expansion as $|m| \to \infty$: 
\begin{equation}
\label{fischio0}
\Omega_m(\lambda) = m^2 +	
  \frac{\varpi_m(I_\tm(\lambda, \e))}{{\la m\ra}}\,,\quad \mbox{with} \ \ \  
  \sup_{\lambda\in \cO_0}\sup_{m \in \Z }|\varpi_m(I_\tm(\lambda, \e))| +
  |\partial_\lambda \varpi_m(I_\tm(\lambda, \e))|\le  C \e^2 \ . 
\end{equation}

 \subsection{Adapted variables}
In this section we introduce adapted local variables in a neighborhood of the torus $\tT^\tk(\cS_0, \lambda) \equiv \tT^\tk(\cS_0, I_\tm(\lambda, \e))$ defined in \eqref{tori}. 
We parametrize the torus $ \tT^\tk(\cS_0, \lambda)$ as $z^\fig(\lambda; \theta) = (z^\fig_m(\lambda; \theta) )_{m \in \Z}$ with 
\begin{equation}
\label{fin.gap2}
z_{\tm_i}^{\rm fg}(\lambda; \theta)= \sqrt{I_{\tm_i}(\lambda)}\  e^{\im \theta_i}\ , \ \ \ \mbox{for} \ 1 \leq i \leq \tk   \ ,\qquad z_m^{\rm fg}=0\,,\quad  \mbox{for} \ m\notin \cS_0 \ ;
\end{equation}
we denote  by  $q^\fig(\lambda; \theta)$ its preimage  through the Birkhoff map\footnote{Abusing notation, from now on we will often write
$q = \Phi^{-1}(z)$ in place of $(q, \bar q) = \Phi^{-1}(z, \bar z)$.}
: 
\begin{equation}
\label{figa}
q^\fig(\lambda; \theta) \equiv (q_m^\fig(\lambda; \theta))_{m \in \Z} \equiv \Phi^{-1}( z^{\rm fg}\big(\lambda; \theta)\big) \ .
\end{equation}
 The explicit expression of  $q^\fig(\lambda; \theta)$ in the $x$ variable was given in  \eqref{finite.gap.formula}. 
% Finally we denote
% $$
% \Gamma(\cS_0, \lambda) := \Phi^{-1}\left(
% \tT^\tk(\cS_0, \lambda) \right)
% $$
% the preimage through $\Phi^{-1}$ of the torus. 
% We have that  
% \begin{equation}
% \label{orbit.qfin}
% \overline{ \bigcup_{t \in \R} q^\fig(\lambda; \omega(\lambda) t)}
%  \equiv 
%  \Gamma(\cS_0, \lambda) \ ;
% \end{equation}
%this is due to the fact that the  orbit of the finite gap is dense on the torus,
% as a consequence of having chosen the tangential frequency $\omega(\lambda)$ nonresonant.		\\
We study now the analytic properties of the map  $\Lambda: (\yy, \theta, \ba) \mapsto (u_\jj)_{\jj \in \Z^2}$ defined in \eqref{mappa}. 
First remark that by \eqref{figa} the set $\Phi^{-1}\left( \tT^\tk(\cS_0, \lambda) \right)$ is described in the $(\yy, \theta, \ba)$ coordinates by $\yy=0$, $\ba =0$, i.e. 
$$
 \Phi^{-1}\left( \tT^\tk(\cS_0, \lambda) \right) \equiv \{ \Lambda(0, \theta, 0) \colon \theta \in \T^\tk \} \ . 
$$
We  show also that $\Lambda$ maps $D(s, r)$ into the set $\cV_\delta$ of functions which are $\delta$-close to the torus $\tT^\tk$, which we recall is 
\begin{equation}
\label{nei}
\begin{aligned}
\cV_\delta := 
\Big\{ v \in H^p(\T^2) \colon 
& {\rm dist}_{H^p(\T^2)} \Big( v,  \Phi^{-1}\left( \tT^\tk(\cS_0, \lambda) \right) \Big) < \delta \ , \\
& \abs{\abs{\Phi_{\tm_i}( (v_{(m,0)})_{m \in \Z} )}^2 -  I_{\tm_i}(\lambda) }\leq \delta^2 \mbox{ for } 1 \leq i \leq \tk 
 \Big\} \ . 
 \end{aligned}
\end{equation}
Note that we are constantly identifying a function with its Fourier coefficients.
%The second condition means that we consider the Fourier coefficients $(v_{(m,0)})_{m \in \Z} $ (corresponding to the restriction $v(\cdot, 0)$), and impose that the modulus of the Birkhoff coordinates with index in $\cS_0$ are $\delta^2$-close to the actions of the finite gap.
\begin{proposition}
\label{prop:adap}
Fix an arbitrary $p > 1$ and the set $\cS_0$. For any $s >0$,  there exist $\e_0(s)>0, c_\star(s), c_\ast(s) >0$ such that  for any $0<\e < \e_0(s)$ the following is true. 
\begin{itemize}
\item[(i)] There exists $0< r_* \leq \sqrt{\e}/4$ s.t. for any $0 <r \leq r_* $ the change of coordinates
$\Lambda \colon D(s, r) \to \Lambda (D(s, r))$, $(\yy,\theta,\ba)\mapsto  (u_{\jj})_{\jj\in \Z^2} $
 is majorant analytic, and the symplectic form in the variables $(\yy, \theta, \ba)$ is given by \eqref{yythetaa}.
 \item[(ii)]
For $0 \leq  r \leq r_*$,  $\cV_{c_\star r} \subseteq \Lambda(D(s, r)) \subseteq  \cV_{c_\ast r}$.
  \item[(iii)] The Hamiltonian \eqref{parto} in the variables $(\yy, \theta, \ba)$ takes the form 
 \begin{align}
\label{H.2}
\cH(\lambda;\yy, \theta, \ba)  
 =  & \cN^\0+\cH^\0(\lambda; \theta, {\bf a})+\cH^\1(\lambda; \theta, {\bf a})+\cH^\2(\lambda; \theta, \yy,  {\bf a})
\end{align}
where 
\begin{align}
\label{def:N}
\cN^\0 & := \omega(\lambda) \cdot \yy +  \cD  \ , \\
\label{def:D}
 & \  \cD :=    \sum_{i=1}^\tk \omega_{\tm_i} (\lambda) \yy_i +\sum_{\jj=(m,n) \in \Z^2\setminus\cS_0} \Omega_\jj^\0\,  |a_\jj|^2   
\end{align} 
and  the normal frequencies $\Omega_\jj^\0$ are defined by
\begin{equation}
\label{def:Omega}
 \Omega^{(0)}_\jj := \left\{ \begin{array}{ll} |\jj|^2 & \text{if}\; \jj=(m,n) \;\text{with} \; n\neq 0 \\ \Omega_m(\lambda) & \text{if}\; \jj=(m,0) \end{array}\right. \ ,
\end{equation}
where $\Omega_m(\lambda)$ is defined in \eqref{norm.freq}.
Moreover  $\cH^\0\in \cQ_{s}^\hor$,   $\cH^\1$ and $\cH^{(\geq 2)}$   belong to $\cA_{s,r_*}^\cO$ with the quantitative  bounds 
\begin{equation}
|\cH^\0|_{s}^\cO \le C\e\,,\qquad |\cH^\1|_{s,r}^\cO \le C\sqrt{\e}r\,,\qquad |\cH^{(\ge 2)}|_{s,r}^\cO \le C r^2  \ , \qquad \forall 0 < r \leq r_* \ , 
\end{equation}
where $C$ is independent of $\e,r$.
\item[(iv)] The mass $M$ and the momentum $P$ (defined in \eqref{mass0} and \eqref{momentum0}) in the variables $(\yy, \theta, \ba)$ take the form \eqref{costanti}.
\end{itemize}
\end{proposition} 
We split the proof of  Proposition 
 \ref{prop:adap} in several steps. First we prove item $(i)$ and $(ii)$.
 
 \begin{proof}[Proof of Proposition \ref{prop:adap} $(i)$ and $(ii)$.]
$(i)$ The map $\Lambda$  is the identity on $\Z^2 \setminus \Z$ (see \eqref{mappa}), thus we need only to study  the analytic  properties of the map
$\Big(\yy,\theta, (a_{(m,0)})_{m \in \Z\setminus \cS_0}\Big)\to (u_{(m,0)})_{m\in \Z}$.
Such a map is the composition of the Birkhoff map $\Phi^{-1}$ and the map $\Upsilon$ which passes the special sites $\cS_0$ to action-angle coordinates:
\begin{align}
\label{map:aa}
 \left((\yy_i, \theta_i)_{1\leq i \leq d}, (a_{(m,0)})_{m \in \Z\setminus \cS_0}\right) &\stackrel{\Upsilon}{\mapsto}  \left( (z_{\tm_1}, \ldots, z_{\tm_d}), (z_{m})_{m \in \Z \setminus \cS_0}\right)
 \stackrel{\Phi^{-1}}{\mapsto} (u_{(m,0)})_{m\in \Z}
\\
\label{aa}
z_{\tm_j}&= \sqrt{I_j(\lambda) +\yy_j}\ e^{\im \theta_j}, \qquad \tm_j \in \cS_0 \\
\label{aa0}
z_{m}&=a_{(m, 0)}, \qquad m\in \Z\setminus \cS_0.
\end{align}
By Theorem  \ref{thm:dnls} the Birkhoff map $\Phi$  and its inverse $\Phi^{-1}$ are majorant analytic canonical diffeomorphism from a small ball  around the origin in $h^p(\Z)$ to $h^p(\Z)$ for any $p \geq 0$.
In particular there exists $0<R\le \rho_*$ (where $\rho_*$ is the domain on majorant analyticity of $\Phi$, see Theorem \ref{thm:dnls})  such that $\und{\Phi^{-1}} \colon B(R)\times B(R) \to B(2R)\times B(2R)$, $B(R)$ being a ball of center 0 and radius $R$ in the topology of $h^p(\Z)$. \\
 For a given $s>0$, fix $\e, r_*$ so that 
\begin{equation}\label{parapendio}
 16 r_*^2 \leq  \e \leq  C_* e^{-2s} R^2 \ , 
 \qquad C_*^{-1}= 48\, \tk \, \max_i |\tm_i|^{2p}.
\end{equation}
 Consider now the map $\Upsilon$.
It is proved in \cite[Lemma 7.6]{BBP} that  the first condition in \eqref{parapendio} implies that   $ \Upsilon$ is majorant analytic and 
$\und{ \Upsilon}_{s} \colon D(s, \sqrt \e /4 ) \to B(C \sqrt \e)\times B(C \sqrt \e)$, where $C >0$ can be chosen uniformly in $\e$ for $\e \leq \e_0$, while the  second  condition in \eqref{parapendio}  ensures that the image of $\Upsilon$ falls in the domain of majorant analyticity of $\Phi^{-1}$. 
As a consequence the map $\Phi^{-1}\circ \Upsilon$ is majorant analytic.
Recall now that  $\Lambda$ is simply obtained by extending  $\Phi^{-1}\circ \Upsilon$  as the identity on $\Z^2 \setminus \Z$, therefore it is    majorant analytic as  a map
$ D(s, \sqrt \e /4 ) \to B(C \sqrt \e) \times B( C \sqrt \e)$, where now we denoted by  $B(R)$ the ball in the topology of $h^p(\Z^2)$. The fact that $\Lambda$  transforms the standard symplectic form into  \eqref{yythetaa} is a direct  computation, using the symplecticity of the Birkhoff map $\Phi$.\\
$(ii)$ First we show that there exists a constant $c_\ast>0$ s.t. $\Lambda(D(s, r)) \subseteq \cV_{c_\ast r}$.
 Thus let $u = \Lambda(\yy, \theta, \ba)$ with $(\yy, \theta, \ba) \in D(s,r)$.  Recall that $q^\fig(\lambda; \theta)  = \Phi^{-1}(\Upsilon(0, \theta,0))$.
   One has that
\begin{align*}
{\rm dist}_{H^p(\T^2)} \Big( u,  \Phi^{-1}\left( \tT^\tk(\cS_0, \lambda) \right) \Big) &\leq  \norm{u - q^\fig(\lambda; \theta)}_{h^p(\Z^2)}   
 = \norm{u}_{h^p(\Z^2\setminus \Z)}  + 
\norm{u - q^\fig(\theta)}_{h^p(\Z)} \\
& =
 \norm{(a_{(m,n)})_{n \neq 0}}_{h^p(\Z^2\setminus \Z)}  + 
\norm{\Phi^{-1}\left(\Upsilon(\yy, \theta, \ba)\right) 
- \Phi^{-1}\left(\Upsilon(0, \theta, 0)\right)}_{h^p(\Z)} \\
& \stackrel{{\rm Thm.}  \ref{thm:dnls}}{\leq}   r + C \norm{\Upsilon(\yy, \theta, \ba)
- \Upsilon(0, \theta, 0)}_{h^p(\Z)} \\
& \leq r + C \norm{(a_{(m,0)})_{m \in \Z}}_{h^p(\Z\setminus \cS_0)} + C \left(\sum_{i=1}^\tk \tm_i^{2p} \abs{\sqrt{I_{\tm_i}(\lambda) + \yy_i} - \sqrt{I_{\tm_i}(\lambda)}}^2 \right)^{1/2} \\
& \leq C' \left(  r + \frac{r^2}{\sqrt \e} \right) 
\stackrel{\eqref{parapendio}}{\leq } c_\ast r  \ ,
\end{align*}
which proves one condition.
Then, since 
\begin{equation}
\label{u.0.b}
(u_{(m,0)})_{m \in \Z} = \Phi^{-1}
\Big(\left(\sqrt{I_{\tm_i}(\lambda) + \yy_i } 
e^{\im \theta_i} 
\right)_{1 \leq i \leq \tk}, 
(a_{(m,0)})_{m \in \Z \setminus \cS_0}  \Big)
\end{equation}
 one has $\Phi_{\tm_i}((u_{(m,0)})_{m \in \Z}) = \sqrt{I_{\tm_i}(\lambda) + \yy_i } \, 
e^{\im \theta_i} $, which clearly implies that
\begin{align*}
\left\vert \abs{\Phi_{\tm_i}((u_{(m,0)})_{m \in \Z})}^2 - I_{\tm_i}(\lambda) \right\vert \leq r^2 \ . 
\end{align*}
Thus  $\Lambda(D(s, r)) \subseteq \cV_{c_\ast r}$.\\
Now we show the converse, namely that $\exists \, c_\star >0$ s.t.
 $ \cV_{c_\star r} \subseteq \Lambda(D(s,r))$.  
 Then take $\delta = c_\star r $,   $u \in \cV_\delta$, and we show that $u = \Lambda(\yy, \theta, \ba)$ for some $(\yy, \theta, \ba) \in D(s,c_\star^{-1}\delta)$.
 First we put $(a_{(m,n)})_{n \neq 0} := (u_{(m,n)})_{n \neq  0}$; the condition  ${\rm dist}\Big(u, \Phi^{-1}\left( \tT^\tk(\cS_0, \lambda) \right)\Big) \leq \delta$ implies immediately that 
$\norm{(a_{(m,n)})_{n \neq 0}}_{h^p(\Z^2\setminus \Z)}\leq \delta$.\\
Consider now   $(u_{(m,0)})_{m \in \Z}$, which are the Fourier coefficients of $u(\cdot, 0)$.
 Denote  $\theta_* := {\rm argmin}_\theta \norm{u - q^\fig(\lambda; \theta_*)}$; such minimum exists since $\T^\tk$ is compact and $\theta \mapsto q^\fig(\lambda; \theta)$ is continuous. 
Then
$$
\norm{(u_{(m,0)})_{m \in \Z}}_{h^p(\Z)} 
\leq \norm{(u_{(m,0)})_{m \in \Z} - q^\fig(\lambda; \theta_*)}_{h^p(\Z)} 
+
\norm{ q^\fig(\lambda; \theta_*)}_{h^p(\Z)} \leq
\delta + C\sqrt \e \leq R \leq \rho_* \ ,
$$ 
where $\rho_*$ is the size of the domain of majorant analyticity 
of the Birkhoff map $\Phi$, see Theorem \ref{thm:dnls}.
Hence the vector  $z\equiv (z_{m})_{m \in \Z} := \Phi((u_{(m,0)})_{m \in \Z})$ is well defined; 
we pose $a_{(m,0)} := z_{m}$ for any $m \notin \cS_0$.
By the  analyticity of $\Phi^{-1}$ and the fact that $\di \Phi^{-1}(0) = \uno$, we  bound 
\begin{align*}
\delta = \norm{u - q^\fig(\lambda; \theta_*)}_{h^p(\Z)} 
& = \norm{\Phi^{-1}(z) - \Phi^{-1}(z^\fig(\lambda; \theta_*))}_{h^p(\Z)} 
\geq c_\star \norm{z - z^\fig(\lambda; \theta)_*}_{h^p(\Z)} \\
& = c_\star \norm{z - z^\fig(\lambda; \theta_*)}_{h^p(\cS_0)} + c_\star \norm{z}_{h^p(\Z\setminus \cS_0)}  \ . 
\end{align*}
Hence we obtain that   $\norm{(a_{(m,0)})_{m \notin \cS_0}}_{h^p(\Z\setminus \cS_0)}\leq c_\star^{-1}\delta$. Finally write
$z_{\tm_i} = \sqrt{I_{\tm_i}(\lambda) + \yy_i} e^{\im \theta_i}$ for some $\yy$ and $\theta$. 
Then  the second condition in \eqref{nei} implies $\delta^2 \geq \abs{|z_{\tm_i}|^2 - I_{\tm_i}(\lambda)} \geq \abs{\yy_i}$. 
Therefore we have shown that $u=\Lambda(\yy, \theta, \ba)$ for some $(\yy, \theta, \ba) \in D(s, c_\star^{-1} \delta)$.
\end{proof}

We  begin now the proof of item $(iii)$. 
Thus we consider  expression \eqref{parto} and apply in two steps the change of coordinates $\Phi^{-1}$ and $\Upsilon$ of \eqref{map:aa}.
 To begin with, we start from the  Hamiltonian in Fourier coordinates \eqref{Ha0}, and set
$$
q_m:= u_{(m,0)}\quad \mbox{ if } \  m\in \Z\,,\qquad a_{\jj}:= u_{\jj} \quad \mbox{ if } \ \jj=(m,n)\in \Z^2\,,\; n\neq 0 \ .
$$ 
We rewrite the  Hamiltonian accordingly in increasing degree in $a$, obtaining 
\begin{align}
\label{H.1}
H(q, a)= 
& \sum_{m \in \Z} m^2 |q_m|^2 -\frac12 \sum_{m \in \Z} |q_m|^4 +  \frac12\sum_{m_i\in \Z \atop m_1-m_2+m_3-m_4=0}^{\star}q_{m_1}\bar q_{m_2}q_{m_3}\bar q_{m_4}\\
\notag
 & + \sum_{\jj\in \Z^2\setminus \Z} |\jj|^2 |a_{\jj}|^2\\
 \notag
 & + 2\sum^\star_{\jj_i=(m_i,n_i)\,,i=3,4\,,\; n_i\neq 0 \atop {
m_1-m_2+m_3-m_4=0\atop  n_3-n_4=0}} q_{m_1}\bar q_{m_2} a_{\jj_3}\bar a_{\jj_4}   +
\frac{1}{2} \sum_{ \jj_i=(m_i,n_i)\,,i=2,4\,,\;n_i\neq 0 \atop {
m_1-m_2+m_3-m_4=0 \atop  n_2+n_4=0} } (q_{m_1} \bar a_{\jj_2}  q_{m_3} \bar a_{\jj_4}+ \bar q_{m_1} a_{\jj_2}  \bar q_{m_3}  a_{\jj_4})\\
\notag
 & +  \sum_{ \jj_i=(m_i,n_i)\,,i=2,3,4\,,\; n_i\neq 0
\atop {m_1-m_2+m_3-m_4=0\atop  -n_2+n_3-n_4=0}} q_{m_1} \bar a_{\jj_2} a_{\jj_3} \bar a_{\jj_4} 
+ 
\sum_{ \jj_i=(m_i,n_i)\,,i=1,2,3\,,\; n_i\neq 0
\atop {m_1-m_2+m_3-m_4=0\atop  n_1-n_2+n_3=0}} a_{\jj_1} \bar a_{\jj_2} a_{\jj_3} \bar q_{m_4}\\
\notag
& + \frac{1}{2}\sum^\star_{ \jj_i=(m_i,n_i)\,,i=1,2,3,4\,,\;  n_i\neq 0\atop{ \jj_1-\jj_2+\jj_3-\jj_4 = 0}} a_{\jj_1} \bar a_{\jj_2} a_{\jj_3} \bar a_{\jj_4}-\frac12\sum_{\jj\in \Z^2\setminus \Z}|a_\jj|^4
\end{align}

\noindent{\bf Step 1:} First we  introduce Birkhoff coordinates on the line $\Z\times \{0\}$: We set
\begin{align}
&\left( (z_m)_{m \in \Z}, \  (a_{\jj})_{\jj\in \Z^2\setminus \Z}\right)\mapsto \left((q_m)_{m \in \Z}, (a_{\jj})_{\jj \in \Z^2\setminus \Z}\right) \ , 
\quad (q_m)_{m \in \Z} =\Phi^{-1}\left((z_m)_{m \in \Z}\right) \ .
\end{align}
%Thanks to the fact that the Birkhoff map is canonical, the 
%symplectic form in the new variables is simply
%\begin{equation}
%\label{simp2}
%\im \sum_{j \in \Z} d z_j\wedge d\bar z_j + \im \sum_{\jj \in \Z^2 \setminus \Z } da_\jj \wedge d\bar a_\jj \ . 
%\end{equation}
Since the first line of \eqref{H.1} is the dNLS with mass shift, when we apply  $\Phi^{-1}$ we obtain 
\begin{align}\label{piri}
H'(z, a)=
 & h_{\rm dm}((|z_m|^2)_{m \in \Z}) +\sum_{\jj\in \Z^2\setminus \Z} |\jj|^2 |a_{\jj}|^2  \\
\notag
 &+ 2\sum^\star_{\jj_i=(m_i,n_i)\,,i=3,4\,,\; n_i\neq 0 \atop {
m_1-m_2+m_3-m_4=0\atop  n_3-n_4=0}} q_{m_1}\bar q_{m_2} a_{\jj_3}\bar a_{\jj_4} +
\frac{1}{2} \sum_{ \jj_i=(m_i,n_i)\,,i=2,4\,,\;n_i\neq 0 \atop {
m_1-m_2+m_3-m_4=0 \atop  n_2+n_4=0} } (q_{m_1} \bar a_{\jj_2}  q_{m_3} \bar a_{\jj_4}+ \bar q_{m_1} a_{\jj_2}  \bar q_{m_3}  a_{\jj_4})\\
\notag
 & +  \sum_{ \jj_i=(m_i,n_i)\,,i=2,3,4\,,\; n_i\neq 0
\atop {m_1-m_2+m_3-m_4=0\atop  -n_2+n_3-n_4=0}} q_{m_1} \bar a_{\jj_2} a_{\jj_3} \bar a_{\jj_4} 
+ 
\sum_{ \jj_i=(m_i,n_i)\,,i=1,2,3\,,\; n_i\neq 0
\atop {m_1-m_2+m_3-m_4=0\atop  n_1-n_2+n_3=0}} a_{\jj_1} \bar a_{\jj_2} a_{\jj_3} \bar q_{m_4}\\
\notag
& + \frac{1}{2}\sum^\star_{ \jj_i=(m_i,n_i)\,,i=1,2,3,4\,,\;  n_i\neq 0\atop{ \jj_1-\jj_2+\jj_3-\jj_4 = 0}} a_{\jj_1} \bar a_{\jj_2} a_{\jj_3} \bar a_{\jj_4}-\frac12\sum_{\jj\in \Z^2\setminus \Z}|a_\jj|^4
\end{align}
where in the last three lines we think $(q_m)_{m \in \Z}=\Phi^{-1}((z_m)_{m \in \Z})$ as a function of the Birkhoff variables $z$.\\

{\bf Step 2:}  We go to action-angle coordinates only on the set $\cS_0= (\tm_1, \ldots, \tm_\tk) \subset \Z$ and rename $z_m$ for $m \notin \cS_0$ as $a_{(m, 0)}$, see \eqref{aa} and \eqref{aa0}. We think each $q_m$ as a function
$$q_m= q_m(\lambda;\yy, \theta, \ba) \equiv \Phi^{-1}_m( (\sqrt{ I_{\tm_i}(\lambda)+\yy_i }  \ e^{\im \theta_i})_{i=1,\ldots,\tk}, (a_{(m,0)})_{m\in \Z\setminus\cS_0} ) \  . $$
Next we   Taylor expand the Hamiltonian \eqref{piri} around the finite-gap torus corresponding to
$(\yy, \theta, \ba) = (0, \theta, 0)$.
First one has (up to an irrelevant constant) 
\begin{align*}
h_{\rm dm}((|z_m|^2)_{m \in \Z}) 
& = h_{\rm dm}\left( I_{\tm_1}(\lambda) + \yy_1, \ldots ,  I_{\tm_\tk}(\lambda) + \yy_{\tk} , \, (|a_{(m,0)}|^2)_{m \in \Z\setminus \cS_0}\right) \\
&  = \omega(\lambda)\cdot \yy + \sum_{m \in \Z\setminus \cS_0}
\Omega_{m}(\lambda) |a_{(m,0)}|^2 + \cO(\yy^2, \yy |\ba|^2, |\ba|^4) \ ,
\end{align*}
where the $\Omega_m(\lambda)$ are defined in \eqref{norm.freq}.
Now we consider the terms which contain $q_m$.
Expanding $q_m$ in Taylor series  we have
  \begin{align*}
   q_m  = & q_m(\lambda; 0, \theta,  0)
    + 
    \sum_{i=1}^d 
    \frac{\partial q_m}{\partial \yy_i}(\lambda; 0, \theta, 0) \yy_i 
    + 
   \sum_{m_1 \in \Z \setminus \cS_0} 
   \frac{\partial q_m}{\partial a_{(m_1,0)}}
   (\lambda; 0, \theta, 0) a_{(m_1, 0)} \\  
&+
   \sum_{m_1 \in \Z \setminus \cS_0} 
   \frac{\partial q_m}{\partial \bar a_{(m_1,0)}}
   (\lambda; 0, \theta, 0) \bar a_{(m_1, 0)}   
     + \cO(\yy^2, \yy \ba , \ba^2)
%   \\
%  =  & q_m(\lambda; \theta) + \sum_{i=1}^d f_{m, i}(\lambda; \theta) \yy_i + \sum_{m_1 \in \Z} g_{m, m_1}(\lambda; \theta) a_{(m_1, 0)} + \cO(\yy^2, \yy a , a^2)
  \end{align*}
  Remark that
  $$
  q_m(\lambda; 0, \theta,  0) = \Phi^{-1}_m( (\sqrt{ I_{\tm_i}(\lambda)}  \ e^{\im \theta_i})_{i=1,\ldots,\tk}, 0 ) = q_m^\fig(\lambda; \theta) \ .
  $$
  Inserting such formula in the Hamiltonian \eqref{piri} and collecting the terms in increasing scaling degree we obtain the new Hamiltonian
  $$
\cH(\lambda;\yy, \theta, \ba)  
 =   \cN^\0+\cH^\0(\lambda; \theta, {\bf a})+\cH^\1(\lambda; \theta, {\bf a})+\cH^\2(\lambda; \theta, \yy,  {\bf a})
$$
where $\cN^\0$ is defined in \eqref{def:N}.
%\begin{equation}
%\label{def:N}
%\begin{aligned} 
%\cN^\0 & :=    \sum_{i=1}^\tk \omega_{\tm_i} (\lambda) \yy_i + \sum_{m\notin \cS_0} \Omega_m(\lambda) |a_{(m,0)}|^2 
%+ \sum_{\jj=(m,n) \in \Z^2 \atop  n\neq 0} |\jj|^2 |a_{\jj}|^2
%\end{aligned}
%\end{equation}
%and $\cD$ is the diagonal operator
%\begin{equation}
%\label{def:D}
%\cD:=\sum_{\jj=(m,n) \in \Z^2\setminus\cS_0} \Omega_\jj^\0\,  |a_\jj|^2   
%\end{equation}
%where the normal frequencies $\Omega_\jj^\0$ are defined by
%\begin{equation}
%\label{def:Omega}
% \Omega^{(0)}_\jj := \left\{ \begin{array}{ll} |\jj|^2 & \text{if}\; \jj=(m,n) \;\text{with} \; n\neq 0 \\ \Omega_m(\lambda) & \text{if}\; \jj=(m,0) \end{array}\right. \ ,
%\end{equation}
Correspondingly $\cH^\0$  is the second and third line of \eqref{H.2}, namely all the remaining terms of degree two in $\ba$:
\begin{align}\label{def:H0}
\cH^\0  := & 
  2\sum^\star_{\jj_i=(m_i,n_i)\,,i=3,4\,,\; n_i\neq 0 
\atop {
m_1-m_2+m_3-m_4=0\atop  n_3-n_4=0}} q_{m_1}^\fig(\lambda;\theta)\bar q_{m_2}^\fig(\lambda;\theta) a_{\jj_3}\bar a_{\jj_4} \\
 \notag
 & + 
\frac{1}{2} \sum_{ \jj_i=(m_i,n_i)\,,i=2,4\,,\;n_i\neq 0 \atop {
m_1-m_2+m_3-m_4=0 \atop  n_2+n_4=0} } 
(q_{m_1}^\fig(\lambda;\theta) \, \bar a_{\jj_2}  \, q_{m_3}^\fig(\lambda;\theta) \, \bar a_{\jj_4}
+ \bar q_{m_1}^\fig(\lambda;\theta) \, a_{\jj_2}  \, \bar q_{m_3}^\fig(\lambda;\theta) \, a_{\jj_4})\notag \ .
\end{align}
{  $\cH^\1$ are the terms with scaling degree 1: 
  \begin{align}\label{def:H1}
\cH^\1 & :=
  \sum_{\ell \in \Z^d , \alpha, \beta \in \N^{\Z^2 \setminus \cS_0} \atop |\alpha| + |\beta| = 3} H_{\alpha, \beta, \ell}(\lambda)  \, e^{\im \theta \cdot \ell} a^\alpha \, \bar a^\beta
\end{align}
in particular such Hamiltonian is cubic in ${\bf a}$ and  it does not contain any monomial of the form $e^{\im \theta \cdot \ell} \yy^{l} a^\alpha \bar a^\beta$ with $|l| = |\alpha + \beta | = 1$.\\
  }
  Finally $\cH^{(\geq 2)} $ %$= \cH^\2 + \cH^{(\geq 3)}$ and
% and  $ \cH^{(\geq 3)}$ 
contains  all the rest i.e.  all the terms with scaling degree $\geq 2$.

 In the next lemma we estimate the norms of the Hamiltonians $\cH^\0, \cH^\1$ and $\cH^{(\geq 2)}$:
 \begin{lemma}
 \label{lem:norm.ham}
Fix $s >0$. There exists $\e_0>0$ and for any  $\e \leq \e_0$, there exists  $r_* \leq  \sqrt{\e}/4 $ s.t.   $\cH^\0\in \cQ_{s}^\hor$ while  $\cH^\1$ and $\cH^{(\geq 2)}$   belong to $\cA_{s,r_*}^\cO$. Finally the following bounds hold:
\begin{equation}
|\cH^\0|_{s}^\cO \le C\e\,,\qquad |\cH^\1|_{s,r}^\cO \le C\sqrt{\e}r\,,\qquad |\cH^{(\ge 2)}|_{s,r}^\cO \le C r^2 \ , 
\qquad \forall 0 < r \leq r_* \ ,
\end{equation}
where $C$ is independent of $\e,r$.
\end{lemma}
\begin{proof}
	We start from the Hamiltonian \eqref{Ha0}. The terms of order four in  \eqref{Ha0}  belong to $\cA_{s,r}$ for all\footnote{note that there is in fact NO dependence on $s$ since there are no angle variables at this point. Furthermore, there is NO dependence on $\lambda \in \cO$} $s,r \geq 0$.  Indeed we have that
$$
H_4:= \sum_{\jj_1+\jj_2=\jj_3+\jj_4}u_{\jj_1}u_{\jj_2}\bar u_{\jj_3}\bar u_{\jj_4} \quad \Rightarrow (X_{H_4})_\jj= (u\star u \star \bar u)_\jj \quad \Rightarrow (\und{X_{H_4}})_\jj =    (u\star u \star \bar u)_\jj\Rightarrow |(\und{X_{H_4}})_\jj|\le  ( v\star v \star \bar v)_\jj
$$
with $v_\jj=|u_\jj |$.
Then we repeat the same argument for $M(u) = \int_{\T^2} |u|^2$ (see the definition of the Hamiltonian \eqref{parto}) and the result follows by the algebra property of our space. This shows that the starting  Hamiltonian \eqref{Ha0} is majorant analytic.
Moreover one has
\begin{equation}
\label{XH4}
\sup_{\abs{u}_{h^p(\Z^2)} \leq \rho} \bnorm{\und{X_{H_4}}(u, \bar u)}_{\rho} \leq \rho^2 \ . 
\end{equation}
Next consider the map $\Lambda$ of \eqref{mappa}. 
By Proposition \ref{prop:adap}(i),  $\Lambda$ is  majorant analytic 
and  maps $D(s, \sqrt{\e}/4) \to B(C_1\sqrt\e)\times B(C_1 \sqrt\e)$. 
Thus the pullback of $X_{H_4}$ through $\Lambda$,  which we denote $Y$, is  a majorant analytic vector field $D(s, r) \to \C^\tk \times \C^\tk \times \fh^p$, $\forall 0< r\leq r_*$. Hence $\cH^{(\geq 0)}:= H_4\circ\Lambda $ is majorant analytic and 
\begin{equation}
\label{YH}
\abs{\cH^{(\geq 0)}}_{s, \sqrt\e/4} \leq \sup_{(\yy, \ba) \in D(\sqrt\e/4)}
\bnorm{\und{Y}_{s}(\yy, \ba)}_{\sqrt\e/4} 
\leq
\sup_{\abs{u}_{h^p(\Z^2)} \leq C_1 \sqrt \e} \bnorm{\und{X_{H_4}}(u)}_{C_1 \sqrt \e} \stackrel{\eqref{XH4}}{\leq} C_1 \e \ . 
\end{equation}
As a consequence each homogeneous component of $\cH^{(\geq 0)}$ is a majorant analytic Hamiltonian; in particular let $\cH^{(d)}$ the homogeneous component of  $\cH^{(\geq 0)}$ of scaling degree $d$; then 
\begin{align*}
\abs{\cH^{(d)}}_{s, r} 
\leq 
\left( \frac{4r}{\sqrt \e}\right)^d  \abs{\cH^{(d)}}_{s,  \sqrt\e/4} 
\stackrel{\eqref{scala0}}{\leq }
\left( \frac{4r}{\sqrt \e}\right)^d  \abs{\cH^{(\geq 0)}}_{s,  \sqrt\e/4} 
\stackrel{\eqref{YH}}{\leq}
\left( \frac{4r}{\sqrt \e}\right)^d C_1 \e \ .
\end{align*}
The estimate for the Lipschtitz norm $\abs{\cH^{(d)}}_{s, r}^\cO$ is analogous, and the lemma follows.
\end{proof}

\begin{proof}[Proof of Proposition \ref{prop:adap} $(iii)$ and $(iv)$]
Item $(iii)$ follows from \eqref{H.2} and Lemma \ref{lem:norm.ham}. 
Item $(iv)$ is proved similarly by performing {\bf Step 1} and {\bf Step 2} 
on the mass  $M$ and momentum $P$, we skip the details.
\end{proof}

% % % % % % % % % % % % % % % % % % % % % % % % % % % % % % % % % % % % % % % % % % % % % % % % % % % % % % % % % % % % % % % % %

\section{Normal form  of the quadratic terms}
In this section we consider only the quadratic part of the  Hamiltonian \eqref{H.2}, namely
\begin{equation}
\label{quad.ham}
\cN^\0+\cH^\0 \equiv \omega \cdot \yy + \cD + \cH^\0 \ ,
\end{equation}
where $\cD$ is defined in \eqref{def:D} and $\cH^\0$ in \eqref{def:H0}. 
Our aim is to  reduce such  Hamiltonian  to a diagonal form  by a convergent KAM procedure by requiring some (potentially rather complicated but relatively explicit) Diophantine condition on $\lambda$. 
 The problem of the KAM algorithm is that it requires to impose the so called {\em second Melnikov conditions} on the frequencies, which do not hold for $\cN^\0$.\\
Thus, first we must  put \eqref{quad.ham}  in a normal form which is suitable in order to start a KAM algorithm to obtain the reducibility result. Following ideas of \cite{procesi12}, we perform a finite number of  normal form transformations whose effect is to put the Hamiltonian \eqref{quad.ham} in a  form suitable to the application of a KAM scheme.\\
In this section we will constantly use the notation $a\lessdot b$ with the meaning $a \leq C b$ for some positive constant $C$ independent of $\e$ and $r$.\\
The aim of the section is to prove the following result:
\begin{theorem}
	\label{thm:q}
Fix $p>1$ and $s_0 >0$. For a generic choice of the set $\Tan$  (in the sense of Definition \ref{defar}), 
	there exists $\e_0 >0$ such that  for all $0<\e\le \e_0$,  $\exists r_0\ll \sqrt\e$ s.t. the following holds true.
There  exists a  compact {\em domain} $\cO_1 \subseteq \cO_0 \subseteq [1/2,1]^\tk$  such that
	for any $\lambda \in \cO_1$ there exists an invertible  symplectic change of variables $\cT^{(0)}$: 
	$$
	\ba \mapsto \cL(\lambda,\e, \theta)\ba\,, \qquad \yy \mapsto  \yy + \im ( \ba, \cQ(\lambda,\e, \theta)\ba)  \ , \qquad \theta \mapsto \theta
	$$
	well defined and majorant analytic together with its inverse, such that $\cT^{(0)},(\cT^\0)^{-1}$:  $D(s/8, \varrho r ) \to D(s,r)$ for all $0<r \leq r_0$, $s_0/64 \le s \leq s_0$ (here $\varrho>0$ depends on $s_0,\max(|\tm_k|^p)$)  and 
	\begin{equation}
	\label{ham.bnf}
	(\omega \cdot \yy + \cD + \cH^\0)\circ \cT^{(0)}(\yy, \theta, a, \bar a) = \omega \cdot \yy + \sum_{\jj \in \Z^2 \setminus \Z} \widetilde\Omega_\jj\, |a_\jj|^2 
+ \sum_{m \in \Z \setminus \cS_0 } \Omega_{m} |a_{(m,0)}|^2 
	+ \wtcH^{(0)}+ \wtcH^\1 + \wtcH^{(\geq 2)} \ .
	\end{equation}
	Here the frequencies $\widetilde\Omega_\jj$ are defined in \eqref{def:omtilde}, and $\Omega_m$ in \eqref{norm.freq}.
	 Furthermore  $\wtcH^\0$ has scaling degree $0$ and has the form
	 $$
	 \wtcH^\0= \wtcH^{(0,\rm hor)}+ \wtcH^{(0,\rm mix)}\,,\quad  	\lceil \wtcH^{(\rm 0,hor)}\rfloor_{s/8,-2}^{\cO_1} + \lceil \wtcH^{(\rm 0,mix)}\rfloor_{s/8,-\bd}^{\cO_1}\lessdot \e^2\,.
	 $$
Similarly 	 $\wtcH^\1$ has scaling degree 1 and does not depend on $\yy$, $\wtcH^{(\geq 2)}$ has scaling degree 2, and 
	$$
	\abs{\wtcH^\1}_{s/8, \varrho r}^{\cO_1}\lessdot \sqrt{\e} \, r \ , \qquad 
	\abs{\wtcH^{(\geq 2)}}_{s/8, \varrho r}^{\cO_1} \lessdot r^2 \ . 
	$$
	The mass $\cM$ and the momentum $\cP$ of \eqref{costanti} in the new coordinates are given by 
	$$
	\cM\circ \cT^\0= \wtcM\,,\quad \cP\circ \cT^\0= \wtcP 
	$$
	defined in \eqref{mp.4}.
\end{theorem}
\begin{remark}
	\label{leggi_sel1}
	In the new variables, the {\em selection rules} of Remark \ref{leggi_sel} become
	\begin{align*}
	&\{ \cH, \wtcM\} = 0 \ \ \  \Leftrightarrow   \ \ \  \cH_{\alpha, \beta, \ell} \, (\widetilde \eta(\alpha, \beta) + \eta(\ell)) = 0 \\
	& \{ \cH, \wtcP_x\} = 0  \ \ \ \Leftrightarrow  \ \ \    \cH_{\alpha, \beta, \ell} \, (\widetilde \pi_x(\alpha, \beta) + \pi(\ell)) = 0 \\
	& \{ \cH, \wtcP_y\} = 0  \ \ \ \Leftrightarrow  \ \ \    \cH_{\alpha, \beta, \ell} \, (\pi_y(\alpha, \beta)) = 0 
	\end{align*}
	where $\eta(\ell)$ is defined in \eqref{def.eta},  $\pi(\ell)$ in \eqref{def.pi}, while
	$$
	\widetilde \eta(\alpha, \beta):= \sum_{\jj\in \Z^2\setminus( \sS \cup \ccC\cup \Tan) }(\al_\jj-\bt_\jj) \ , 
	$$
	$$
	 \widetilde{\pi}_x(\alpha, \beta):= \sum_{\jj=(m,n) \atop \jj\in \Z^2 \setminus (\sS \cup \Tan\cup \ccC)} m(\al_\jj-\bt_\jj) + \sum_{i< k \atop \jj=(m,n) \in \ccC^+_{i,k} } (m- \tm_i) (\al_\jj - \bt_\jj)\,+\!\!\!  \sum_{i< k \atop \jj=(m,n) \in \ccC^-_{i,k} } (m- \tm_k) (\al_\jj - \bt_\jj) \ . 
	$$
\end{remark}

As far as it concerns the structure of the transformation $\cT^\0$, we have the following result:
\begin{theorem}
	\label{thm:L}
We have that $\cL,\cQ$ are block diagonal in the $y$-Fourier modes. More precisely the matrix  $\cL={\rm diag}_{n\in \N}(L_{n})$, with each $ L_{n}$ acting on the sequence $(a_{(m,n)},a_{(m,-n)})_{m\in \Z}$.
The $ L_n$ satisfy the following properties:
\noindent		 $ L_0=\id$ , while $L_n$ with $n \neq 0$ is the composition of four maps: $ L_n =  L_{n}^{(\rm D)}\circ  L_{n}^{(\rm B)}  \circ R_n\circ U_n$ where 
			\begin{itemize}
				\item Setting $ L^{(\rm D)}= {\rm diag}_{n\in \N}( L^{(\rm D)}_{n})$, we have 	$ L^{(\rm D)}-\id=  L^{(\rm D, hor)}+ L^{(\rm D, mix)}$,  with $ L_{n}^{(\rm D,\rm hor)}$ independent of $n$, for all $n\neq 0$,  and
			$$  
				\lceil  L^{(\rm D,hor)} \rfloor_{s/4,-2}^{\cO_0}, \quad \lceil  L^{(\rm D,mix)} \rfloor^{\cO_0}_{s/4,-\bd}\lessdot \e\,,
				$$
				see formula \eqref{eye of the tiger}.
				%	\item $\norm{\cT_j- \uno}_{s/2} \leq \e$, $j=1,2$, i.e. $\cT_j$ are 1-smoothing maps $\e$ close to the identity
				%	\textcolor{red}{ho l'impressione che qui si possa dire meglio...non mi ricordo cosa intendevamo con one-smoothing quelloche e' vero e' se non sbaglio $\cT_j-I= A^{\rm Hor}+A^{\rm mix}$.}
				\item Setting $ L^{(\rm B)}= {\rm diag}_{n\in \N}( L^{(\rm B)}_{n})$, we have 	$ L^{(\rm B)}-\id=  L^{(\rm B, hor)}+ L^{(\rm B, mix)}$,  with $ L_{n}^{(\rm B,\rm hor)}$ independent of $n$, for all $n\neq 0$, and% $\widetilde L_j^\hor \in \cQ^\hor_{s,-1}$, $\widetilde L_j^\mix \in \cQ_{s,-\bu}^\cO$  and
			$$  
			\lceil  L^{(\rm B,hor)} \rfloor_{s/8,-2}^{\cO_0},\quad \lceil  L^{(\rm B,mix)} \rfloor_{s/8,-\bd}^{\cO_0}\lessdot \e\, . 
			$$
				\item $R_n$ is a  finite dimensional phase shift  given explicitly by formula \eqref{rotation}, moreover if \begin{equation}\label{proto}
				\ccC \cap \{(m,n), (m, -n)\}_{m \in \Z} =\emptyset
				\end{equation}
				then  $R_n$ is independent of $n$.
				\item $U_n$ acts non trivially only on  $(\sS \cup \ccC) \cap \{(m,n), (m, -n)\}_{m \in \Z}$ it is invertible and depends analytically on $\lambda\in\cO_1$ together with its inverse. Moreover if \eqref{proto} holds  then  $U_n$ is independent of $n$ and orthogonal.
			\end{itemize}
			The matrix $\cQ$ has the same structure of $\cL$.
\end{theorem}

The rest of the section is devoted to the proof of Theorem \ref{thm:q} and Theorem \ref{thm:L}.

\subsection{Proof of Theorems \ref{thm:q} and \ref{thm:L}}
We prove our statement in three steps, where we construct one by one the changes of variables  in Theorem \ref{thm:L} justifying their role in the diagonalization process.
 \begin{itemize}
 	%\item[Step 1.] We diagonalize $\cH^{\rm line}$ apart from a finite dimensional  nilpotent term
 	\item
 	[Step 1.] The  change of coordinates $\cL^{(D)}$ puts \eqref{quad.ham} into the form $$\omega\cdot \yy +\cD+\cH^\hor + \cH^\mix$$ so that $\cH^\hor \in \cQ^{\cO_0, \hor}_{s',-2}$ and  $\cH^\mix \in \cQ^{\cO_0}_{s',-\bd}$. Furthermore we isolate the terms of order $\e$.
 	\item
 	[Step 2.] The  change of coordinates $\cL^{(B)}$ removes from the terms at order $\e$ in $\cH^\hor + \cH^\mix$ all the monomials which are  Birkhoff non-resonant, in the sense of Definition \ref{def:bnr} below.
 	\item
 	[Step 3.] The  changes of coordinates $\cR,\cU$ are  not close to the identity  and are used to put the Birkhoff resonant terms at order $\e$ in diagonal form.
 \end{itemize}
 
We start by decomposing  $\cD$ and $\cH^\0$ in their line, diagonal and out-diagonal   components as defined in Section \ref{ham.pseud}. Note that $\cD = \cD^\line + \cD^\diag$ (it does not have an out-diagonal component), while 
$$
\cH^\0(\lambda; \, \theta, a, \bar a) = \cH^{\rm diag}(\lambda; \,\theta, a, \bar a) + \cH^{\rm out}(\lambda; \,\theta, a, \bar a) \ ,
$$
where 
\begin{align}
\label{def:Hdiag}
\cH^{\rm diag}(\lambda; \,\theta, a, \bar a)  &\equiv     \sum_{m_1, m_2 , n \in \Z} H^-_{m_1, m_2}(\lambda; \,\theta)   a_{(m_1,n)} \bar a_{(m_2, n)} \\
\notag
& =   2 \sum_{ 
m_1-m_2+m_3-m_4=0 \atop m_1 \neq m_2 , \, n\neq 0} q_{m_3}^\fig(\lambda; \theta) \bar q_{m_4}^\fig(\lambda; \theta) a_{(m_1,n)} \bar a_{(m_2, n)}
\\
\label{def:Hout}
\cH^{\rm out}(\lambda; \,\theta, a, \bar a) & \equiv   2\sum_{m_1, m_2 , n >0} {\rm Re } \left( H^{+}_{m_1, m_2}(\lambda; \,\theta)   a_{(m_1,n)}  a_{(m_2, -n)} \right) \\
\notag
& =   \sum_{
m_1+m_2=m_3+m_4 \atop  n>0 } (q_{m_3}^\fig(\lambda; \theta)  q_{m_4}^\fig(\lambda; \theta) \bar a_{(m_1,n)}  \bar a_{(m_2, -n)}+ \bar q_{m_1}^\fig(\lambda; \theta)  \bar q_{m_2}^\fig(\lambda; \theta)  a_{(m_1,n)}  a_{(m_2, -n)})  \ ;
\end{align}
here $( q_m^\fig(\lambda; \theta) )_{m \in \Z} = \Phi^{-1}(z^{\rm fg}(\lambda; \theta))$. Note that in \eqref{def:Hout} we are summing only on $n >0$, and not on all $n$ as in \eqref{def:H0}, hence the coefficient is changed accordingly.
Furthermore note that    $\cH^\0$ does not have a line component.

\subsubsection{Step 1:  descent method.}
{From now on we will constantly write $a \lessdot b$ with the meaning $a \leq C b$, with a constant $C$ independent of $\e, r$.}
\\
We prove the following:
\begin{lemma}\label{scendo}
There exists an invertible symplectic transformation $\cT^{(D)}: D(s/4, r/2) \to D(s, r)$ $\, \forall s_0/64 \le s\le s_0, 0<r\le  r_0$,  which transforms the Hamiltonian \eqref{quad.ham} in the following form:
\begin{align}
\label{ham.30}
(\omega\cdot \yy + \cD + \cH^\0)\circ \cT^{(D)}  = \omega \cdot \yy + \cD +\cH_{1}  +  \cH_2 \ , 
\end{align}
where 
\begin{itemize}
   \item[(i)] the map $\cT^{(D)}$ is the time-1 flow of a quadratic Hamiltonian 
   $\chi_1= \chi_1^\hor + \chi_1^\mix$ such that 
   $\lceil \chi_1^\hor \rfloor_{3s/4, -1}^{\cO_0 }+ \lceil \chi_1^\mix \rfloor_{3s/4, -\bd}^{\cO_0}  \lessdot \e$.
	\item[(ii)] $\cH_1 = \cH_1^\hor + \cH_1^\mix \in \cQ^{\cO_0, \hor}_{s/4, -2} + \cQ^{\cO_0}_{s/4, -\bd}$  and  $\lceil \cH_1^\hor \rfloor_{s/4, -2}^{\cO_0} + \lceil \cH_1^\mix \rfloor_{s/4, -\bd}^{\cO_0}   \lessdot \e$.    Explicitly 
	\begin{align}
	\label{kernel0}
	\cH_1^\hor (\lambda; \theta, a, \bar a ) & =  2\e \sum_{m_1, m_2 , n \in \Z\setminus \{0\} \atop |m_1| \neq |m_2|}  \frac{ - (\omega^\0 \cdot \partial_\theta)^2 Q^{-}_{m_1, m_2}(\lambda;\theta)}{ (m_1^2 -m_2^2)^2} \   a_{(m_1,n)} \bar a_{(m_2, n)} \\
& \qquad  +  2\e \sum_{m_1, n\neq 0} Q_{m_1}^-(\lambda; \theta) \   a_{(m_1,n)} \bar a_{(-m_1, n)} 
\\
\label{kernel00}
\cH_1^\mix (\lambda; \theta, a, \bar a )&= 2\e \sum_{m_1, m_2 , n >0}{\rm Re } \left( \frac{\im \omega^\0 \cdot \partial_\theta Q_{m_1, m_2}^{+}(\lambda;\theta) }{m_1^2+m_2^2+2n^2} \   a_{(m_1,n)}  a_{(m_2, -n)}\right)	\ ,
	\end{align}
	where
	\begin{align}
	\label{kernel10}
&	Q^{-}_{m_1, m_2} (\lambda; \theta):=  \sum_{i\neq j: \; \tm_i - \tm_j = m_2 - m_1 }
	 \sqrt{\lambda_{i} \lambda_{j}}\ e^{\im (\theta_{i} - \theta_{j})} \,,\\
	 \label{kernel11}
	& Q_{m_1}^-(\lambda; \theta):= \sum_{i\neq j: \; \tm_i - \tm_j = -2m_1 } \sqrt{\lambda_{i} \lambda_{j}}\ e^{\im (\theta_{i} - \theta_{j})}  \ , \\
	\label{kernel12}
	& Q^{+}_{m_1, m_2}(\lambda; \theta) :=  \sum_{i,j: \; \tm_i + \tm_j = m_2 +m_1 }
	 	 \sqrt{\lambda_{i} \lambda_{j}}\ e^{-\im (\theta_{i} + \theta_{j})}  \ . 
\end{align}
	 Note that $Q^\s_{m_1,m_2}(\lambda; \theta)=0$ if there exists no couple $\tm_i,\tm_j\in \Tan$ such that $\tm_i +\s \tm_j = m_1+\s m_2$.
	 
	\item[(iii)] $\cH_2 = \cH_2^\hor + \cH_2^\mix \in \cQ^{\cO_0, \hor}_{s/4, -2} + \cQ^{\cO_0}_{s/4, -\bd}$  and  $\lceil \cH_2^\hor \rfloor_{s/4, -2}^{\cO_0} + \lceil \cH_2^\mix \rfloor_{s/4, -\bd}^{\cO_0}   \lessdot \e^2$. 
	\item[(iv)] One has $\cM \circ \cT^{(D)} = \cM$ and  $\cP \circ \cT^{(D)} = \cP$.
	\end{itemize}
\end{lemma}

\begin{proof}
Decompose $\cH^\0 = \cH^\diag + \cH^\out $ where   $\cH^\diag$ and $\cH^\out$ defined in \eqref{def:Hdiag} respectively \eqref{def:Hout} are both horizontal. We recall that $\cH^\0$ does not have a line component.
We use the method of the Lie series,  constructing   the symplectic map $\cT^{(D)}$ as the time-1 flow of a    Hamiltonian $\chi_{1}$ of the form 
\begin{equation}
\begin{aligned}
& \chi_1=  \chi_1^\diag + \chi_1^\out \ , \\
& \chi_1^\diag := \sum_{m_1, m_2 , n \neq 0 } \chi_{m_1, m_2,n}^-(\lambda; \theta)  a_{(m_1,n)} \bar a_{(m_2, n)}\,,\quad \chi_{m_1, m_2,n}^-= \overline{\chi_{m_2, m_1,n}^-} \\
&  \chi_1^\out :=  2\sum_{m_1, m_2 , n>0 } {\rm Re}(\chi_{m_1, m_2,n}^+(\lambda;\theta)  a_{(m_1,n)}  a_{(m_2,- n)})\ . 
\end{aligned}
\end{equation}
By construction  $\chi_1^{\rm line}=0$, so that $\chi_1$ Poisson commutes with any Hamiltonian $\cF^{\rm line}$. This implies that $\cT^{(D)}$ preserves the line component $\cD^{\rm line}$ .
Using formula \eqref{quad.ham}  we have that
\begin{align}
\notag
(\omega\cdot \yy + \cD+ \cH^\0)\circ \cT^{(D)}  = &   \ \omega\cdot \yy + \cD \\
\label{step1.he}
&  + \{\cD^\diag,  \chi_1 \} + \cH^{\rm diag}+  \cH^{\rm out }  + \left\{ \omega \cdot \yy, \chi_1 \right\}+\frac12\{\cH^\out,\chi^\diag_1\} \\ 
& +  \{\cH^\diag,\chi_1\} +\frac12\{\{\cD^\diag+ \omega \cdot \yy,\chi_1 \},\chi_1\}+\frac12\{\cH^\out,\chi^\diag_1\} \label{step1.hb}\\
& + \{\cH^\out,\chi^\out_1\}+\re2\left( \chi_1;\cH^\diag+\cH^{\rm out} \right)+ \re3\left( \chi_1; \ \omega \cdot \yy + \cD^\diag \right) \label{step1.hc}
\end{align}
Our aim is  to reduce the {\em order} of $\cH^\0$, to this purpose we set
\begin{equation}
\label{hom.eq.0}
\left\{\cD^\diag, \chi_1 \right\} + \cH^{\rm diag}+\cH^{\rm out} +  \left\{ \omega \cdot \yy, \chi_1 \right\} +\frac12\{\cH^\out,\chi_1^\diag\} = \cZ+\cR
\end{equation}
where    $\cZ$ is in the kernel of ad$(\cD^\diag)$ while $\cR \in \cQ^{\cO_0, \hor}_{s/2, -2} + \cQ^{\cO_0}_{s/2, -\bd}$.  Note that  $\chi_1,\cZ,\cR$  are all  to be determined.
Writing \eqref{hom.eq.0} in Fourier-Taylor components,  explicitly we obtain for the diagonal part
$$
 \im ( m_1^2-m_2^2) {\chi_{m_1, m_2,n}^-}+ H^{-}_{m_1, m_2} +\omega\cdot\partial_\theta {\chi_{m_1, m_2,n}^-}  = \delta(|m_1|,|m_2|) \cZ^-_{m_1,m_2} + \cR^-_{m
_1, m_2} \ 
$$
and for the out-diagonal part
$$
\im ( m_1^2+m_2^2+ 2n^2) {\chi_{m_1, m_2,n}^+}+ H^{+}_{m_1, m_2} +\omega\cdot\partial_\theta {\chi_{m_1, m_2,n}^+} + \frac12(\{\cH^\out,\chi_1^\diag\})^+_{m_1,m_2} =  \cR^+_{m_1, m_2,n} \ ,
$$
note that the last summand in the left hand side depends only on $\chi^-$.
 A simple computation using the fact that the diagonal coefficients of the matrix $H^{\rm diag}$ are zero, i.e. $H^{-}_{m,m} =0$, shows that a solution is given by
 \begin{align*}
 & \chi_{m_1, m_2,n}^- :=
  \begin{cases}
   \displaystyle{\im \frac{H^{-}_{m_1, m_2}}{ m^2_1 - m^2_2 } - \frac{\omega\cdot\partial_\theta H^{-}_{m_1, m_2}}{ (m^2_1 - m^2_2 )^2}} \ , & |m_1| \neq |m_2| \\
   0 \ ,  &  |m_1| = |m_2| 
  \end{cases} \ , \\
 & \chi_{m_1, m_2,n}^+ := 
  \im \frac{H^{+}_{m_1, m_2}+\frac12(\{\cH^\out,\chi_1^\diag\})^+_{m_1,m_2}}{ (m^2_1 + m^2_2 +2n^2)}  \ , \\
  & \cZ=   0 \ , 
 \end{align*}
\begin{equation}
\label{R.def}
\cR^-_{m_1,m_2}=
\begin{cases}
 H^{-}_{m_1,-m_1} , & m_1 = -m_2  \\
- \displaystyle{\frac{(\omega\cdot\partial_\theta)^2 H^{-}_{m_1, m_2}}{ (m^2_1 - m^2_2 )^2}} \ , & |m_1| \neq |m_2| \\
0 & {\rm otherwise} 
\end{cases} \,, 
 \qquad  \qquad 
 \cR^+_{m_1,m_2,n}=\omega\cdot\partial_\theta \chi_{m_1, m_2,n}^+ \, . 
\end{equation}
By construction $\chi_1^\diag$ and $\cR^\diag$ are horizontal since their coefficients {\em do not} depend on $n$.
We show now that 
$$
\{\cM, \chi_1 \}=0\,,\quad \{\cP, \chi_1\} = 0 \ .
$$
This follows immediately by remarking that $\cH^\diag$ and $\cH^\out$ commute with $\cM$ and $\cP$ and hence they satisfy the selection rules of Remark \ref{leggi_sel}. 
By the explicit formula for $\chi_1$ it follows that the same selection rules hold for $\chi_1$.
This shows that item $(iv)$ holds.\\
We pass now to the quantitative estimates.
We define 
$$\chi_1^\hor:= (\chi_1)^\diag  \ , \qquad \chi_1^\mix:=(\chi_1)^\out \ . $$
In order to estimate $\chi_1^\hor$, we split it in two components by writing
\begin{align*}
& \chi_1^\hor:= \chi_{1,-1}^\hor + \chi_{1,-2}^\hor \ , \\
& (\chi_{1,-1}^\hor)_{m_1,m_2}:= 
\begin{cases}
\im \frac{H^{-}_{m_1, m_2}}{ m^2_1 - m^2_2 } \ , &  |m_1| \neq |m_2| \\
0 \ , & |m_1| = |m_2|
\end{cases} \ , \ \ \ \
(\chi_{1,-2}^\hor)_{m_1,m_2}:=
\begin{cases}
- \frac{\omega\cdot\partial_\theta H^{-}_{m_1, m_2}}{ (m^2_1 - m^2_2 )^2} \ , & |m_1| \neq |m_2| \\
0  \ , & |m_1| \neq |m_2| \ . 
\end{cases}
\end{align*}
We  prove that $\chi_{1,-1}^\hor \in \cQ^{\cO_0, \hor}_{s, -1}$, $\chi_{1,-2}^\hor \in \cQ^{\cO_0, \hor}_{3s/4, -2}$, $\chi_{1}^\mix \in \cQ^{\cO_0}_{3s/4, -\bd}$ with quantitative estimates 
\begin{equation}
\label{chi1.est}
\lceil\chi_{1,-1}^\hor\rfloor_{s,-1}^{\cO_0}
+
 \lceil\chi_{1,-2}^\hor\rfloor_{3s/4,-2}^{\cO_0} 
 +
 \lceil\chi_1^\mix \rfloor_{3s/4,-\bd}^{\cO_0} 
 \lessdot (|\cH^{\rm diag}|_s^{\cO_0} + |\cH^{\rm out}|_s^{\cO_0}) \lessdot \e  \ . 
\end{equation}
We begin with $\chi_{1,-1}^\hor $. 
By definition $\lceil\chi_{1,-1}^\hor\rfloor_{s,-1}^{\cO_0}$  is the $\abs{\cdot}_s^{\cO_0}$ norm of the Hamiltonian with Fourier-Taylor  coefficients 
$\left\{\la m_1 \ra \, \abs{\dfrac{H^{-,\ell}_{m_1, m_2}}{ m^2_1 - m^2_2 } }\right\}$.  Such coefficients are estimated by
$$
\abs{\la m_1 \ra \, (\chi_{1,-1}^\hor)_{m_1, m_2}^{-,\ell} }^{\cO_0}_\C \leq   \frac{\la m_1 \ra}{ | m^2_1 - m^2_2 | }  \abs{H^{-,\ell}_{m_1, m_2}}^{\cO_0}_\C \leq \abs{H^{-,\ell}_{m_1, m_2}}^{\cO_0}_\C
$$
since $\frac{\la m_1\ra}{|m_1^2- m_2^2|}\le 1$ for $|m_1| \neq |m_2|$. Thus, by Proposition \ref{riassunto} (iii), $\lceil\chi_{1,-1}^\hor\rfloor_{s,-1}^{\cO_0} \lessdot |\cH^{\rm diag}|_s^{\cO_0}$.\\
Consider now $\chi_{1,-2}^\hor $. By construction 
$$
\abs{\la m_1 \ra^2 \, (\chi_{1,-2}^\hor)_{m_1, m_2}^{-,\ell} }^{\cO_0}_\C \lessdot \,  \frac{\la m_1 \ra^2 |\ell|}{ | m^2_1 - m^2_2 |^2 }  \abs{H^{-,\ell}_{m_1, m_2}}^{\cO_0}_\C \lessdot \, |\ell| \abs{H^{-,\ell}_{m_1, m_2}}^{\cO_0}_\C \ ,  
$$
which implies that
$\lceil\chi_{1,-2}^\hor\rfloor_{3s/4,-2}^{\cO_0} \lessdot |\cH^\diag|^{\cO_0}_s $.

Finally consider $\chi_1^\mix$. By definition, 
$\lceil\chi_1^\mix\rfloor_{s,-\bd}^{\cO_0}$ is  the  $\abs{\cdot}_s^{\cO_0}$ norm of the  Hamiltonian with Fourier-Taylor coefficients 
$\{(\la m_1 \ra^2 + \la n \ra^2 ) \, \abs{\chi_{m_1, m_2,n}^{+,\ell}} \}$. We denote  $G^{+,\ell}_{m_1,m_2} :=\frac{1}{2}(\{\cH^\out,\chi^\diag\})^{+,\ell}_{m_1,m_2}$ and get
\begin{align*}
\abs{(\la m_1 \ra^2 + \la n \ra^2 )  \, \chi_{m_1, m_2,n}^{+,\ell}}^{\cO_0}_\C 
& \leq   \frac{(\la m_1 \ra^2 + \la n \ra^2 )  }{ | m^2_1 + m^2_2 + 2n^2 | }\, \big( \abs{H^{+,\ell}_{m_1, m_2}}^{\cO_0}_\C+  \abs{G^{+,\ell}_{m_1, m_2}}^{\cO_0}_\C\big ) \\
& \lessdot \big( \abs{H^{+,\ell}_{m_1, m_2}}^{\cO_0}_\C   +  \abs{G^{+,\ell}_{m_1, m_2}}^{\cO_0}_\C\big )
\end{align*}
since $\frac{(\la m_1 \ra^2 + \la n \ra^2 )  }{m_1^2+ m_2^2+ 2n^2}\le 1$.
Once again we get by Proposition \ref{riassunto} (iii) that 
$$
\lceil\chi_1^\mix\rfloor_{3s/4,-\bd}^{\cO_0} \lessdot |\cH^\out|_s^{\cO_0}(1+2|\chi_1^\diag|_{3s/4}^{\cO_0}) \lessdot 
 |\cH^\out|_s^{\cO_0}(1+ |\cH^\diag|_{s}^{\cO_0}) \ .
$$
We have  thus proven \eqref{chi1.est}.

We show now that  line \eqref{step1.he} belongs to $\cQ^{\cO_0, \hor}_{s/2, -2} + \cQ^{\cO_0}_{s/2, -\bd}$, i.e. it  is either horizontal and of order $-2$ or of order $-\bd$.
By the homological equation \eqref{hom.eq.0}, line \eqref{step1.he} equals 
\begin{align*}
\cR & = \cR^\diag + \cR^\out = \cR^\diag_1 + \cR^\diag_2 + \cR^\out \ , \\
& \cR^\diag_1 = \left\{\delta(m_1,-m_2) H^{-}_{m_1,-m_1}\right\} \ , \qquad \cR^\diag_2=\omega \cdot \partial_\theta  \chi^\hor_{1,-2}  \ ,\\
& \cR^\out = \omega \cdot \partial_\theta \chi_1^\mix \equiv 
\{ \omega\cdot\partial_\theta \chi_{m_1, m_2,n}^+\}
\end{align*}
First we estimate $\cR^\diag_1$. By momentum conservation, see Remark \ref{leggi_sel}, its coefficients are not-zero 
 only if $2m_1= -\pi(\ell)$ so 
 $\lceil \cR^\diag_1 \rfloor_{3s/4,-2}^{\cO_0} \lessdot  |\cH^\diag|_{s}^{\cO_0}$.  \\
Consider now $\cR^\diag_2$. One has 
$\lceil \cR^\diag_2 \rfloor_{s/2,-2}^{\cO_0} \lessdot \lceil \chi^\hor_{1,-2} \rfloor_{3s/4,-2}^{\cO_0} \lessdot |\cH^\diag|_{s}^{\cO_0}$. \\
Finally consider $\cR^\out$. One has 
 $$
 \lceil \cR^\out \rfloor_{s/2,-\bd}^{\cO_0} \lessdot 
\lceil \chi_1^\mix \rfloor_{3s/4,-\bd}^{\cO_0} \lessdot 
  |\cH^\out|_{s}^{\cO_0} \ .
 $$
 Collecting all the estimates we have that 
\begin{equation}
\label{R.est}
\lceil \cR^\diag \rfloor_{s/2,-2}^{\cO_0} + 
\lceil \cR^\out \rfloor_{s/2,-\bd}^{\cO_0} 
\lessdot |\cH^\0|_{s}^{\cO_0} \lessdot \e  \ .
\end{equation}
Thus we proved that line \eqref{step1.he} belongs to  $\cQ_{s/2, -2}^{ \cO_0, \hor} + \cQ_{s/2, -\bd}^{\cO_0}$.\\
 Let us now prove the same for \eqref{step1.hb}. 
 The first term to consider is 
\begin{equation}\label{partocesareo}
\{\cH^{\rm diag},\chi_1 \}=\{\cH^{\rm diag},\chi_{1,-1}^\hor \}+\{\cH^{\rm diag},\chi_{1,-2}^\hor \} + 
\{\cH^{\rm diag},\chi_1^\mix \} \ , 
\end{equation}
where we used that $\chi_1 =\chi_{1,-1}^\hor +\chi_{1,-2}^\hor + \chi_1^\mix$. 
We will prove that $\{\cH^{\rm diag},\chi_{1,-1}^\hor \}+\{\cH^{\rm diag},\chi_{1,-2}^\hor \} \in \cQ^{\cO_0, \hor}_{s/2, -2}$ while $\{\cH^{\rm diag},\chi_1^\mix \}  \in \cQ^{\cO_0}_{s/2, - \bd}$ with quantitative estimates
\begin{equation}
\label{linea4.15}
\lceil \{\cH^{\rm diag},\chi_{1,-1}^\hor \} \rfloor_{s/2,-2}^{\cO_0}  
+ \lceil \{\cH^{\rm diag},\chi_{1,-2}^\hor \} \rfloor_{s/2,-2}^{\cO_0} + 
\lceil \{\cH^{\rm diag},\chi_{1}^\mix \} \rfloor_{s/2,-\bd}^{\cO_0}  \lessdot (|\cH^\diag|_s^{\cO_0})^2 \lessdot \e^2 
\end{equation}
First note that since both $\cH^\diag$ and $\chi_1^\diag$ are horizontal, $\{\cH^{\rm diag},\chi_{1,-1}^\hor \}+\{\cH^{\rm diag},\chi_{1,-2}^\hor \} $ is horizontal as well. We pass to quantitative estimates.
By Lemma \ref{lem:smo} we have 
\begin{equation}
\label{partocesareo.est}
\lceil \{\cH^{\rm diag},\chi_{1,-2}^\hor \} \rfloor_{s/2,-2}^{\cO_0} \lessdot  |\chi_{1, -2}^\hor|_{3s/4, -2}^{\cO_0} \, |\cH^\diag|_s^{\cO_0} \lessdot \e^2 \ .
\end{equation}
Now  $\{\cH^{\rm diag},\chi_{1,-1}^\hor \}$ is   apparently only of order $-1$; however a direct computation shows that it is in fact of order $-2$:
	\begin{align*}
	(\{\cH^{\rm diag},\chi_{1,-1}^\hor \})_{m_1,m_2}^\ell
&	=   -   
\sum_{{m_3,\ell_1,\ell_2 \atop \ell_1+\ell_2=\ell\,,\; m_1-m_3= \pi(\ell_1)}\atop {m_3-m_2=\pi(\ell_2) \atop |m_3|\neq |m_2|} }
 \frac{H^{-, \ell_1}_{m_1, m_3 } H^{-, \ell_2}_{m_3, m_2 }}{(m_3^2-m_2^2)}
 +
 \sum_{{m_4,\ell_1,\ell_2 \atop \ell_1+\ell_2=\ell\,,\; m_1-m_4= \pi(\ell_2)}\atop {m_4-m_2=\pi(\ell_1) \atop |m_1| \neq |m_4| }} 
  \frac{H^{-, \ell_2}_{m_1, m_4 } H^{-, \ell_1}_{m_4, m_2 }}{(m_1^2-m_4^2)}\\
& =  -   \sum_{{m_3,\ell_1,\ell_2 \atop \ell_1+\ell_2=\ell\,,\; m_1-m_3= \pi(\ell_1)}\atop{m_3-m_2=\pi(\ell_2) \atop |m_3|\neq |m_2|, \ |m_1| \neq |m_1-m_3+m_2|} } 
 \frac{H^{-, \ell_1}_{m_1, m_3 } H^{-, \ell_2}_{m_3, m_2 }}{(m_3^2-m_2^2)}- \frac{H^{-, \ell_2}_{m_1, m_1-m_3+m_2 } H^{-, \ell_1}_{m_1-m_3+m_2, m_2 }}{(m_1^2-(m_1-m_3+m_2)^2)}   
	   \\
& =  -  \sum_{{m_3,\ell_1,\ell_2 \atop \ell_1+\ell_2=\ell\,,\; m_1-m_3= \pi(\ell_1)}\atop{m_3-m_2=\pi(\ell_2) \atop |m_3|\neq |m_2|, \ |m_1| \neq |m_1-m_3+m_2|}  }
 H^{-, \ell_1}_{m_1, m_3 } H^{-, \ell_2}_{m_3, m_2 } \left(\frac{1}{(m_3^2-m_2^2)}- \frac{1}{(m_1^2-(m_1-m_3+m_2)^2)}\right)   
	 \\
& =  -  \sum_{{m_3,\ell_1,\ell_2 \atop \ell_1+\ell_2=\ell\,,\; m_1-m_3= \pi(\ell_1)}\atop{m_3-m_2=\pi(\ell_2) \atop |m_3|\neq |m_2|, \ |m_1| \neq |m_1-m_3+m_2|}  } 
H^{-, \ell_1}_{m_1, m_3 } H^{-, \ell_2}_{m_3, m_2 }
 (\frac{ {2(m_1 - m_3)}}{(m_3-m_2)(m_3+m_2)(2m_1-m_3+m_2))}) 
	\end{align*}
Note that the conditions $|m_3|\neq |m_2|$ and $|m_1|\neq |m_4|$ come from the definition of $\chi_1^\diag$.
	By Remark \ref{paralin} we may assume $|\ell|,|\ell_i|\le  c\langle m_1\rangle$ (where $\displaystyle{c^{-1}= 2 \max_{1 \leq i \leq \tk }(|\tm_i|)}$ )  since otherwise we have  an infinitely  smoothing term. 
We have
$$
m_3+m_2 =2m_1+ m_3-m_1+m_2-m_1= 2m_1-\pi(\ell_1)-\pi(\ell)
$$
thus using also that $|\pi(\ell_1) + \pi(\ell)| \leq |m_1| $ and $|\pi(\ell_2)| \leq |m_1|$ one deduce the estimate
\begin{align*}
\langle m_1\rangle^2 \frac{2|m_1-m_3|}{|m_3-m_2||m_3+m_2||2m_1-m_3+m_2|}
\leq 
\frac{2|\pi(\ell_1)| \langle m_1\rangle^2 }{|\pi(\ell_2)||2m_1-\pi(\ell_1)-\pi(\ell)| \, |2m_1-\pi(\ell_2)| }
\le 2 |\pi(\ell_1)| \,.
\end{align*}
This implies that  $\{\cH^{\rm diag},\chi_{1,-1}^\diag \}\in \cQ_{3s/4,-2}^{\cO_0, \hor}$ with quantitative estimate
\begin{equation}
\label{partocesareo.est2}
\lceil \{\cH^{\rm diag},\chi_{1,-1}^\hor \} \rfloor_{3s/4,-2}^{\cO_0} \lessdot  (|\cH^\diag|_s^{\cO_0})^2 \lessdot \e^2 \ .
\end{equation}
Finally consider $ \{\cH^{\rm diag},\chi_{1}^\mix \} $. By Lemma \ref{lem:pseudo0.00}$(i)$ one has 
\begin{equation}
\label{partocesareo.est3}
\lceil \{\cH^{\rm diag},\chi_{1}^\mix \} \rfloor_{s/2,-\bd}^{\cO_0} 
\lessdot |\chi_1^\mix|_{3s/4, -\bd}^{\cO_0} \  |\cH^\diag|_s^{\cO_0} \lessdot \e^2 \ .
\end{equation}
Thus  estimate \eqref{linea4.15} follows from \eqref{partocesareo.est}, \eqref{partocesareo.est2}, \eqref{partocesareo.est3}.\\
Consider now the last two terms of line \eqref{step1.hb}. Substituting \eqref{hom.eq.0} into 
\begin{align}
\notag
\frac12\{\{\cD^\diag+ \omega \cdot \yy,\chi_1 \},\chi_1\}+\frac12\{\cH^\out,\chi^\diag_1\} 
 = & \frac12\{-\cH^{\rm diag}-\frac12\{\cH^\out,\chi_1^\diag\} +\cR,\chi_1\} {-\frac{1}{2} \{\cH^\out, \chi_1^\mix\}}\\
\label{seclina}
 = & \frac12 \{\chi_1,\cH^\diag\} 
+ \frac14 \{\{\chi_1^\diag, \cH^\out\}, \chi_1\}
 \\ 
 &+ \frac12 \{\cR, \chi_1\}{+ \frac12 \{\chi^\mix_1, \cH^\out\} } \ . 
\end{align}
The first term of \eqref{seclina} is the same as \eqref{partocesareo}, thus it is estimated by \eqref{linea4.15}. \\
Consider now the second term in \eqref{seclina}. Using that $\cH^\out$ is horizontal one decompose  
$$
\{\{\chi_1^\diag, \cH^\out\}, \chi_1\} = 
\{\{\chi_1^\diag, \cH^\out\}, \chi_1^\diag\} + 
\{\{\chi_1^\diag, \cH^\out\}, \chi_1^\out\}  
 \in  \cQ^{\cO_0, \hor}_{s/4, -2} +  \cQ^{\cO_0}_{s/4, -\bd} 
$$
 with quantitative estimates
\begin{equation}
 \label{seclina2}
 \lceil \{\{\chi_1^\diag, \cH^\out\}, \chi_1^\diag\}  \rfloor_{s/4, -2}^{\cO_0} 
 +
 \lceil \{\{\chi_1^\diag, \cH^\out\}, \chi_1^\out\}   \rfloor_{s/4, -\bd}^{\cO_0} \lessdot (|\cH|_{s}^{\cO_0})^3 \lessdot \e^3 \ .
 \end{equation} 
Next consider the  term  $\{\cR, \chi_1\}$. By construction 
$\cR = \cR^\diag + \cR^\out$, 
$\chi_1 = \chi_1^\diag + \chi_1^\out$ 
with $\cR^\diag$, $\chi_1^\diag$ horizontal. 
So
$$
\{\cR, \chi_1\} 
= 
\{\cR^\diag, \chi_1^\diag \} 
+ \{\cR^\diag, \chi_1^\out\} 
+ \{\cR^\out, \chi_1^\diag\} 
+ \{\cR^\out, \chi_1^\out\} $$
and the first term belongs to $ \cQ^{\cO_0, \hor}_{s/4, -2} $, 
while the last three belong to $ \cQ^{\cO_0}_{s/4, -\bd}$, and \eqref{chi1.est}, \eqref{R.est} give the quantitative estimates
\begin{equation}
\label{seclina3}
\lceil \{\cR^\diag, \chi_1^\diag \} \rfloor_{s/4, -2}^{\cO_0} + 
\lceil \{\cR^\diag, \chi_1^\out \} \rfloor_{s/4, -\bd}^{\cO_0} + 
\lceil \{\cR^\out, \chi_1^\diag \} \rfloor_{s/4, -\bd}^{\cO_0} + 
\lceil \{\cR^\out, \chi_1^\out \} \rfloor_{s/4, -\bd}^{\cO_0} 
\lessdot \e^2 \ . 
\end{equation}
Finally we study $\{\cH^\out, \chi_1^\mix \}$, which is easily seen to belong to $\cQ_{s/4, - \bd}^{\cO_0}$ with quantitative estimate
\begin{equation}
\label{seclina5}
\lceil \{\cH^\out, \chi_1^\mix \} \rfloor_{s/4, -\bd}^{\cO_0} 
\lessdot \e^2 \ . 
\end{equation}
Altogether we have proved that $\eqref{step1.hb} = \eqref{step1.hb}^\hor + \eqref{step1.hb}^\mix \in \cQ_{s/4, - 2}^{\cO_0, \hor} + \cQ_{s/4, - \bd}^{\cO_0}$ with estimates
\begin{equation}
\label{seclina6}
\lceil \eqref{step1.hb}^\hor  \rfloor_{s/4, -2}^{\cO_0} + 
\lceil \eqref{step1.hb}^\mix \rfloor_{s/4, -\bd}^{\cO_0} 
\lessdot \e^2 \ . 
\end{equation}
\\
In order to study line \eqref{step1.hc}  we apply Lemma \ref{lem:smo} to  $\re2\left( \chi_1; \cH^{\rm diag}+\cH^{\rm out}\right)$. We have
$$
\lceil \re2\left( \chi_{1}; \cH^{\rm diag}+\cH^{\rm out}\right)^\hor \rfloor_{s/2, -2}^{\cO_0} \lessdot \e^2 
\,,\qquad
\lceil \re2\left( \chi_1; \cH^{\rm diag}+\cH^{\rm out}\right)^\mix \rfloor_{s/2, -\bd}^{\cO_0} \lessdot \e^2 \ . 
$$
%Now consider $\{\omega \cdot \yy, \chi_1\}$. An explicit computation of the Poisson bracket shows that
%\begin{equation}
%\label{omega.J.chi0}
%\{\omega \cdot \yy, \chi_1\} = \sum_{m_1, m_2 , n \in \Z\setminus \{0\}}  \omega \cdot \partial_\theta\chi_{m_1, m_2}^-(\lambda;\theta)\   a_{(m_1,n)} \bar a_{(m_2, n)}+2  \sum_{m_1, m_2 , n >0} {\rm Re}(\omega \cdot \partial_\theta\chi_{m_1, m_2,n}^+(\lambda;\theta)\   a_{(m_1,n)}  a_{(m_2, n)}  )\ .
%\end{equation}
%Note that the first summand is horizontal while the second is of order $-\bd$, so  by Cauchy estimates
%$$
%\lceil\{\omega \cdot \yy, \chi_1\}^\hor \rfloor_{2s/3,-1}+ \lceil\{\omega \cdot \yy, \chi_1\}^\mix \rfloor_{2s/3,-\bd} \leq C\e \ .
%$$
Now consider  $\re3\left( \chi_1; \omega \cdot \yy +\cD^\diag\right)$. Write it as
\begin{align*}
\re3\left( \chi_1; \omega \cdot \yy +\cD^\diag\right) =& \sum_{k \geq 3} \frac{{\rm ad}(\chi_1)^k[\omega \cdot \yy +\cD^\diag]}{k!} = \sum_{k \geq 3} \frac{{\rm ad}(\chi_1)^{k-1}[\{ \omega \cdot \yy+\cD^\diag , \chi_1\}]}{k!} \\
=& \sum_{k \geq 3} \frac{{\rm ad}(\chi_1)^{k-1}[ -\cH^{\rm diag}-\cH^{\rm out}+\frac12\{\chi_1^\diag,\cH^\out\} + \cZ+\cR]}{k!} \ . 
\end{align*}
%Since   $\{ \omega \cdot \yy,\chi_1 \}= \{\omega \cdot \yy, \chi_1\}^\hor+ \{\omega \cdot \yy, \chi_1\}^\mix$ we proceed as in 
Again we apply Lemma \ref{lem:smo} to obtain
$$
\lceil \left(\re3(\chi_1; \omega \cdot \yy+\cD^\diag)\right)^\hor\rfloor_{s/4, -2}^{\cO_0}+\lceil \left(\re3(\chi_1; \omega \cdot \yy+\cD^\diag)\right)^\mix\rfloor_{s/4, -\bd}^{\cO_0}  \lessdot \e^2 \ . 
$$
So far we have proved that
$$(\omega\cdot \yy + \cD^\diag + \cH^\0)\circ \cT^{(D)}= \omega\cdot \yy + \cD^\diag + \cR + \wt\cH $$ 
 with 
 $\wt\cH = \wt\cH^\hor + \wt\cH^\mix \in  \cQ_{s/4,-2}^{\cO_0,\hor} +  \cQ_{s/4,-\bd}^{\cO_0}$ fulfilling
 $$
 \lceil \wt\cH^\hor  \rfloor_{s/4, -2}^{\cO_0} +  \lceil \wt\cH^\mix \rfloor_{s/4, -\bd}^{\cO_0} \lessdot \e^2 \ .
 $$
The next step is to extract from the so obtained  Hamiltonian the  terms which are exactly of order $\e$. Obviously such terms can be contained only in $\cR$, so 
 we extract from $\cR$  the monomials of order exactly $\e$, which we denote by $ \cH_1$:  
$$
 \cH_1 := \partial_\e \cR \vert_{\e = 0} \ .
$$
	By formula \eqref{finite.gap.formula}, it follows that
\begin{equation}
q_m^\fig(\theta) = 
\begin{cases}
\sqrt{\e} \, \sqrt{ \lambda_i} \, e^{\im \theta_i} + \e^{3/2} p_m(\theta) \ , \quad m=\tm_i \in \Tan\\
\e^{3/2} p_m(\theta) \ , \quad m \notin \Tan
\end{cases} \ , 
\end{equation}	
	thus
	\begin{align}
	H_{m_1, m_2}^{-}(\lambda;\theta) & = 2\sum^\star_{m_3 - m_4 = m_2 - m_1 } q_{m_3}^\fig (\theta) \bar q_{m_4}^\fig (\theta) \\
	\label{kernel1}
	& = 2\e \sum_{i\neq j: \; \tm_i - \tm_j = m_2 - m_1 }
	 \sqrt{\lambda_{i} \lambda_{j}}\ e^{\im (\theta_{i} - \theta_{j})} + {\cO}(\e^2) 
	\end{align}
	and
		\begin{align}
		\notag
		H_{m_1, m_2}^{+}(\lambda; \theta) & = \sum_{m_3 + m_4 = m_2 + m_1 } \bar q_{m_3}^\fig (\theta)  \bar q_{m_4}^\fig(\theta) \\
		\label{kernel2}
		& = \e \sum_{i,j: \; \tm_i + \tm_j = m_2 + m_1 }
		 \sqrt{\lambda_{i} \lambda_{j}}\ e^{-\im (\theta_{i} + \theta_{j})} + \cO(\e^2).
		\end{align}
	Now we substitute
	 such kernels into the expression of $\cR $ \eqref{R.def} and  separate the terms of order $\e$ which define $\cH_1 $. All the other terms define $\cH_2$.
\end{proof}

\subsubsection{Step 2: removal of Birkhoff non resonant monomials at order $\e$}
We begin with the following definition:

\begin{definition}
\label{def:bnr}
 A monomial of the form  $ e^{\im \theta \cdot \ell } a^\alpha \bar a^\beta$, $|\alpha|+|\beta|=2$ will be said to be  {\em Birkhoff resonant} iff
$$
\omega^\0 \cdot \ell + \Omega^\0 \cdot(\alpha - \beta) = 0 \ .
$$
\end{definition}
In the next lemma we describe the monomials in the  Hamiltonian $\cH_1$ (defined in \eqref{ham.30}) which are Birkhoff resonant. 

\begin{lemma}
\label{lem:br}
The following holds true:
\begin{itemize}
\item[(i)] 
Consider a  monomial $ e^{\im \theta \cdot \ell }a_{(m_1,n)} \bar a_{(m_2,n)}$, $|m_1| \neq |m_2|$, $|\ell| =2$.    Then the monomial is Birkhoff resonant iff there exist $ 1 \leq i,j \leq \tk$ such that
\begin{equation}
\label{res1}
m_1=\tm_j \ , \ m_2 = \tm_i  \ , \ \ \  \ell = {\bf e}_i - {\bf e}_j \ .
\end{equation}
Such monomials  have support in $\bigcup_{n\in \Z} \cS_{0n}$. Thus, the  support of such monomials forms a horizontal rectangle with two points in $\Tan$. 
\item[(ii)] Consider a  monomial $ e^{\im \theta \cdot \ell }a_{(m_1,n)} \bar a_{(-m_1,n)}$, $|\ell| =2$.    Then the monomial is Birkhoff  resonant iff there exist $1 \leq i,j \leq \tk$, $i \neq j$, such that 
 $\tm_i =-\tm_j$.
\item[(iii)] Consider a  monomial $ e^{\im \theta \cdot \ell }\bar a_{(m_1,n)} \bar a_{(m_2,-n)}$,  $|\ell| = 2$.   Then the monomial is Birkhoff resonant iff there exist $1 \leq i,j \leq \tk$ such that
\begin{equation}
\label{res2} 
m_2=\tm_i + \tm_j - m_1 \ ,  \ \ \ (m_1 - \tm_i)(m_1-\tm_j) + n^2 =0  \ , \ \ \ \ell = {\bf e}_i + {\bf e}_j \ .
\end{equation}
Such monomials  have support in $\bigcup_{i<j} \ccC_{ij}$. Thus, their support form a rotated rectangle with two point in $\Tan$.
\end{itemize}
\end{lemma}
\begin{proof}
 $(i)$ Obviously if \eqref{res1} holds, then the monomial is Birkhoff resonant. Assume now that the monomial is Birkhoff resonant. We will use the conservation of mass and momentum as described in Remark \ref{rem:mm.use}. 
  By conservation of mass  $\sum_i \ell_i = 0$, which together  with $|\ell| =2$ implies $\ell = {\bf e}_i - {\bf e}_j$, $i \neq j$.
   Now by  conservation of momentum  $\pi(\ell) +m_1 - m_2 = 0$, which shows that $\tm_i - \tm_j + m_1 - m_2 = 0$. Thus
\begin{align*}
0 = \oo^\0 \cdot \ell+\Omega^\0\cdot (\al-\bt) & = \oo^\0 \cdot({\bf e}_{i} - {\bf e}_{j}) + m_1^2 - m_2^2 = \tm_i^2 - \tm_j^2 + m_1^2 - m_2^2 \\
& = 2(\tm_i - \tm_j)(\tm_j- m_1) \ .
\end{align*}
If   $\tm_j = \tm_i$, then by momentum conservation $m_1 = m_2$, which contradicts  $|m_1| \neq |m_2|$. Thus $m_1 = \tm_j$, $m_2 = \tm_i$.

$(ii)$ By conservation of mass and momentum $\ell = {\bf e}_i - {\bf e}_j$,  $i \neq j$, $\tm_i - \tm_j + 2m_1= 0$. Then
\begin{align*}
0 = \oo^\0 \cdot \ell+\Omega^\0\cdot (\al-\bt) & =  \tm_i^2 - \tm_j^2 \ = (\tm_i - \tm_j)(\tm_j + \tm_i) \ .
\end{align*}
Since $i \neq j$, it follows that $\tm_j + \tm_i = 0$.

$(iii)$ Once again if \eqref{res2} holds, than the monomial is Birkhoff resonant. Assume now to have a  Birkhoff resonant monomial.
By conservation of mass $\sum_i \ell_i = 2$, which together with $|\ell|=2$ implies $\ell = {\bf e}_i + {\bf e}_j$. Now the conservation of momentum reads $\pi(\ell) - m_1 - m_2 = 0$, hence one has $\tm_i + \tm_j - m_1 - m_2 =0$. Thus
\begin{align*}
0 = \oo^\0 \cdot \ell+\Omega^\0\cdot (\al-\bt) & = \oo^\0 \cdot({\bf e}_{i} + {\bf e}_{j}) - m_1^2 - m_2^2 - 2 n^2 = \tm_i^2 + \tm_j^2 - m_1^2 - m_2^2 - 2n^2\\
& =-2[(m_1 - \tm_i)(m_1-\tm_j) + n^2 ]  \ ,
\end{align*}
which shows the claimed condition.
\end{proof}

\begin{remark}
By the condition of arithmetic genericity of $\Tan$ (see Definition \ref{defar}), one has $\tm_i \neq - \tm_j$ $\forall i,j$. Thus the monomials described in Lemma \ref{lem:br} (ii) are always Birkhoff non resonant.
\end{remark}

In the next lemma we perform a canonical transformation which removes all the Birkhoff non  resonant monomials.
	\begin{lemma}\label{medio}
There exists an invertible symplectic transformation $\cT^{(B)}: D(s/8, r/4) \to D(s/4, r/2)$  $\forall s_0/64 \le s \leq s_0, 0 < r \leq r_0$ which transforms the Hamiltonian \eqref{ham.30} in the following form:
\begin{align}
\label{ham.40}
\left(\omega \cdot \yy + \cD +\cH_{1}  +  \cH_2\right) \circ \cT^{(B)} =  \omega \cdot \yy + \cD + \cZ_{1} + \whcH_{2}  \ , 
\end{align}
where 
\begin{itemize}
 \item[(i)] the map $\cT^{(B)}$ is the time-1 flow of a quadratic Hamiltonian $\chi_2= \chi_2^\hor + \chi_2^\mix$ such that $\lceil \chi_2^\hor \rfloor_{s/4, -2}^{\cO_0} + \lceil \chi_2^\mix \rfloor_{s/4, -\bd}^{\cO_0}  \lessdot \e$.
	\item[(ii)] $\cZ_1$ is the {\em Birkhof resonant} part of $\cH_1$ and has the following form:  $\cZ_1  = \cZ^\hor_1 + \cZ^\mix_1$, where 
	\begin{align*}
\cZ_1^\hor=	\cZ^{\rm diag}_1 := 2\e \sum_{n\neq 0}\sum_{i\neq j} \sqrt{\lambda_i \lambda_j}\,  e^{\im(\theta_i-\theta_j)} \,  a_{(\tm_j,n)}\bar a_{(\tm_i,n)} \\
\cZ_1^\mix=	\cZ^{\rm out}_1 := 4\e \, \sum_{i<j}\sum_{(m,n)\in \ccC_{i,j}^+ }	 \sqrt{\lambda_i \lambda_j} \ {\rm Re}(e^{\im(\theta_i+\theta_j)} \ \bar a_{(m,n)}\bar a_{(\tm_i+\tm_j-m,-n)} )\ ,
	\end{align*}
	where $\ccC_{i,j}^+$ is defined in \eqref{def:lambda}. Note that $\cZ_1^\hor$ is both horizontal and diagonal while $\cZ_1^\mix$ is out-diagonal and supported only on the finite set $\ccC_{i,j}$.
	\item[(iii)] $\whcH_2 = \whcH_2^\hor + \whcH_2^\mix \in \cQ_{s/8, -2}^{\cO_0, \hor} +\cQ_{s/8, -\bd}^{\cO_0}  $  and  $\lceil \whcH_2^\hor \rfloor_{s/8, -2}^{\cO_0} + \lceil \whcH_2^\mix \rfloor_{s/8, -\bd}^{\cO_0}   \lessdot \e^2$.  
	\item[(iv)]  One has $\cM \circ \cT^{(B)} = \cM$ and  $\cP \circ \cT^{(B)} = \cP$.
	\end{itemize}
	\end{lemma} 
\begin{proof}
Once again we use the method of the Lie series. Thus we look for $\cT^{(B)}$ as the time-1 flow map of an  Hamiltonian  $\chi_2=\chi_2^\diag+\chi_2^\out $ to be determined. As in the previous step $\chi_2$ Poisson commutes with $\cD^{\rm line}$. Then one has
\begin{align}
\notag
\left(\omega \cdot \yy + \cD +\cH_{1}  +  \cH_2\right) \circ \cT^{(B)}     = & \  \omega \cdot \yy + \cD   \\
\label{red.10}
& + \{ \omega \cdot \yy + \cD^{\rm diag} , \chi_2 \} + \cH_{1}  \\
\label{red.20}
& + \re2(\chi_2; \, \omega \cdot \yy + \cD^{\rm diag} ) + \re1(\chi_2; \,  \cH_1 )\\
\label{red.30}
& + \cH_{2}\circ \cT^{(B)} 
\end{align}
This time we fix $\chi_2$ in order to  remove the Birkhoff non resonant terms of $\cH_1$, so   we solve the homological equation with $\omega \cdot \yy + \cD^{\rm diag}$ (on the contrary of \eqref{hom.eq.0}).  The homological equation is 
$$
\{ \omega \cdot \yy + \cD^{\rm diag}, \chi_2\} + \cH_{1} =\cZ_1
$$
for some $\chi_2, \cZ_1$ to be determined.
We claim that we may divide   $\chi_2 = \chi_{2}^\hor + \chi_{2}^\mix$ so that we can solve the horizontal part of the equation and the not-horizontal separately, i.e. 
\begin{equation}\label{fifi}
\{ \omega \cdot \yy + \cD^{\rm diag}, \chi_{2}^\hor\} + \cH^\hor_{1} =\cZ^\hor_1\,,\quad \{ \omega \cdot \yy + \cD^{\rm diag}, \chi_{2}^\mix\} + \cH^\mix_{1} =\cZ^\mix_1 \ , 
\end{equation} 
with
 $$\lceil \chi_2^\hor\rfloor^{\cO_0}_{s/4,-2} +  \lceil \chi_2^\mix\rfloor^{\cO_0}_{s/4,-\bd} \lessdot \e. $$  
In the first equation of \eqref{fifi},  remarking that $\cH_1^\hor$ is  also a  diagonal Hamiltonian, we make the ansatz that $\cZ_1^\hor$ and  $\chi_2^\hor$ are diagonal as well. Passing to Taylor-Fourier expansion we get for the coefficients $\{\chi_{m_1,m_2,n}^{-,\ell}\}$ of $\chi_2^\hor$ the equation 
\begin{equation}
\im (\oo\cdot \ell+ m_1^2  - m_2^2)\chi_{m_1,m_2,n}^{-,\ell}+(\cH^\hor_1)_{m_1,m_2}^{-, \ell}= (\cZ^\hor_1)_{m_1,m_2}^{-, \ell}\ .
\end{equation}
Now by the explicit expression of $\cH^\hor_1$ given in  \eqref{kernel0} we  have that the Taylor-Fourier support of $\cH^\hor_{1}$ is the set
\begin{align*}
{\rm supp}(\cH^\hor_1):= \Big\{\big((m_1, n), (m_2, n), \ell \big) : \ \exists  \, 1 \leq i, j \leq \tk \ ,  \, i \neq j  \mbox{ s.t. }\  m_2 - m_1 = \tm_i - \tm_j \ , \   \ell = {\bf e}_{i} - {\bf e}_{j}  \Big\} \ . 
\end{align*}
By Lemma \ref{lem:br}$(i)-(ii)$, all the monomials with $m_1 \neq \tm_j$ are Birkhoff non resonant, thus using also that $|\ell|=2$ 
\begin{equation}
\label{step2.he1}
|\oo\cdot \ell+ m_1^2  - m_2^2| \geq |\oo^\0 \cdot \ell+ m_1^2  - m_2^2| - \cO(\e) >1/2 \ ,
\end{equation}
 and moreover such divisor  does not depend on $n$.  We obtain
\begin{align}
\label{Zhor1}
 (\cZ^\hor_1)_{m_1,m_2}^{-, \ell}=
\begin{cases}
 2\e \sqrt{\lambda_i \lambda_j} \ , & m_1=\tm_j\ ,\ m_2=\tm_i \ , \  \ell = {\bf e}_{i} - {\bf e}_{j} \\
 0 & \mbox{ otherwise}
 \end{cases}
\end{align}
and 
$$
\chi_{m_1,m_2,n}^{-,\ell} = \frac{(\cZ^\hor_1)_{m_1,m_2}^{-, \ell} - (\cH^\hor_1)_{m_1,m_2}^{-, \ell}}{\im (\oo\cdot \ell+m_1^2 - m_2^2)} \ .
$$
It is clear that $\chi_2^\hor  \in \cQ^{\cO_0, \hor}_{s/4,-2}$ is horizontal and using also the estimate \eqref{step2.he1} one gets
$$
\lceil \chi_2^\hor \rfloor^{\cO_0}_{s/4,-2} \lessdot \e \ .
$$
Consider now the second equation  in \eqref{fifi}.  By Lemma \ref{lem:br} $(iii)$  the resonant monomials are those  fulfilling \eqref{res2}. 
Thus, using also the rectangle condition $|(\tm_j, 0)|^2 - |(m_1, n)|^2 + |(\tm_i, 0)|^2 - |(m_2, -n)|^2 = 0$,  we define for $n>0$
\begin{align*} (\cZ^\mix_1)_{m_1,m_2,n}^{+, \ell}=
\begin{cases}
 2\e \sqrt{\lambda_i \lambda_j} \ , &  (m_1-\tm_i) (m_1-\tm_j)+n^2=0  , \; \ \ell =- {\bf e}_{i} - {\bf e}_{j}\\
0 & \mbox{ otherwise} 
\end{cases}
\end{align*}
 and  
$$
\chi_{m_1,m_2,n}^{+,\ell} = \frac{(\cZ^\mix_1)_{m_1,m_2,n}^{-, \ell} - (\cH^\mix_1)_{m_1,m_2,n}^{-, \ell}}{\im (\oo\cdot \ell+m_1^2 + m_2^2 + 2n^2)} \ .
$$
As in \eqref{step2.he1}, we have the estimate of the small divisors %\red{non mi torna: non e' come prima, nel senso che la parte a $\e = 0$ e' un intero non nullo e quindi stimo con 1/2?} {\color{blue} si sembra anche a me}
$$
\abs{\omega \cdot \ell  + m_1^2 + m_2^2 + 2n^2 } \geq  \frac12
$$
which implies that $\chi_2^\mix\in \cQ_{s/4,-\bd}^{\cO_0}$ and 
$$
\lceil \chi_2^\mix\rfloor_{s/4,-\bd}^{\cO_0} \lessdot  \e \ .
$$
Now remark that  $\cH_1$ commutes with $\cM$ and $\cP$ and hence it satisfies the selection rules of Remark \ref{leggi_sel}. 
By the explicit formula for $\chi_2$ it follows that the same selection rules hold for $\chi_2$, hence $\{\cM, \chi_2 \}=0$ and $\{\cP, \chi_2\} = 0 $ and item $(iv)$ follows.

Now we analyze line \eqref{red.20}. First one has that, since $\chi_2 = \chi_2^\hor + \chi_2^\mix$,  we can apply Lemma \ref{lem:smo} to get 
$$
\lceil \re1\left( \chi_2; \cH_1\right)^\hor \rfloor_{s/8, -2}^{\cO_0} \lessdot \e^2 
\,,\qquad
\lceil \re1\left( \chi_2; \cH_1\right)^\mix \rfloor_{s/8, -\bd}^{\cO_0} \lessdot \e^2 \ .
$$
Using that $\{ \chi_2, \omega \cdot \yy + \cD^{\rm diag} \}= \cH_1 - \cZ_1 $, one obtains 
\begin{align*}
\re2\left( \chi_2; \omega \cdot \yy + \cD^{\rm diag} \right) & 
%= \sum_{k \geq 2} \frac{{\rm ad}(\chi_2)^k[ \omega \cdot \yy + \cD^{\rm diag} ]}{k!}
 = \sum_{k \geq 2} \frac{{\rm ad}(\chi_2)^{k-1}[ \{ \chi, \omega \cdot \yy + \cD^{\rm diag}  \}]}{k!} = \sum_{k \geq 1} \frac{{\rm ad}(\chi_2)^{k}[\cH_1 - \cZ_1]}{(k+1)!}  \ . 
\end{align*}
Since  $\cH_1 - \cZ_1 = (\cH_1 - \cZ_1)^\hor + (\cH_1 - \cZ_1)^\mix$ we apply  Lemma \ref{lem:smo} to get
$$
\lceil\left( \re2(\chi_3; \oo\cdot \yy+ \cD^\diag)\right)^\hor\rfloor^{\cO_0}_{s/8, -2} +\lceil\left( \re2(\chi_3; \oo \cdot \yy + \cD^\diag)\right)^\mix\rfloor^{\cO_0}_{s/8, -\bd} \lessdot \e^2 \ . 
$$
Now consider line \eqref{red.30}. Using again Lemma \ref{lem:smo} (with $\ti = 0$) we get  $\whcH_2:= \cH_2 \circ \cT^{(B)} = \whcH_2^\hor + \whcH_2^\mix $ with the claimed estimates.
\end{proof}
%We verify now that all the  monomials with support in ${\rm supp}(\cH^{\rm diag}_1)$ which do not form a non degenerate rectangle are birkhoff non resonant with respect the  non resonance condition \eqref{nr1} holds. Indeed for  any $\Big((m_1, n), (m_2, n), \ell \Big) \in {\rm supp}(\cH^{\rm diag}_1)$, we have that 
%\begin{align*}
%\oo^\0 \cdot \ell+\Omega\cdot (\al-\bt) & = \oo^\0 \cdot({\bf e}_{i} - {\bf e}_{j}) + m_1^2 - m_2^2 = \tm_i^2 - \tm_j^2 + m_1^2 - m_2^2 \\
%& = 2(\tm_i - \tm_j)(\tm_j- m_1)
%\end{align*}
%Since $\tm_j \neq \tm_i$, the only case is $m_1 = \tm_j$ (and thus $m_2 = \tm_i$), which is exactly a rectangle.
%
%{\tt ripetere i calcoli per la parte out-diagonal}

\subsubsection{Step 3: diagonalization of the Birkhoff resonant terms}
%\red{c'e' anbiguita' in tutto il capitolo: a volte usiamo $\cD$, altre $\cD^\diag$: decidere se mettere i termini line o no}\\
In the final  step we consider the resonant Hamiltonian in normal form  $ \omega \cdot \yy + \cD   + \cZ_1$ and we diagonalize it through a  transformation which is not close to the identity. 
\begin{remark}
	Due to our genericity condition we have that $(m,n)\in \cup_{i<k}\ccC_{i,k}$ implies that $m\notin \Tan$.   Moreover given $(m,n)$ with $n>0$ there exists at most one couple $(i,k)$, $i<k$, such that $(m,n)\in \ccC_{i,k}^+$. In the same way given $(m,n)$ with $n<0$ there exists at most one  couple $(i,k)$, $i<k$, such that  $(m,n)\in \ccC_{i,k}^-$ and consequently  $(\tm_i+\tm_k-m,-n)\in \ccC_{i,k}^+$.
\end{remark}
We now perform a phase shift which removes the  dependence on the angles in $\cZ_1$.
\begin{lemma}
	\label{lem:diag0}
	Consider the  Hamiltonian \eqref{ham.40}. For all $s_0/2 \le s\le s_0, 0<r\le  r_0$, there exists an invertible symplectic change of variables $\cR: D(s/8, e^{-s_0} r/4) \to D(s/8, r/4)$: $ (\yy^{(+)}, \theta, {\bf b}) \mapsto (\yy, \theta, \ba) $ s.t.  
	\begin{align}
	\label{quad00}
	(\omega\cdot \yy + & \cD + \cZ_1 )\circ \cR = 
	\omega \cdot \yy^{(+)} + \whcD + \whcZ_1 
	\end{align}
	where
	\begin{align}
	\label{quad0}
	& \whcD:= \sum_{(m,n) \in \Z^2\setminus (\Tan\cup\sS \cup \ccC) } \Omega^\0_{(m,n)} |b_{(m,n)}|^2  \\
	\label{quad1}
	& \whcZ_1:= \sum_{1 \leq i\leq \tk \atop  n\in \Z\setminus \{0\} }  \left(\tm_i^2 - \omega_i + n^2 \right)  |b_{(\tm_i,n)} |^2
	+ 2 \e \sum_{1 \leq i,k \leq \tk, \, i\neq  k \atop   n\in \Z\setminus \{0\}} \sqrt{\lambda_i \lambda_k} \,   b_{(\tm_i,n)} \, \bar b_{(\tm_k, n)}  \\
	\label{quad2}
	& \qquad + \sum_{i<k , \; (m,n) \in \ccC_{i,k}^+} \Big( \left(m^2 + n^2 -\omega_i\right)|b_{(m,n)}|^2 + \left((\tm_i + \tm_k - m)^2 +n^2 -\omega_k \right)|b_{(\tm_i + \tm_k - m, -n)}|^2  \\
	\label{quad22}
	& \qquad  + 4\e \sqrt{\lambda_i \lambda_k}\, {\rm Re}(b_{(m,n)}b_{(\tm_i+\tm_k-m,-n)}) \Big)
	\end{align}
	Furthermore the following is true:
	\begin{itemize}
		\item[(i)] $\cR$ is the identity on the variables in $\Z^2 \setminus (\sS \cup \ccC)$, and on $\sS \cup \ccC$ is a phase shift.
		\item[(ii)] $\cM \circ \cR = \wtcM \ , \quad \cP\circ \cR= \wtcP$, 
		where $\wtcM$ and $\wtcP$ are defined in \eqref{mp.4}.
	\end{itemize}
\end{lemma}
\begin{proof}
	We define the symplectic
	\footnote{ a simple computation shows that the symplectic form in the new variables is  given by 
		$$
		d \yy^{(+)} \wedge d\theta +  \im \sum_{(m , n) \in \Z^2\setminus \Tan   } d  b_{(m,n)} \wedge d \bar b_{(m,n)}
		$$}
	transformation $\cR: (\yy^{(+)}, \theta, {\bf b}) \mapsto (\yy, \theta, \ba) $ as 
	\begin{align*}
	\yy_i = \yy^{(+)}_i - \sum_{ n\neq 0} |b_{(\tm_i,n)}|^2 -\sum_{i<k}\sum_{\jj\in \ccC_{i,k}^+ }|b_{\jj}|^2  -\sum_{k<i}\sum_{\jj\in \ccC_{k,i}^- }|b_{\jj}|^2 \ , 
	\end{align*}
	\begin{equation}
	\label{rotation}
	a_{(m,n)}= \begin{cases}  e^{\im \theta_i} \, b_{(\tm_i, n)} & m= \tm_i \,,\quad n\neq 0\\
	e^{\im \theta_i} \, b_{ (m, n)} & (m,n)\in \ccC_{i,k}^+ \quad i<k \\
	e^{\im \theta_k} \, b_{ (m, n)} & (m,n)\in \ccC_{i,k}^-\,,\quad i<k\\
	b_{(m,n)} & \text{otherwise}
	\end{cases}
	\end{equation}
	It's a simple computation to show that $\cR$ conjugates $\omega \cdot \yy + \cD + \cZ_1$, the mass $\cM$ and the momentum $\cP$ to the claimed functions.
\end{proof}
Note that after this change of variables the dynamics of  the action-angles $(\theta,\yy)$, of $\whcD$ and of $\whcZ_1$ is decoupled.
\begin{lemma}
\label{lem:step4}
There exists $\varrho>0$ depending on $s_0,\max|\tm_i|^p$ and an open set $\cO_1$ such that for all $s_0/64 \le s\le s_0, 0<r\le  r_0$ and for any $\lambda\in \cO_1$ there exists a linear invertible symplectic transformation $\cU: D(s/8, \varrho r) \to D(s/8, e^{-s_0}r/4)$ of the   $(\ba,\yy,\theta)\mapsto (U \ba, \yy,\theta)$  which depends analytically on $\lambda $ and transforms the Hamiltonian \eqref{ham.40} in the following form (we are calling the variables $\yy,\ba$ again):
\begin{align}
\label{ham.50}
(\omega\cdot \yy^{(+)} + \whcD + \whcZ_1 )\circ \cU =  \omega\cdot \yy + \whcD + \whcZ_1 \circ \cU = \omega \cdot \yy +  \wtcD(\lambda,\e) 
\end{align}
where 
\begin{itemize}
\item[(i)] $U= \diag(U_n)$ where  $U_n$  acts non trivially only on  $(\sS \cup \ccC) \cap \{(m,n), (m, -n)\}_{m \in \Z}$ and is the identity elsewhere.  The $U_n$  depend analytically on $\lambda$ and satisfy bounds of the form
$$
\|U_n\|^{\rm op},\|(U_n)^{-1}\|^{\rm op}\lessdot 1 \ . 
$$
	\item[(ii)]  $\wtcD(\lambda,\e)$ is the diagonal  Hamiltonian   
	\begin{align}
	\label{D.lam} 
	&\wtcD(\lambda, \e) := \sum_{\jj\in \Z^2\setminus \Z }\wtOmega_\jj(\lambda, \e) |a_\jj|^2  
	+
	\sum_{m \in \Z \setminus \Tan }\Omega_m(\lambda) |a_{(m,0)}|^2  
	\end{align}
	where  the normal frequencies $\wtOmega_\jj(\lambda, \e)$ are defined in \eqref{def:omtilde} and $\Omega_m(\lambda)$ in \eqref{norm.freq}.
	\item[(iii)] $\wtcH_{2}:= \whcH_2 \circ \cR\circ \cU= \wtcH_2^\hor+\wtcH_2^\mix \in \cQ^{\cO_1,\hor}_{s/8, -2} +\cQ^{\cO_1}_{s/8, -\bd} $ with the bounds  $\lceil \wtcH_{2}^\hor  \rfloor_{s/8, -2}^{\cO_1}+\lceil \wtcH_{2}^\hor  \rfloor_{s/8, -\bd}^{\cO_1} \lessdot \e^2$. 
	\item[(iv)] One has $\wtcM\circ \cU = \wtcM$ and  $\wtcP\circ \cU = \wtcP$. 
	\end{itemize}
	\end{lemma} 

%\begin{remark}
%Using the structure of $\ccC_{i,j}$, {\textcolor{red}{per le nostre definizioni $\ccC_{i,j}$ contiene solo punti con $n>0$  quindi cambio }}
%$$
%\sum_{i <j \atop (m,n) \in \ccC_{i,j}}
%(m - \tm_i) |b_{(m,n)}|^2 = 
%\sum_{i <j, n >0 \atop (m,n) \in \ccC_{i,j}}(m - \tm_i)  \left[ |b_{(m,n)}|^2 - |b_{\tm_i + \tm_j - m,-n}|^2 \right]
%$$
%\end{remark}
%\begin{remark}
%\label{rem:sel.rul.mod}
%Given a monomial $e^{\im \theta\cdot \ell}\yy^la^\al \bar a^\bt$ we have the new commutation rules
%\begin{align*}
%&\{\omega\cdot \yy + \wtcD, \ e^{\im \theta\cdot \ell}\yy^la^\al \bar a^\bt\}=\im (\omega\cdot \ell +\wtOmega\cdot (\al-\bt) )e^{\im \theta\cdot \ell}\yy^la^\al \bar a^\bt\\ 
%&\{\wtcM,e^{\im \theta\cdot \ell}\yy^la^\al \bar a^\bt\}=\im \left(\sum_i \ell_i +\sum_{\jj\in \Z^2\setminus( \sS \cup \ccC\cup \Tan) }(\al_\jj-\bt_\jj) \right)e^{\im \theta\cdot \ell}\yy^la^\al \bar a^\bt\\
%&\{\wtcP_{x},e^{\im \theta\cdot \ell}\yy^la^\al \bar a^\bt\}=\im \left(\sum_i \mathtt m_i \ell_i +\!\!\!\!\!\sum_{\jj=(m,n) \atop j\in \Z^2 \setminus (\sS \cup \ccC\cup \Tan)} m(\al_\jj-\bt_\jj) + \!\!\!\!\!\!\!\!\sum_{{i< k \atop \jj=(m,n) \in \ccC_{i,k}^+ }} \!\!\!\!\!\!\!\!(m- \tm_i) (\al_\jj - \bt_\jj )+ \!\!\!\!\!\!\!\!\sum_{{i< k \atop \jj=(m,n) \in \ccC_{i,k}^- }} \!\!\!\!\!\!\!\!(m- \tm_k) (\al_\jj - \bt_\jj ) \right)e^{\im \theta\cdot \ell}\yy^la^\al \bar a^\bt \\ 
%&\{\wtcP_{y},e^{\im \theta\cdot \ell}a^\al \bar a^\bt\}=\im (\sum_{j=(m,n)\in \Z^2\setminus\Tan} n(\al_j-\bt_j) )e^{\im \theta\cdot \ell}a^\al \bar a^\bt
%\end{align*} 
%\end{remark}
Before proving the Lemma, we discuss some basic fact on normal forms for quadratic Hamiltonians. {%\color{blue}
 Let ${\bf z}= (z,\bar z)$ be a finite dimensional phase space -say of dimension $2k$- with respect to the Poisson form $\im d z\wedge d \bar z$. Let $\dot {\bf z} = \im A {\bf z}$ be a linear  Hamiltonian system corresponding to the real Hamiltonian 
 \begin{equation}\label{quadrat}
\cQ=\frac12 ({\bf z}, J^{-1} A {\bf z})\in \R\,,\quad A^T J = -J A\,,\quad E A E= -\widebar A \,,
 \end{equation}
 where
  $$
 J= \begin{pmatrix}
 0 &{\rm Id} \\ -{\rm Id} & 0
 \end{pmatrix}\,,\quad E= \begin{pmatrix}
 0 &{\rm Id} \\ {\rm Id} & 0
 \end{pmatrix}
 $$
  and $(\cdot,\cdot)$ is the real scalar product. 
  Assume that the eigenvalues of the matrix $A$ are distinct,  real and different from zero, say $\pm a_1,\dots,\pm a_k$. 
 Then by the standard theory of quadratic normal forms there exists a symplectic matrix $U$ which diagonalizes $A$ and preserves the real structure:
 $$ 
  U^{-1}A U =D= \diag(a_1,\dots,a_k,-a_1,\dots, -a_k)\,,\quad  U^TJU =  J \,,\quad E U E= \bar U \ . 
 $$
 Consequently   $\cU:  {\bf w} \mapsto  U {\bf w} = {\bf z} $ is canonical   with ${\bf w}=(w,\bar w)$ and
 $$
 Q\circ \cU= \sum_{i=1}^k a_i |w_i|^2.
 $$ 
  Finally, since the eigenvalues  of $A$ are distinct, $U$  depends analytically on the matrix elements of $A$.
  
%  Indeed we note that if $\whU$ diagonalizes $A$ then so does $ U= \whU M$ with $M$ diagonal. 
%  Now  since $A^T J = -J A$ we have
%  $$
%  U^{-T}D U^T J + J U D U^{-1}=0 \quad \Rightarrow \quad D U^T J U+ U^TJU D=0.
%  $$
%  Since $D$ has distinct eigenvalues and $U^TJU$ anticommutes with it then
%  \begin{equation}\label{simplet}
%  U^TJU = S J \,,
%  \end{equation}
%  With $S$ a diagonal matrix. Similarly 
%  $$
% \widebar U^{-1} E A E\widebar U  = -D   = E D E
%  $$
%  so $E\widebar U E $ diagonalizes $A$. Since the eigenvalues are simple this means that 
% \begin{equation}\label{reale}
% E\widebar U E=  U \widehat S
% \end{equation}
%  with $\widehat S$ a diagonal matrix.
%  Using the normalization $ U= \whU M$ we obtain a matrix  $\hat U$ which is is symplectic and preserves the real structure. Indeed  in order to preserve the real structure we need
%  $$  
%   M \hat S E \bar M^{-1} =E 
%  $$
%  while in order have a symplectic map we need
%  $$
% S^{-1} M^TJM =J
%  $$
%  in conclusion we get
%  $$
%  m_k^\s = 
%  $$
%
%Now
%  $$
%  U^T J^{-1} =  J^{-1} U^{-1} \quad \Rightarrow  Q\circ \cU= \frac12(U{\bf w},J^{-1}A U {\bf w})=  \frac12({\bf w},U^TJ^{-1}A U {\bf w})=  \frac12({\bf w},J^{-1}U^{-1}A U {\bf w})
%  $$
  We now specialize this normal form result to {\em block diagonal}  Hamiltonians.
  Recall that a Lagrangian subspace is {a subspace which coincides with its symplectic orthogonal, i.e. $W = W^{\angle}$. In particular it has dimension $k$.}
  \\
   Let    $\bz=(z^\1,z^\2)$ with 
  $$z^\1= (z_{j_1}^{\s_1}, z^{\s_2}_{j_2},\dots, z^{\s_k}_{j_k}) \ , \quad z^\2=(z_{j_1}^{-\s_1}, z^{-\s_2}_{j_2},\dots, z^{-\s_k}_{j_k})= \bar z^\1   \quad \mbox{where} \; z^+_j= z_j \,,\, z_j^-= \bar z_j$$
  be a Lagrangian decomposition, namely with  Span$(z_{j_1}^{\s_1} z^{\s_2}_{j_2},\dots, z^{\s_k}_{j_k})$ a Lagrangian subspace.
  We write all our matrices in terms of this decomposition; so for example setting $\Sigma = \diag(\s_i 1)$ we have
  $$
  J= \begin{pmatrix}
  0 &\Sigma \\ -\Sigma & 0
  \end{pmatrix}\,,\quad {E= \begin{pmatrix}
  0 & {\rm Id}\\ {\rm Id} & 0
  \end{pmatrix}} \ . 
  $$
 Assume that $A=\diag(A^\1,-\bar A^\1)$ is block diagonal w.r.t the decomposition. Then $U$ is block diagonal as well, namely $U= \diag(U^\1,\bar U^\1)$. 
 Now the fact that $U$ is symplectic reads in the block decomposition that 
  $U^\1$ is {\em orthogonal} w.r.t.  $\Sigma$ i.e. $(U^\1)^T \Sigma \widebar U^\1= \Sigma$. 
  We are interested in a simple consequence of this fact.
  Define
  \begin{equation}\label{invariante}
 \cQ_0= \frac12 (\bz , J^{-1} \begin{pmatrix}
 {\rm Id} & 0 \\ 0 & -{\rm Id}
 \end{pmatrix} {\bf z}) = {\rm Re}(z^\1,\Sigma \,\bar z^\1)\ . 
\end{equation}
Then using the orthogonality condition  $(U^\1)^T \Sigma \widebar U^\1= \Sigma$ we have 
$$
\cQ_0 = {\rm Re}(z^\1,(U^\1)^T \Sigma \widebar U^\1 \,\bar z^\1)= {\rm Re}(U^\1 z^\1,\Sigma \,\bar U^\1\bar z^\1) = \cQ_0\circ\cU
 $$
i.e. for any block diagonal symplectic $\cU$ we have $\cQ_0\circ\cU= \cQ_0$.

\smallskip
We wish to apply this theory to $\whcZ_1$. 
We shall show that $\whcZ_1$ is the sum of non-interacting quadratic Hamiltonians of two types.
\begin{itemize}
	\item[(Type I)]  The first type  has dimension $2\tk$  with Hamiltonian
	\begin{equation}\label{cq1}
	\cQ_1= \mathtt K \sum_{i=1}^\tk |z_i|^2 + (z, M \bar z) 
	\end{equation}
	with $\tK$ some real number and $M(\lambda) :=( M_{ij}(\lambda))_{i,j}$ given for any $1 \leq i, j \leq \tk$ by 
	\begin{equation}\label{emme}
	M_{ij}(\lambda) := 
	\begin{cases}
	\lambda_i \ , & i = j \\
	2 \sqrt{\lambda_i \lambda_j } \ , & i \neq j 
	\end{cases} \ . 
	\end{equation}
	It is easily seen that this Hamiltonian is block diagonal w.r.t  the Lagrangian decomposition with $z^\1= z$, moreover the first summand in \eqref{cq1}  is the invariant Hamiltonian $\cQ_0$ while the second summand corresponds to a  Hamiltonian as in \eqref{quadrat} with matrix $A= \diag(M,M)$.

\item[(Type II)] The second type of Hamiltonian has dimension $4$ and is block diagonal w.r.t the Lagrangian decomposition $z^\1=(z_1,\bar z_2)$, 
with Hamiltonian
\begin{equation}\label{cq2}
\cQ_2= \mathtt K  (|z_1|^2-|z_2|^2) + \frac12 (\bz ,J^{-1} B \bz)  \,,\quad B= \begin{pmatrix}
N & 0 \\ 0 & -N
\end{pmatrix}\,, \quad N\in GL_2(\R)
\end{equation}
As in the previous case, the  first  summand in \eqref{cq2}  is the invariant Hamiltonian $\cQ_0$ -recall that now $\Sigma=\diag(1,-1)$- while the second summand corresponds to a  Hamiltonian as in \eqref{quadrat} with matrix $A= \diag(N,N)$. 
We shall show that this type of Hamiltonian appears for each $\jj=(m,n)\in \ccC_{i,k}^+$ by identifying $z_1= z_{\jj}, z_2= z_{\und \jj}$ where $\und \jj= (\tm_i+\tm_k-m,-n)$. In this case  $\tK= m^2-\tm_i^2+n^2$ while
\begin{equation}
N \equiv N(\lambda_i, \lambda_k) := 
\begin{pmatrix}
\lambda_i  & 2 \sqrt{\lambda_i \lambda_k } \\
- 2 \sqrt{\lambda_i \lambda_k } &- \lambda_k
\end{pmatrix}
\end{equation}
\end{itemize}
}

%  
%   Mtate  some useful property concerning the polynomials $P(t,\lambda)$ defined in \eqref{P.poly0bis} and $Q(t, \lambda_i, \lambda_k)$ defined in  \eqref{char02bis}.
%We start by defining  the matrix $M(\lambda) :=( M_{ij}(\lambda))_{i,j}$ given for any $1 \leq i, j \leq \tk$ by 
%\begin{equation}\label{emme}
%M_{ij}(\lambda) := 
%\begin{cases}
%\lambda_i \ , & i = j \\
%2 \sqrt{\lambda_i \lambda_j } \ , & i \neq j 
%\end{cases} \ . 
%\end{equation}
%The matrix $M$ is a real symmetric, $\tk\times \tk$ matrix, thus it can be diagonalized by an orthogonal  change of coordinates, which we denote by $O_1$.  We denote its eigenvalues by   $(\mu_i(\lambda))_{1 \leq i \leq \tk}$. These are the roots of the  characteristic polynomial of $M(\lambda)$, $P(t, \lambda) =\det( t\uno - M(\lambda))$. We claim that it is the polynomial \eqref{P.poly0bis}.  Indeed we can write
%$$
%M(\lambda) = {\rm diag}(\sqrt{\lambda_i})\cdot (2A - \uno) \cdot  {\rm diag}(\sqrt{\lambda_i})
%$$
%where 
%$$
%A := (A_{ik})_{i,k} \ , \quad A_{ik} = 1 \quad \forall 1 \leq i,k \leq \tk \ .
%$$
%then our claim follows by an explicit computation.
%%Then by construction  setting for all $n\neq 0$
%%$$ (b_{\tm_1,n},\dots, b_{\tm_\tk,n} ) =  U (c_{\tm_1,n},\dots, c_{\tm_\tk,n}) \,,\quad  (\bar b_{\tm_1,n},\dots, \bar b_{\tm_\tk,n} ) =  U (\bar c_{\tm_1,n},\dots, \bar c_{\tm_\tk,n})$$
%%we have a symplectic change of  coordinates which we denote by $\cU_1$.
\begin{lemma}
\label{lem:irr}
%\begin{itemize}
(i) The  characteristic polynomial of $M(\lambda)$, $P(t, \lambda) =\det( t\uno - M(\lambda))$ coincides with \eqref{P.poly0bis} and  is irreducible over $\Z[t, \lambda_1, \ldots, \lambda_\tk]$. Consequently the eigenvalues of $M$, which we denote by
$\mu_i(\lambda)$ are distinct algebraic functions of $\lambda$ homogeneous of degree one. 
\item[(ii)] Consider any open domain contained in a single connected component where all the eigenvalues $\mu_i(\lambda)$ are distinct.   In any of such domain there exists an orthogonal matrix $U_1\in O_\tk(\R)$,  depending analytically on $\lambda$, such that 
$$\cU_1:\bz \to (U_1z, \bar U_1\bar z)=(w,\bar w)$$ is  symplectic and  for any $\tK$ we have  
$$
\cQ_1\circ\cU_1 = \sum_{i=1}^{\tk}(\tK+ \mu_i(\lambda)) |w_i|^2 \,,\quad \cQ_0\circ \cU_1=\sum_{i=1}^\tk |z_i|^2\circ \cU_1 = \sum_{i=1}^{\tk} |w_i|^2 \ . 
$$  
\item[(iii)] The  characteristic polynomial of $N(\lambda_i,\lambda_k)$, $Q(t, \lambda_i,\lambda_k) =\det( t\uno - N(\lambda_i,\lambda_k))$ coincides with \eqref{char02bis} and  is irreducible over $\Z[t, \lambda_1, \ldots, \lambda_\tk]$. Consequently the eigenvalues of $N$, which we denote by
$\mu_{i,k}^\pm(\lambda)$ are distinct algebraic functions of $\lambda_i,\lambda_k$, homogeneous of degree one. 
\item[(iv)] There exist open connected domains in which  $\mu_{i,k}^\pm(\lambda)$ are real and distinct.   In any  of such domain there exists a matrix  $U_2\in O_\tk(\R)$ such  that $$\cU_2:\bz=(z^\1,\bar z^\1) \to (U_2 z^\1, \bar U_2 \bar z^\1)=(w_1,\bar w_2,\bar w_1 , w_2)$$
 is  symplectic and  for any $\tK$ we have  
$$
\cQ_2\circ\cU_2 = (K+ \mu_{i,k}^+(\lambda)) |w_1|^2 - (K+ \mu_{i,k}^-(\lambda)) |w_2|^2 \,,\quad \cQ_0\circ \cU_2 = (|z_1|^2-|z_2|^2)\circ \cU_2 = |w_1|^2-|w_2|^2.
$$  
%\end{itemize}
\end{lemma}
\begin{proof}
(i) The fact that $\det( t\uno - M(\lambda))$ coincides with \eqref{P.poly0bis} is a direct computation. In order to prove that  $P(t, \lambda)$ is  irreducible over $\Z[t, \lambda_1, \ldots, \lambda_\tk]\subset \Q( \lambda_1, \ldots, \lambda_\tk)[t]$ we proceed by induction on $\tk$.  If $\tk=1$  the statement is trivial.  Now let us suppose that it is true up to $\tk= N$ and prove it for $N+1$. We consider the polynomial
$$
P_{N+1}(t,\lambda)=\prod_{i=1}^{N+1} (t + \lambda_i) - 2 \sum_{i=1}^{N+1} \lambda_i \, \prod_{k \neq i} (t + \lambda_k)
$$
and specify to $\lambda_{N+1}=0$. We obtain 
$$
P_{N+1}(t,\lambda_1,\dots,\lambda_N,0)= t \, \Big(\prod_{i=1}^{N} (t + \lambda_i) - 2  \sum_{i=1}^{N} \lambda_i \, \prod_{k \neq i} (t + \lambda_k)\Big)= t \, P_N(t,\lambda_1,\dots,\lambda_N)
$$
and by the inductive hypothesis the second factor is irreducible. Assume by contradiction that  $P_{N+1}$ is not irreducible: then it must factorize a linear term of the form $(t- c \lambda_{N+1})$.  Note that $c$ is a number, this is due to the fact that $P_{N+1}$ is homogeneous of degree $N+1$ in $(t,\lambda)$ and $P_N$ is homogeneous of degree $N$.
We repeat the same argument specifying to $\lambda_1=0$ and obtain a contradiction.
Indeed we would have
$$
P_{N+1}(t,\lambda)=(t- c_1 \lambda_{N+1})\whP_1(t,\lambda) = (t-c_2\lambda_1)\whP_2(t,\lambda)
$$
with $\whP_1,\whP_2$ irreducible. Now this equality can hold only if $(t- c_1 \lambda_{N+1})$ divides $\whP_2$ which is impossible by the irreducibility.

Now remark that for a polynomial with coefficients in a field with characteristic $0$, a sufficient condition to have distinct roots is the polynomial to be irreducible. Indeed if there were a double root then $f(t)$ and $f'(t)$ would have a common divisor, thus contradicting the irreducibility. Hence the eigenvalues of the symmetric matrix $M$ are distinct (and obviously real for positive $\lambda$) outside a finite number of algebraic surfaces. Since $M$ is homogeneous of degree one then so are the $\mu_i(\lambda)$.

\noindent
(ii) In order to conclude our proof we  restrict to a connected component which does not cross any  surface where two eigenvalues coincide. Then we apply the theory of quadratic Hamiltonians as described above.  \\
(iii) The irreducibility can be verified immediately by computing the roots of $Q(t, \lambda_i, \lambda_k)$, which are given by
\begin{align}
\label{mutilde}
\mu_{i,k}^+(\lambda) = \frac{\lambda_i - \lambda_k - \sqrt{\lambda_i^2 + \lambda_k^2- 14 \lambda_i \lambda_k}}{2} \ , \qquad 
\mu_{i,k}^-(\lambda) = \frac{\lambda_i - \lambda_k + \sqrt{\lambda_i^2 + \lambda_k^2- 14 \lambda_i \lambda_k}}{2} 
\end{align}
(iv) We restrict $\lambda_i,\lambda_k$ to a region where we have 2 distinct, real eigenvalues. To this purpose  we  impose the condition ${\rm Tr}^2  N(\lambda_i, \lambda_k) > 4 \det N(\lambda_i, \lambda_k)$, which is equivalent to 
$$
(\lambda_k - c_+\lambda_i)(\lambda_k - c_- \lambda_i) >0  \ , \quad c_\pm  = (5\pm\sqrt{21})/2 \ .
$$
This  selects two conic regions of parameters $\lambda_i, \lambda_k$. In each such region we may apply the theory of quadratic Hamiltonians in order to obtain our result.
\end{proof}

%We come now to the proof of  Lemma \ref{lem:step4}.
\begin{proof}[Proof of  Lemma \ref{lem:step4}.]We first notice that $\whcD$, $\whcZ_1$ depend on different variables and hence do not interact; in the same way, due to our genericity condition, the first line in $\whcZ_1$, see \eqref{quad1}, does not interact with the second and third one, see \eqref{quad2}-\eqref{quad22}. Finally \eqref{quad1}  is the infinite sum over $n\neq 0$ of the finite dimensional Hamiltonian:
$$
\sum_{1 \leq i\leq \tk  }  \left(\tm_i^2 - \omega_i + n^2 \right)  |b_{(\tm_i,n)} |^2
+ 2 \e \sum_{1 \leq i,k \leq \tk, \, i\neq  k \atop   n\in \Z\setminus \{0\}} \sqrt{\lambda_i \lambda_k} \,   b_{(\tm_i,n)} \, \bar b_{(\tm_k, n)}
$$
supported on the Fourier indices 
{$$
\sS_n := \bigcup_{\pm n}\cS_{0n}  \ .
$$. }
\\
Since $\tm_i^2 - \omega_i + n^2= \e\lambda_i +n^2$ the Hamiltonian above is of type I, see \eqref{cq1}, we just identify $z=(b_{(\tm_1,n)},b_{(\tm_2,n)},\dots ,b_{(\tm_tk,n)}) $ for all $n\neq 0$ and fix $\tK= n^2$.
Regarding \eqref{quad2}-\eqref{quad22} we have Hamiltonians of  type II, see \eqref{cq2}. Indeed for each  $\jj=(m,n)\in \ccC_{i,k}^+$ we identify $z_1= z_{\jj}, z_2= z_{\und \jj}$ where $\und \jj= (\tm_i+\tm_k-m,-n)$. In this case the coefficient $\tK= m^2-\tm_i^2+n^2$.  %This implies that each block of this type is diagonalized by the matrix $U_2(\lambda_i,\lambda_k)$ of 
In conclusion $\whcD+\whcZ_1$ is block diagonal with respect to the blocks
$$
\{\Z^2\setminus(\sS \cup \ccC\cup \Tan)\} \cup_{n >  0}\sS_n\cup_{i<k}\ccC_{i,k}.
$$
We consider an open connected region in $\cO_0$ where the eigenvalues $\mu_i(\lambda)$ are distinct and for each $i\neq k$,
the eigenvalues $\mu^\pm_{i,k}(\lambda)$ are  real and distinct.
 Such a region  exists and is the intersection between $\cO_0$ and a {\em cone} due to Lemma \ref{lem:irr}.
  Now we may choose a compact domain $\cO_1$ strictly contained in a connected component of the open cone and define 
  \begin{equation}\label{tutti}
\wt\gamma\le  \,\mbox{ distance between $\cO_1$ and the border of the cone}\,,
  \end{equation} \
  so that
  \begin{equation}
  \label{tutti2}
   \min_{\lambda\in\cO_1}  \Big(|\mu^\s_{i,k}(\lambda)-  \mu^{\s'}_{i',k'}(\lambda)|, |\mu_{i}(\lambda)-  \mu^{\s'}_{i',k'}(\lambda)|, |\mu_{i}(\lambda) -  \mu_{i'}(\lambda)| \Big) >\tilde{\g}
  \end{equation}
  for any choice of the distinct eigenvalues.
  \\
  In $\cO_1$ the changes of variables $\cU_1(\lambda)$ and $(\cU_2(\lambda_i,\lambda_k))_{i<k}$ of Lemma \ref{lem:irr} are well defined, analytic  and we may estimate $\cU_i,\partial_{\lambda}\cU_i$  by Cauchy estimates (recall that $\cU_i$ and $\wt\gamma$ are $\e$ independent). 
   We are ready to define $U$, which is  a block diagonal matrix with respect to the blocks
$$
(\Z^2\setminus(\sS \cup \ccC\cup \Tan)) \cup_{n\neq 0}\sS_n\cup_{i<k}\ccC_{i,k}.
$$
On the first block $U$ is the identity, on each block $\sS_n$ it is the matrix $U_1$ of Lemma \ref{lem:irr}, on each 4x4 block  $\jj=(m,n)\in \ccC_{i,k}^+$, $\und \jj= (\tm_i+\tm_k-m,-n)$, the matrix  $U$ coincides with the matrix $U_2(\lambda_i,\lambda_k)$. We have proved items (i) and (ii).
Regarding item (iii) we only need to bound the norm of $\cR$ and $\cU$ together with their $\lambda$ derivatives. The bound on $\cR$ is trivial, as for the one on $\cU$ it follows by our definition of $\cO_1$.
Finally we prove item $(iv)$. Clearly  $\wtcM\circ \cU = \wtcM$  since this  functions  {\em does not} depend on the variables $(b_{
\jj})_{\jj\in \sS\cap\ccC} $, as one sees inspecting formulas \eqref{mp.4}.
Regarding $\wtcP_x $  and $\wtcP_{y}$ we notice that they are decomposed in the same blocks as $\whcZ_1$:
\begin{align*}
	&\wtcP_x= \sum_i  \tm_i  \yy_i + \sum_{(m,n) \in \Z^2 \setminus (\sS \cup \Tan\cup \ccC)}\!\!\!\!m \, |b_{(m,n)}|^2 +\!\!
	\sum_{i <k,  (m,n) \in \ccC^+_{i,k}}\!\!(m - \tm_i)  \left(\ |b_{(m,n)}|^2 - |b_{(\tm_i + \tm_j - m,-n)}|^2 \right) \ , \\
	&\wtcP_y= \sum_{(m,n) \in \Z^2 \setminus (\sS \cup \Tan\cup \ccC)}\!\!\!\! n |b_{(m,n)}|^2 + \sum_{n\neq 0} n \sum_{(m,n)\in\sS_n} |b_{(m,n)}|^2 + \sum_{i <k,  (m,n) \in \ccC^+_{i,k}}\!\!   n \left(\ |b_{(m,n)}|^2 - |b_{(\tm_i + \tm_j - m,-n)}|^2 \right) \ . 
\end{align*}
On the first block $\cU$ acts as the identity. 
On each block $\sS_n$,  $\cU$ acts as $\cU_1$  and moreover $\wtcP_x$ is null on this block while $\wtcP_y$ is proportional to the invariant Hamiltonian $\cQ_0= \sum_{i=1}^\tk |b_{(\tm_i,n)}|^2$. 
Finally on the last blocks $\cU$ acts as $\cU_2$ and $\wtcP_x$, $\wtcP_y$ are {\em both} proportional to the invariant Hamiltonian $\cQ_0$ (recall that these blocks are of type II and hence $\cQ_0= |b_{(m,n)}|^2 - |b_{(\tm_i + \tm_j - m,-n)}|^2$, see formula \eqref{cq2}). 
Finally we rename the canonical variables $(\yy,\theta,\ba)$.
\end{proof}

\begin{remark}
\label{rem:unstable}
The fact that the eigenvalues of $ N(\lambda_i,\lambda_k)$ might  in principle be imaginary shows that the tori {containing the family of the finite gap solutions } might be  linearly hyperbolic, hence linearly unstable. Here we want to rule out such behavior.
\end{remark}

%\begin{remark}
%Since we imposed the condition that $\ccC \cap \sS = \emptyset$, the changes of coordinates of Lemma \ref{lem:diag1} and Lemma \ref{lem:diag2} are independent of each other (in the sense that they act on different variables).
%\end{remark}
%\begin{proof}[Proof of Lemma \ref{lem:step4}]
%We restrict $\lambda$ to a set $\cO_1 {\subseteq \wt\cO}$ such that the functions $\mu_i(\lambda)$ and $\mu_{i,k}^\pm(\lambda)$ for $1 \leq i,  k\leq \tk$, $i \neq k$ are all real distinct, non zero	and are such that
%\begin{equation}\label{ponte}
%\min_{i < k}\inf_{\lambda\in \cO_1}\min(|\mu_i(\lambda)-\mu_k(\lambda)|, |\mu_{i,k}^+(\lambda)-\mu_{i,k}^-(\lambda) |) \ge \g_0
%\end{equation}
%for some $\g_0>0$ independent of $\e$. Since these are a finite number of distinct analytic functions such a region surely exists provided that we take $\g_0$ sufficiently small. The estimates on the Lipschitz norm of $U$ (or equivalently $U_{i,k}$ ) come from the Cauchy estimates, using  \eqref{ponte} to guarantee that we are not close to the border of the domain of analiticity.\\
%Now we define $\cT^{(R)} := \cT^{(R)}_1 \circ \cT^{(R)}_2 \circ \cT^{(R)}_3$. 
%Then $\cT^{(R)}$ fulfills the properties of item $(i)$. Item $(ii)$ and $(iv)$ follow from Lemma \ref{lem:diag0}, \ref{lem:diag1} and \ref{lem:diag2}.\\
%Item $(iii)$ is a consequence of the fact that $\cT^{(R)}$ is the identity on the variables which are not in $\sS \cup \ccC$.\\
%Finally we rename the canonical variables $(\yy,\theta,\ba)$.
%\end{proof}
\begin{proof}[Proof of Theorems \ref{thm:q} and \ref{thm:L}]
We just apply Lemmata \ref{scendo}, \ref{medio}, \ref{lem:diag0} and \ref{lem:step4}.
\end{proof}

\section{Non-degeneracy of the Hamiltonian $\wt\cD(\lambda, \e)$ }%{Second and third order Melnikov conditions for $\wt\Omega_\jj(\lambda, \e)$}
\label{sec:IIm}
By Theorem \ref{thm:q}  the quadratic  Hamiltonian $\omega\cdot \yy + \cD + \cH^\0 $ is now reduced to a $\cO(\e^2)$ perturbation of the  diagonal  Hamiltonian $\omega \cdot \yy + \wt\cD(\lambda, \e)$, where $\wt\cD(\lambda, \e)$ is defined in \eqref{D.lam}.\\
In this section we prove that we can impose  second and third order Melnikov conditions
on the frequencies $\wtOmega_\jj(\lambda, \e)$ of the operator $\wt\cD(\lambda, \e)$. 
%First we  have the following definition:
%\begin{definition}
%\label{rem:adm3}
%	Given $\bj = (\jj_1, \jj_2) \in (\Z^2\setminus\Tan)^2$, $\ell \in \Z^\tk$ and $\sigma= (\sigma_1, \sigma_2) \in \{-1, 1\}^2$, we say that $(\bj, \ell, \sigma)$ is {\em admissible} iff the monomial $e^{\im \theta \cdot \ell} \, a_{\jj_1}^{\sigma_1} \, a_{\jj_2}^{\sigma_2} $ Poisson commutes with $\wtcM,\wtcP$ and $(\bj, \ell, \sigma)\neq (\jj_1,\jj_1,0,\s_1,-\s_1)$ . We write $(\bj, \ell, \sigma)\in \fR_2$.
%\end{definition}
Now we have the following lemma:
\begin{lemma}
\label{quefe}
	Let $\wtOmega_\jj(\lambda, \e)$ be defined as in \eqref{def:omtilde}. For a generic  choice of $\cS_0$  the following holds: for each admissible $(\bj, \ell, \bs)\neq ((\jj,\jj),0, (\s_1,-\s_1)) $ in the sense of Definition \ref{rem:adm3} one has 
\begin{equation}
\label{mel.2}
\omega \cdot \ell + \sigma_1\wtOmega_{\jj_1}(\lambda, \e) + \sigma_2 \wtOmega_{\jj_2}(\lambda, \e) \not\equiv 0 \ , 
\end{equation}	
		in the set $\cO_1$ of Theorem \ref{thm:q}.
\end{lemma}
\begin{proof}
Without loss of generality we may assume that $\s_1=1$ and  set $\s_2=\s$ and $\jj_i=(m_i,n_i)$ for $i=1,2$.
In the expression \eqref{mel.2} we separate the terms of order $\e^0$ and $\e$:
\begin{equation}
\label{res.2}
\omega \cdot \ell + \wtOmega_{\jj_1}(\lambda, \e) + \sigma \wtOmega_{\jj_2}(\lambda, \e)  \equiv \tK_{\bj, \ell}^\bs + \e \tF_{\bj, \ell}^\bs(\lambda) \ ,
\end{equation}
where we denote by $\tK_{\bj, \ell}^\bs$  the part of \eqref{res.2} which is of order $\e^0$ and by $\tF_{\bj, \ell}^\bs(\lambda)$ the part of order $\e$. Explicitly
\begin{align}
&\qquad \tK_{\bj, \ell}^\bs := \omega^\0 \cdot \ell +   \wtOmega_{\jj_1}(\lambda, 0) + \sigma \,  \wtOmega_{\jj_2}(\lambda,0)   \\
& \qquad  \tF_{\bj, \ell}^\bs(\lambda) :=  \partial_{\e} \left. \Big(\omega(\lambda) \cdot \ell +   \wtOmega_{\jj_1}(\lambda, \e) +\sigma \wtOmega_{\jj_2}(\lambda, \e)  \Big) \right|_{\e = 0} =- \lambda\cdot \ell + \vartheta_{\jj_1}(\lambda) + \sigma \vartheta_{\jj_2}(\lambda)
\end{align}
 where  $\tK_{\bj, \ell}^\bs$ is an integer while the functions $\vartheta_{\jj}(\lambda)$ belong to the  finite list of functions
$$
\vartheta_{\jj}(\lambda)\in \Big\{ (\mu_i(\lambda))_{1 \leq i \leq \tk},\ \ \ (\mu_{i,k}^{\pm}(\lambda))_{1 \leq i < k \leq \tk} \Big \}.
$$
Clearly in order for the resonance \eqref{mel.2} to hold identically, we must have $\tK_{\bj, \ell}^\bs =\tF_{\bj, \ell}^\bs(\lambda)\equiv 0$.

Remark that $\jj_1$ and $\jj_2$ can belong to either $\Z^2 \setminus  (\Tan\cup\sS \cup \ccC)$, or $\sS$, or $\ccC$ and all the possible combinations are possible. Thus we perform  a case analysis and show that,  however one chooses admissible $((\jj_1,\jj_2),\ell,(\s_1,\s_2))\neq ((\jj,\jj),0, (\s_1,-\s_1)) $, the function \eqref{mel.2} cannot be identically 0.
	\begin{enumerate}
		\item $\jj_1,\jj_2\in\Z^2\setminus (\Tan\cup\sS\cup\ccC)$. 
		In this case  $\tF_{\bj, \ell}^\bs(\lambda) = - \lambda\cdot \ell$, which is identically zero iff $ \ell = 0$. 
		The conservation of mass and momentum reads
		\begin{equation}
		\label{mc1}
		\eta(\ell) + 1 +\sigma  = 0 \ , \quad \pi(\ell) + m_1 + \sigma m_2 = 0\ ,\quad n_1+\s n_2=0 \ .
		\end{equation}
		Since $\ell=0$, the first equation above  fixes $\s=-1$ and the other two imply  $\jj_1=\jj_2$, which is the trivial case we excluded. 
		
		\item  $\jj_1 \in \sS \ ,\jj_2 \in \Z^2 \setminus  (\Tan\cup\sS \cup \ccC)$. 
		We have  
		$$
		\tF_{\bj, \ell}^\bs(\lambda)= - \lambda \cdot \ell + \mu_i (\lambda) \quad \mbox{for some} \; 1 \leq i \leq \tk \ . 
		$$ 
		 This expression may be identically zero only if for some $1 \leq i \leq \tk$  and some $\ell\in \Z^{\tk} $ one has $\mu_i(\lambda)\equiv \lambda\cdot \ell$. Since  by definition  $\mu_i(\lambda)$ is a root of $P(t,\lambda)$ (defined  in \eqref{P.poly0bis}), we would have that  $(t- \lambda \cdot \ell) \in \Z[t, \lambda_1, \ldots, \lambda_\tk]$ would be a divisor of $P(t,\lambda)$. However by Lemma \ref{lem:irr} $P(t,\lambda)$ is    irreducible in $\Z[t, \lambda_1, \ldots, \lambda_\tk]$, and we have obtained a contradiction.

		\item $\jj_1 \in \ccC_{i, k}^\pm$ for some $i<k$, $\ \jj_2 \in \Z^2 \setminus  (\Tan\cup\sS \cup \ccC)$. 		We have  $$
\tF_{\bj, \ell}^\bs(\lambda)=-\lambda \cdot \ell + \mu_{i,k}^\pm (\lambda) \quad \mbox{for some} \; 1 \leq i< k\leq \tk \ . 
$$  This expression may be zero only if for some $1 \leq i< k\leq \tk$, $s=\pm $  and some $\ell\in \Z^{\tk} $ one has $\mu^s_{i,k}(\lambda)\equiv\lambda\cdot \ell$. But in this case the polynomial $(t- \lambda \cdot \ell) \in \Z[t, \lambda_1, \ldots, \lambda_\tk]$ would be a divisor of $Q(t,\lambda_i,\lambda_k)$ defined in \eqref{char02bis}, which is irreducible by  Lemma \ref{lem:irr}.  We have obtained a contradiction.

		\item $\jj_1 , \jj_2 \in \sS $. 		W.l.o.g.  we may assume that $m_1=\tm_i$ and $m_2=\tm_k$. By conservation of mass and momentum of Remark \ref{leggi_sel1}
		\begin{equation}
		\label{mc4}\eta(\ell) = 0\,,\quad \pi(\ell) = 0\,,\quad  n_1 +\s n_2 = 0.
		\end{equation}
		Then we have $\tK_{\bj, \ell}^\bs=\omega^\0\cdot \ell  + (1+ \sigma) n_1^2$ and finally  
		$$
		\tF_{\bj, \ell}^\bs(\lambda)=-\lambda \cdot \ell + \mu_i(\lambda) + \sigma \mu_k(\lambda) \ .
		$$ 
As usual assume that \eqref{res.2} is identically $0$. 
		If $\ell=0$ then  $\tK_{\bj, \ell}^\bs=0$ implies $\s=-1$. Then $\tF_{\bj, \ell}^\bs(\lambda)\equiv 0$ iff  $\mu_i(\lambda) = \mu_k(\lambda)$.  By the irreducibility of $P(t,\lambda)$  this may hold only if $i=k$ and hence $m_1=m_2=\tm_i$. Since $n_1=n_2$ this implies $\jj_1=\jj_2$, which contradicts the assumptions.\\ 
		Now suppose $\ell\neq 0$. Then $\tF_{\bj, \ell}^\bs(\lambda)\equiv 0$ iff 
		$ \mu_i(\lambda) \equiv - \sigma \mu_k(\lambda)+ \lambda \cdot \ell$.  
		This means that $\mu_k(\lambda)$ is a root both of $P(t,\lambda)$ and  $P(-\s t + \lambda \cdot \ell,\lambda ) $, so the two polynomials have a common divisor.
		We claim that $P(-\s t + \lambda \cdot \ell,\lambda ) $ is irreducible over $\Z[t, \lambda_1, \ldots, \lambda_\tk]$. 
		Indeed, suppose that 
		$$
		P(-\s t + \lambda \cdot \ell,\lambda )= P_1(t,\lambda)P_2(t,\lambda) \ , 
		$$
		then setting $\tau= -\s t + \lambda \cdot \ell$ we would have
		$$
		P(\tau,\lambda)= P_1(-\s \tau+\s\lambda\cdot\ell,\lambda)P_2(-\s \tau+\s\lambda\cdot\ell,\lambda)  \ , 
		$$
		 but this contradicts the irreducibility of $P$ since $P_1(-\s \tau+\s\lambda\cdot\ell,\lambda)\in \Z[t, \lambda_1, \ldots, \lambda_\tk]$. 
		
		Now in order for $P(t,\lambda)$ and $P(-\s t + \lambda \cdot \ell,\lambda ) $ to have a common divisor they have to be equal (or opposite).
		Since $\ell\neq 0$ and $\eta(\ell)=0$ then $\lambda \cdot \ell$ depends on at least two variables, say $\lambda_1,\lambda_2$. We specialize the equality $P(t,\lambda)= s P(-\s t + \lambda \cdot \ell,\lambda )$, $s= \pm$, to $\lambda_i=0$ for all $i\neq 1,2$ and get 
		$$
		(-\s t +\lambda\cdot\ell)^{\tk-2}[(-\s t +\lambda\cdot\ell + \lambda_1)(-\s t +\lambda\cdot\ell + \lambda_2) - 2  \lambda_1  (-\s t + \lambda \cdot \ell + \lambda_2) -2
		\lambda_2  (-\s t + \lambda \cdot \ell + \lambda_1)] =
		$$
		$$
		s t^{\tk-2}[(t+\lambda_1)(t+\lambda_2) -2 \lambda_1(t+\lambda_2) -2\lambda_2(t+\lambda_1) \ . 
		$$
		Equating the leading terms we get  $(-\s)^{\tk} = s$, then from the subsequent degree
		$$
		(-\s)^{\tk-1} (2 \lambda\cdot\ell -(\lambda_1 + \lambda_2) )= -(-\s)^{\tk}(\lambda_1 +\lambda_2) \quad \Rightarrow \quad 2\s \lambda\cdot\ell = (1+\s)(\lambda_1 +\lambda_2) \ , 
		$$
		which is not compatible with $\eta(\ell)=0$, $\ell\neq 0$.
		%Indeed 
		%Remark that such $\ell^{(i,k)}$ does not depend on the tangential sites (but just on their number) and satisfies
		%$|\ell^{(i,k)}|\le \tM_0$.
		%
		%Since $\ell^{(i,k)} \neq 0$, the condition $\pi(\ell^{(i,k)}) \equiv \sum_r \tm_r \ell^{(i,k)} = 0$ defines a hyperplane in $\C^{\tk}$. Then  in our genericity condition \eqref{pop} we must fix $\tL>\tM_0$, so that $\ell^{(i,k)}$ is not compatible with momentum conservation \eqref{mc4} . 
		
		\item $ \jj_1 \in \ccC \ , \jj_2 \in \sS$.  W.l.o.g. assume that $\jj_1\in \ccC_{i,k}^s$ for some $1 \leq i<k \leq \tk$, $s=\pm$, and $m_2= \tm_h$ for some $1 \leq h \leq \tk$. In this case
		$$
		\tF_{\bj, \ell}^\bs(\lambda) = - \lambda \cdot \ell + \mu_{i,k}^s (\lambda) +\s \mu_h(\lambda) \ .
		$$
Assume that  $\tF_{\bj, \ell}^\bs(\lambda) \equiv 0$. Then  we would have 
$$
\mu_{i,k}^s (\lambda) = -\s \mu_h(\lambda) + \lambda \cdot \ell \ , 
$$
 which means that $P(t, \lambda)$  would have $\mu_h(\lambda)$ as a common root with $Q(-\sigma t + \lambda \cdot \ell, \lambda_i, \lambda_k)$. 
 Since both polynomials are irreducible (one can prove  the irreducibility  of  $Q(-\sigma t + \lambda \cdot \ell, \lambda_i, \lambda_k)$ as we did  in 4.), then  they must coincide up to a scale factor. This is absurd unless $\tk=2$. In this last case we can verify directly that the two polynomials never coincide.

		\item $ \jj_1,\jj_2 \in \ccC $.  W.l.o.g.  assume that $\jj_r\in \ccC_{i_r,k_r}^{s_r}$  for $r=1,2$ and  $1 \leq i_r< k_r \leq \tk$,  $s_r=\pm$. We have
		$$
		\tF_{\bj, \ell}^\bs(\lambda)=-\lambda \cdot \ell + \mu_{i_1,k_1}^{s_1} (\lambda) +\s \mu_{i_2,k_2}^{s_2 }(\lambda)
		$$
		so that $\tF_{\bj, \ell}^\bs(\lambda) \equiv 0$ would require 
		$$
		\mu_{i_1,k_1}^{s_1} (\lambda)\equiv  \lambda \cdot \ell -\s \mu_{i_2,k_2}^{s_2 }(\lambda).
		$$
		This  is trivially false if $(i_1,k_1)\neq (i_2,k_2)$ (one just remarks that the square roots in formula \eqref{mutilde} cannot cancel out). If $(i_1,k_1)= (i_2,k_2)$ 
we divide two cases. If $\ell = 0$, $\mu_{i_1,k_1}^{s_1} (\lambda)\equiv   -\s \mu_{i_2,k_2}^{s_2 }(\lambda)$ can happen only if  $\sigma = -1$ and $s_2 = s_1$. By conservation of $\wtcP_y$ we have $n_1=n_2$ and by 
		conservation of $\wtcP_x$  we get $(m_1 - \tm_i) - (m_2 - \tm_i) =0$, which implies $m_1 = m_2$. Hence $\jj_1 = \jj_2$, which is a contradiction.		
	
		If $\ell \neq 0$, for the square root to cancels identically
		 we must have
		$\s=1$, $s_1=-s_2$ and $ \lambda \cdot \ell= -( \lambda_{i_1}-\lambda_{k_1})$. Then $\jj_1,\jj_2$ are on the same circle $\ccC_{i_1,k_1}$ and by conservation of $\wtcP_y$ we have $n_1=-n_2$. Since both $\jj_1$ and $\jj_2$ are on a circle, this implies that either $m_1 = m_2$ or they are on the opposite sides of a diameter, i.e.  $m_1 + m_2 = \tm_{i_1} + \tm_{i_k}$. 
		But by  conservation of $\wtcP_x$  we get
		$$
		 0 = (-\tm_{i_1}+ \tm_{k_1}) + (m_1 - \tm_{i_1}) + (m_2 - \tm_{k_1}) =  m_1 + m_2 -2 \tm_{i_1} 
		$$
and in both cases we get contradictions.
	\end{enumerate}
\end{proof}

We pass to third order Melnikov conditions. 
\begin{lemma}
	\label{lem.F}
	There exists $\tL\in \N$ such that for a $\tL$-generic  choice of $\cS_0$  the following holds: for any $(\bj,\ell, \bs)\in\fA_3\setminus\fR_3$, i.e. admissible and non action preserving  in the sense of Definitions \ref{rem:adm3}, \ref{reso} we have
\begin{equation}
\label{mel.3}
\omega(\lambda) \cdot \ell +  \s_1 \wtOmega_{\jj_1}(\lambda, \e) +\sigma_2\wtOmega_{\jj_2}(\lambda, \e) + \sigma_3\wtOmega_{\jj_3}(\lambda, \e)\not\equiv 0
\end{equation}	
		in the set $\cO_1$ of Lemma \ref{lem:step4}.
\end{lemma}
\begin{proof}
As in the previous proof we set
\begin{equation}
\label{res.3} 
\omega(\lambda) \cdot \ell +  \s_1 \wtOmega_{\jj_1}(\lambda, \e) +\sigma_2\wtOmega_{\jj_2}(\lambda, \e) + \sigma_3\wtOmega_{\jj_3}(\lambda, \e)\equiv \tK_{\bj, \ell}^\bs + \e \tF_{\bj, \ell}^\bs(\lambda)
\end{equation}
where $\tK_{\bj, \ell}^\bs$ is the term of order $1$ and $\tF_{\bj, \ell}^\bs(\lambda)$ is the term of order $\e$. Explicitly 
\begin{align}
\label{K3}
 \tK_{\bj, \ell}^\bs & := \omega^\0 \cdot \ell +   \s_1\wtOmega_{\jj_1}(\lambda, 0) + \sigma_2 \,  \wtOmega_{\jj_2}(\lambda,0) +\s_3 \wtOmega_{\jj_3}(\lambda,0)  \\
 \label{F3}
\tF_{\bj, \ell}^\bs(\lambda) &:=  \partial_{\e} \left. \Big(\omega(\lambda) \cdot \ell +  \s_1 \wtOmega_{\jj_1}(\lambda, \e) +\sigma_2\wtOmega_{\jj_2}(\lambda, \e) + \sigma_3\wtOmega_{\jj_3}(\lambda, \e)  \Big) \right|_{\e = 0}\\
& =- \lambda\cdot \ell + \s_1\vartheta_{\jj_1}(\lambda) + \sigma_2 \vartheta_{\jj_2}(\lambda) + \sigma_3 \vartheta_{\jj_3}(\lambda)
\end{align}
As before	  $\tK_{\bj, \ell}^\bs$ is an integer while the functions $\vartheta_{\jj}(\lambda)$ belong to the  finite list of functions
$$
\vartheta_{\jj}(\lambda)\in \Big\{ (\mu_i(\lambda))_{1 \leq i \leq \tk},\ \ \ (\mu_{i,k}^{\pm}(\lambda))_{1 \leq i < k \leq \tk} \Big \}.
$$
Clearly in order for the resonance \eqref{mel.3} to not hold identically, we need  to ensure that 
	\begin{equation}
	\label{K.F.0}
	\tK_{\bj, \ell}^\bs = 0 \quad \mbox{ and }  \quad \tF_{\bj, \ell}^\bs(\lambda) \equiv 0 
	\end{equation}
	cannot occur for $(\bj,\ell, \bs)$ admissible.
	As before we  study all the possible combinations, each time we assume that \eqref{K.F.0} holds and we deduce a contradiction.
	\begin{enumerate}
		\item $\jj_1, \jj_2, \jj_3 \in \Z^2\setminus(\Tan \cup \sS \cup \ccC$). 
		Then $\tF_{\bj, \ell}^\bs(\lambda) = -\lambda \cdot \ell$ is identically $0$ iff $\ell = 0$. However by  mass conservation $\eta(\ell)$ is odd  hence $\ell \neq 0$.
		
		\item $\jj_1, \jj_2 \in \Z^2\setminus (\Tan \cup \sS \cup \ccC)$, $\jj_3 \in \sS$.
		Then $\tF_{{\bf j}, \ell}^\bs(\lambda) = -\lambda \cdot \ell + \sigma_3 \, \mu_{i}(\lambda)$ for some $1 \leq i \leq \tk$. If $\tF_{{\bf j}, \ell}^\bs(\lambda) \equiv 0$ then $\mu_i(\lambda) = \sigma_3  \lambda \cdot \ell$ is a  root in $\Z[\lambda]$ of the polynomial $P(t, \lambda)$ defined in \eqref{P.poly0bis}, which is irreducible over $\Z[t, \lambda_1, \ldots, \lambda_\tk]$ (see Lemma \ref{lem:irr}), thus leading to a contradiction.
		
		\item $\jj_1, \jj_2 \in \Z^2\setminus (\Tan \cup \sS \cup \ccC)$, $\jj_3 \in \ccC$.  W.l.o.g. assume that $\jj_3\in \ccC_{i,k}^s$ for some $1 \leq i<k \leq \tk$, $s=\pm$. 
		Then 
	$$
	\tF_{{\bf j}, \ell}^\bs(\lambda) = -\lambda \cdot \ell + \sigma_3\,  \mu_{i,k}^s(\lambda) \ .
	$$ 
	If $\tF_{{\bf j}, \ell}^\bs(\lambda) \equiv 0$, $\mu_{i,k}^s(\lambda) = \sigma_3 \lambda \cdot \ell$ is a  root in $\Z[t, \lambda_i, \lambda_k]$ of the polynomial $Q(t,\lambda_i, \lambda_k)$ defined in \eqref{char02bis}, which is irreducible over $\Z[t, \lambda_i, \lambda_k]$ (see Lemma \ref{lem:irr}), thus leading to a contradiction.

		\item $\jj_1, \jj_2 \in \sS$, $\jj_3 \in \Z^2\setminus (\Tan \cup \sS \cup \ccC)$.  W.l.o.g. let $\jj_1 = (\tm_{i},n_1)$, $\jj_2=(\tm_{k},n_2)$ for some $1 \leq i, k \leq \tk$.  Then 
		$$
		\tF_{{\bf j}, \ell}^\bs(\lambda) = -\lambda \cdot \ell + \sigma_1\,  \mu_{i}(\lambda) + \sigma_2 \mu_{k}(\lambda) \ . 
		$$
		By conservation of mass $\eta(\ell)  = \pm 1$, thus $\ell \neq 0$. 
		Assume $\tF_{{\bf j}, \ell}^\bs(\lambda) \equiv 0$. 
		 Then  $\mu_{k}(\lambda) = -\sigma_1\sigma_2 \mu_{i}(\lambda) +\sigma_2 \lambda \cdot \ell$. 
		 This means that  $P(t,\lambda)$ has $\mu_{i_1}(\lambda)$ as a common divisor with  $P(-\sigma_1\sigma_2 t + \sigma_2 \lambda\cdot \ell)$. Proceeding as in case 4. of the previous proof  one gets a contradiction.

		\item $\jj_1, \jj_2 \in \ccC$, $\jj_3 \in \Z^2\setminus (\Tan \cup \sS \cup \ccC)$.  W.l.o.g. let  $\jj_r \in  \ccC_{i_r,k_r}^{s_r}$ for some $1 \leq i_r < k_r \leq \tk$, $s_r =\pm$, $r=1,2$.
		Then 
		$$
		\tF_{{\bf j}, \ell}^\bs(\lambda) \equiv -\lambda \cdot \ell + \sigma_1  \mu_{i_1,k_1}^{s_1}(\lambda) + \sigma_2\mu_{i_2,k_2}^{s_2}(\lambda). 
		$$
		By mass conservation $\eta(\ell) = \pm 1$, hence $\ell \neq 0$.
	Then $ \tF_{{\bf j}, \ell}^\bs(\lambda) \equiv 0$ would require
		 $$
		\mu_{i_1,k_1}^{s_1} (\lambda)\equiv  \sigma_1\lambda \cdot \ell -\sigma_1\sigma_2 \mu_{i_2,k_2}^{s_2 }(\lambda) \ .
		$$
		This  is trivially false if $(i_1,k_1)\neq (i_2,k_2)$ (one just remarks that the square roots in formula \eqref{mutilde} cannot cancel out).
	If $(i_1,k_1)=(i_2,k_2)$ then by the explicit formula \eqref{mutilde}
			we obtain
			$$
			0  \equiv -\lambda \cdot \ell + (\sigma_1+ \sigma_2)  \frac{\lambda_{i_1} - \lambda_{k_1}}{2} -(\sigma_1s_1+ \sigma_2 s_2  ) \frac{\sqrt{\lambda_{i_1}^2 + \lambda_{k_1}^2- 14 \lambda_{i_1} \lambda_{k_1}}}{2} \ . 
			$$
			If $\sigma_1s_1+ \sigma_2 s_2  \neq 0$ then we get a contradiction since the root cannot be canceled identically by a linear function.
			If  $\sigma_1s_1+ \sigma_2 s_2  = 0$ and $\sigma_1 + \sigma_2 = 0$, we also get a contradiction as $\ell \neq 0$. Thus we are left with the case $\sigma_1 + \sigma_2 = \pm 2$, which implies $0 \equiv - \lambda \cdot(\ell \pm (\be_{i_1} - \be_{i_k}))$. This can hold only if  $\ell = \mp (  \be_{i_1} - \be_{i_k}).$ But for such $\ell$ one has $\eta(\ell) = 0$, contradicting the mass conservation.
%	
%		\begin{itemize}
%			\item If $(i_1, k_1) = (i_2, k_2)$, then we obtain
%			$$
%			0  \equiv -\lambda \cdot \ell + (\sigma_1+ \sigma_2)  \frac{\lambda_{i_1} - \lambda_{k_1}}{2} +(\sigma_1s_1+ \sigma_2 s_2  ) \frac{\sqrt{\lambda_{i_1}^2 + \lambda_{k_1}^2- 14 \lambda_{i_1} \lambda_{k_1}}}{2} $$
%			If $\sigma_1s_1+ \sigma_2 s_2  \neq 0$, then we get a contradiction, since the root cannot be cancelled identically by a linear function.
%			If  $\sigma_1s_1+ \sigma_2 s_2  = 0$ and $\sigma_1 + \sigma_2 = 0$, we also get a contradiction as $\ell \neq 0$. Thus we are left with the case $\sigma_1 + \sigma_2 = \pm 2$, which implies $0 = - \lambda \cdot(\ell \pm (\be_{i_1} - \be_{i_k}))$. Since $\lambda$ is diophantine, this can be zero only if $\ell = \mp (  \be_{i_1} - \be_{i_k}).$ But in such case $\sum_i \ell_i = 0$, which contradicts the mass conservation.
%			\item W.l.o.g. assume that $k_2 \neq i_1,k_1$. Then put $\lambda_{k_2} = 0$ and obtain
%			$$
%			0  \equiv -\lambda \cdot \ell + \sigma_1  \frac{\lambda_{i_1} - \lambda_{k_1} +s_1  \sqrt{\lambda_{i_1}^2 + \lambda_{k_1}^2- 14 \lambda_{i_1} \lambda_{k_1}}}{2} + 
%			\sigma_2  \frac{ \lambda_{i_2}  +s_2  |\lambda_{i_2}| }{2} 
%			$$
%			Due to the root, this function cannot vanish identically.
%		\end{itemize}
		
\item $\jj_1 \in \Z^2\setminus (\Tan \cup \sS \cup \ccC)$, $\jj_2 \in \sS$, $\jj_3 \in \ccC$. W.l.o.g. let $\jj_2 = (\tm_{i_2}, n_2)$ for some $1 \leq i_2 \leq \tk$, $\jj_3 = (m_3, n_3) \in \ccC_{i_3, k_3}^{s_3}$ for some $1\leq i_3 < k_3 \leq \tk$, $s_3=\pm$. Then $$
\tF_{{\bf j}, \ell}^\bs(\lambda) \equiv -\lambda \cdot \ell + \sigma_2  \mu_{i_2}(\lambda) + \sigma_3 \mu_{i_3,k_3}^{s_3}(\lambda) \ .
$$
		Assume that $\tF_{{\bf j}, \ell}^\bs(\lambda)\equiv 0$. Then $\mu_{i_2}(\lambda) = - \sigma_2 \sigma_{3}  \mu_{i_3,k_3}^{s_3}(\lambda) + \sigma_2 \lambda \cdot \ell$. 
		Thus $ \mu_{i_3,k_3}^{s_3}(\lambda) $  is a root of $Q(t, \lambda_{i_3}, \lambda_{k_3})$  and of  $P(-\sigma_2\sigma_3 t + \sigma_2 \lambda \cdot \ell, \lambda) \in \Z[t, \lambda_1, \ldots,\lambda_\tk]$.
		But both polynomials are irreducible, thus they must coincide up to a sign. This is absurd unless $\tk \neq 2$. In the last case we can verify directly that the two polynomials never coincide.

		\item $\jj_1, \jj_2, \jj_3 \in \sS$. W.l.o.g. let  $\jj_1=(\tm_{i_1}, n_1)$, $\jj_2=(\tm_{i_2}, n_2)$, $\jj_3= (\tm_{i_3}, n_3)$ for some $1 \leq i_1,i_2,i_3 \leq \tk$ and  $n_1, n_2, n_3 \neq 0$.   Then  
	$$
	\tF_{\bj, \ell}^\bs(\lambda) = -\lambda \cdot \ell + \sigma_1 \mu_{i_1}(\lambda)+ \sigma_2 \mu_{i_2}(\lambda) + \sigma_3 \mu_{i_3}(\lambda) \ . 
	$$
	 Assume $\tF_{{\bf j}, \ell}^\bs(\lambda) \equiv 0$. 
	  This fix $\ell^{({\bf i}, \bs)} \in \Z^\tk$ uniquely, ${\bf i}:=(i_1, i_2, i_3)$.  By conservation of $\wtcP_x$ we have  $\sum_{k} \tm_k \ell_k^{({\bf i}, \bs)} = 0$.
	 If $\ell^{({\bf i}, \bs)} \neq 0$, such expression  defines a hyperplane ${\bf v}^{({\bf i}, \bs)} \subset \C^\tk$, which depends just on the functions $\mu_i(\lambda)$.
	 By taking all the possible  $({\bf i}, \bs)$ (there are only a finite number of possibilities!) we get a finite number of hyperplanes. 
	  Then it is enough to choose the sites $(\tm_k)_{1 \leq k \leq \td}$ outside the set $\bigcup_{{\bf i}, \bs}{\bf v}^{({\bf i}, \bs)}$ 
	  (remark that the $\mu_i(\lambda)$ {\bf do not} depend on 
	  the sites $(\tm_k)_{1 \leq k \leq \td}$, see Remark \ref{rem:mu}). In order to impose such condition we only have to choose in the $\tL$-genericity condition \ref{gen.cond}: $$\tL>\max_{k,{\bf i},\bs}|\ell_k^{({\bf i}, \bs)} |.$$ 

		Now assume $\ell^{({\bf i}, \bs)} = 0$. In this case we are not able to prove that $\tF_{{\bf j}, \ell}^\bs(\lambda) \not\equiv 0$, however we show that in such case  $\tK_{{\bf j}, 0}^\bs \neq 0$. Assume the contrary.   By conservation of $\wtcP_y$ we have  $\sum_{k=1}^3 \sigma_k n_k = 0$ and furthermore  $0 = \tK_{{\bf j}, 0}^\bs = \sigma_1 n_1^2 + \sigma_2 n_2^2 + \sigma_3 n_3^2$. Exploring all the possible choices of $\sigma_1, \sigma_2, \sigma_3$, one deduces that at least one among $n_1, n_2, n_3$ must equal 0. But this is impossible.

		\item $\jj_1, \jj_2, \jj_3 \in \ccC$. W.l.o.g. let $\jj_r \in \ccC_{i_r, k_r}^{s_r}$ for some $1 \leq i_r < k_r \leq \tk$, $s_r= \pm$, $r=1,2,3$.
Then
		\begin{align*}
		\tF_{{\bf j}, \ell}^\bs(\lambda) \equiv -\lambda \cdot \ell + \sigma_1  \frac{\lambda_{i_1} - \lambda_{k_1} -s_1  \sqrt{\lambda_{i_1}^2 + \lambda_{k_1}^2- 14 \lambda_{i_1} \lambda_{k_1}}}{2} + 
		\sigma_2  \frac{\lambda_{i_2} - \lambda_{k_2} -s_2  \sqrt{\lambda_{i_2}^2 + \lambda_{k_2}^2- 14 \lambda_{i_2} \lambda_{k_2}}}{2} \\
		+  \sigma_3  \frac{\lambda_{i_3} - \lambda_{k_3} -s_3  \sqrt{\lambda_{i_3}^2 + \lambda_{k_3}^2- 14 \lambda_{i_3} \lambda_{k_3}}}{2}
		\end{align*}
				 Assume $\tF_{{\bf j}, \ell}^\bs(\lambda) \equiv 0$. 
		Assume first that $(i_1,k_1)= (i_2,k_2) = (i_3,k_3)$. Then there is an odd number of roots in the expression for $\tF_{{\bf j}, \ell}^\bs(\lambda)$ and thus they cannot cancel identically.\\
			Otherwise there is a couple of indexes which appears just once. Assume it is $(i_1,k_1)$. Then specify to $\lambda_{i_1} = \lambda_{k_1} = 1$ and all the rest at $0$. Then the first root, namely $\sqrt{\lambda_{i_1}^2 + \lambda_{k_1}^2- 14 \lambda_{i_1} \lambda_{k_1}}$ is complex, while all the others terms are real. Hence $\tF_{{\bf j}, \ell}^\bs(\lambda)$ cannot vanish identically.

		\item $\jj_1, \jj_2 \in \sS$, $\jj_3 \in \ccC$.W.l.o.g. let  $\jj_1 =  (\tm_{i_1}, n_1)$, $\jj_2= (\tm_{i_2}, n_2)$ for some $1\leq i_1, i_2 \leq \tk$  and $\jj_3  \in \ccC_{i_3,k_3}^{s_3}$ for some $1\leq i_3 < k_3 \leq \tk$ and $s_3 = \pm$.  Then 
		$$
		\tF_{{\bf j}, \ell}^\bs(\lambda) \equiv -\lambda \cdot \ell + \sigma_1  \mu_{i_1}(\lambda) + \sigma_2  \mu_{i_2}(\lambda) + \sigma_3  \mu_{i_3 k_3}^{s_3}(\lambda) \ . 
		$$
		 Assume $\tF_{{\bf j}, \ell}^\bs(\lambda) \equiv 0$.
		Specify all the $\lambda$-depending functions to  $E:=\{ \lambda \in \C^\tk \colon \  \lambda_i = 0 \mbox{ for any } i \neq i_3, k_3\}$.
		Then $\mu_{i_1}(\lambda)\vert_{E}$, $\mu_{i_2}(\lambda)\vert_{E}$ are  roots of 
		$$
	P(t, \lambda)\vert_{E}=	P(t, 0, \ldots, \lambda_{i_3}, 0 , \ldots,  \lambda_{k_3}, \ldots, 0) = t^{k-2} \left( t^2 - (\lambda_{i_3} + \lambda_{k_3}) t - 3\lambda_{i_3} \lambda_{k_3}\right) 
		$$
and thus belong to the set
		$$
\left\lbrace 		0, \qquad  \frac{\lambda_{i_3} + \lambda_{k_3} \pm  \sqrt{\lambda_{i_3}^2 + \lambda_{k_3}^2 + 14 \lambda_{i_3} \lambda_{k_3}}}{2} \right\rbrace \ .
		$$
		Thus we  get three cases:
		\begin{itemize}
			\item  $\mu_{i_1}(\lambda)\vert_{E} = \mu_{i_2}(\lambda)\vert_{E} = 0$, thus 
			$$
			\tF_{{\bf j}, \ell}^\bs(\lambda)\vert_{E} = - \lambda_{i_3}\ell_{i_3} - \lambda_{k_3}\ell_{k_3} + \sigma_3  \mu_{i_3 k_3}^{s_3}(\lambda)\ . 
			 $$
			  Then one proceeds as in  3. to get a contradiction.
			\item $\mu_{i_1}(\lambda)\vert_{E} = 0$, $\mu_{i_2}(\lambda)\vert_{E} =  \frac{1}{2} \left(\lambda_{i_3} + \lambda_{k_3} \pm  \sqrt{\lambda_{i_3}^2 + \lambda_{k_3}^2 + 14 \lambda_{i_3} \lambda_{k_3}}\right) $. In such a  case 
			\begin{align*}
			\tF_{{\bf j}, \ell}^\bs(\lambda)\vert_{E} = - \lambda_{i_3}\ell_{i_3} - \lambda_{k_3}\ell_{k_3} & +  \sigma_2 \frac{\lambda_{i_3} + \lambda_{k_3} + \tilde\sigma_2  \sqrt{\lambda_{i_3}^2 + \lambda_{k_3}^2 + 14 \lambda_{i_3} \lambda_{k_3}}}{2} \\
		&	+  \sigma_3  \frac{\lambda_{i_3} - \lambda_{k_3} -s_3  \sqrt{\lambda_{i_3}^2 + \lambda_{k_3}^2- 14 \lambda_{i_3} \lambda_{k_3}}}{2}
			\end{align*}
			But the roots cannot vanish identically, since they are different, hence we exclude such a case.
			\item $\mu_{i_1} (\lambda)\vert_{E} =  \mu_{i_2} (\lambda)\vert_{E}=  \frac{1}{2} \left(\lambda_{i_3} + \lambda_{k_3} \pm  \sqrt{\lambda_{i_3}^2 + \lambda_{k_3}^2 + 14 \lambda_{i_3} \lambda_{k_3}}\right) $. Thus
			\begin{align*}
			\tF_{{\bf j}, \ell}^\bs(\lambda)\vert_{E} = 	- \lambda_{i_3}\ell_{i_3} - \lambda_{k_3}\ell_{k_3} + 
	(\sigma_1 + \sigma_2) \frac{\lambda_{i_3} + \lambda_{k_3}}{2} 
&	+	 (\sigma_1 s_1 + \sigma_2 s_2) \frac{ \sqrt{\lambda_{i_3}^2 + \lambda_{k_3}^2 + 14 \lambda_{i_3} \lambda_{k_3}}}{2}   \\
&			+  \sigma_3  \frac{\lambda_{i_3} - \lambda_{k_3} - s_3  \sqrt{\lambda_{i_3}^2 + \lambda_{k_3}^2- 14 \lambda_{i_3} \lambda_{k_3}}}{2}
			\end{align*}
			One checks directly that however  one chooses $\sigma_1, \sigma_2, \sigma_3, s_1, s_2, s_3$ the roots cannot vanish identically.
		\end{itemize}

		\item $\jj_1, \jj_2 \in \ccC$, $\jj_3 \in \sS$. W.l.o.g. let $\jj_r  \in \ccC_{i_r,k_r}^{s_r}$ for some $1 \leq i_r < k_r \leq \tk$, $s_r=\pm$, $r=1,2$, and $\jj_3 = (\tm_{i_3}, n_3)$ for some $1 \leq i_3 \leq \tk$. 
		Then  
		$$
		\tF_{{\bf j}, \ell}^\bs(\lambda) \equiv -\lambda \cdot \ell + \sigma_1 \mu_{i_1 k_1}^{s_1}(\lambda) + \sigma_2 \mu_{i_2 k_2}^{s_2}(\lambda) + \sigma_3 \mu_{i_3}(\lambda) \ . 
		$$
 Assume $\tF_{{\bf j}, \ell}^\bs(\lambda) \equiv 0$. We study different cases separately:
		\begin{itemize}
			\item if $(i_1, k_1) = (i_2,k_2)$, then
			$$
			\tF_{{\bf j}, \ell}^\bs(\lambda) = 
			- \lambda\cdot \ell + (\sigma_1 + \sigma_2) 
			\frac{\lambda_{i_1} + \lambda_{k_1}}{2} - (\sigma_1 s_1 + \sigma_2 s_2) \frac{ \sqrt{\lambda_{i_1}^2 + \lambda_{k_1}^2 - 14 \lambda_{i_1} \lambda_{k_1}}}{2} + \sigma_3 \mu_{i_3}(\lambda)
			$$
Now if  $\sigma_1 s_1 + \sigma_2 s_2 = 0$, we are reduced to case 3., thus we have a contradiction.\\
If $\sigma_1 s_1 + \sigma_2 s_2 \neq  0$, we specify to  $\lambda_{i_1} = \lambda_{k_1}=1$,  $\lambda_i = 0$ if $i \neq i_1,k_1$.   We obtain that $ {\rm Im } \, \tF_{{\bf j}, \ell}^\bs(\lambda) \neq 0$, hence it does not vanish identically. Indeed  $\mu_{i_3}(\lambda)$ for $\lambda \geq 0$ is an eigenvalue of a self-adjoint matrix, and thus it is real.

			\item if $i_2 = i_1$, $k_2 \neq k_1$, then as before we put $\lambda_{i_1} = \lambda_{k_1}=1$, all the other $\lambda_i = 0$.  We obtain that $ {\rm Im} \, \tF_{{\bf j}, \ell}^\bs(\lambda) \neq 0$, hence $\tF_{{\bf j}, \ell}^\bs(\lambda)$ does not vanish identically.
			\item if $i_2, k_2$ are both different from $i_1,k_1$, then once again we we put $\lambda_{i_1} = \lambda_{k_1}=1$, all the other $\lambda_i = 0$ and proceed as before.
		\end{itemize}
	\end{enumerate}
\end{proof}
\begin{remark}
	\label{rem:ff}
	For any  $ \jj_1, \jj_2$ with $|m_1| > |m_2|$,  $\forall \lambda \in \cO_1, \  \forall \e\leq \e_1$ 
	\begin{equation}
	| \wtOmega_{(m_1, n)}(\lambda, \e) - \wtOmega_{(m_2,n)}(\lambda, \e)| \geq  
	\begin{cases}
	|m_1| \ ,  &   {\rm if } \  |m_1| \geq \max_{1 \leq k \leq \tk} |\tm_k| \\
	\e \tilde\g  \ , &   {\rm if } \  |m_1| \leq \max_{1 \leq k \leq \tk} |\tm_k| 
	\end{cases}
	\end{equation}
	(see formulas \eqref{tutti}, \eqref{tutti2}) while
	\begin{equation}
	| \wtOmega_{(m_1, n)}(\lambda, \e) + \wtOmega_{(m_2,-n)}(\lambda, \e)| \geq  
	\begin{cases}
	m_1^2 +  n^2 &   {\rm if } \  |m_1| \geq \max_{1 \leq k \leq \tk} |\tm_k| \\
	n^2 &  {\rm if } \ |m_1| \leq \max_{1 \leq k \leq \tk} |\tm_k| 
	\end{cases}
	\end{equation}
\end{remark}

We conclude the section with the following lemma; first let us fix $\tM_0 = \tM_0(\cO_1, \e_0)$ as 
\begin{equation}
\label{M0}
\begin{aligned}
	\tM_0 & := \sum_{1 \leq i \leq \tk} | \mu_i(\cdot) |^{\cO_1}_\C + \sum_{1 \leq i < k \leq \tk \atop \pm} |  \mu_{ik}^\pm(\cdot) |^{\cO_1}_\C + 
	\\
	& \qquad \qquad 
	+ \sup_{\e \leq \e_0} \frac{1}{\e^2}\Big( 
	\sup_{m \in \Z} |\varpi_m(\cdot, \e)|^{\cO_1}+ 2\lceil \wtcH^{(\rm 0,hor)} \rfloor_{s,-2}^{\cO_1} +
	2 \lceil \wtcH^{(\rm 0,mix)} \rfloor_{s,-\bd}^{\cO_1}\Big) \ . 
	\end{aligned}
\end{equation}
We have the following
\begin{lemma}
	\label{lem:cono}
	There exists  $\gamma_1>0$ (independent of $\e$) and a compact domain   $\cC_{\rm c}\subset\cO_1$  s.t.  
	\begin{equation}
	\label{C.2}
	\min_{|\ell| \leq 4\tM_0 , \ \bj, \bs } \ \ \inf_{ \lambda \in \cC_{\rm c}\cap \cO_1} \abs{\tF_{\bj, \ell}^\bs(\lambda)} \geq \gamma_1  > 0 \ .
	\end{equation}
\end{lemma}
\begin{proof}
	For each $\bj, \ell, \bs$ the function $\tF_{\bj, \ell}^\bs(\lambda)$ is a $\Z$-linear combination of algebraic functions, thus it is algebraic.
	%, see Lemma \ref{lem:sum.al}.  
	Furthermore  $\tF_{\bj, \ell}^\bs(\lambda)$ is homogeneous of degree 1 (see Lemma  \ref{lem:irr}), 
	its zeroes set is a union of a finite numbers of  algebraic surfaces  of codimension 1. By homogeneity, such surfaces are in fact cones.
	 Note that outside these surfaces, the eigenvalues $\mu_{i}(\lambda)$ and $\mu_{i,k}^\pm(\lambda)$ are distinct . 
	 As in the construction of $\cO_1$,  each
 equation $\tF_{\bj, \ell}^\bs(\lambda) = 0$ divides $\cO_0$ and $\cO_1$  in a finite number of distinct open cones in which $\tF_{\bj, \ell}^\bs(\lambda) \neq 0$. 

	Since we consider just a finite number of distinct, non zero, analytic  functions $\tF_{\bj, \ell}^\bs(\lambda)$,  any compact domain $\cC_c$ strictly  contained in $\cO_1$ satisfies an estimate of type \eqref{C.2} with some $\wt\gamma_0>0$. 
	We just need to fix $\g_1\le \wt\gamma_0,\wt\gamma$ (see formula \eqref{tutti}).
\end{proof}

\section{KAM algorithm for the reducibility}

In the previous section we have reduced the original  Hamiltonian to the normal form
\begin{equation}
	\label{ham.bnf2}
	\wtcH := \omega \cdot \yy + \sum_{\jj \in \Z^2 \setminus \cS_0} \widetilde\Omega_\jj\, |a_\jj|^2 + \wtcH^{(0)}+ \wtcH^\1 + \wtcH^{(\geq 2)} \ ,
	\end{equation}
as described in Theorem \ref{thm:q}. The aim of this section is to put the  Hamiltonian $\wtcH^\0$ to diagonal form with a KAM algorithm. 
We prove the following result:
\begin{theorem}
	\label{thm:kam}
	 Assume that $\Tan$ is  $\tL$-generic (in the sense of Definition \ref{defar}) with $\tL$ fixed in Lemma \ref{lem.F}.
	There exist $0<\e_\star \ll \e_0$ and $\gamma_0  >0 $ s.t that the following holds true for all $0<\e\le \e_\star$. 
	   There  exist  functions $(\Omega_\jj(\lambda, \e))_{\jj\in \Z^2\setminus \cS_0}$ defined and Lipschitz in $\lambda$ on the set $\cO_1$ of Theorem  \ref{thm:q},  such that: 
\item[(i)] The functions $\Omega_{\jj}(\lambda, \e)$ satisfy \eqref{as.omega0}, \eqref{as.omega}, \eqref{theta.est} with $\tM_0$ as in \eqref{M0}. For $\g_0/2\le \g\le \g_0$ %\red{$\g_0$ non e' definito}
%\blue{with $\g_0$ defined in Lemma \ref{lem:step4} (oppure mettiamo esiste all'inizio)}
 and {$\tau$ sufficiently large},  the set 
	\begin{equation}
	\label{set.C}
	\cC:=\left\{\lambda\in \cO_1:\;	\abs{\omega \cdot \ell + \s_1\Omega_{\jj_1}(\lambda, \e)+ \s_2 \Omega_{\jj_2}(\lambda, \e)+\s_3 \Omega_{\jj_3}(\lambda, \e) } \geq \e \frac{\gamma}{\la \ell \ra^\tau} \ ,\;\forall (\bj , \ell, \bs)\in \fA_3\setminus\fR_3\right\}
		\end{equation}
has positive measure. 

\item[(ii)]  For each $\lambda \in \cC$ there exists a symplectic change of variables $\whcT$ well defined and majorant analytic $D(s/16, \varrho r/2) \to D(s/8,\varrho r)$ for all $0<r \leq r_0$, $s_0/64 \le s \leq s_0$ 
($s_0, r_0$ and $\varrho$ are defined in Theorem \eqref{thm:q}) and such that
\begin{equation}
\label{ham.bnf31}
( \wtcH\circ \whcT)(\yy, \theta, a, \bar a) =:\cK= \omega \cdot \yy + \sum_{\jj \in \Z^2 \setminus \cS_0} \Omega_\jj\, |a_\jj|^2 + \cK^{\1} + \cK^{(\geq 2)} \ ,
\end{equation}
where   $\cK^{\1}$ contains just monomials of scaling degree  1, while $\cK^{(\geq 2)}$ of scaling degree at least 2 and 
\begin{equation}
\label{R.est00}
\abs{\cK^{( 1)}}_{s/16,\varrho r/2}^{\cC} \lessdot  \sqrt{\e} \, r  \ , \qquad
 \abs{\cK^{(\geq 2)}}_{s/16,\varrho r/2}^{\cC} \lessdot r^2 \ ,
 \qquad \forall 0 < r \leq r_0 \ . 
\end{equation}
\item[(iii)] Finally one has $\wtcM \circ \whcT = \wtcM$ and $\wtcP \circ \whcT = \wtcP$.
\end{theorem}

In the course of the proof, for notational convenience we shall rename $s/8\rightsquigarrow s$, $\varrho r\rightsquigarrow r$.
 We first describe the {\em KAM step}, namely a standard change of variables which we shall apply recursively in order to diagonalize the Hamiltonian \eqref{ham.bnf}.

\medskip
{\bf KAM step.} Fix an $0< \e_1\ll \e_0$ and let  $\g_1$ be defined in Lemma \ref{C.2} . For all $\e\le \e_1$ consider a Hamiltonian of the form
\begin{equation}
\label{basic.ham}
\cH(\lambda,\e):= \omega\cdot\yy +\cD + \sQ\,,\qquad \sQ:= \sum_{\ell, \, |\al|+|\bt|=2} Q_{\al,\bt, \ell}(\lambda, \e) \, e^{\im \ell\cdot\theta}\, a^\al \bar a^\bt
\end{equation}
where $\cD(\lambda, \e) := {\rm diag}(\cD_{\jj}(\lambda, \e))_{\jj \in \Z^2\setminus\Tan}$. Assume that the functions  
$$\cD_\jj(\lambda, \e) = \wtOmega_\jj(\lambda, \e) + \cD_\jj^{\rm hor}(\lambda, \e) + \cD_\jj^{\rm mix}(\lambda, \e)$$
are defined and Lipschitz for $\lambda\in \cO_1$ and fulfill the estimates
\begin{equation}
\label{L00} 
\sum_{1 \leq i \leq \tk}
 | \mu_i(\cdot) |^{\cO_1}_\C
  + \sum_{1 \leq i < k \leq \tk \atop \pm} |  \mu_{ik}^\pm(\cdot) |^{\cO_1}_\C
  +
 \sup_{\e \leq \e_1  } \frac{1}{\e^2}\Big( \lceil \cD^{\rm mix}\rfloor_{s, -\bd}^{\cO_1} +  \lceil \cD^{\rm hor}\rfloor_{s,-2}^{\cO_1} \Big) \leq  {\tM}\ ,
\end{equation}
with $\e_1\tM\ll1$.
 Note that for a diagonal matrix
\begin{equation}
\label{cheneso}
\lceil \cD^{\rm hor}\rfloor_{s,-2}^{\cO_1} = \sup_{m \in \Z} \la m\ra^2 |\cD^\hor_m(\cdot, \e)|^{\cO_1}_\C \ , 
\end{equation}
same for the $\mix$ part.

Assume moreover that $\sQ$ is a quadratic  Hamiltonian, defined for $\lambda$ in a compact set $\cC\subseteq\cO_1$,  which commutes with $\wtcM$ and $\wtcP$, and admits the decomposition
$$
\sQ = \sQ^{\rm hor} + \sQ^{\rm mix} \ , \quad \sQ^{\rm hor}  \in \cQ^{\cC, {\rm hor}}_{s,-2} \ , \ \ \  \sQ^{{\rm mix}} \in \cQ_{s, -\bd} ^{\cC} \ , \qquad \mbox{ for some } s >0 \ . 
$$
Fix $\tau>\tk+1$, $0<\g<\g_1$, $\tN \gg \max(1,\tM)$
and assume that for $\e \leq \e_1$
\begin{equation}\label{piccola0}
\tN^{3\tau +2 } \, (\g^2\e)^{-1} \Big( \lceil \sQ^{\rm mix}\rfloor_{s, -\bd}^{\cC} +  \lceil \sQ^{\rm hor}\rfloor_{s,-2}^{\cC}  \Big) \leq  1 \ . 
\end{equation}
For $\e\le \e_1$ define  $\cC^{\new}\equiv \cC^{\new}(\gamma,\tau,\e, \cD,\tN) \subset \cC$ as the set of $\lambda \in \cC$ fulfilling
\begin{align}
\label{MC1}
& \abs{\omega \cdot \ell +\cD_\ii(\lambda, \e) +\sigma \cD_\jj(\lambda, \e)} \geq \e  \frac{\g}{2} \tN^{-\tau}\ , \qquad \forall ((\ii,\jj),\ell,\bs)\in \fA_2\setminus\fR_2\,,\;  |\ell|\leq \tN \ .
\end{align}
Note that at this point the set $\cC^\new$ might be empty.

In the next lemma we will write $a \lessdot b$ meaning  $a \leq C b$ 
with a constant $C$ independent of $\gamma, \e, \tN$.

\begin{proposition}[Homological equation]
	\label{prop:hom}
The following holds true:
\begin{itemize}
\item[(i)] 
For $\lambda\in \cC^\new$, there exists $\chi = \chi^{\rm hor} + \chi^{\rm mix}$, commuting with $\wtcM$ and $\wtcP$, with $\chi^{\rm hor}   \in  \cQ_{s, -2}^{\cC^\new, {\rm hor}}$ and $\chi^{\rm mix}   \in  \cQ_{s, -\bd}^{\cC^\new, {\rm mix}}$
which satisfies the bounds
\begin{equation}\label{hom.eq.KAM2}
\lceil \chi^{\rm hor} \rfloor_{s, -2}^{\cC^\new} \lessdot \,   \tN^{2\tau +1} \, (\g^2\e)^{-1} \lceil \sQ^{\rm hor}\rfloor_{s, -2}^{\cC}   \ ,
 \qquad \lceil \chi^{\rm mix} \rfloor_{s, -\bd}^{\cC^\new} \lessdot \,  \tN^{3\tau +1} \, (\g^2\e)^{-1} ( \lceil \sQ^{\rm mix}\rfloor_{s, -\bd}^{\cC} +  \lceil \sQ^{\rm hor}\rfloor_{s,-2}^{\cC})
\end{equation}
and such that
 \begin{equation}
\label{hom.eq.KAM}
\{ \omega \cdot \yy + \cD,  \chi\} = \Pi_{\le \tN}( [\sQ] - \sQ )\,,\qquad  [\sQ](\lambda, \e; \ba) := \sum_{\jj \in \Z^2} Q_{\be_\jj,\be_\jj, 0}(\lambda, \e) \,  |a_\jj|^2
\end{equation}
\item[(ii)]  For $\lambda\in \cC^\new$, the time 1-flow $\cT_\chi$ of  $\chi$ defines a symplectic analytic change of variables
%$D(s, r) \to D(s+\s,r+\s)$ 
 which transforms the  Hamiltonian $\omega\cdot\yy +\cD + \sQ$ into the new  Hamiltonian 
\begin{equation}
\left(\omega\cdot\yy +\cD + \sQ \right)\circ \cT_\chi = \omega\cdot\yy +\cD^\new + \sQ^\new
\end{equation}
with
$$
\cD^\new:= \sum_{\jj \in \Z^2}\cD_\jj^\new(\lambda, \epsilon) |a_\jj|^2\,,
\quad 
\sQ^\new = \sQ_\new^{\rm hor} + \sQ_\new^{\rm mix} \ , 
\quad 
\sQ_\new^{\rm hor}  \in \cQ^{\cC^\new, {\rm hor}}_{s',-2} \ , \ \sQ_\new^\mix \in \cQ_{s', -\bd} ^{\cC^\new}\ .
$$
for all $0<s'<s$.
The new frequencies  $\cD^\new_\jj(\lambda, \e)$ are defined for all $\lambda\in \cO_1$ and are given by
\begin{equation}
\label{dig.str+}
\begin{aligned}
& \cD^\new_\jj(\lambda, \e) = \wtOmega_\jj(\lambda, \e) + \cD_\jj^{\rm hor,\tN}(\lambda, \e) + \cD_\jj^{\rm mix,\tN}(\lambda, \e) \ , \\
%& \qquad \cD_{m}^{\rm hor,\tN}(\lambda, \e) := \frac{\Theta^\new_m(\lambda, \e)}{\la m\ra^2 } \ , 
%\quad  \Omega_{(m,n)}^{\rm mix,\tN}(\lambda, \e) : = \frac{\Theta^\new_{m,n}(\lambda, \e)}{\la m \ra^2 +\la n\ra^2}
\end{aligned}
\end{equation}
and satisfy the estimates % \red{le $\sQ$ sono definite su $\cC$. faccio l'estensione a tutto $\cO_1$?} \blue{NO voglio lasciare $\cC$}
\begin{equation}\label{kaiser}
\lceil \cD^{\mix,\tN} - \cD^\mix \rfloor_{s, -\bd}^{\cO_1} +  \lceil{\cD^{\hor,\tN}- \cD^\hor}\rfloor_{s,-2}^{\cO_1} \lessdot  \lceil \sQ^{\rm mix}\rfloor_{s, -\bd}^{\cC} +  \lceil{\sQ^{\rm hor}}\rfloor_{s,-2}^{\cC}
%\sup_{m \in \Z } \abs{\Theta^\new_m(\cdot, \e)-\Theta_m(\cdot, \e)}^{\cO_1}_\C + \sup_{(m, n) \in \Z^2 } \abs{\Theta_{m,n}(\cdot, \e)-\Theta^\new_{m,n}(\cdot, \e)}^{ \cO_1}_\C 
%\leq 
%( \lceil \sQ^{\rm mix}\rfloor_{s, -\bd}^{\cC} +  \abs{\sQ^{\rm hor}}_{s,-2}^{\cC})  \ . 
\end{equation}
Finally we  have 
\begin{equation}\label{hbh}
\begin{aligned}
( \lceil \sQ_\new^\mix\rfloor_{s', -\bd}^{\cC^\new} +  \abs{\sQ_\new^\hor}_{s',-2}^{\cC^\new}) 
\lessdot
&  e^{-(s-s') \tN}( \lceil \sQ^{{\rm mix}}\rfloor_{s, -\bd}^{\cC} 
 +
   \lceil \sQ^{{\rm hor}}\rfloor_{s,-2}^{\cC}) \\
& \quad  +
 \tN^{3\tau +1} \,(\g^2\e)^{-1} ( \lceil \sQ^{\rm mix}\rfloor_{s, -\bd}^{\cC} + 
  \lceil \sQ^{\rm hor}\rfloor_{s,-2}^{\cC})^2
  \end{aligned}
\end{equation}
\end{itemize}
\end{proposition}

In order to prove  Proposition \ref{prop:hom} we need a preliminary result regarding the asymptotics of the frequencies $\cD_\jj$.
 For $\jj \in \Z^2\setminus \Tan$ we define  $\wt\Omega_\jj^*(\lambda, \epsilon)$ as 
%\red{bisognerebbe spiegare che e' un metodo ispirato alle quasi-topliz}
 \begin{equation}
 \label{hor.fre}
 \wt\Omega^*_{\jj}(\lambda, \epsilon) := 
 \begin{cases}
m^2 + n^2 , & \jj=(m,n) \not\in \sS  \ , \\
% m^2 \ , & \jj=(m,0) \ , m \not\in \Tan \ , \\
\e \mu_i(\lambda) + n^2   \ , & \jj=(\tm_i, n)\ , n\neq 0  \\
\end{cases}
 \end{equation}
 and the frequencies $\cD^*_\jj$ as
 $$
 \cD^*_\jj :=  \wt\Omega^*_{\jj} + \cD_\jj^\hor \ .
 $$
\begin{lemma}
For any $\lambda \in \cC^\new$ and $\forall \e\leq \e_1$,  we have 
\begin{equation}
\label{MC1bis}
\begin{aligned}
%\abs{\omega \cdot \ell +\wtOmega_\ii(\lambda, \e) -  \wtOmega_\jj(\lambda, \e)  +  \frac{\Theta_{m_1}(\lambda, \e)}{\la m_1\ra^2}- \frac{\Theta_{m_2}(\lambda, \e)}{\la m_2\ra^2 } } \geq \e  \frac{\g}{4} \tN^{-\tau_0}\ , \quad \forall (\ii,\jj,\ell,1,-1)\in \fR_2\,,\;  |\ell|\leq \tN \ , \\
\abs{\omega \cdot \ell +\cD_\ii^*(\lambda, \e) -  \cD_\jj^*(\lambda, \e)   } \geq \e  \frac{\g}{4} \tN^{-\tau}\ , \quad \forall ((\ii,\jj),\ell,(1,-1))\in \fA_2\setminus\fR_2\,,\;  |\ell|\leq \tN \ . 
\end{aligned}
\end{equation}
\end{lemma}
\begin{proof}
We  deduce \eqref{MC1bis} from the Melnikov conditions \eqref{MC1}. To do so, recall that for 
$|\jj| \geq C \displaystyle{\max_{1 \leq k \leq \tk} |\tm_k|}$, one has $\wtOmega_\jj^* = \wtOmega_\jj$. Moreover it is easily seen that $\cD_\ii^* - \cD_\jj^*$ is independent of $n$ for all compatible $((\ii, \jj), \ell, (1,-1) ) \in \fA_2$.  Thus defining $\ii_1 = \ii + (0, K) \ , \ \jj_1 = \jj + (0, K) $ with a sufficiently large $K$, we have
\begin{align*}
\cD_\ii^* - \cD_\jj^* & = \cD_{\ii_1}^* - \cD_{\jj_1}^*   = 
 \wt\Omega^*_{\ii_1}  - \wtOmega_{\jj_1}^* + \cD_{\ii_1}^\hor - \cD_{\jj_1}^\hor  \\
 &= 
 \wt\Omega_{\ii_1}  - \wtOmega_{\jj_1} + \cD_{\ii_1}^\hor - \cD_{\jj_1}^\hor  = 
 \cD_{\ii_1} - \cD_{\jj_1} - \cD_{\ii_1}^\mix + \cD_{\jj_1}^\mix
\end{align*}
If we take $K \geq \tN^{\tau/2}$ we bound  (recall  that $\e \tM\ll 1$)
$$
\abs{\cD_{\ii_1}^\mix} + \abs{\cD_{\jj_1}^\mix}
\stackrel{\eqref{cheneso}}{\leq }
\lceil \cD^\mix \rfloor_{s, -\bd} \left( \frac{1}{|\ii_1|^2} + \frac{1}{|\jj_1|^2} \right)
 \stackrel{\eqref{L00}}{ \leq} \e \frac{\gamma}{4} \tN^{-\tau} \ . 
$$
Then \eqref{MC1bis} follows easily from \eqref{MC1}. 
\end{proof}

\begin{proof}[Proof of Proposition \ref{prop:hom}]
$(i)$ As usual we represent the quadratic Hamiltonians as
$$
F=\sum_{\ell,|\al|+|\bt|= 2} F_{\al,\bt,\ell}\,  e^{\im \ell\cdot\theta} \, a^\al \bar a^\bt \ .
$$
 In Taylor-Fourier components, the equation \eqref{hom.eq.KAM} reads
\begin{equation}
\label{hom.eq.fou}
\im \left(\omega \cdot \ell + \cD(\lambda, \e)\cdot(\alpha - \beta)\right)  \chi_{\alpha, \beta, \ell} =  Q_{\alpha,\beta,\ell} \,,\quad ({\alpha,\beta,\ell})\neq (\al,\al,0)\,,\quad |\ell|\le \tN  \ .\end{equation}
The solution of this equation is given by 
\begin{equation}
\label{hom.sol3}
\chi_{\alpha, \beta, \ell} = \frac{1 }{\im (\omega \cdot \ell + \cD(\lambda, \e) \cdot (\alpha - \beta))} Q_{\alpha, \beta, \ell} \ ,\quad \al\ne \bt\,,\quad   |\ell|\le \tN
\end{equation}
 provided $\lambda \in \cC^\new$, thanks to the estimate \eqref{MC1}. Indeed we have that for all admissible choices of $\ell,\al,\bt$ ({recall $\tN\gg \tM$}): 
\begin{equation}\label{lip1}
\abs{\frac{1 }{\im (\omega \cdot \ell + \cD(\lambda, \e) \cdot (\alpha - \beta))}}^{\cC^\new}_\C \le  \tN^{2\tau+1}(\e\g^2)^{-1}.
\end{equation}
We show now that such $\chi$ can be written as $\chi = \chi^\hor + \chi^\mix$, and we  bound the two components. 
To do so, consider first the out-diagonal part $\chi^\out$. In this case we set $\ii=(m_1, n)$, $\jj=(m_2, - n)$, so that 
\eqref{hom.sol3} reads 
\begin{equation}
\label{chi.out00}
\chi^+_{m_1, m_2, n} = \frac{1 }{\im \left(\omega \cdot \ell + \cD_\ii(\lambda, \e) + \cD_{\jj}(\lambda, \e) \right)} Q_{m_1, m_2, n}^+ \ ,\quad \al\ne \bt\,,\quad  |\ell|\le \tN \ .
\end{equation}
Such term is clearly  mixing, and  we define $\chi^{\mix,0}$ the  Hamiltonian with coefficients \eqref{chi.out00}. To estimate it we need first to control the small divisor.
By Remark \ref{rem:ff} 
\begin{equation*}
\min(|\ii|^2,|\jj|^2)=\min(m_1^2,m_2^2) + n^2 \geq 4\,  \tN \,\max_k |\tm_k|^2 \Longrightarrow \abs{\omega \cdot \ell + \cD_\ii + \cD_\jj} \geq
(m_1^2 + n^2)/2   \ , 
\end{equation*}
thus for $\lambda \in \cC^\new$ there exists $\tC=\tC(\tk,\max_k |\tm_k|)$ such that 
%\red{$C(\tk)$ dipense anche dal $\max |\tm_i|$}
\begin{align*}
\abs{\frac{\la m_1\ra^2+  \la n\ra^2 }{\omega \cdot \ell + \cD_\ii + \cD_\jj}} \leq 2+ \max_{ \substack{\min(|\ii|^2,|\jj|^2) \leq 4 \tN \,\max_k |\tm_k|^2\\ ((\ii,\jj),\ell,(1,1))\in {\fA_2}\\ |\ell|\le \tN}} \abs{\frac{\la m_1\ra^2+  \la n\ra^2 }{\omega \cdot \ell + \cD_\ii + \cD_\jj}} \leq \tC  \frac{\tN^{\tau+1}}{\gamma \e} \ .
\end{align*}
Its  Lipschitz norm is then
\begin{equation}
\label{chi.l.1}
\abs{\frac{\la m_1\ra^2+  \la n\ra^2 }{\omega \cdot \ell + \cD_\ii + \cD_\jj}}^{\cC^\new}_\C \leq \tC \frac{\tN^{2\tau +2}}{\e \, \gamma^2} \ , 
\end{equation}
which implies that 
\begin{equation}
\label{chi.mix.out}
\lceil \chi^{\mix,0} \rfloor_{s,-\bd}^{\cC^\new} \lessdot \, \tN^{2\tau+2}\, (\e \gamma^2)^{-1}\abs{\sQ}_{s}^{\cC}  \ .
\end{equation}
We pass to $\chi^\diag$. In this case we set $\ii=(m_1, n)$, $\jj=(m_2,  n)$, so \eqref{hom.sol3} reads
\begin{equation}
\label{chi.diag00}
\chi^-_{m_1, m_2, n} = \frac{1 }{\im \left(\omega \cdot \ell + \cD_\ii(\lambda, \e) - \cD_{\jj}(\lambda, \e) \right)} Q_{m_1, m_2, n}^- \ ,\quad \al\ne \bt\,,\quad |\ell|\le \tN \ .
\end{equation}
To analyze it we divide the small divisor  and $Q^\diag$ in their horizontal and mixing parts. 
 Recalling $\cD_\jj = \wt\Omega_\jj + \cD_\jj^\hor + \cD_\jj^\mix$ and $\cD_\jj^* = \wtOmega_\jj^* + \cD_\jj^\hor$,  write
 \begin{align}
 \notag
 \frac{1 }{\im (\omega \cdot \ell + \cD_\ii- \cD_{\jj})} 
 &= 
\frac{1 }{\im(\omega \cdot \ell + \cD^*_\ii - \cD^*_{\jj} )}  
\ + \  
\frac{\wt\Omega_\ii^* - \wt\Omega_\ii + \wt\Omega_\jj - \wt\Omega_\jj^* + \cD_\ii^\mix - \cD_\jj^\mix }{\im \left(\omega \cdot \ell + \cD_\ii- \cD_{\jj}\right) \, \left(\omega \cdot \ell + \cD_\ii^*- \cD_{\jj}^*\right) \, } \\
\label{A+B}
& =: A_{m_1, m_2,\ell}(\lambda, \e) + B_{\ii, \jj, \ell}(\lambda, \e) \ .
 \end{align}
We  prove that   $ A_{m_1, m_2,\ell}(\lambda, \e)$ is independent of $n$, while  $B_{\ii, \jj, \ell}(\lambda, \e)$ is decreasing in $n$.\\
It is easy to see that  $\cD^*_\ii - \cD^*_\jj$ is independent of $n$ for all $((\ii, \jj), \ell, (1,-1)) \in \fA_2$ and    by \eqref{MC1bis}, as in \eqref{lip1}
 \begin{equation}\label{lip2}
\abs{\frac{1 }{\omega \cdot \ell + \cD^*_\ii - \cD^*_\jj }}^{\cC^\new}_\C \le \tC \frac{\tN^{2\tau+1}}{\e\g^2} \ .
\end{equation}
Consider now $B_{\ii, \jj, \ell}(\lambda, \e)$.  First we study its numerator. Note that  $\wt\Omega^*_{\jj}=\wt\Omega_{\jj}$ if $\jj \not\in \ccC$, so it is convenient to separate the case $|n| > \max_k |\tm_k|$ and $|n| \leq \max_k |\tm_k|$.  
One has  in the former case
\begin{equation}
\label{om.num1}
\abs{\wt\Omega_\ii^* - \wt\Omega_\ii + \wt\Omega_\jj - \wt\Omega_\jj^* + \cD_\ii^\mix - \cD_\jj^\mix}^{\cO_1}_\C \leq
 \frac{4\tM \e^2 }{\min(\la m_1 \ra^2, \la m_2\ra^2) + \la n \ra^2} \ , \qquad \mbox{ for }   |n| > \max_k |\tm_k| \ .  
\end{equation}
Then \eqref{lip1}, \eqref{lip2} and \eqref{om.num1} give
\begin{equation}
\label{B.lip}
\abs{B_{\ii,\jj,\ell}(\cdot, \e)}_\C^{\cO_1} \leq
 \tC \, \frac{\tN^{3\tau + 1}}{\gamma^2 \, \la n \ra^2}\ , \qquad \mbox{ for }   |n| > \max_k |\tm_k|  \ , 
\end{equation}
while using the definition $  B_{\ii, \jj, \ell}:=\left({\omega \cdot \ell + \cD_\ii- \cD_{\jj}}\right)^{-1} - A_{m_1, m_2,\ell}$ and the estimates \eqref{lip1}, \eqref{lip2} one has
\begin{equation}
\label{B.lip2}
\abs{B_{\ii,\jj,\ell}(\cdot, \e)}^{\cC^\new}_\C \leq |\left({\omega \cdot \ell + \cD_\ii- \cD_{\jj}}\right)^{-1}|+|A_{m_1, m_2,\ell}|\leq 
 \tC \, \frac{\tN^{2\tau + 1}}{\gamma^2 \, \e }\ , \qquad  \forall n  \ .
\end{equation}
We are ready to estimate $\chi^\diag$ defined in \eqref{chi.diag00}. Divide $\sQ^\diag$ into its horizontal and mixing components, $\sQ^\diag = \sQ^{\diag,\hor} + \sQ^{\diag, \mix}$ so that
\begin{align}
\chi^-_{m_1, m_2, n} &= \left(A_{m_1, m_2, \ell}  + B_{\ii,\jj,\ell} \right) \, \left( Q_{m_1, m_2, \ell}^{-, \hor} + Q_{\ii, \jj, \ell}^{-, \mix} \right) \\
& = 
A_{m_1, m_2, \ell} \, Q_{m_1, m_2, \ell}^{-, \hor} +
B_{\ii,\jj,\ell}  \, Q_{m_1, m_2, \ell}^{-, \hor} + 
 \frac{1 }{\im (\omega \cdot \ell + \cD_\ii- \cD_{\jj})} 
  Q_{\ii, \jj, \ell}^{-, \mix}  \ . 
\end{align}
We estimate the three terms separately. Note that the first one is horizontal, while the last two are mixing, so we define
$$
\chi^\hor_{m_1, m_2, \ell} :=A_{m_1, m_2, \ell}  \, Q_{m_1, m_2, \ell}^{-, \hor} \ .
$$
To estimate it,  use the assumption  $\sQ^\hor \in \cQ^{\cC,\hor}_{s,-2}$  and \eqref{lip2}; one obtains  immediately the first of \eqref{hom.eq.KAM2}. \\
Define now the  Hamiltonian  $\chi^{\mix,1}$  with coefficients
$$
\chi^{\mix,1}_{m_1, m_2, \ell}:= B_{\ii,\jj,\ell} \,  Q_{m_1, m_2, \ell}^{-, \hor}  \ ;
$$
using $\sQ^\hor \in \cQ^{\cC,\hor}_{s,-2}$  and \eqref{B.lip2} one has
\begin{equation}
\label{chi.mix1}
\lceil \chi^{\mix,1} \rfloor_{s, -2}^{\cC^\new} \lessdot
  %\tC
   \, \frac{\tN^{2\tau + 1}}{\gamma^2
  	 \, \e } \, \lceil \sQ^{\hor} \rfloor_{s, -2}^{\cC} \ , 
\end{equation}
while by \eqref{B.lip} one has
\begin{equation}
\label{chi.mix1b}
\lceil \chi^{\mix,1} \rfloor_{s, (0,-2)}^{\cC^\new} \lessdot
 % \tC
  \, \frac{\tN^{3\tau +1}}{\gamma^2
  	 \, \e } \, \lceil \sQ^{\hor} \rfloor_{s}^{\cC} \ . 
\end{equation}
Then \eqref{bagnetto} implies that
\begin{equation}
\label{chi.mix1c}
\lceil \chi^{\mix,1} \rfloor_{s, -\bd}^{\cC^\new} \lessdot
% \tC\,\max_k |\tm_k| \,
 \frac{\tN^{3\tau+1}}{\gamma^2 \,
 	 \e } \, \lceil \sQ^{\hor} \rfloor_{s, -2}^{\cC} \ .
\end{equation}
Finally define $\chi^{\mix,2}$ to be the  Hamiltonian with coefficients 
$$
\chi^{\mix,2}_{m_1, m_2, \ell}:=
\frac{1 }{\im (\omega \cdot \ell + \cD_\ii- \cD_{\jj})}   Q_{\ii, \jj, \ell}^{-, \mix} \ . 
$$
Using that $\sQ^\mix \in \cQ^{\cC,\mix}_{s, -\bd}$  and \eqref{lip1} one has
\begin{equation}
\label{chi.mix2}
\lceil \chi^{\mix,2} \rfloor_{s, -\bd}^{\cC^\new} \lessdot
% \tC \, 
\frac{\tN^{2\tau + 1}}{\gamma^2 \, 
	\e } \, \lceil \sQ^{\mix} \rfloor_{s, -\bd}^{\cC} \ .
\end{equation}
Define now
$\chi^\mix := \chi^{\mix,0} + \chi^{\mix,1} + \chi^{\mix,2}$.
Estimates \eqref{chi.mix.out}, \eqref{chi.mix1c} and \eqref{chi.mix2} imply the second of \eqref{hom.eq.KAM2}.
This concludes the proof of item $(i)$.\\

$(ii)$ Now by Lemma \ref{lem:smo} and estimate \eqref{piccola0}, the change of variables $\cT_\chi$ is well defined.  We have
\begin{align*}
\left(\omega\cdot\yy +\cD + \sQ \right)\circ \cT_\chi 
&=
  \omega\cdot\yy +\cD + \sQ +\{\omega\cdot\yy +\cD ,\chi\} + \re{2}(\chi; \omega\cdot\yy +\cD ) 
  + \re1(\chi; \sQ) 
\\
&=  \omega\cdot\yy +\cD + [\sQ] +\Pi_{>\tN}\sQ +\re{1}( \chi; \sQ) +  \sum_{j=1}^\infty \frac{\ad^j(\chi)}{(j+1)!}  ([\sQ]-\Pi_{\le \tN} \sQ )
\end{align*}

We recall that $ [\sQ](\lambda, \e; \ba)=\sum_{\jj} Q_{\be_\jj,\be_{\jj}, 0}(\lambda, \e) |a_\jj|^2 $ is defined for $\lambda\in \cC$ and by construction
$$
[\sQ] =  [\sQ^\hor]+  [\sQ^\mix]\,,\quad   Q_{\be_\jj,\be_{\jj}, 0}= Q_{m}^\hor + Q_{\jj}^\mix
% =  \frac{\widetilde\Theta_m(\lambda, \e)}{\la m\ra^2 } 
%+\frac{\widetilde\Theta_{m,n}(\lambda, \e)}{\la m \ra^2 +  \la n\ra^2}
$$ 
with
$$ \sup_{m \in \Z}\langle m^2\rangle  |Q_m^\hor|^{\cC} + \sup_{\jj \in \Z^2\setminus \cS_0} \langle \jj\rangle^2 |Q^\mix_\jj|^{\cC} \leq   \lceil \sQ^{\rm mix}\rfloor_{s, -\bd}^{\cC} +  \lceil{\sQ^{\rm hor}}\rfloor_{s,-2}^{\cC}\,.  $$

By  Kirszbraun theorem we may extend the $Q^\hor_m, Q_\jj^\mix $ to functions $\widetilde Q^\hor_m, \widetilde Q_\jj^\mix $ defined on the whole $\cO_1$ and with the same Lipschitz norm.  This in turn defines a diagonal Hamiltonian $[\widetilde{\sQ}]$. We set, for $\lambda\in \cO_1$,
$
\cD^\new:= \cD + [\widetilde\sQ] \,,
$
while for all $\lambda\in \cC^\new$:
$$
\sQ^\new:= \Pi_{>\tN}\sQ +\re{1}(\chi;  \sQ) +  \sum_{j=1}^\infty \frac{\ad^j(\chi)}{(j+1)!}  ([\sQ]-\Pi_{\le \tN} \sQ ).
$$
The bounds \eqref{kaiser} follow. The bounds \ref{hbh} follow by applying Lemma \ref{lem:smo2} and  Proposition \ref{riassunto} (ii).
\end{proof}

We now apply the KAM step recursively starting from
$$
\sH_0:= \omega \cdot \yy +  \cD_0(\lambda,\e) +  \sQ_{0}  := \omega \cdot \yy +  \wtcD(\lambda,\e) +  \wtcH^\0
 $$
 which is well defined for all $\lambda\in \cC_0:=\cO_1$.
 More precisely we construct a sequence of sets $\cC_\nu \subset \cC_0$ and on such sets we define a sequence of canonical transformations $\cT_\nu$ such that
 we may recursively  set  for $\nu\ge 0$
 $$
\sH_\nu \circ  \cT_{\nu+1} := \sH_{\nu+1}= \omega \cdot \yy +  \cD_{\nu+1}(\lambda,\e) +  \sQ_{\nu+1}
 $$
% \begin{equation}
% \label{ham.k}
% \cH_\nu := \omega\cdot \yy +  \wtcD_\nu+ \cZ^{\rm line} + \cQ_\nu
% \end{equation}
 with $\cD_{\nu+1}$ diagonal and defined on the whole $\cO_1$, 
 and  $\sQ_{\nu+1} \ll \sQ_\nu$ (in appropriate norms, see Proposition \ref{prop:iterative} $(iv)$ below).

Given $s,\kam_0>0$ we  fix for all $\nu\geq 1$ the following sequences of parameters
	\begin{align*}
s_{\nu}:= s_{\nu-1} - \frac{c_* s}{\nu^2} \ , \quad c_*^{-1} = 2 \sum_{\nu=1}^{\infty} \frac{1}{\nu^2}  \ , %\\
\quad  \kam_\nu := \kam_0^{(3/2)^\nu} \ ,\quad	\tN_\nu := (s_\nu - s_{\nu+1})^{-1} \log \kam_\nu^{-1} \ .
	\end{align*}
\begin{proposition}[Iterative scheme]
\label{prop:iterative}
There exist $\e_1, \kam_\star>0$, such that $\forall \e \leq \e_1$ and $\forall \g\le \g_1$ setting 
{for $s_0/16<s \leq s_0/8$}
\begin{equation}\label{miseria}
\kam_0:= \frac{\lceil \wtcH^{(0,\rm mix)}\rfloor_{s, -\bd}^{\cO_1} +  \lceil \wtcH^{(0,\rm hor)}\rfloor_{s,-2}^{\cO_1}}{\g^2\e} \le \kam_\star\ ,
\end{equation}
	one has that $\forall \nu \geq 0$ we may define a set $\cC_{\nu+1}$, a map $\cT_{\nu+1}$ and  Hamiltonians $\cD_{\nu+1}={\rm diag}(\cD_\jj^{(\nu+1)})$, $\sQ_{\nu+1}$ s.t. the following holds true:
	\begin{itemize}	
	\item[(i)] Setting $\cC_0=\cO_1$, we have recursively:
	\begin{align}
	\label{MC1nu}
	\cC_{\nu+1}:=\left\lbrace \omega\in \cC_\nu: \abs{\omega \cdot \ell +\cD^{(\nu)}_\ii(\lambda, \e) +\sigma \cD^{(\nu)}_\jj(\lambda, \e)} \geq \e  \frac{\g}{2} \tN_\nu^{-\tau} \ ,  \quad\forall ((\ii,\jj),\ell,\bs)\in \fA_2\setminus\fR_2\,,\;  |\ell|\leq \tN_\nu \right\rbrace \ .
	\end{align}
		\item[(ii)] $\cT_{\nu+1}$ is a canonical transformation defined for all $\lambda\in \cC_{\nu+1}$ and s.t. 
	$$
		\left(\omega\cdot \yy +  \cD_{\nu}+ \sQ_{\nu}\right)\circ \cT_{\nu+1} = \omega\cdot \yy +  \cD_{\nu+1}+  \sQ_{\nu+1}
	$$
	and %$\cT_{\nu+1} =  \uno +  \whcT_{\nu+1}^\hor + \whcT_{\nu+1}^\mix$
	$$
	\lceil \cT_{\nu+1}^\hor -\uno  \rfloor_{s_\nu, -2}^{\cC_{\nu+1}} + 
	 \lceil \cT_{\nu+1}^\mix   \rfloor_{s_\nu, -\bd}^{\cC_{\nu+1}}\leq \kam_\nu\tN_\nu^{3\tau+2}
	$$
	\item[(iii)] $\cD_{\nu+1}={\rm diag}(\cD_\jj^{(\nu+1)})$ is a diagonal  Hamiltonian
	$$\cD^{(\nu+1)}_\jj(\lambda, \e) = \wtOmega_\jj(\lambda, \e) + \cD_\jj^{(\nu+1,\rm hor)}(\lambda, \e) + \cD_\jj^{(\nu+1,\rm mix)}(\lambda, \e).$$
	The functions $\cD_\jj^{(\nu+1,\rm hor)},\cD_\jj^{(\nu+1,\rm mix)}$
%	\begin{equation}
%	\label{dig.strn}
%	\cD^{(\nu+1,\rm hor)}_\jj(\lambda, \e) := \frac{\Theta^{(\nu+1)}_m(\lambda, \e)}{\la m\ra^2 } \ , 
%	\quad  
%	\cD^{(\nu+1,\rm mix)}_\jj(\lambda, \e) : = \frac{\Theta^{(\nu+1)}_{m,n}(\lambda, \e)}{\la m \ra^2 +  \la n\ra^2}
%	\end{equation}
	are defined for $\lambda\in \cO_1$ and fulfill the estimates
%	\begin{equation}\label{dig.str+2}
%	 \sup_{m \in \Z} |\Theta^{(\nu+1)}_m(\cdot, \e)-\Theta^{(\nu)}_m(\cdot, \e)|^{\cO_1}_\C + \sup_{(m,n) \in \Z^2} |\Theta^{(\nu+1)}_{m,n}(\cdot, \e)-\Theta^{(\nu)}_{m,n}(\cdot, \e)|^{\cO_1}_\C \leq   \e \kam_\nu \ . 
%	\end{equation}
%	\red{ oppure non si definiscono qui le $\Theta^{(\nu)}$ ma si dice solo
		\begin{equation}\label{dig.str+2}
		\lceil \cD^{(\nu+1,\rm mix)} - \cD^{(\nu,\rm mix)} \rfloor_{s, -\bd}^{\cO_1} +  \lceil{\cD^{(\nu+1,\rm hor)}- \cD^{(\nu,\rm hor)}}\rfloor_{s,-2}^{\cO_1} \leq  \g^2\e \kam_\nu
		\end{equation}
	%}
	\item[(iv)]$\sQ_{\nu+1} = \sQ^{\rm hor}_{\nu+1} + \sQ^{\rm mix}_{\nu+1}$ and 
	$$
	\lceil \sQ^{\rm hor}_{\nu+1} \rfloor_{s_{\nu+1}, -2}^{\cC_{\nu+1}} 
	+
	\lceil \sQ^{\rm mix}_{\nu+1} \rfloor_{s_{\nu+1}, -\bd}^{\cC_{\nu+1}} \leq \g^2\e \kam_{\nu+1}
	$$
	\end{itemize}
\end{proposition}
\begin{proof}
%	Fix the followingn sequence of parameters

	{\em Step $\nu=0$:} We fix $\e_1$ sufficiently small so that 
	$\e_1 \tM_0 \ll 1$.  We apply the KAM step with    $\cD:= \cD_0$, $\sQ := \sQ_0$, $\cC = \cC_0:= \cO_1$, $\tN=\tN_0$. 
	The smallness condition \eqref{piccola0} follows from \eqref{miseria} provided that $\kam_\star$ is  appropriately small depending only on $\tau,s_0$. Then we apply the KAM step procedure, getting items (i)-(iii). As for item (iv) it follows from the bound \eqref{hbh} by setting $\sigma=\sigma_{0}= s -s_{1}\ge s_0/24$.

\noindent	{\em Step $\nu \rightsquigarrow\nu+1$:} we just have to show that we may apply the KAM step with 
$\cD:= \cD_\nu$, $\sQ := \sQ_\nu$, $\cC=\cC_\nu$, $\tN=\tN_\nu$. The smallness condition \eqref{piccola0} follows by noting that
$$\tN_\nu^{3\tau+2} \kam_{\nu}< \tN_0^{3\tau+2} \kam_{0}<1.$$ 
We may estimate $\tM=\tM^{(\nu)}$ in \eqref{L00} by using \eqref{dig.str+2},  obtaining   $\tM^{(\nu)}\le \tM^{(0)}+\e\sum_{i \le \nu-1}\eta_i  \leq 2\tM^\0 \leq \tM_0$. Hence again one has $\epsilon_1 \tM_0 \ll 1$.
Then we note that $\cC^\new \equiv \cC_{\nu+1}$, we apply the KAM step procedure and define $\cD_{\nu+1}:= \cD^\new$, $\sQ_{\nu+1} := \sQ^\new$. We get items (i)-(iii) directly. As for item (iv) it follows from the bound \eqref{hbh} by setting $ \sigma=\sigma_{\nu}= s_\nu -s_{\nu+1}$, provided $\kam_\star$ is  appropriately small depending only on $\tau,s_0$.
\end{proof}

The iterative  KAM scheme above can be applied to diagonalize $\wtcH^\0$ provided that
\begin{equation}\label{eg}
\frac{\lceil \wtcH^{(0,\rm mix)}\rfloor_{s, -\bd}^{\cO_1} +  \lceil \wtcH^{(0,\rm hor)}\rfloor_{s,-2}^{\cO_1}}{\g^2\e}\sim \e\g^{-2} \le \e_1\g^{-2}\le \kam_\star\,,\quad \e_1\tM_0\le 1  \ . 
\end{equation}
More precisely the following holds:

\begin{corollary}\label{stoca}
Assume that   
\begin{equation}
\label{cond0}
\frac{\e_1 \, \tM_0}{\gamma^2} \ll 1 \ , 
\end{equation}	
 	with  $\tM_0$ as in \eqref{M0}.	For all $\lambda\in \cO_1$  and for any $\jj\in \Z^2\setminus \Tan$ we have that $\cD_{\jj}^{(\nu)} $ is a Cauchy sequence. We denote by $\Omega_{\jj}$ its limit. We have that 
	\[
	\Omega_\jj(\lambda, \e) =\begin{cases}
	 \wtOmega_\jj(\lambda, \e) + \Omega_\jj^{\rm hor}(\lambda, \e) + \Omega_\jj^{\rm mix}(\lambda, \e) \ , \quad \jj\neq (m,0)\\
 m^2+ \frac{\varpi_m(\lambda, \e)}{\la m\ra }\ , \quad \jj=(m,0)
	\end{cases}
	\]
%	$$\Omega_{(m,0)}(\lambda, \e) = m^2+ \frac{\varpi_m(\cdot, \e)}{\la m\ra} $$ while 
%	$$\Omega_\jj(\lambda, \e) = \wtOmega_\jj(\lambda, \e) + \Omega_\jj^{\rm hor}(\lambda, \e) + \Omega_\jj^{\rm mix}(\lambda, \e) \ , \quad \jj\neq (m,0)
%	$$
	where
	\begin{equation}
	\label{dig.str}
	\Omega_\jj^{\rm hor}(\lambda, \e) := \frac{\Theta_m(\lambda, \e)}{\la m\ra^2 } \ , 
	\quad  
	\Omega_\jj^{\rm mix}(\lambda, \e) : = \frac{\Theta_{m,n}(\lambda, \e)}{\la m \ra^2 +  \la n\ra^2}
	\end{equation}
	are defined for $\lambda\in \cO_1$ and  fulfill the estimates
	\begin{equation}	\label{theta.es2t}
\begin{aligned} 
		& \sum_{1 \leq i \leq \tk} | \mu_i(\cdot) |^{\cO_1} +  \sum_{1 \leq i < k \leq \tk \atop \pm} |  \mu_{ik}^\pm(\cdot) |^{\cO_1} \\
		& \qquad
		+ \sup_{\e \leq \e_1  } \frac{1}{\e^2}\Big( 
		{\sup_{m \in \Z} |\varpi_m(\cdot, \e)|^{\cO_1}+ }
		\sup_{m \in \Z} |\Theta_m(\cdot, \e)|^{\cO_1} + \sup_{(m,n) \in \Z^2 \setminus \Z} |\Theta_{m,n}(\cdot, \e)|^{\cO_1} \Big) \leq  \tM_0 
	\end{aligned}
\end{equation}

	For all $\lambda\in \cap_{\nu }\cC_{\nu}$ the sequence
	$$\hat\cT_{\nu}:= \cT_1\circ\cT_2\dots\circ \cT_\nu
	$$
	is a Cauchy sequence of changes of variables, converging to a $\widehat\cT$, which satisfies the bounds:
	$$
\lceil \whcT^\hor -\uno  \rfloor_{s_\nu, -2}^{\cC_{\nu+1}} + 
\lceil \whcT^\mix   \rfloor_{s_\nu, -\bd}^{\cC_{\nu+1}}\leq \e 
 \ , \qquad \forall \nu \geq 0 \ . 
$$
\end{corollary}
\begin{proof}
	This is a standard KAM convergence argument, see for instance \cite{Pos96}.
\end{proof}
It turns out that the set on which we can diagonalize can be described  in terms of the final frequencies. Indeed we have the following
\begin{lemma}
	Consider the set 
	\begin{equation}
	\cC^{\mathtt{ fin}}:= \left\{ \lambda \in \cO_1 \, : \,  \begin{array}{l}
	\abs{\omega \cdot \ell +\Omega_\ii(\lambda, \e) +\sigma \Omega_\jj(\lambda, \e)} \geq \e  \g \, \la \ell \ra^{-\tau}\ , \qquad \forall ((\ii,\jj),\ell, \bs )\in \fA_2\setminus\fR_2  
		\end{array}
 \right\}\ .  
	\end{equation}
	We have that $\cC^{\mathtt{ fin }}\subseteq \cap_\nu \cC_\nu$.
\end{lemma}

\begin{proof}
By definition $\cC^\fin \subseteq \cO_1$.  By induction, we assume it is contained in $\cap_{\nu \leq n} \cC_\nu = \cC_n$.
Recall that by definition 
$$
\cC_{n+1} = \left\{ \lambda \in \cC_n \, : \, 
 \abs{\omega \cdot \ell +\cD_\ii^{(n)}(\lambda, \e) +\sigma \cD_\jj^{(n)}(\lambda, \e)} \geq \e  \frac{\g}{2} \tN_n^{-\tau}\ , \qquad\forall ((\ii,\jj),\ell,\bs)\in \fA_2\setminus\fR_2\,,\;  |\ell|\leq \tN_n
 \right\} \ . 
$$
First recall that
\begin{align*}
\abs{\Omega_\ii(\lambda, \e)  - \cD^{(n)}_\ii(\lambda, \e)  }
\leq
\sum_{k \geq n}\abs{\cD^{(k+1)}_\ii(\lambda, \e)  - \cD^{(k)}_\ii(\lambda, \e)} \leq \e\g^2 \sum_{\nu\geq n} \kam_\nu \leq  2\e \g^2 \kam_n \leq \frac{\g}{4} \e \,  \tN_n^{-\tau} \ . 
\end{align*}
Thus  for $|\ell| \leq \tN_n$
\begin{align*}
\abs{\omega \cdot \ell +\cD^{(n)}_\ii(\lambda, \e) +\sigma \cD^{(n)}_\jj(\lambda, \e)} &
\geq
	\abs{\omega \cdot \ell +\Omega_\ii(\lambda, \e) +\sigma \Omega_\jj(\lambda, \e)} 
	-
	2 \sup_{\ii} \abs{\Omega_\ii(\lambda, \e)  - \cD^{(n)}_\ii(\lambda, \e)  } \\
	& \geq
\e  \g \la \ell \ra^{-\tau} 	-\e  \frac{\g}{2} \tN_n^{-\tau} \geq
\e  \frac{\g}{2}  \tN_n^{-\tau} \ . 
\end{align*}
\end{proof}
Finally we prove that the set of $\lambda$ in which our diagonalization scheme holds has positive measure:
\begin{proposition}
	\label{lem:meas.mel3}
	There exists $\gamma_0 >0$ s.t. for all $0<\g < \g_0$, $\e < \e_1$  fulfilling \eqref{cond0}, 
	the set $\cC$ defined in \eqref{set.C}  has measure
	$\meas(\cC)\ge \meas(\cC_c)/2$ (where $\cC_c$ is defined in Lemma \ref{lem:cono}).
\end{proposition}
The proposition, being quite technical, it is proved in Appendix \ref{app:mes.m}.

\begin{proof}[Proof of Theorem \ref{thm:kam}]
Fix $\g_0$ as in Proposition 
\ref{lem:meas.mel3} and $\e_\star$ s.t. $4 \e_\star \tM_0 \gamma_0^{-2} \ll 1$, so that \eqref{cond0} holds for $\e_1 \leq \e_\star$ and $\g_0/2 \leq \g \leq \g_0$. This ensures that \eqref{eg} holds, so we apply Corollary \ref{stoca}.
The change of variables $\widehat\cT$ and the  final frequencies $\Omega_\jj$ are the ones defined in Corollary \ref{stoca}, so all the desired bounds follow. The measure of the set $\cC$ is studied in Proposition \ref{lem:meas.mel3}.
\end{proof}

\section{Birkhoff normal form for the cubic terms}

After Theorem \ref{thm:kam} we have the  Hamiltonian 
\begin{equation}
\label{ham.bnf3}
\cK := \omega \cdot \yy + \sum_{\jj \in \Z^2 \setminus \cS_0} \Omega_\jj(\lambda, \e) \, |a_\jj|^2 + \cK^{\1} + \cK^{(\geq 2)} \ .
\end{equation}
%Clearly $\wtcH^\1$ Poisson commutes with mass and momentum, so its monomials are admissible in the sense of Definition \ref{rem:adm3}. In particular for such monomials one can impose the third order Melnikov conditions.  We formalize this with the following
%\begin{definition}
%\label{rem:adm3}
%	Given $\bj = (\jj_1, \jj_2, \jj_3) \in (\Z^2\setminus \Z\times\{0\})^3$, $\ell \in \Z^\tk$ and $\sigma= (\sigma_1, \sigma_2, \sigma_3) \in \{-1, 1\}^3$, we say that $(\bj, \ell, \sigma)$ is {\em admissible} iff the monomial $e^{\im \theta \cdot \ell} \, a_{j_1}^{\sigma_1} \, a_{j_2}^{\sigma_2} \, a_{j_3}^{\sigma_3}$ preserves the mass and the momentum, i.e. the selection rules of Remark \ref{leggi_sel1} hold. We write $(\bj, \ell, \sigma) \in \fR_3$.
%\end{definition}
We now perform a Birkhoff change of variables which cancels out $\cK^\1 $. In order to define such a change of variables we use third order Melnikov conditions, which hold true in the set $\cC$ of Proposition \ref{lem:meas.mel3}.
The main result of this section is the following theorem:
\begin{theorem}
\label{thm:3b}
For all $\e \leq \e_\star$, there exists $r_0\ll\sqrt{\e}$ such that, for all  $\lambda\in \cC$ there exists a symplectic change of variables $\cT^\1$  well defined and majorant analytic: 
	{$D(s/32,\varrho r /4 )\to D(s/16,\varrho r/2)$} for all $r\le  r_0 $, $s_0/64\le s\le s_0$
	and  such that
	$$
	\cK \circ \cT^\1 :=\omega \cdot \yy +  \sum_{\jj\in \Z^2\setminus \Tan}\Omega_\jj(\lambda, \e) |a_\jj|^2  +\cR^{( \geq 2)} \ , 
	$$
	where 
\begin{itemize}
\item[(i)] the map $\cT^\1$ is the time-1 flow of a cubic Hamiltonian $\chi^\1$ such that 
$|{\chi}^\1|_{\frac{s}{24}, \frac{\varrho r}{2}}^\cC \lessdot \frac{r }{\sqrt{\e}}$.
\item[(ii)]  $\cR^{( \geq 2)} $ contains just monomials of scaling degree at least 2 and 
	$$
	\abs{\cR^{(\geq 2)}}_{\frac{s}{32}, \frac{\varrho r}{4}} \lessdot r^2 \ . 
	$$
	\item[(iii)] One has  $\wt \cM \circ \cT^\1 = \wt \cM$ and 
$\wt \cP \circ \cT^\1 = \wt \cP$.
\end{itemize}
 \end{theorem}
\begin{proof}
As usual 	we look for $\cT^\1$ as the time one flow of an Hamiltonian $\chi^\1$.  
With $\cN := \omega \cdot \yy +  \sum_{\jj\in \Z^2\setminus \Tan}\Omega_\jj(\lambda, \e) |a_\jj|^2$, 	we have 
\begin{align}
\label{blu.10}
\cK\circ \cT^\1    = & \  \cN +  \{ \cN , \chi^\1 \} + \cK^\1  \\
\label{blu.20}
& + \re2(\chi^\1; \, \cN ) + \re1(\chi^\1; \,  \cK^\1 )  \\
\label{blu.30}
& + \cK^{(\geq 2)} \circ \cT^\1
\end{align}	
	We choose $\chi^\1$  to solve the homological equation $  \{ \cN , \chi^\1 \} + \cK^\1 = 0$. Thus 
	$$
\mbox{setting}\;\; \cK^\1=	\sum_{(\bj,\ell,\bs)\in\; \fA_3\atop \,\sigma_i=\pm 1} K^{\bs}_{\ell,\bj}(\lambda, \e)\,   e^{\im \theta\cdot \ell}a^{\sigma_1}_{\jj_1}a^{\sigma_2}_{\jj_2}a^{\sigma_3}_{\jj_3}
\,,
\quad \Rightarrow\;
\chi^\1=	\sum_{(\bj,\ell,\bs)\in\; \fA_3\atop \,\sigma_i=\pm 1} \chi^{\bs}_{\ell,\bj}(\lambda, \e)\,  e^{\im \theta\cdot \ell}a^{\sigma_1}_{\jj_1}a^{\sigma_2}_{\jj_2}a^{\sigma_3}_{\jj_3}
	$$
	with 
	$$
	\chi^{\bs}_{\ell,\bj}(\lambda, \e) :=
	 \frac{-\im K^{\bs}_{\ell,\bj}(\lambda, \e)}{\omega \cdot \ell + \s_1\Omega_{\jj_1}(\lambda, \e)+ \s_2 \Omega_{\jj_2}(\lambda, \e)+ \s_3 \Omega_{\jj_3}(\lambda, \e)} \ .
	$$
Since $\lambda \in \cC$, the third order Melnikov conditions of Theorem \ref{thm:kam} hold, thus we have %\red{check the size of domains, si stringe $s$}
	$$
|{\chi}^\1|_{\frac{s}{24}, \frac{\varrho r}{2}}^\cC \leq C(s,\tau)\e^{-1} |\cK^\1|_{\frac{s}{16}, \frac{\varrho r}{2}} \lessdot \frac{r}{\sqrt{\e}} \ . 
	$$
	Thus choosing 
	\begin{equation}
	\label{ro}
	0<r \leq r_0 \leq c \,  \sqrt\e
	\end{equation}
	with $c$ sufficiently small, Proposition \ref{prop:mer} $(ii)$ guarantees that $\chi^\1$ generates a flow. 	
	We come to  the terms of line \eqref{blu.20}. First we use the homological equation  $  \{ \cN , \chi^\1 \} + \cK^\1 = 0$ to get that
	\begin{align*}
\re2\left( \chi^\1; \cN \right) & 
 = \sum_{k \geq 2} \frac{{\rm ad}(\chi^\1)^{k-1}[ \{ \chi^\1, \cN \}]}{k!} = \sum_{k \geq 1} \frac{{\rm ad}(\chi^\1)^{k}[\cK^\1 ]}{(k+1)!}  \ ,
\end{align*}
and Proposition  \ref{prop:mer}$(iii)$ implies that $\re2(\chi^\1; \, \cN )$, $\re1(\chi^\1; \,  \wt\cH^\1 )$ and $\cK^\2 \circ \cT^\1$ have scaling degree at least 2 and fulfill the quantitative estimates
	$$
	\abs{\re2(\chi^\1; \, \cN )}_{\frac{s}{32}, \frac{\varrho r}{4}}^\cC \lessdot r^2 \ , 
	\qquad 
	\abs{\re1(\chi^\1; \,  \wt\cH^\1 )}_{\frac{s}{32}, \frac{\varrho r}{4}}^\cC  \lessdot r^2 \ , 
	\qquad
	\abs{\cK^\2 \circ \cT^\1}_{\frac{s}{32}, \frac{\varrho r}{4}}^\cC  \lessdot r^2 \ . 
	$$
	To conclude the proof we show that $\{ \wt\cM, \chi^\1\} = \{ \wt \cP, \chi^\1\} = 0$. 
	This follows since $\cK^\1$ commutes with $\wt\cM$ and $\wt \cP$, hence its monomials fulfill the selection rules of Remark \ref{leggi_sel1}. 
	By the explicit formula for $\chi^\1$ it follows that the same selection rules hold for $\chi^\1$.
This conclude the proof.
	\end{proof}

\begin{proof}[Proof of Theorems \ref{main2} and \ref{main3}] 
Given $s_0 >0$,   we fix $\tL$, $\e_\star$ as in Theorem \eqref{thm:kam}, and  $r_0$ to fulfill \eqref{parapendio} and \eqref{ro}. 
Then Theorem \ref{main2}  follows by applying the sequence of
 transformations $\cT^\1$, $\widehat \cT$, $\cT^\0$ defined  in Theorems \ref{thm:3b}, \ref{thm:kam}, \ref{thm:q} and setting $
\varrho_0 = \varrho/4$.

 The asymptotics of the frequencies claimed in Theorem \ref{main3} is proved in Theorem \ref{thm:kam}, more precisely in Corollary \ref{stoca}.
\end{proof}

\section{Proof of Theorem \ref{thm:main}}
In this last section we show how  Theorem \ref{main2} implies Theorem \ref{thm:main}.
We start by fixing  $s_0 >0$, $p >1$,  a  $\tL$-generic support set $\cS_0$,    $0< \e <\e_\star$ ($\e_\star$ is given by 
Theorem \ref{main2}) and   $\lambda \in \cC$ of \eqref{3.mc}.
This fixes the torus $\tT^\tk(\cS_0, \lambda) \equiv \tT^\tk(\cS_0, I_\tm(\lambda))$. Note that the actions $I_\tm$ are fixed by \eqref{boia}.
%Then clearly we have that $\norm{q}_{L^2} \leq 2\tk \sqrt{\e}$, and the first condition in \eqref{condom} is trivially satisfied simply rescaling $\e$. 

First we  study the dynamics of the Hamiltonian \eqref{H.2} showing that $\yy = 0$, $a = 0$ is orbitally stable. 
More precisely we prove that there exist $K_0, T_0 >0$, independent of $r, \e$  s.t. 
\begin{equation}
\label{os.1}
\left(\yy(0), \theta(0), a(0)\right) \in D(\frac{ s_0}{64\cdot 32}, \varrho_0 K_0 r) 
\ \ \Longrightarrow  \ \ 
\left(\yy(t), \theta(t), a(t)\right) \in D(s_0, r)  \ \ \ 
\forall |t| \leq T_0/r^2  \ .
\end{equation} 
To  prove  \ref{os.1} we apply the change of coordinates $\cT$ of Theorem \ref{main2}. Recall that
both $\cT$ and its inverse $\cT^{-1}$ map $D(s/32, \varrho_0 r) \to D(s,r)$ for any $0 <r \leq r_0$ and {$\frac{s_0}{64} \leq s \leq s_0$.}  Denoting
$(\yy, \theta, \ba) = \cT(\yy', \theta, \ba')$, the Hamiltonian in the variables $(\yy', \theta, \ba')$ is given by \eqref{ham.bnf.fin}, and its  equations of motion are
\begin{equation}
\label{eq.yythetaa}
\begin{cases}
\dot \yy' = - \partial_\theta \cR^{(\geq 2)}(\yy', \theta, a', \bar a' ) \\
\dot \theta = \omega  +  \partial_\yy \cR^{(\geq 2)}(\yy', \theta, a', \bar a' )\\
\im \dot a'_\jj  = \Omega_\jj \, a'_\jj +  \partial_{\bar a'} \cR^{(\geq 2)}(\yy', \theta, a', \bar a' )
\end{cases}
\end{equation}
Then we prove the following bootstrap lemma: 
\begin{lemma}
\label{lem:b}
Let  $K_0 < \varrho_0$ and consider the system \eqref{eq.yythetaa} with initial datum $(\yy'(0), \theta(0), \ba'(0)) \in D( s_0/64 , K_0 r)$. 
Assume that exists $T_0 >0$ s.t. the quantities $\cJ(t):= \norm{a'(t)}^2 + \abs{\yy'(t)}_1 $ and $\Theta(t) := {\rm Im } |\theta(t)| $ fulfill
\begin{equation}
\label{b0}
\sup_{|t| \leq T_0 r^{-2}} \cJ(t) \leq   \varrho_0^2 r^2 \  , 
\qquad
\sup_{|t| \leq T_0 r^{-2}} \Theta(t) \leq   \frac{s_0}{32}  \ .
\end{equation}
 Then, provided $ K_0, T_0$ are small enough (independently from $r$)  one has
\begin{equation}
\label{b1}
\sup_{|t| \leq T_0 r^{-2}} \cJ(t) \leq \frac{ \varrho_0^2 r^2}{2}  \  , 
\qquad
\sup_{|t| \leq T_0 r^{-2}} \Theta(t) \leq  \frac{ s_0 }{40} \ .
\end{equation}
\end{lemma}
\begin{proof}
In the course of the proof we drop the superscript $'$ from the variables.
Consider first the dynamics of $\yy$. 
By the very definition of  $\abs{\cdot}_{s,r}$ 
(see \eqref{musically} and \eqref{listen}) 
 we have  
 $$
\sup_{(\yy, \theta, \ba) \in D(s_0/32, \varrho_0r)} \abs{ \partial_\theta \cR^{(\geq 2)}(\yy, \theta, \ba)}_1 
\leq 
\varrho_0^2 r^2 \abs{\cR}_{\frac{s_0}{32},\varrho_0 r}  \stackrel{\eqref{R.est0}}{\leq } C_\cR\, \varrho_0^2\,   r^4 \ .
$$ 
Thus for any $|t| \leq T_0 r^{-2}$  one estimates
\begin{align}
\notag
\abs{\yy(t)}_1 & \leq  \abs{\yy(0)}_1 + |t| \sup_{|t| \leq T_0/r^2 } \abs{ \partial_\theta \cR^{(\geq 2)}(\yy(t), \theta(t), a(t), \bar a(t) )}_1 \\
& 
\label{yy.est.b}
\leq 
K_0^2 r^2  +T_0 C_\cR \,  \varrho_0^2 r^2  \leq \frac{\varrho_0^2 \, r^2}{4}
\end{align}
provided $K_0^2  \leq \varrho_0^2/8$ and $T_0 < 1 / 8C_\cR $.
Similarly, using the equations of motion for $a$ and the estimate
$$
\sup_{(\yy, \theta, \ba) \in D(s_0/32, \varrho_0 r)}
\abs{ \partial_{\bar a} \cR^{(\geq 2)}(\yy, \theta, \ba)}_1 
\leq 
\abs{\cR}_{\frac{s_0}{32}, \varrho_0 r} \, \varrho_0 \, r 
\stackrel{\eqref{R.est}}{\leq } C_\cR\,  \varrho_0\,  r^3 \ , 
$$
 for any  $|t| \leq T_0 r^{-2}$
 one gets
\begin{align}
\label{a.est.b}
\norm{a(t)} & \leq  \norm{a(0)} + |t| \sup_{|t| \leq T_0/r^2 } \abs{ \partial_{\bar a } \cR^{(\geq 2)}(\yy(t), \theta(t), a(t), \bar a(t) )} 
\leq 
K_0 r + T_0 C_\cR\, \varrho_0\, r  \leq \frac{\varrho_0 r}{2} 
\end{align}
provided $K_0 \leq  \varrho_0/4$ and $T_0$ as above. \\
Finally  the equation for $\theta$ and the estimate
$\displaystyle{\sup_{(\yy, \theta, a) \in D(s_0/32, \varrho_0 r)}
\abs{ \partial_{\yy} \cR^{(\geq 2)}(\yy, \theta, a)}_\infty 
\leq  C_\cR \, r^2}$ 
give
\begin{equation}
\label{t.est.b}
\abs{{\rm Im }\, \theta(t)}
\leq 
\abs{{\rm Im }\, \theta(0)} 
+ T_0 C_\cR \leq \frac{s_0}{40}
\end{equation}
provided  $T_0< s_0/C_\cR 10^3$ as above.
Estimates \eqref{yy.est.b},  \eqref{a.est.b} 
and \eqref{t.est.b}
imply \eqref{b1}.
\end{proof}
This bootstrap lemma and the properties of $\cT$ and $\cT^{-1}$ imply that, taking  $ K_0 \leq \min\left( \frac{\varrho_0}{4}, \frac{\varrho_0^2}{8}\right)$ and $T_0 \leq \min\left( \frac{s_0}{C_\cR 10^3}, \frac{1}{8 C_\cR} \right)$, one has 
\begin{align*}
(\yy(0), \theta(0),  a(0)) \in D(\frac{s_0}{64 \cdot 32}, \varrho_0 K_0 r) & \Rightarrow 
(\yy'(0), \theta(0),  a'(0)) \in D(\frac{ s_0}{64}, K_0 r) \\
&\Rightarrow 
(\yy'(t), \theta(t),  a'(t)) \in D(s_0/32,  \varrho_0 r) \Rightarrow
(\yy(t), \theta(t),  a(t)) \in D(s_0, r) 
\end{align*}
which is valid for all $ |t| \leq T_0/r^2$; this  proves \eqref{os.1}.
 
Finally we show how the stability result \eqref{os.1} implies Theorem \ref{thm:main}. 
Thus fix $\delta_0>0$ so that
\begin{equation}
\label{delta0}
c_\star^{-1} \delta_0 < K_0 \varrho_0 r_0
\end{equation}
where $c_\star$ is defined in Proposition \ref{prop:adap}.
Take  $0<\delta < \delta_0$ and $u \in \cV_\delta$. Then by Proposition \ref{prop:adap}(ii), 
$u(0) \in \cV_{\delta}$ implies that
 $(\yy(0), \theta(0), \ba(0)) \in D(\frac{s_0}{64\cdot 32}, c_\star^{-1}  \delta )$.
 Choose now  $ r= c_\star^{-1} (K_0 \varrho_0)^{-1} \delta$ (note that  \eqref{delta0} guarantees that $r \leq r_0$ is fulfilled) and apply  \eqref{os.1} to get that  $(\yy(t), \theta(t),  \ba(t)) \in D(s_0, r)$ 
for $|t| \leq T_0/r^2$. Applying  again  Proposition \ref{prop:adap} one gets $u(t) \in \cV_{c_\ast(s_0) r} \equiv \cV_{C \delta}$.

\appendix

\section{Proof of Lemma \ref{rem:inf.s0}}
\label{generic.infinity}

First we consider \eqref{gen.cond}.
Let us consider $\Tan$ as a point  $\vec{\tm}=(\tm_1,\dots,\tm_\tk)\in \C^\tk$.  With this identification $B_R$   is the set of integer vectors in the cube $[-R,R]^\tk$ with ordered components.
 An intersection $\sS\cap\ccC$ is a point
$$
v_{i,j,l}(\vec\tm)=(x,y)\in \C^2:  \quad x= \tm_l\,,\quad y =\sqrt{(\tm_l-\tm_i)(\tm_j-\tm_l) }
$$
for some triple of indexes $i,j,l \in \{1, \ldots, \tk\}$. So fix such  indexes and consider the map 
$$
\Gamma:\C^\tk\to \C^{\tk+2}, \qquad  \vec \tm \to (\vec{\tm}, v_{i,j,l}(\vec\tm)).
$$
Now  we restrict our attention to $\vec{\tm}\in \Z^\tk$ and we wish to avoid  to choose $\vec{\tm}$  so that $v_{i,j,l}(\vec\tm)\in \Z^2$.
Our first remark is that if $\vec{\tm}\in B_R$ then we have that $|v_{i,j,l}(\vec\tm)|\le 3R$.
In the square $\sup_{i}|\tm_i|\leq R$ there are $(2R)^\tk$ integer points and we claim that once we remove all points with $v_{i,j,l}(\vec\tm)\in \Z^2$, we still have many integer points left. So in $\C^{\tk+2}$ let us  count the integer valued points $(\vec{\tm},v)$ such that 
$\sup_{i}|\tm_i|\leq R$, $|v|\le 3R$ and $v= v_{i,j,l}(\vec\tm)$. We can fix all the $\tm_s$ with $s\neq l$ in any way we want (we get $(2R)^{\tk-1}$ points) then for each such choice, $\tm_l$ is the $x$-coordinate  of one of the integer points on the circle with diameter $(\tm_i,0),(\tm_j,0)$.
But on a circumference of radius $R$ the  number of  integer points are of the order $R^\delta$ with $\delta$ arbitrarily small, so the bad points  in $\C^{\tk+2}$ are of order $R^{\tk-1+\delta}$. Then the pre-image through $\Gamma$ of these points is composed at most of $R^{\tk-1+\delta}$ points. Since the total amount of points is of order $R^\tk$ we have lots of good points, i.e. points $\vec\tm\in \Z^\tk$ such that the intersection $v_{i,j,l}(\vec\tm)$ is not an integer. When we play this game for all the triples of indexes $(i,j,l)$ we still get many points.
Regarding the intersection $\ccC_{i,j} \cap \ccC_{l,s}$  the reasoning is the same.
An intersection is a point $v= v_{i,j,l,s}(\vec\tm)$ which solves the equation
$$
(x-\tm_i)(x-\tm_j)+ y^2 =0 \,,\quad (x-\tm_l)(x-\tm_s)+ y^2 =0  \ . 
$$
So we consider the map 
$$
\widetilde \Gamma:\C^\tk\to \C^{\tk+2}, \qquad \; \vec \tm \to (\vec{\tm}, v_{i,j,l,s}(\vec\tm)).
$$
As before we make the ansatz that $\sup_{i}|\tm_i|\leq R$. Then we have that $|v_{i,j,l,s}(\vec\tm)|\le 3R$.
Now we can fix all the points $\tm_q$ for $q\neq s$ ($R^{\tk-1}$ points) and  for each such choice, fix  $v$ as  one of the $R^\delta$ integer points on the circle with diameter $(\tm_i,0),(\tm_j,0)$. Then $\tm_s$ is fixed by the last equation (which is linear for $\tm_s$). In conclusion we have again $R^{\tk-1+\delta}$ bad points.  Here we have to play such game for all the quadruples $(i,j,l,s)$.

Finally we consider condition \eqref{pop}. Fix $\ell$ with $|\ell| \leq \tL$. Then the equation $\sum \ell_i x_i = 0$ defines
an hyperplane in $\C^{\tk}$, and in the hypercube of size $R$ there are $\sim R^{\tk -1}$integer  bad points to exclude. Taking the union over all the possible hyperplanes  gives $\sim \tL^\tk R^{\tk-1}$ bad points to exclude.  In conclusion the number of bad points in $B_R$ is $\le C(\tk)R^{\tk-1+\delta}$, so
$$
\lim_{R\to \infty}\frac{|G_R|}{|B_R|} \ge 1-  \lim_{R\to \infty}\frac{C(\tk)R^{\tk-1+\delta} }{|B_R|} =1.
$$

\section{Proof from Section \ref{fun}}
\label{app:fun}
We begin by defining projections of Hamiltonians. More precisely we define projections on {\em any subspace} defined as the closure in $\cA_{s,r}$ of the monomials $ e^{\im \ell\cdot \theta}\, \yy^l \,  a^\al \bar a^\bt$ satisfying  some rule $\ell,l,\al,\bt\in \mathbb I$ with $\mathbb I\subset \Z^\tk\times\N^\tk\times \N^{\Z^2\setminus\cS_0}\times \N^{\Z^2\setminus\cS_0}$:
\begin{equation}
\label{proj} 
\Pi_\mathbb I h :=  \sum_{\ell\in \Z^\tk ,l\in \N^\tk,\al,\bt\in\N^{\Z^2\setminus\cS_0}  \atop 
	\ell,l,\al,\bt\in \mathbb I} h_{\al,\bt, l, \ell} \ e^{\im \ell\cdot \theta}\, \yy^l \,   a^\al \bar a^\bt \ . 
\end{equation}
In particular we denote  by $\Pi_N$ the projection on trigonometric polynomials of degree $\leq N$, i.e.
\begin{equation}
\label{def:PN}
\Pi_N h :=  \sum_{\ell\in \Z^\tk, \, |\ell| \leq N  \atop 
	l,\al, \bt} h_{\al,\bt, l, \ell}\  e^{\im \ell\cdot \theta}\yy^l  a^\al \bar a^\bt\ . 
\end{equation} 

\begin{proposition}
	\label{riassunto} 
	For every $ s,r>0$ the following holds true:
	\begin{itemize}
		\item[(i)] {\em Continuity:} all projections  $\Pi_\mathbb I$ are continuous, namely
		$$
		|\Pi_\mathbb I h|_{s,r}^\cO\leq |h|_{s,r}^\cO \ . 
		$$
		\item[(ii)] {\em Smoothing}: for any $0< s' <s$ one has
		$$
		|\Pi_{ N} h|_{s,r}^\cO\leq e^{{s'}N} | h|_{s-s',r}^\cO\,,\qquad |(\uno - \Pi_{ N}) h|_{s,r}^\cO\leq e^{-s'N} | h|_{s+s',r}^\cO
		$$
		\item[(iii)]{\em Partial ordering:} if we have 
		$$
		|f_{\al,\bt,l,\ell}|^\cO_\C \leq |h_{\al,\bt,l,\ell}|^\cO_\C \,,\quad \forall \; \al,\bt,l,\ell
		$$
		and $\{h_{\al,\bt,l,\ell}\}$ are the coefficients of the  Taylor-Fourier expansion of a function $h\in \cA_{s,r}^\cO$,  then there exists a unique function $f \in  \cA_{s,r}^\cO $ whose Fourier expansion has coefficients 
		$\{f_{\al,\bt,l,\ell}\}$ and such that
		$$ |f|_{s,r}^\cO\leq |h|_{s,r}^\cO$$
		\item [(iv)]{\em Graded Poisson algebra}: 
		Given $f,g\in \cA_{s,r}^\cO$, for any  $0<s'<s$ and $0<r'<r$ one has
		$$
		|\{f,g\}|^\cO_{s', r'}\leq \delta^{-1} C(s) \, |f|^\cO_{s,r} \ |g|^\cO_{s,r} \ ,  
		$$
		where $\delta:= \min \left(1-\frac{r'}{r}, 1-\frac{s'}{s}\right)$. 
		Moreover if $f$ and $g$ are monomials of the form $e^{\im \ell\cdot \theta}\, \yy^l \,   a^\al \bar a^\bt $ one has that  ${\rm deg}(\{f,g\})= {\rm deg}(f)+{\rm deg}(g)$ and $\pi(\{f,g\})= \pi(f)+\pi(g)$.
	\end{itemize}
\end{proposition}
The proposition can be proved by adapting the methods of \cite{BBP}.

\subsection{Proof of Proposition \ref{boni}}
	In order to prove \eqref{HResto} we have to control 
	\begin{align}
	\notag
 &| \{(\la m_1\ra^{N_1+k}+ \la n \ra^{N_2}) H_{m_1, m_2, n}^{\pm, \ell} \}_{|\ell|> c \la m_1 \ra  }|^\cO_{s'}  \\ 
\label{posta}
 & \stackrel{\eqref{bagnetto}}{\le}
  | \{\la m_1\ra^{N_1+k} |H_{m_1, m_2, n}^{\pm, \ell}| \}_{|\ell|> c \la  m_1 \ra   }|^\cO_{s'}+  | \{ \la n \ra^{N_2} |H_{m_1, m_2, n}^{\pm, \ell}| \}|^\cO_{s'}	\ . 
	\end{align}
In the first summand since $|\ell|> c \la m_1\ra$ we have 
$$
| \{\la m_1\ra^{N_1+k} |H_{m_1, m_2, n}^{\pm, \ell}| \}_{|\ell|> c\la m_1\ra }|^\cO_{s'} \le 
c^{-k}| \{\la m_1\ra^{N_1} |\ell|^k |H_{m_1, m_2, n}^{\pm, \ell}| \}|^\cO_{s'} \ .
$$
Now we note that
$$
|\ell|^k e^{-(s-s')|\ell|} \le \frac{k!}{(s-s')^k} = k! \delta^{-k}
$$
so that
$$
| \{\la m_1\ra^{N_1+k} H_{m_1, m_2, n}^{\pm, \ell} \}_{|\ell|> c\la m_1\ra }|^\cO_{s'} \le 
k! \, c^{-k}\delta^{-k} | \{\la m_1\ra^{N_1}  |H_{m_1, m_2, n}^{\pm, \ell}| \}|^\cO_{s}\stackrel{\eqref{bagnetto}}{\le} k! \,  c^{-k}\delta^{-k} \lceil H\rfloor_{s,-(N_1,N_2)}^{\cO} \ .
$$
In the second summand  of \eqref{posta} we just have 
$$
 | \{ \la n \ra^{N_2} |H_{m_1, m_2, n}^{\pm, \ell}| \}|^\cO_{s'}\le 	\lceil H\rfloor_{s',-(N_1,N_2)}^\cO \le \lceil H\rfloor_{s,-(N_1,N_2)}^\cO \ . 
$$

\subsection{Proof of Lemma \ref{lem:pseudo00} }
{\bf {\em Proof of Lemma \ref{lem:pseudo00} (i).}} By \ref{bagnetto} we have
\begin{equation}\label{pianto}
	\lceil \{ F, G \} \rfloor_{s', -\bN-\bM}^\cO \leq \lceil \{ F, G \} \rfloor_{s', -(N_1+M_1,0)}^\cO + \lceil \{ F, G \} \rfloor_{s', -(0, N_2+M_2)}^\cO \ . 
\end{equation}
	We start by discussing the first term.
	Following Lemma \ref{boni} we divide $F= F^B+F^R$ and $G= G^B+G^R$, so that
	$$
	\{F,G\}= \{F^B,G^B\}+\{F,G^R\}+\{F^R,G^B\} \ . 
	$$ 
	We claim that
	\begin{align}\label{mai}
	& \lceil \{ F^B, G^B \} \rfloor_{s', -(N_1+M_1,0)}^\cO \leq 4^{\max(M_1,N_1)}\,  \, \lceil F^B \rfloor_{s',-\bN}^\cO \,  \lceil G^B \rfloor_{s',-\bM}^\cO  \\\label{fine}
	& \lceil \{ F, G^R \} \rfloor_{s', -(N_1+M_1,0)}^\cO+ \lceil \{ F^R, G^B \} \rfloor_{s', -(N_1+M_1,0)}^\cO  \leq C_{N_1,M_1}\delta^{-N_1-M_1}\,  \, \lceil F \rfloor_{s}^\cO \,  \lceil G \rfloor_{s}^\cO
	\end{align}
	Let us start by proving \eqref{mai}: we divide
	\begin{align*}
	\{ F^B, G^B \} & = \{F^B, G^B\}^{\rm diag} + \{F^B, G^B\}^{\rm out} + \{F^B, G^B\}^{\rm line} \\
	&= \{F^{\rm diag}_B, G^{\rm diag}_B\} +  \{F^{\rm out}_B, G^{\rm out}_B\}  + \{F^{\rm diag}_B, G^{\rm out}_B\}  + \{F^{\rm out}_B, G^{\rm diag}_B\} +  \{F^{\rm line}_B, G^{\rm line}_B\}
	\end{align*}
	where we put
	\begin{align*}
	&	F^{\rm diag}_B(\lambda;\theta, a, \bar a) = \sum_{ m_1,m_2, \ell,  n \neq 0 \atop |\ell|\le c\la m_1\ra  } F^{-, \ell }_{m_1,m_2, n }(\lambda) \ e^{\im \theta \cdot \ell} a_{(m_1, n)} \bar a_{(m_2, n)} \ , \ \ \\\
	&	 F^{\rm out}_B(\lambda;\theta, a, \bar a) = \sum_{ m_1,m_2, \ell , n > 0 \atop |\ell|\le c\la m_1\ra    } F^{+, \ell}_{m_1,m_2, n }(\lambda) \ e^{\im \theta \cdot \ell} a_{(m_1,n)}  a_{(m_2, -n)} + c.c. \ , 
	\end{align*}
	and similarly for $G$.	Then
	\begin{align*}
&	\{ F^{\rm diag}_B,  G^{\rm diag}_B\} =  \im \sum_{m_1, m_2, \ell , n \neq 0	 } A_{m_1, m_2, n}^\ell(\lambda) \ e^{\im \theta \cdot\ell}  a_{(m_1,n)} \bar a_{(m_2,n)} \ , \\
	& \qquad A_{m_1, m_2, n}^\ell(\lambda) := \left(  \sum_{m_3, \ell_1+\ell_2=\ell\atop |\ell_1|\le c\la m_1\ra , |\ell_2|\le c\la m_3\ra} F^{-, \ell_1}_{m_1, m_3, n }(\lambda) G^{-, \ell_2}_{m_3, m_2, n } (\lambda)-    
	G^{-, \ell_1}_{m_1, m_3, n }(\lambda) F^{-, \ell_2}_{m_3, m_2, n }(\lambda) \right) 
	\end{align*}
	and by momentum conservation $m_1-m_3+\pi(\ell_1)=0$, $m_3-m_2 + \pi(\ell_2)=0$, thus  $m_1-m_2 +\pi(\ell)=0$.\\
		By definition of $\lceil \cdot \rfloor_{s, - \bN}^\cO$, we need to compute 
\begin{align}\label{porcata}
\Big|\big\{	\la m_1 \ra^{N_1+M_1}  A_{m_1, m_2, n}^\ell(\lambda)\}\Big|_{s'}^\cO
\end{align}
We  note that
 	\begin{align}
	\label{for1A}
 \la m_1 \ra^{N_1+M_1}  \abs{A_{m_1, m_2, n}^\ell(\lambda) }
	 \leq & \sum_{{m_3 \atop \ell_1+\ell_2=\ell}\atop
		|\ell_1|\le c\la m_1\ra}  \la m_1 \ra^{N_1} \left(\frac{\la m_1\ra}{\la m_3\ra}\right)^{M_1}  \abs{F^{-, \ell_1}_{m_1, m_3, n }(\lambda)} \la m_3 \ra^{M_1} \abs{ G^{-, \ell_2}_{m_3, m_2, n }(\lambda)} \\
	\notag
	&+ \sum_{{m_3 \atop \ell_1+\ell_2=\ell}\atop
		|\ell_1|\le c\la m_1\ra}   \ \la m_1 \ra^{M_1} \abs{G^{-, \ell_1}_{m_1, m_3, n }(\lambda)} \left(\frac{\la m_1\ra}{\la m_3\ra}\right)^{N_1} \la m_3 \ra^{N_1} \abs{   F^{-, \ell_2}_{m_3, m_2, n }(\lambda)} 
%	& + \sum_{{m_3,  \atop \ell_1+\ell_2=\ell}}   \ \la m_1 \ra^{M_1} \abs{G^{-, \ell_1}_{m_1, m_3, n }(\lambda)} \ \la m_1-m_3 \ra^{N_1} \abs{   F^{-, \ell_2}_{m_3, m_2, n }(\lambda)} \\ \notag
	\end{align}
 Now we remark that  since $m_1-m_3 = -\pi(\ell_1)$, one has
	\begin{equation}
	\label{for2a} 
	 \la m_1\ra -|\pi(\ell_1)|\le \la m_3\ra\le  \la m_1\ra + |\pi(\ell_1)| \ . 
	\end{equation}
	 Moreover since $|\ell_1|\le c\la m_1\ra$ with $c^{-1}= 2 \max (|\tm_i|)$ this implies $2|\pi(\ell_1)|\le  \la m_1\ra$ so that 
		\begin{equation}
	\label{for3a} 
	\frac12 \le \frac{\la m_1\ra}{\la m_3\ra}\le  2 \ .
	\end{equation} 
%	\begin{equation}
%	\label{for3a}
%	2(\la \pi(\ell_1)\ra + 1)^{N} \, e^{-(s-s') |\ell_1|} \leq  \, C ( |\ell_1| + 1)^{N} \, e^{-(s-s') |\ell_1|} \leq  C \, \frac{1}{(s-s')^N} +1:= C\delta^N \ .
%	\end{equation}
%	We also have
%	$$
%	\la m_1\ra^{N_1},\la n\ra^{N_2}\le \la m_1\ra^{N_1}+\la n\ra^{N_2}\,,\quad \la m_1\ra^{M_1},\la n\ra^{M_2}\le \la m_1\ra^{M_1}+\la n\ra^{M_2}.
%	$$
	In conclusion we may bound
	\begin{align*}
	\eqref{porcata}  \le &   2^{M_1}\big|\{\sum_{{m_3 \atop \ell_1+\ell_2=\ell}\atop |\ell_1|\le c\la m_1\ra }  \la m_1 \ra^{N_1}   \abs{F^{-, \ell_1}_{m_1, m_3, n }(\lambda)} \la m_3 \ra^{M_1} \abs{ G^{-, \ell_2}_{m_3, m_2, n }(\lambda)}\}\big|_{s'} \\
		\notag
		&+ 2^{N_1}\big|\{\sum_{{m_3 \atop \ell_1+\ell_2=\ell}\atop
			|\ell_1|\le c\la m_1\ra}   \ \la m_1 \ra^{M_1} \abs{G^{-, \ell_1}_{m_1, m_3, n }(\lambda)}  \la m_3 \ra^{N_1} \abs{   F^{-, \ell_2}_{m_3, m_2, n }(\lambda)} \}\big|_{s'} \\
		\le & 	C_{M_1,N_1} \lceil F^B \rfloor_{s',-(N_1,0)}^\cO \,  \lceil G^B \rfloor_{s',-(M_1,0)}^\cO \le C_{M_1,N_1} \lceil F^B \rfloor_{s',-\bN}^\cO \,  \lceil G^B \rfloor_{s',-\bM}^\cO\,, 
	\end{align*}
	by using the algebra property \eqref{alge} and Remark \ref{bagnetto}.
%	\begin{align*}
%	&e^{|\ell|s'} \sum_{m_3\atop \ell_1+\ell_2=\ell} \left(\frac{\la m_1\ra}{\la m_3 \ra}\right)^{N_1}  \ \la m_1 \ra^{M_1} \abs{G^{-, \ell_1}_{m_1, m_3, n }(\lambda)} \ \la m_3 \ra^{N_1} \abs{   F^{-, \ell_2}_{m_3, m_2, n }(\lambda)} \\
%	& \le  C \sum_{m_3\atop \ell_1+\ell_2=\ell} e^{|\ell_1| s'} \ (\la m_1 \ra^{M_1}+\la n\ra^{M_2}) \abs{G^{-, \ell_1}_{m_1, m_3, n }(\lambda)} \ (\la m_3 \ra^{N_1}+\la n\ra^{N_2}) e^{|\ell_2| s'}\abs{   F^{-, \ell_2}_{m_3, m_2, n }(\lambda)}\\ & := C \sum_{m_3\atop \ell_1+\ell_2=\ell} A^{\ell_1}_{m_1,m_3}B^{\ell_2}_{m_3,m_2}
%	\end{align*}
%	now by definition 
%	\begin{align*}
%	\sup_{|\ba|_1\le 1} |\{\sum_{m_1\atop \ell_1+\ell_2=\ell} A^{\ell_1}_{m_1,m_3}B^{\ell_2}_{m_3,m_2} a_{m_1,n}\}|_1 \le  \sup_{|\ba|_1\le 1}|\{\sum_{m,\ell} B^{\ell}_{m,m'} a_{m,n}\} |_1 \sup_{|\ba|_1\le 1} |\{\sum_{m,\ell} A^{\ell}_{m,m'} a_{m,n} \}|_1
%	\end{align*}
%	so we get the first summand in \eqref{prop.smooth.bra}.
%	The same reasoning holds for the third line in \eqref{porcata}.
	The term $ \{F^{\rm out}_B, G^{\rm out}_B\}$ is treated exactly in the same way. 
	Now we  analyze $ \{F^{\rm diag}_B, G^{\rm out}_B\}$.
	An explicit computation shows that
	\begin{align*}
	\{ F^{\rm diag}_B, G^{\rm out}_B\} = &  - \im \sum_{m_1, m_2, \ell , n >0 }\left(  \sum_{m_3 , \ell_1 + \ell_2 = \ell\atop |\ell_1|\le c\la m_1\ra,\, |\ell_2|\le c \la m_2\ra} F^{-, \ell_1}_{m_1, m_3, n }  G^{+, \ell_2}_{m_3, m_2, n } +     
	G^{+, \ell_1}_{m_1 ,- m_3, n } F^{-, \ell_2}_{m_2, -m_3, -n } \right)  e^{\im \theta \cdot \ell}   a_{(m_1,n)}  a_{(m_2,-n)} 
	\end{align*}
	Let us consider the equivalent of \eqref{for1A}, namely
		\begin{align}
	\label{for1Ab}
	&\la m_1 \ra^{N_1+M_1}   \abs{\sum_{{m_3,\, \ell_1+\ell_2=\ell}\atop |\ell_1|\le c \la m_1\ra,\, |\ell_2|\le c \la m_2\ra } F^{-, \ell_1}_{m_1, m_3, n }(\lambda) G^{+, \ell_2}_{m_3, m_2, n }(\lambda) +   
		G^{+, \ell_1}_{m_1, -m_3, n }(\lambda) F^{-, \ell_2}_{m_2,- m_3, -n }(\lambda) }\\\notag
	 \leq & \sum_{{m_3,\, \ell_1+\ell_2=\ell}\atop |\ell_1|\le c \la m_1\ra,\, |\ell_2|\le c \la m_2\ra } \left(\frac{\la m_1\ra}{\la m_3 \ra}\right)^{M_1} \la m_1 \ra^{N_1}   \abs{F^{-, \ell_1}_{m_1, m_3, n }(\lambda)} \la m_3 \ra^{M_1} \abs{ G^{+, \ell_2}_{m_3, m_2, n }(\lambda)} \\
	\notag
	&+ \sum_{{m_3,\, \ell_1+\ell_2=\ell}\atop |\ell_1|\le c \la m_1\ra,\, |\ell_2|\le c \la m_2\ra } \left(\frac{\la m_1\ra}{\la m_2 \ra}\right)^{N_1}  \ \la m_1 \ra^{M_1} \abs{G^{+, \ell_1}_{m_1,- m_3, n }(\lambda)} \ \la m_2 \ra^{N_1} \abs{   F^{-, \ell_2}_{m_2, -m_3, -n }(\lambda)}   
	\end{align}
	The first  line is treated as in \eqref{for1A}. For the second we note that $\la m_2\ra -|\pi(\ell_2)|\le \la m_3\ra \le  \la m_2\ra +|\pi(\ell_2)| $ so that since
	$$
	\frac{\la m_1\ra}{\la m_2 \ra}= \frac{\la m_1\ra}{\la m_3 \ra} \frac{\la m_3\ra}{\la m_2 \ra}
	$$
	we have 
	$$
	\frac14 \le \frac{\la m_1\ra}{\la m_2 \ra}\le 4.
	$$
	The term  $ \{F^{\rm out}_B, G^{\rm diag}_B\}$ can be estimated in an analogous way,  we skip the details. This proves \eqref{mai}.
	
		Regarding \eqref{fine} let us suppose first that  in \eqref{for1A} we have $|\ell_1|> c \la m_1\ra$, then
		\begin{align}
	\label{for1C}
	&	 \la m_1 \ra^{N_1+M_1}  \abs{\sum_{{m_3 \atop \ell_1+\ell_2=\ell}\atop
			|\ell_1|>c\la m_1\ra} F^{-, \ell_1}_{m_1, m_3, n }(\lambda) G^{-, \ell_2}_{m_3, m_2, n }(\lambda) -    
		G^{-, \ell_1}_{m_1, m_3, n }(\lambda) F^{-, \ell_2}_{m_3, m_2, n }(\lambda) }\\\notag
	& \leq c^{-{N_1 -M_1} }\sum_{{m_3 \atop \ell_1+\ell_2=\ell}}   |\ell_1|^{M_1+N_1}  \left(\abs{F^{-, \ell_1}_{m_1, m_3, n }(\lambda)}  \abs{ G^{-, \ell_2}_{m_3, m_2, n }(\lambda)} 
	+ \abs{G^{-, \ell_1}_{m_1, m_3, n }(\lambda)}  \abs{  F^{-, \ell_2}_{m_3, m_2, n }(\lambda)} \right)
	%	& + \sum_{{m_3,  \atop \ell_1+\ell_2=\ell}}   \ \la m_1 \ra^{M_1} \abs{G^{-, \ell_1}_{m_1, m_3, n }(\lambda)} \ \la m_1-m_3 \ra^{N_1} \abs{   F^{-, \ell_2}_{m_3, m_2, n }(\lambda)} \\ \notag
	\end{align}	
	We have proved that
\begin{align}\label{porcatab}
          &\Big|\big\{	\la m_1 \ra^{N_1+M_1}  \sum_{m_3 , \ell_1+\ell_2=\ell\atop |\ell_1|> c \la m_1\ra} \big(F^{-, \ell_1}_{m_1, m_3, n }(\lambda) G^{-, \ell_2}_{m_3, m_2, n }(\lambda) -    
		G^{-, \ell_1}_{m_1, m_3, n }(\lambda) F^{-, \ell_2}_{m_3, m_2, n }(\lambda) \big)\big\}\Big|_{s'}^\cO \\
		 \notag 
		  & 
		\le \frac{1}{c^{N_1+M_1}} \Big|\big\{ |\ell_1|^{N_1+M_1} F^{-, \ell_1}_{m_1, m_3, n }(\lambda)\big\}\Big|_{s'}^\cO
		\Big|\big\{G^{-, \ell_2}_{m_3, m_2, n }(\lambda) \big\}\Big|_{s'}^\cO\\
		\notag
		&	\quad	 + \frac{1}{c^{N_1+M_1}}\Big|\big\{
			|\ell_1|^{N_1+M_1}	G^{-, \ell_1}_{m_1, m_3, n }(\lambda)\big\}\Big|_{s'}^\cO\Big|\big\{ F^{-, \ell_2}_{m_3, m_2, n }(\lambda) \big)\big\}\Big|_{s'}^\cO
\end{align}	
Now we remark that
	$$
	\sup_{\ell }|\ell|^{M_1+N_1} e^{-(s-s')|\ell|} \leq  C_{M_1, N_1} \,  \delta^{-M_1-N_1}
	$$	
	so that
	$$
	\Big|\big\{ |\ell_1|^{N_1+M_1} F^{-, \ell_1}_{m_1, m_3, n }(\lambda)\big\}\Big|_{s'}^\cO \le C \delta^{-M_1-N_1} \Big|\big\{ F^{-, \ell_1}_{m_1, m_3, n }(\lambda)\big\}\Big|_{s}^\cO\,,
	$$
	and \eqref{mai} follows.
	
	Consider finally the case $|\ell_1|\le c \la m_1\ra$, $|\ell_2|\ge c \la m_3\ra$.   Then \eqref{for3a} holds and
	\begin{align}
	\label{for1Cb}
	&	 \la m_1 \ra^{N_1+M_1}  \abs{\sum_{{m_3 \atop \ell_1+\ell_2=\ell}\atop
			|\ell_2|>c\la m_3\ra} F^{-, \ell_1}_{m_1, m_3, n }(\lambda) G^{-, \ell_2}_{m_3, m_2, n }(\lambda) -    
		G^{-, \ell_1}_{m_1, m_3, n }(\lambda) F^{-, \ell_2}_{m_3, m_2, n }(\lambda) }\\\notag
	& \leq 2^{N_1+M_1}c^{-N_1 - M_1}\sum_{{m_3 \atop \ell_1+\ell_2=\ell}}   |\ell_2|^{M_1+N_1}  \left(\abs{F^{-, \ell_1}_{m_1, m_3, n }(\lambda)}  \abs{ G^{-, \ell_2}_{m_3, m_2, n }(\lambda)} 
	+ \abs{G^{-, \ell_1}_{m_1, m_3, n }(\lambda)}  \abs{  F^{-, \ell_2}_{m_3, m_2, n }(\lambda)} \right)
	\end{align}	
	and we proceed as in the previous case.
	We are left with the second summand in \eqref{pianto}.
	We claim
	\begin{equation}\label{finestra}
 \lceil \{ F, G \} \rfloor_{s', -(0, N_2+M_2)}^\cO \le  2 \lceil F\rfloor_{s', -(0, N_2)}^\cO	 \lceil G\rfloor_{s', -(0, M_2)}^\cO	
	\end{equation} indeed 
	\begin{align}\label{porcataC}
&	\Big|\big\{	\la n \ra^{N_2+M_2}  \big(\sum_{m_3 , \ell_1+\ell_2=\ell\atop |\ell_i|\le c \la m_1\ra} F^{-, \ell_1}_{m_1, m_3, n }(\lambda) G^{-, \ell_2}_{m_3, m_2, n }(\lambda) -    
			G^{-, \ell_1}_{m_1, m_3, n }(\lambda) F^{-, \ell_2}_{m_3, m_2, n }(\lambda) \big)\big\}\Big|_{s'}^\cO
			 \\ 
			 \notag 
			  & \le  2 \Big|\big\{ \la n\ra^{N_2} F^{-, \ell_1}_{m_1, m_3, n }(\lambda)\big\}\Big|_{s'}^\cO\Big|\big\{ \la n\ra^{M_2} G^{-, \ell_2}_{m_3, m_2, n }(\lambda) \big\}\Big|_{s'}^\cO,
	\end{align}
and the result follows.	\\

\noindent{\bf {\em Proof of Lemma \ref{lem:pseudo00} (ii).}}		 The proof is standard, we repeat it here for completeness. We need to estimate
\begin{equation}
\label{p.est}	
		\re{1}(F; G) = \sum_{k \geq 1} \frac{{\rm ad}(F)^{k}(G) }{k!}  \ .
	\end{equation}
	By the algebra property \eqref{alge} we have 
	$$
	|{\rm ad}(F)^{k}(G)|^\cO_{s'}\le (C_0 |F|^\cO_{s'})^k |G|^\cO_{s'} \ . 
	$$
	Then by using \eqref{prop.smooth.bra} we have
	\begin{align}\notag
	\lceil{\rm ad}(F)^{k}(G)\rfloor^\cO_{s',-\bN} &= \lceil \{F,{\rm ad}(F)^{k-1}(G)\}\rfloor^\cO_{s',-\bN}  \leq C_{\bN} \, ( \, \lceil F \rfloor_{s',-\bN}^\cO \, |{\rm ad}(F)^{k-1}(G)|_{s'}^\cO+ \delta^{-N_1} \, | F |_{s}^\cO \, |{\rm ad}(F)^{k-1}(G)|_{s}^\cO) 
\\
\label{para}
	& \le C_\bN(C_0|F|_{s}^\cO)^{k-1}(\lceil F \rfloor_{s',-\bN}^\cO+ \delta^{-N_1} \, | F |_{s}^\cO)|G|_{s}^\cO
	\end{align}

\noindent{\bf {\em Proof of Lemma \ref{lem:pseudo00} (iii).}}	 	Let us prove inductively that
	\begin{equation}\label{suca}
		\lceil {\rm ad}(F)^{k}(G) \rfloor_{s', -k \bN}^\cO\le \hat C_{k, \bN}(\delta^{-N_1} \lceil F\rfloor_{s,-\bN}^\cO)^k k  |G|_{s}^\cO
	\end{equation}
	
		We use \eqref{prop.smooth.bra};  writing ${\rm ad}(F)^{k}(G) = \{  {\rm ad}(F)^{k-1}(G), F\}$ we get
		\begin{align} 
		\label{lem:pseudo0.00}
		\lceil {\rm ad}(F)^{k}(G) \rfloor_{s', -k \bN}^\cO &  \leq  
		 C_{k \bN} \,\left( \lceil {\rm ad}(F)^{k-1}(G) \rfloor_{s',-(k-1) \bN}^\cO |F|_{s,-\bN}^\cO+  \delta^{-k N_1}| {\rm ad}(F)^{k-1}(G)|_{s} |F|_s \right) \,  \\
		 \notag
		 &\le C_{k \bN} \,\left(  |F|_{s,-\bN}^\cO \hat C_{ k-1,\bN}  (\delta^{-N_1}\lceil F\rfloor_{s,-\bN}^\cO)^{k-1} (k-1)  |G|_{s}^\cO+  \delta^{-k N_1} C_0^k |G|_{s} |F|^k_s \right) \ . 
		% & \leq C_\bN \,  \delta^{(\ti-1)N_1} \, k \, \lceil {\rm ad}(F)^{k-1}(G) \rfloor_{s_{i},-\bN}^\cO \lceil F \rfloor_{s_{i-1}, -\bN}^\cO  \ .
		\end{align}
The result follows  provided 
$$
C_{k\bN} \left(\hat C_{k-1 ,\bN} (k-1) \delta^{-(k-1) N_1} + \delta^{-k N_1} C_0^k\right) \le \hat C_{k,\bN} \, k \, \delta^{-k N_1}
$$
which in turn follows by setting
$$
 C_{k\bN} \left(\hat C_{k-1 ,\bN}(k-1) +  C_0^k \right)\le  k  \,  \hat C_{k,\bN}.
$$
Now setting $s_1= s- (s-s')/2$ we have
$$
\lceil \re{\,\rm i}(F; G) \rfloor_{s',-\ti \bN}^\cO \le \lceil ({\rm ad}{F})^\ti G\circ \cT_F \rfloor_{s',-\ti \bN}^\cO \leq\ti (\hat C_{
	\ti, \bN} \lceil F\rfloor_{s_1,-\bN}^\cO)^\ti(2\delta)^{-\ti N_1} |G \circ \cT_F|_{s_1}^\cO
$$
and the result follows.

\subsection{Proof of Lemma \ref{lem:smo}}

First note that $\re{\ti}(F; G)$ is linear in the second argument. Thus
$$
\re{\ti}(F; G) = 
\re{\ti}(F; G^{\rm hor}) + \re{\ti}(F; G^\mix) = 
\re{\ti}(F^{\rm hor}; G^{\rm hor}) + \re{\ti}(F; G^\mix) + \big(\re{\ti}(F; G^{\rm hor}) - \re{\ti}(F^{\rm hor}; G^{\rm hor})\big) \ . 
$$
By Remark \ref{rem:hor}, $\re{\ti}(F^{\rm hor}; G^{\rm hor})$ is horizontal. We define $\re{\ti}(F; G)^\hor := \re{\ti}(F^\hor; G^\hor)$, which by Lemma \ref{lem:pseudo00}(iii) belongs to $\cQ^\hor_{s,-\ti a}$  and fulfills the claimed estimate.

 Now consider $\re{\ti}(F; G^\mix)$. We prove  that  $\re{\ti}(F; G^\mix) \in \cQ_{s',-\bd}^\cO$, $\forall \, 0<s' <s$. We proceed as in the proof of Lemma \ref{lem:pseudo00}(iii).  We prove inductively
 \begin{equation}\label{parti}
 \lceil {\rm ad}(F)^{k}(G^\mix) \rfloor_{s', -\bd}^\cO\le  (\hat C_\bd|F|^\cO_{s})^k \ \big( \lceil G^\mix\rfloor_{s, -\bd}^\cO+ k \delta^{-2} |G|_{s}^\cO \big) \ .
 \end{equation}
 Indeed:
	\begin{align*}
	\lceil {\rm ad}(F)^{k}(G^\mix) \rfloor_{s', -\bd}^\cO 
	&  \leq  \lceil \{F,{\rm ad}(F)^{k-1}(G^\mix)\} \rfloor_{s', -\bd}^\cO \\
	& 	\stackrel{\eqref{prop.smooth.bra}}{\le}
	 C_{\bd} \, ( \, \abs{F}_{s'}^\cO \, \lceil{\rm ad}(F)^{k-1}(G^\mix)\rfloor_{s',-\bd}^\cO+ \delta^{-2} \, | F |_{s}^\cO \, |{\rm ad}(F)^{k-1}(G)|_{s}^\cO) 
	\\ 
	& 
	\stackrel{\eqref{parti}}{\le} C_\bd  \, |F|_{s}^\cO \ \left( \max (C_0, \hat C_\bd) |F|_{s}^\cO\right)^{k-1} \, \left( \lceil G^\mix\rfloor_{s, -\bd}^\cO+ (k-1)\delta^{-2} |G|_{s}^\cO +\delta^{-2} |G|_{s}^\cO \right) 
	 	\end{align*}
	and the desired bound  follows by taking $\hat C_\bd=\eta^{-1}=\max(C_0,C_\bd)$, where $C_0$ is the constant in \eqref{alge}. Then substituting we get
	\begin{align}
	\label{mix1} 
	\lceil \re{\ti}(F; G^\mix) \rfloor_{s', -\bd}^\cO & 
	 \leq C_{\bd}\, \delta^{-2}\frac{\lceil G^\mix\rfloor_{s, -\bd}^\cO\left(\eta^{-1} \abs{F}_{s}^\cO\right)^\ti}{1 - \eta^{-1} \abs{F}_{s}^\cO} \ .
	\end{align}
Finally consider $\re{\ti}(F; G^{\rm hor}) - \re{\ti}(F^{\rm hor}; G^{\rm hor})$. Since the  series are summable
$$
\re{\ti}(F; G^{\rm hor}) - \re{\ti}(F^{\rm hor}; G^{\rm hor}) = \sum_{k \geq \ti} \frac{{\rm ad}(F)^{k}(G^{\rm hor}) - {\rm ad}(F^{\rm hor})^{k}(G^{\rm hor}) }{k!}  \ .
$$
Now remark that $F = F^{\rm hor} + F^\mix$, so ${\rm ad}( F^{\rm hor} + F^\mix)^{k}(G^\hor) - {\rm ad}(F^{\rm hor})^{k}(G^{\rm hor})$ equals
\begin{equation}
\label{aiutp01}
 \sum_{h=0}^{k-1} \ad(F)^h \, \ad(F^\mix)\, \ad(F^\hor)^{k-h-1}(G^\hor)
\end{equation}
%\end{document}
Now  by  \eqref{prop.smooth.bra} 
\begin{align}
\notag
\lceil \ad(F^\mix)\ad(F^\hor)^{k-h-1}(G^\hor)\rfloor_{s',-\bd} &\le C_\bd \Big(\lceil F^\mix\rfloor_{s',-\bd}|\ad(F^\hor)^{k-h-1}(G^\hor)|_{s'}+\delta^{-2}|F|_{s}|\ad(F^\hor)^{k-h-1}(G^\hor)|_s\Big) \\
\label{aiutooo}
 &\le
C_\bd (\lceil F^\mix\rfloor_{s',-\bd} +\delta^{-2}|F|_{s})(C_0 |F|_{s})^{k-h-1}|G|_s  \ . 
\end{align}
Now we use \eqref{parti} with $k\rightsquigarrow h$ and $G^\mix \rightsquigarrow \wt G^\mix :=\ad(F^\mix)\ad(F^\hor)^{k-h-1}(G^\hor)$. We get 
\begin{align*}
\lceil \ad(F)^h \, \wt G^\mix \rfloor_{s',-\bd}  &
\le
 (\eta^{-1}|F|_{s})^h 
 \ \Big( \lceil \wt G^\mix  \rfloor_{s, -\bd} + k \delta^{-2} 
 |\wt G^\mix|_{s} \Big)\\
 & \stackrel{\eqref{aiutooo}}{\le} (\eta^{-1}|F|_{s})^h \ 
 \Big(C_\bd (\lceil F^\mix\rfloor_{s',-\bd} +\delta^{-2}|F|_{s})(C_0|F|_{s})^{k-h-1}|G^\hor|_s 
  \\ & 
  \qquad 
 + k \delta^{-2} (C_0|F|_{s})^{k-h}|G^\hor|_s \Big)\\
& \le (\eta^{-1}|F|_{s})^{k-1} \ 
\Big(
C_\bd \lceil F^\mix\rfloor_{s',-\bd} +(k+1)\delta^{-2}|F|_{s}
\Big)  \ |G^\hor|_s
\end{align*}

Then it follows easily that 
	\begin{align}
	\label{mix2} 
	\lceil \re{\ti}(F; G^{\rm hor}) - \re{\ti}(F^{\rm hor}; G^{\rm hor})  \rfloor_{s', -\bd} 
	\leq &  
	2 \Big(
C_\bd \lceil F^\mix\rfloor_{s',-\bd} +\delta^{-2}|F|_{s}
\Big) \ |G^\hor|_s \ 
 \frac{\left(\eta^{-1} \abs{F}_s\right)^{\ti-1}}{1 - \eta^{-1} \abs{F}_s} \ .
	\end{align}
	We define $\re{\ti}(F; G)^\mix := \re{\ti}(F; G^\mix) + (\re{\ti}(F; G^{\rm hor}) - \re{\ti}(F^{\rm hor}; G^{\rm hor}))$. Estimates \eqref{mix1}, \eqref{mix2} show that $\re{\ti}(F; G)^\mix  \in \cQ_{s,-\bd}^\cO$.

\subsection{Proof of Lemma \ref{lem:smo2}}
The lemma follows using the same strategy of the proof  of Lemma \ref{lem:smo}, but replacing the estimate of the Poisson bracket $\{F, G\}$ by the following one: 
let   $\bN=(N_1, N_2), \bM=(M_1, M_2) \in \N^2$,   $F\in\cQ^\cO_{s, -\bN}$, $G \in \cQ^\cO_{s,-\bM}$. Then 
  $\{ F, G \} \in \cQ^\cO_{s, -{\bf O}}$, where ${\bf O}= \big(\min(N_1, M_1), \min(N_2, M_2)\big)$  with the quantitative estimate 
	\begin{equation}
	\lceil \{ F, G \} \rfloor_{s, -{\bf O}}^\cO  \leq   C_0 \,  \lceil F \rfloor_{s,-\bN}^\cO \,  \lceil G \rfloor_{s,-\bM}^\cO   \ .
	\end{equation}
Such an estimate follows easily by exploiting the algebra property of the norm, see also Remark \ref{rem.m} and Remark \ref{rem:mappare}.

\section{Proof of Proposition \ref{lem:meas.mel3}}
\label{app:mes.m}
Recall that, by Theorem \ref{thm:kam}, the frequencies $\Omega_\jj(\lambda, \e)$ of  Hamiltonian \eqref{ham.bnf3} have the form \eqref{as.omega0}, \eqref{as.omega}.
Expanding $\Omega_\jj(\lambda, \e)$ in Taylor series in powers of $\e$ we get that 
\begin{equation}
\label{bb0}
\omega(\lambda) \cdot \ell + \sigma_1  \Omega_{\jj_1}(\lambda, \e) +\sigma_2 \Omega_{\jj_2}(\lambda, \e) + \sigma_3 \Omega_{\jj_3}(\lambda, \e) = 
\tK_{\bj, \ell}^\bs+ \e \  \tF_{\bj, \ell}^\bs(\lambda) + \e^2 \  \tG_{\bj, \ell}^\bs(\lambda,\e) \ , 
\end{equation}
where $\tK_{\bj, \ell}^\bs $ is defined in \eqref{K3}
and $\tF_{\bj, \ell}^\bs(\lambda) $ is defined in \eqref{F3}. 
 Recall that  given any $0<\g<\g_1$, the functions above are well defined provided that \eqref{eg} holds. 
From now on we assume \eqref{cond0} which  clearly implies  \eqref{eg}.\\
We wish to prove that the set of $\lambda\in \cO_1$ such that 
\begin{equation}
\label{eq:IIIm}
\abs{\omega \cdot \ell + \s_1\Omega_{\jj_1}(\lambda, \e)+ \s_2 \Omega_{\jj_2}(\lambda, \e)+\s_3 \Omega_{\jj_3}(\lambda, \e) } \geq \e \frac{\gamma}{\la \ell \ra^\tau} \ , \qquad \forall \, (\bj,\ell,\bs)\in \fA_3\setminus\fR_3
\end{equation}
 has positive measure. 
To do this, we will show that the set of $\lambda \in \cC_c$ such that \eqref{eq:IIIm} holds has positive measure; this clearly implies that also $\cC$ has positive measure.\\
We discuss two cases separately, recall that $\tM_0$ is defined in \eqref{M0}.
\vspace{1em}\\
\noindent{\bf Case  $|\ell| \leq 4\tM_0$.} Then $\sup_\lambda |\tF_{\bj, \ell}^\bs(\lambda)| \leq    8\, \tM_0 $.  Assume first that  $\tK_{\bj, \ell}^\bs \in \Z\setminus\{0\}$, then  for $\e$ sufficiently small one has
$$
\abs{ \eqref{bb0} } \geq |\tK_{\bj, \ell}^\bs| - \e 8 \tM_0 -  \e^2 3 \tM_0 \geq \frac{1}{2}
$$
hence, for such $\ell$'s,  \eqref{eq:IIIm} is trivially fulfilled $\forall \lambda \in \cO_1$. \\
If instead $\tK_{\bj, \ell}^\bs = 0$, 
we use  Lemma \ref{lem:cono}  and \eqref{cond0} to estimate  for any  $\lambda\in \cC_c$ 
$$
\abs{\eqref{bb0}} \geq \e \gamma_1 -   \e^23\tM_0  \geq \frac{\e \gamma_1}{2} \  . 
$$
Thus so far we have proved that for all $|\ell| \leq 4 \tM_0$ and for any $\lambda \in \cC_c$, \eqref{eq:IIIm} hold.

\vspace{2em}
\noindent{\bf Case $|\ell| > 4\tM_0$.} 
For $1 \leq i \leq \tk$, $0 \leq k \leq \tk$ define the functions
\begin{equation}
\label{hatF}
\widehat\tF_{i,k}(\lambda)  = 
\begin{cases}
\e \mu_{i}(\lambda) & \mbox{ if }  k = 0 \\
\e \mu_{i,k}^+ (\lambda) & \mbox{ if }  1 \leq i < k \leq \tk \\
\e \mu_{i,k}^- (\lambda) & \mbox{ if }  1 \leq k < i \leq \tk \\
0 & \mbox{ if}\; 1 \leq  i=k  \leq \tk
\end{cases}
\end{equation}
Consider an expression of the form 
\begin{equation}
\label{cc}
\begin{aligned}
   \omega & (\lambda) \cdot \ell + K  \\
&+\eta_1 \widehat\tF_{i_1,k_1}(\lambda) + \eta_2 \widehat\tF_{i_2,k_2}(\lambda) +  \eta_3 \widehat\tF_{i_3,k_3}(\lambda)   \\
&+ \eta_{11} \frac{\Theta_{m_1}(\lambda, \e)}{\la m_1 \ra^2} + \eta_{12} \frac{\Theta_{m_2}(\lambda, \e)}{\la m_2 \ra^2} +  \eta_{13} \frac{\Theta_{m_3}(\lambda, \e)}{\la m_3 \ra^2} \\
&+ \eta_{21} \frac{\Theta_{m_1,n_1}(\lambda, \e)}{\la m_1\ra^2 +  \la n_1 \ra^2} +   \eta_{22} \frac{\Theta_{m_2,n_2}(\lambda, \e)}{\la m_2 \ra^2 +  \la n_2 \ra^2} +   \eta_{23} \frac{\Theta_{m_3,n_3}(\lambda, \e)}{\la m_3 \ra^2+ \la n_3 \ra^2 } \\
& {+  \eta_{31} \frac{\varpi_{m_1}(\lambda, \e)}{\la m_1 \ra} + \eta_{32} \frac{\varpi_{m_2}(\lambda, \e)}{\la m_2 \ra} +  \eta_{33} \frac{\varpi_{m_3}(\lambda, \e)}{\la m_3 \ra}}
\end{aligned}
\end{equation}
where $K \in \Z$, $m_i\in \Z,n_i\in \Z\setminus \{0\}$  while $\eta_r, \eta_{r_1 r_2} \in \{ -1, 0, 1 \}$. 
First we have the following
 \begin{lemma}
 \label{c.1}
	If  $|K| \geq 4 \,   |\ell| \displaystyle{\max_{1 \leq i \leq \tk} }
	(\tm_i^2) $ then for any $m_i\in \Z,n_i\in \Z\setminus \{0\}$, $\eta_r, \eta_{r_1 r_2} \in \{ -1, 0, 1 \}$, one has  $|\eqref{cc} | \geq \tM_0$.
\end{lemma}
\begin{proof}
	We have
	\begin{align}
	\label{K.grande}
	|\eqref{cc} | & \geq |K| - |\omega(\lambda)| \, |\ell| - \sum_{r=1}^3 \abs{\widehat\tF_{i_r,k_r}}^{\cO_1}_\C - \sum_{r=1}^3 \frac{\abs{\Theta_{m_r}(\cdot, \e )}^{\cO_1}_\C}{\la m_r \ra^2}  - \sum_{r=1}^3 \frac{\abs{\Theta_{m_r, n_r}(\cdot, \e )}^{\cO_1}_\C}{\la m_r \ra^2 + \la n_r \ra^2 } \\
	\nonumber
	& \quad {- \sum_{r=1}^3 \frac{\abs{\varpi_{m_r}(\cdot, \e )}^{\cO_1}_\C}{\la m_r \ra} }\\
	\notag
	& \geq 4 \, \max_{1 \leq i \leq \tk} (\tm_i^2) \,  |\ell| - \max_{1 \leq i \leq \tk} (\tm_i^2) \,  |\ell| - \e |\ell| - 3\e \tM_0-3\e^2 \tM_0 \geq  \tM_0 \ .
	\end{align}
\end{proof}
Next we have  the following result:
\begin{lemma}
\label{c.2}
	Fix arbitrary $K \in \Z$,  $m_i\in \Z,n_i\in \Z\setminus \{0\}$, $\eta_r, \eta_{r_1 r_2} \in \{ -1, 0, 1 \}$. 
	For any $\al >0$ then 
	\begin{equation} 
	%	\label{tataki alkazoo}
	{\rm meas}(\{ \lambda \in \cO_{1}:  |\eqref{cc}| < \e\alpha \}) < 16 \alpha |\ell|^{-1} \ .
	\end{equation}
\end{lemma}
\begin{proof}
	Let $\hat\ell:=\ell/|\ell|$ and let us denote the expression in \eqref{cc} as $f(\lambda)$.  We have
	\begin{align}
	\nonumber
	\inf_{\lambda\neq\mu\in \cO_{1} \atop \lambda-\mu = |\lambda-\mu| \hat\ell} \frac{|f(\lambda) - f(\mu)|}{|\lambda - \mu|} & \geq
	\e |\ell| - \sum_{r=1}^3 \abs{\widehat\tF_{i_r,k_r}}^{\cO_{1}}_\C  - \sum_{r=1}^3 \frac{\abs{\Theta_{m_r}(\cdot, \e )}^{\cO_1}_\C}{\la m_r \ra^2}  - \sum_{r=1}^3 \frac{\abs{\Theta_{m_r, n_r}(\cdot, \e )}^{\cO_1}_\C}{\la m_r \ra^2 +  \la n_r \ra^2 } \\
	\nonumber
	& \quad {- \sum_{r=1}^3 \frac{\abs{\varpi_{m_r}(\cdot, \e )}^{\cO_1}_\C}{\la m_r \ra} }
	\\
	& \geq \e\,|\ell| -  3 \e \tM_0 - 3 \e^2 \tM_0  \geq \frac{\e}{8} \, |\ell|.
	\label{linf}
	\end{align}
	In order to prove our claim we first perform an orthogonal change of variables so that $\hat{\ell}$ becomes the first basis vector.  Formula \eqref{linf} amounts to
	$$ |f(x,\lambda_2,\dots,\lambda_n)- f(y,\lambda_2,\dots,\lambda_n)|\ge \frac{\e}{8} |x-y|$$ 
	for all $x\neq y$ such that $(x,\lambda_2,\dots,\lambda_n), (y,\lambda_2,\dots,\lambda_n)\in \cO_1$.
	
	%\marginpar{non si capisce molto: io ancora non l'ho toccata...}
	We consider  the map $F : \lambda \mapsto \lambda'=(f(\lambda),\lambda_2,\ldots,\lambda_n)$  which maps $\cO_1$ bijectively to some set $B$. $F$ is a lipeomorphism and its   inverse has Lipschitz constant $< 8\e^{-1}|\ell|^{-1}$. 
	In  the $\lambda'$ variables the volume of the set of $\mu\in B$ such that $|\mu_1|<\e \alpha$ can be estimated by $2\e \alpha $  hence on  $ \cO_1$  it can be estimated by $16\alpha |\ell|^{-1}$.
\end{proof}

We will employ such lemma to prove the following result:
\begin{lemma}
\label{c.3}
	There exist $\g_2, \tau >0$ such that for $0< \g<\g_2$, the set
	\begin{equation}
	\label{c1}
	\cC_\star= \cC_\star(\gamma, \e) := \left\lbrace
	\lambda \in \cO_1 \colon 
	\abs{\eqref{cc}} \geq \e \frac{\gamma}{\la \ell \ra^\tau}  \ , \quad \forall |\ell| \geq 4 \tM_0 
	\right\rbrace 
	\end{equation}
	has positive measure. More precisely
	$$
	\meas( \cO_1 \setminus \cC_\star ) \leq \tC_\star \gamma
	$$
	for some positive $\tC_\star$ independent of $\epsilon$.
\end{lemma}
\begin{proof}
We prove such a claim by 
finite induction on the number of $\eta_{r_1,r_2}$ different from 0.
More precisely for every $0 \leq n \leq 9$ we shall show that for $\gamma$ small enough, there exist  a positive increasing sequence $\tau_n$ and a sequence of nested sets $\cC^n = \cC^n(\gamma, \tau_n)$   such that 
provided 
$$
{|\eta_{1,1}|+ \cdots + |\eta_{3,3}| = n \ , }
$$
then 
\begin{equation}
\label{meas.cn}
\meas(\cO_1 \setminus \cC^{0}) \leq C \gamma  \ , \quad 
\meas(\cC^n \setminus \cC^{n+1}) \leq C \gamma  
\end{equation}
with some $C>0$ independent of $\e$. Moreover 
for $\lambda \in \cC^n$, $|\ell| \geq 4 \tM_0$, one has 
\begin{align}
\label{eta.ind}
\Big| \eqref{cc}  \Big| \geq \frac{\e \ \gamma}{\la \ell\ra ^{\tau_n}} \ .
\end{align}
Then Lemma \ref{c.3} follows by taking $\cC_\star := \cC^9$,  $\tau = \tau_9$ and $\tC_\star = 10 C$. 
$\gamma_2\leq \g_1$  is fixed in order to ensure that $10 C \gamma < \meas(\cO_1)$, so that the measure of $\cC_\star$ is positive.\\
 
\underline{ Case $n=0$.}
%First consider the case when also   $\eta_r =0$ $\, \forall r $. Then it is well known that there exists a set $\wt\cC_\e^0$, $\gamma, \tau_0 >0$ s.t. for any $\lambda \in \wt\cC^0_\e$
%$$
%| \omega(\lambda) \cdot \ell + K | \geq \e  \gamma \frac{\la K \ra}{\la \ell \ra^{\tau_0}} \ , \qquad \forall K \in \Z \ ;
%$$
%such estimate implies  \eqref{eta.ind} for $n=0$.\\
%
%\noindent Now consider the case where the $ \eta_{r}$'s might be non zero. 
Given  $K \in \Z$, $\bi = (i_1, i_2, i_3) \in \{1, \ldots \tk\}^3$, $\bk= (k_1, k_2, k_3) \in \{0,\ldots ,\tk\}^3$, $\ell \in \Z^\tk$ with $|\ell| \geq 4 \tM_0$, $\eta=(\eta_1,\eta_2,\eta_3) \in \{-1, 0, 1\}^3$, we define the set
$$
G_{K, \bi, \bk, \eta,\ell}^0(\gamma, \tau_0) := 
\left\{\lambda  \in \cO_1 \ : \ |\eqref{cc}| \leq \frac{\e \ \gamma}{\la \ell\ra ^{\tau_0}} \  \mbox{ and } \ \eta_{r_1 r_2} = 0 \  \ \  \forall r_1, r_2  \right\} \ . 
$$
If $|K| \geq 4 \, \displaystyle{\max_{1 \leq i \leq \tk}} (\tm_i^2) \,  |\ell| $ then by Lemma \ref{c.1} 
we have $G_{K, \bi, \bk,\eta, \ell}^0(\gamma, \tau_0) = \emptyset$, provided that $\tM_0\ge \e \g/2$. \\
If $|K|\leq  4 \, \displaystyle{\max_{1 \leq i \leq \tk} (\tm_i^2)} \,  |\ell| $ , then by Lemma \ref{c.2} with $\alpha = \g \la \ell \ra^{-\tau_0}$
we have
$$
\meas \left( G_{K, \bi, \bk, \eta,\ell}^0(\gamma, \tau_0)\right) \leq \frac{  16 \gamma}{ \, \la \ell\ra ^{\tau_0+1}} \ .
$$
Taking the union over all the possible values of $K, \bi, \bk,\eta, \ell$ one gets that %\red{non possiamo essere piu brutali nella stima del primo numero e mettere $(2\tk+1)^3$ che e' ovvio? }
$$
\meas \left ( \bigcup_{|\ell| \geq 4 \tM_0, \ \bi, \bk,\eta \atop
	|K| \leq 4 \, \max_i (\tm_i^2) \,  |\ell| }  G_{K, \bi, \bk, \eta, \ell}^0(\gamma, \tau_0) \right) \leq  C(\tk) \, \gamma \,  \sum_{|\ell| \geq 4 \tM_0} \frac{1}{   \, \la \ell\ra ^{\tau_0}} \leq C \gamma  \ ,
$$
which is finite provided $\tau_0 \geq \tk+1$.
Letting 
$$\cC^0 := \cO_1 \setminus \bigcup_{|\ell| \geq 4 \tM_0, \ \bi, \bk,\eta \atop
	|K| \leq 4 \, \max_i (\tm_i^2) \,  |\ell| }   G_{K, \bi, \bk,\eta, \ell}^0(\gamma, \tau_0)  $$
one has clearly that 
$\meas (\cO_1 \setminus \cC^0) \leq C\gamma$ and   for $\lambda \in \cC^0 $ we have
\begin{equation}
\abs{ \omega(\lambda) \cdot \ell + K  
	+ \eta_1 \widehat\tF_{i_1,k_1}(\lambda) +  \eta_2 \widehat\tF_{i_2,k_2}(\lambda) +  \eta_3 \widehat\tF_{i_3,k_3}(\lambda)  } \geq \frac{\e \ \gamma}{ \la \ell\ra ^{\tau_0}}
\end{equation}
for any choice of $\ell,K,\bi,\bk,\eta$. This proves the inductive step for $n=0$.\\

\underline{Case $n \leadsto n+1$.}  Assume that \eqref{eta.ind} holds for any possible choice of 
$\eta_{11}, \ldots, \eta_{33}$ s.t. $|\eta_{11}|+ \cdots + |\eta_{33}| \leq  n$ for some $(\tau_r)_{r = 1}^n$. 
We prove now the step $n+1$.
Suppose first that 
\begin{equation}
\label{cc.a1}
\exists \ m_i \mbox{ s.t. } |m_i| \geq  \la \ell \ra^{\tau_n} \mbox{ and  } |\eta_{1i}| + |\eta_{2i}| + |\eta_{3i}| \neq 0 \ . 
\end{equation}
W.l.o.g. assume it is $m_3$. Then  
$$
\abs{\frac{\Theta_{m_3}(\lambda, \e)}{\la m_3\ra^2 }} +  \abs{\frac{\Theta_{m_3,n_3}(\lambda, \e)}{\la m_3 \ra^2 +   \la n_3 \ra^2 }}+
\abs{\frac{\varpi_{m_3}(\lambda, \e)}{\la m_3\ra }} 
 \leq \frac{\tM_0 \, \e^2}{\la \ell \ra^{\tau_n}} 
$$
and by the inductive assumption  and \eqref{cond0}, for any $\lambda \in \cC^n$
\begin{align*}
\Big| \eqref{cc}  \Big|  & \geq 
\Big| \omega(\lambda) \cdot \ell + K  +\eta_1 \widehat\tF_{i_1,k_1}(\lambda) + \eta_2 \widehat\tF_{i_2,k_2}(\lambda) +  \eta_3 \widehat\tF_{i_3,k_3}(\lambda)  \\ 
&\qquad  + \eta_{11} \frac{\Theta_{m_1}(\lambda, \e)}{\la m_1 \ra^2} + \eta_{12} \frac{\Theta_{m_2}(\lambda, \e)}{\la m_2 \ra^2} 
+ \eta_{21} \frac{\Theta_{m_1,n_1}(\lambda, \e)}{\la m_1\ra^2 + \la n_1 \ra^2} +   \eta_{22} \frac{\Theta_{m_2,n_2}(\lambda, \e)}{\la m_2 \ra^2 + \la n_2 \ra^2}\\
& \qquad  
 + \eta_{31} \frac{\varpi_{m_1}(\lambda, \e)}{\la m_1 \ra} + \eta_{32} \frac{\varpi_{m_2}(\lambda, \e)}{\la m_2 \ra} 
\Big| 
-  \frac{\tM_0 \, \e^2}{\la \ell \ra^{\tau_n}} \\
& 
\geq  \frac{\e \ \gamma}{ \la \ell\ra ^{\tau_n}} -  \frac{\tM_0 \, \e^2}{\la \ell \ra^{\tau_n}} 
\geq \frac{\e \ \gamma}{2 \la \ell\ra ^{\tau_{n}}} \geq \frac{\e \ \gamma}{ \la \ell\ra ^{\tau_{n+1}}}
\end{align*}
provided $\tau_{n+1} \geq \tau_n +1$.\\
If \eqref{cc.a1} is not fulfilled, assume that
\begin{equation}
\label{cc.a2}
\exists \ n_i \mbox{ s.t. } |n_i|^2 \geq  \la \ell \ra^{\tau_n} \mbox{ and  } |\eta_{2i}| \neq 0 \ . 
\end{equation}
W.l.o.g. assume it is $n_3$. Then
$$
\abs{\frac{\Theta_{m_3,n_3}(\lambda, \e)}{\la m_3 \ra^2 +  \la n_3 \ra^2}} \leq \frac{\tM_0 \, \e^2}{\la \ell \ra^{\tau_n}} 
$$
and again  by the inductive assumption and \eqref{cond0}
\begin{align*}
\Big| \eqref{cc}  \Big| \geq & 
\Big| \omega(\lambda) \cdot \ell + K  +\eta_1 \widehat\tF_{i_1,k_1}(\lambda) + \eta_2 \widehat\tF_{i_2,k_2}(\lambda) +  \eta_3 \widehat\tF_{i_3,k_3}(\lambda)  \\ 
& \qquad + \eta_{11} \frac{\Theta_{m_1}(\lambda, \e)}{\la m_1 \ra^2} 
+ \eta_{12} \frac{\Theta_{m_2}(\lambda, \e)}{\la m_2 \ra^2} 
+\eta_{13} \frac{\Theta_{m_3}(\lambda, \e)}{\la m_3 \ra^2}  \\
&\qquad
+ \eta_{21} \frac{\Theta_{m_1,n_1}(\lambda, \e)}{\la m_1\ra^2+  \la n_1 \ra^2} 
+ \eta_{22} \frac{\Theta_{m_2,n_2}(\lambda, \e)}{\la m_2 \ra^2 + \la n_2 \ra^2} 	\\
& \qquad
+ \eta_{31} \frac{\varpi_{m_1}(\lambda, \e)}{\la m_1 \ra} 
+ \eta_{32} \frac{\varpi_{m_2}(\lambda, \e)}{\la m_2 \ra} 
+\eta_{33} \frac{\varpi_{m_3}(\lambda, \e)}{\la m_3 \ra}  
\Big|  
-  \frac{\tM_0 \, \e^2}{\la \ell \ra^{\tau_n}} \\
\geq & \frac{\e \ \gamma}{2 \la \ell\ra ^{\tau_{n}}} \geq \frac{\e \ \gamma}{\la \ell\ra ^{\tau_{n+1}}}
\end{align*}
provided $\tau_{n+1} \geq \tau_n +1$. 
If \eqref{cc.a1} and \eqref{cc.a2} are not fulfilled, then one has $|m_i|\, , |n_i|^2 \leq \la \ell \ra^{\tau_n}$ for all the $m_i, n_i$ that appear non-trivially in \eqref{cc}. Furthermore, we can safely assume that $|K| \leq 4 \, \max_i \tm_i^2 \, | \ell | $, otherwise Lemma \ref{c.1}  ensures that \eqref{eta.ind} is met for any $\lambda 
\in \cO_1$. Thus we are left with a finite number of cases and we can impose a finite number of Melnikov conditions.

For $K \in \Z$,  $\bi = (i_1, i_2, i_3) \in \{1,\ldots, \tk\}^3$, $\bk= (k_1, k_2, k_3) \in \{0,\ldots,  \tk\}^3$, $\ell \in \Z^\tk$ with $|\ell| \geq 4\tM_0$,  $\eta=(\eta_1,\eta_2,\eta_3) \in \{-1, 0, 1\}^3$,  $\bm= (m_1, m_2, m_3)$, $\bn= (n_1, n_2, n_3)$ define the set
$$
G^{n+1}_{K,  \bi, \bk,  \eta, \ell, \bm, \bn}(\gamma, \tau_{n+1}) := \left\{\lambda  \in \cC^n \ : \ |\eqref{cc}| \leq \frac{\e \ \gamma}{ \la \ell\ra ^{\tau_{n+1}}} \ , \qquad  |\eta_{11}|+ \cdots + |\eta_{33}| =   n+1 \right\} \ . 
$$
By Lemma \ref{c.2} with $\alpha = \g / \la \ell \ra^{\tau_{n+1}}$ we have
$$
\meas \left( G^{n+1}_{K,  \bi, \bk,  \eta, \ell, \bm, \bn}(\gamma, \tau_{n+1}) \right) \leq \frac{ 16\gamma}{ \, \la \ell\ra ^{\tau_{n+1}+1}} \ ,
$$
and taking the union over the possible values of $K,  \bi, \bk,  \eta, \ell, \bm, \bn$ one gets that
$$
\meas \Big( \bigcup_{|\ell| \geq 4 \tM_0 \atop  \bi, \bk, \eta} \bigcup_{|m_i| \, ,\  |n_i|^2 \leq \la \ell \ra^{\tau_n}  \atop
	|K| \leq 4 \, \max_i (\tm_i^2) \,  |\ell| } G_{ K,  \bi, \bk,  \eta, \ell, \bm, \bn}(\gamma, \tau_{n+1}) \Big) \leq C(\tk) \, \gamma \sum_{|\ell| \geq 4 \tM_0} \frac{ \la \ell \ra^{1+ 9 \tau_n/2}}{ \, \la \ell\ra ^{\tau_{n+1}+1}} \ ,
$$
which converges provided $\tau_{n+1} \geq \tk+1 +9 \tau_n/2$.
Thus we define the set
$$
\cC^{n+1} := \cC^n\setminus \bigcup_{|\ell| \geq 4 \tM_0 \atop  \bi, \bk, \eta} \bigcup_{|m_i|, |n_i|^2 \leq \la \ell \ra^{\tau_n}  \atop
	|K| \leq 4 \, \max_i (\tm_i^2) \,  |\ell| }  G^{n+1}_{ K,  \bi, \bk,  \eta, \ell, \bm, \bn}(\gamma, \tau_{n+1}) 
$$
which fulfills \eqref{meas.cn} and \eqref{eta.ind}.  
\end{proof}

We conclude the section with the proof of Proposition \ref{lem:meas.mel3}.

\begin{proof}[Proof of Proposition \ref{lem:meas.mel3}]
Fix $\tau$ as in Lemma \ref{c.3}. By definition for $\gamma < \gamma_2$, any $\lambda \in \cC_c \cap \cC_\star$ fulfills 
\eqref{eq:IIIm} for any admissible  $(\bj, \ell, \bs)$. 
By taking $\gamma$ sufficiently small, we can ensure that $\meas(\cC_c \cap \cC_\star)> \meas(\cC_c)/2$.
This fixes $\gamma_0$.
\end{proof}

\section{A crash course on  polynomial rings and algebraic extensions}
\label{app:pol.ring}
The aim of this section is to recall the reader \footnote{and the authors!} some basic algebraic properties of polynomial rings and algebraic extensions which are used repeatedly in the paper.

\paragraph*{Basic properties of polynomial rings.}
 Let $R$ be a commutative ring with unit element. We denote by $R[\lambda]$ the  polynomial ring in $\lambda$ over $R$, which 
is the set of formal polynomials in $\lambda$ with coefficients in $R$. Its elements are of the form $\sum_j a_j \lambda^j$.  $R[\lambda]$ is a commutative ring with unit element.
Similarly one  defines the ring of polynomials in the $n$-variables $\lambda_1, \ldots, \lambda_n$ over $R$, denoted by $R[\lambda_1,\ldots , \lambda_n]$, whose elements are of the form 
$\sum a_{i_1 i_2 \ldots i_n}\lambda_1^{i_1} \lambda_2^{i_2} \ldots \lambda_n^{i_n}$.\\

The following result  is very well known:
\begin{lemma}
\label{app:pol.id}
If $R$ is an integral domain\footnote{An integral domain is a commutative ring with an identity $1 \neq 0$ with no zero-divisors. That is $ab = 0 \Rightarrow  a = 0$ or $b = 0$. The ring $\Z$ is an integral domain. 
}, then so is $R[\lambda_1, \ldots, \lambda_n]$. 
\end{lemma}

Let  $R$ be  an integral domain.  The characteristic of $R$  is the smallest positive integer such
that $n\cdot 1 = 0$, where $n\cdot 1$ stands for $1 + 1 + \ldots +1$ $n-times$.  In case $n\cdot 1$ is never 0, then $R$ has  characteristic 0.   

If $R$ is an integral domain it is possible to construct its field of quotients $F$ and $R$ has characteristic 0 if and only if $F$ contains the field $\mathbb Q$ of rational numbers.

In as integral domain   we can speak about {\em irreducible elements}. More precisely  an element $a \in R$, $a\neq 1$ will be called {\em irreducible} if $a= bc$ with $b, c \in R$ implies that one of $b$ or
$c$ must be invertible in $R$, that is {\em a unit}.  Two elements $b,c$ are {\em associated} if $b=uc$ with $u$ invertible.

\begin{definition}
 An integral domain $R$ is a unique factorization domain (UFD) if
\begin{itemize}
\item[(a)] Every nonzero element in $R$ is either a unit or can be written as the product of a finite
number of irreducible elements of $R$.
\item[(b)] The decomposition in part $(a)$ is unique up to the order and associates of the irreducible
elements.
\end{itemize}
\end{definition}  The ring $\Z$ is the simplest UFD.
A main theorem is that if $R$ is a UFD so is $R[\lambda]$ thus:

\begin{example}
The ring $\Z[\lambda_1, \cdots, \lambda_\tk]$ of polynomials in the variables $\lambda_1, \cdots, \lambda_\tk$ with coefficients in $\Z$ is a UFD (its units are $\pm 1$).
\end{example}
%A crucial property of UFD is that the existence of the greatest common divisor (GCD) of two elements is assured:
%\begin{lemma} 
%If $R$ is a unique factorization domain and if $a, b$ are in $R$, then $a$ and $b$ have a
%greatest common divisor $(a, b)$ in $R$.
%\end{lemma}

Given a UFD $R$ let  $F$ be its field of quotients  We can then consider $R[\lambda]$ to be a sub-ring of $F[\lambda]$. In particular given any polynomial
$f(\lambda) \in F[\lambda]$, then $f(\lambda) = (f_0(\lambda)/a)$, where $f_0(\lambda)\in R[\lambda]$ and  $a\in R$.  It is
natural to ask if a polynomial irreducible in $R[\lambda]$ is still irreducible when it is considered as a polynomial in the larger ring $F[\lambda]$.
\begin{lemma}
If $f(\lambda) \in R[\lambda]$ is irreducible as an element of $R[\lambda]$, then it
is irreducible as an element of $F[\lambda]$. Conversely, if   $f(\lambda) \in R[\lambda]$ is  both primitive\footnote{$f(\lambda) \in R[\lambda]$  is said to be primitive if the GCD of its coefficients is 1. Every polynomial $f(\lambda) \in R[\lambda]$ can be decomposed as the product of a primitive polynomial and the GCD of its coefficients, such decomposition being unique.} and irreducible as
an element of $F[\lambda]$, it is also irreducible as an element of $R[\lambda]$.
\end{lemma}
We want to apply this to $R= \Z[\lambda_1, \cdots, \lambda_\tk]$  with field of quotients $\Q(\lambda_1, \cdots, \lambda_\tk) $
the rational functions with rational coefficients.

\begin{remark}
We identify  $\Z[t, \lambda_1, \cdots, \lambda_\tk]$ with the ring
$\Z[\lambda_1, \cdots, \lambda_\tk][t]$, namely with the ring of polynomial in the variable $t$ with coefficients in the ring 
$\Z[\lambda_1, \cdots, \lambda_\tk]$.
\end{remark}

Next we recall some basic facts about extension fields  that is $F\subset E$  two fields.
We begin with the following definition:
\begin{definition}
If $E$ is a field extension of $F$, then an element $a$ of $E$ is called an algebraic element over $F$ if there exists some non-zero polynomial $f(t) \in F[t]$ such that $f(a)=0$.
\end{definition}
The set of elements of $E$ which are algebraic over $F$  form also a field that is are closed under sum difference multiplication  and inverse.   There is also an extension $\bar F$  algebraic over $F$ and {\em algebraically closed}  that is the only irreducible polynomials are the linear ones, thus every polynomial $f(t)\in F[t]$ of degree $n$ factors in $\bar F$ through its $n$ roots (with possible multiplicities).  So when we speak of the roots of a polynomial $f(t)\in F[t]$ we think of elements of $\bar F$.

Given $a \in E$ an  algebraic element  over $F$, we define the set 
$$
I_a := \left\lbrace f(t) \in F[t] \colon f(a) = 0 \right\rbrace \ . 
$$ 
This is an ideal of $ F[t] $.
\begin{lemma}
$F[t]$ is a {\em principal ideal domain}, namely every ideal of $ F[t] $ is  generated by a unique  monic polynomial $p_a \in F[t]$, which is called {\em the minimal polynomial} of $a$:
$$
I_a  \equiv  \left\lbrace f(t)\, p_a(t)   \colon f(t)  \in F[t] \right\rbrace  \ . 
$$
Clearly the minimal polynomial is the monic polynomial of least degree in $I_a$, and it is irreducible over $F$.
\end{lemma}

Let  $f(t) \in F[t]$ be  irreducible, then the lemma above implies that $f$ is the minimal polynomial of every root of $f$ in $E$.

%\paragraph{Useful facts.}
%Let $p(t)=a_{0}t^{n}+a_{1}t^{n-1}+\cdots +a_{n}$ and $q(t) = b_{0}t^{m}+b_{1}t^{m-1}+\cdots +b_{m}$ be two polynomials in $ F[t]$. Then the {\bf resultant} of $p$ and $q$, $R(p,q)$ is the determinant of the matrix
% $$
% \begin{vmatrix}a_{0}&0&\cdots &0&b_{0}&0&\cdots &0\\a_{1}&a_{0}&\cdots &0&b_{1}&b_{0}&\cdots &0\\a_{2}&a_{1}&\ddots &0&b_{2}&b_{1}&\ddots &0\\\vdots &\vdots &\ddots &a_{0}&\vdots &\vdots &\ddots &b_{0}\\\vdots &\vdots &\cdots &a_{1}&\vdots &\vdots &\cdots &b_{1}\\a_{n}&a_{n-1}&\cdots &\vdots &b_{m}&b_{m-1}&\cdots &\vdots \\0&a_{n}&\ddots &\vdots &0&b_{m}&\ddots &\vdots \\\vdots &\vdots &\ddots &a_{n-1}&\vdots &\vdots &\ddots &b_{m-1}\\0&0&\cdots &a_{n}&0&0&\cdots &b_{m}\end{vmatrix}
%$$
%The key property of the resultant $R(p,q)$ is that, if $a_0b_0\neq 0$ then it is zero if and only if $p$ and $q$ have a common root (in  the algebraic closure  $\bar F$ of $F$).
%%Clearly if $a\in E$ is algebraic, then $-a$ is also algebraic, being the root of any polynomial $p(-t)$ fulfilling $p(a) = 0$. Even more is true:
%%\begin{lemma}
%%\label{lem:sum.al}
%%Assume that $a,b $ are algebraic elements of $E$ over $F$. Then $a+b$ is   algebraic.
%%\end{lemma}
%%\begin{proof}
%%Assume that $p(t), q(t)  \in F[t]$  fulfill $p(a) = q(b) = 0$.  Consider the resultant $r(z) = R(p(\cdot), q(z-\cdot))$  of the polynomials $p(t)$ and $q(z-t)$, which is a  polynomial in $z$. Then $r(a+b) = R(p(\cdot), q(a+b - \cdot)) = 0$, since $p(t)$ and $q(a+b-t)$ have $a$ as common root.
%%\end{proof}

\begin{lemma}
Let $F$ be a field  with characteristic $0$. An irreducible  polynomial $f(t) \in F[t]$ has no multiple roots.
\end{lemma}
\begin{proof}
Recall that to any  polynomial 
 $f(t) = \sum_{j \geq 0} a_j t^j$,  we can associate its formal derivative $f'(t) =  \sum_{j \geq 0} j a_{j}  \, t^{j-1}$.
 If $\deg f = 1$, the statement is trivial, thus let $\deg f> 1$. Assume that $f(t)$ has multiple roots. 
 Then  $f$ and $f'$ have a common root $a \in \bar F$. 
 Since $f$ is the minimal polynomial of $a$, then $f / f'$. 
 But since $\deg f' < \deg f$, then $f' = 0$, hence $\deg f = 1$, contradiction.
\end{proof}

\begin{lemma}
\label{lem:tra}
Let  $f(t) \in F[t]$ be irreducible. For any $a \in F$,  $f(t+a)$ is irreducible.
\end{lemma}
\begin{proof}
Assume that $f(t+a) = g_1(t) g_2(t)$ with $g_1, g_2$ not trivial polynomials. Then $f(t) = g_1(t-a) g_2(t-a)$, and $g_1(t-a), g_2(t-a) \in F[t]$, getting a contradiction.
\end{proof}

\newcommand{\etalchar}[1]{$^{#1}$}
\def\cprime{$'$}

\end{document}